\documentclass[10pt]{amsart}
\usepackage{txfonts}
\usepackage[T1]{fontenc}
\usepackage{fancyhdr}

\usepackage{amsmath,amsthm,amssymb,hyperref, mdframed}
\usepackage{mwe}
\usepackage{mathrsfs} 
\usepackage{tikz}
\usepackage[all]{xy}
\usetikzlibrary{decorations.markings}
\usetikzlibrary{backgrounds,shapes}
\usepackage{subfig}
\usepackage{tikz}
\usetikzlibrary{shapes,arrows}
\usetikzlibrary{shapes.geometric, arrows}
\usetikzlibrary{decorations.markings}
\usepackage{subfig}
\usepackage{tabularx}
\usepackage{enumitem}
\usepackage{amsmath,amsthm,amscd}
\usepackage{amssymb,amsfonts}
\usepackage{fancyhdr} 
\usepackage{pgf,tikz}
\usepackage{caption}
\usepackage{tikz-cd}
\usepackage{tikzscale}
\usetikzlibrary{patterns}
\usepackage{rotating}
\usepackage{xcolor}
\usepackage{lscape}
\usepackage{array} 
\usepackage{tabularx}
\usepackage{multirow}
\usepackage{multicol}
\usepackage{booktabs} 
\usepackage{caption}
\usepackage{xspace}
\usepackage{xfrac}
\usepackage{blindtext}
\usepackage[ruled,vlined]{algorithm2e}

\usetikzlibrary{decorations.markings}
\usetikzlibrary{patterns}
\usetikzlibrary{calc}
\usetikzlibrary{shapes.geometric}

\usepackage{color}
\usepackage{hyperref}
\hypersetup{
    colorlinks=true, 
    linktoc=all,     
    linkcolor=blue,  
    citecolor=blue,
    filecolor=blue,
    urlcolor=blue
}

\definecolor{aliceblue}{rgb}{0.9, 0.95, 1.0}

\usepackage{geometry}
\numberwithin{equation}{section}

\newcommand\Z{{\mathbb Z}}

\newcommand\R{{\mathbb R}}
\newcommand\Q{{\mathbb Q}}

\newcommand{\C}{{\mathbb C}}

\newcommand{\sorr}{\mathrm{SO}(2,\R)}

\newcommand{\spz}{\mathrm{Sp}(2g,\Z)}
\newcommand{\glplus}{\mathrm{GL}^+(2,\mathbb{R})}

\newcommand{\cp}{\mathbb{C}\mathrm{\mathbf{P}}^1}
\newcommand{\modul}{\textnormal{Mod}(S_{g,n})}

\newcommand{\cext}{\mathcal{C}_{\textnormal{ext}}}

\newcommand{\homolz}{\text{H}_1(X,\mathbb{Z})}
\newcommand{\homolzz}{\textnormal{H}_1(X,\mathbb{Z})}

\newcommand{\shomolzn}{\textnormal{H}_1(S_{g,n},\mathbb{Z})}
\newcommand{\shomolzo}{\textnormal{H}_1(S_{g,1},\mathbb{Z})}

\newcommand{\shomolzon}{\textnormal{H}_1(S_{1,\,n},\mathbb{Z})}
\newcommand{\shomolzoo}{\textnormal{H}_1(S_{1,1},\mathbb{Z})}
\newcommand{\shomolznps}{\textnormal{H}_1(S_{0,n},\mathbb{Z})}
\newcommand{\shomolz}{\textnormal{H}_1(S_{g},\mathbb{Z})}

\theoremstyle{plain}                    
\newtheorem{thm}{Theorem}[section]

\newtheorem{thma}{Theorem}

\newtheorem{cora}[thma]{Corollary}

\newtheorem{lem}[thm]{Lemma}
\newtheorem{prop}[thm]{Proposition}
\newtheorem{cor}[thm]{Corollary}
\newtheorem*{thmnn}{Theorem}

\theoremstyle{definition}
\newtheorem{defn}[thm]{Definition}

\newtheorem{ex}[thm]{Example}

\newtheorem{rmk}[thm]{Remark}

\newcommand*{\defeq}{\mathrel{\vcenter{\baselineskip0.5ex \lineskiplimit0pt
                     \hbox{\scriptsize.}\hbox{\scriptsize.}}}%
                     =}

\tikzstyle{rb} = [rectangle, rounded corners, minimum width=3cm, minimum height=1cm, text width=3cm, text centered, draw=black, fill=blue!30]
\tikzstyle{sb} = [rectangle, minimum width=2cm, minimum height=1cm, text width=3cm, text centered, draw=black, fill=violet!30]

\title[Period realization of meromorphic differentials with prescribed invariants]{Period realization of meromorphic differentials with prescribed invariants}

\author{Dawei Chen}
\address[Dawei Chen]{Department of Mathematics, Boston College, Chestnut Hill, MA 02467, USA}
\email{dawei.chen@bc.edu}

\author{Gianluca Faraco}
\address[Gianluca Faraco]{Dipartimento di Matematica e Applicazioni U5, Universita` degli Studi di Milano-Bicocca, Via Cozzi 55, 20125 Milano, Italy}
\email{gianluca.faraco@unimib.it}
\date{\today}

\begin{document}

\subjclass[]{}%
\dedicatory{}

\begin{abstract}
We provide a complete description of realizable period representations for meromorphic differentials on Riemann surfaces with prescribed orders of zeros and poles, hyperelliptic structure, and spin parity. 
\end{abstract}

\maketitle
\tableofcontents

\section{Introduction} 

\noindent Let $S_{g,n}$ be the connected and oriented surface of genus $g$ and with $n$ punctures, and let $\mathcal{M}_{g,n}$ be the moduli space of unmarked complex structures on $S_{g,n}$. For a complex structure $X\in\mathcal{M}_{g,n}$, let $\Omega(X)$ denote the space of holomorphic abelian differentials with at most finite-order poles at the punctures that we refer as \textit{meromorphic differentials} on $X$. Let $\Omega\mathcal{M}_{g,n}$ denote the moduli space of pairs $(X,\omega)$, where $X$ is a punctured Riemann surface and $\omega\in\Omega(X)$ is an abelian differential on $X$. Such a moduli space admits a natural stratification, given by the strata $\mathcal{H}_g(m_1,\dots,m_k; -p_1,\dots,-p_n)$, enumerated by unordered partitions of $2g-2$ where we allow negative integers corresponding to the orders of poles. A stratum $\mathcal{H}_g(m_1,\dots,m_k; -p_1,\dots,-p_n)$ consists of equivalent classes of pairs $(X,\omega)$ where $\omega$ has $k$ zeros and $n$ poles of orders $m_1,\dots, m_k, -p_1,\dots,-p_n$ respectively. A stratum of differentials can be disconnected and in fact it can have at most three connected components according to some additional properties like the presence of a \textit{hyperelliptic involution} or a topological invariant known as \textit{spin structure}. For holomorphic differentials on compact Riemann surfaces the connected components of the strata have been classified by Kontsevich-Zorich in \cite{KZ} and subsequently by Boissy in \cite{BC} in the case of meromorphic differentials. 

\medskip 

\noindent The \textit{period character} of an abelian differential $\omega$ on a compact Riemann surface $X$ is defined as the homomorphism
\begin{equation}
    \chi\colon\homolz\longrightarrow \mathbb{C} \,\, \text{ such that } \,\, \gamma\longmapsto\int_\gamma \omega,
\end{equation}
\noindent and the \textit{period map} is the association 
\begin{equation}\label{permap}
    \text{Per}\colon\Omega\mathcal{M}_{g,n}\longrightarrow \text{Hom}\Big(\shomolzn,\,\mathbb{C}\,\Big) 
\end{equation}
mapping an abelian differential $\omega$ to its character. For holomorphic differentials on compact Riemann surfaces (\textit{i.e.} $n=0$) the image of the period map has originally been studied by Haupt in \cite{OH} and subsequently rediscovered by Kapovich in \cite{KM2} by using Ratner's theory. It turns out that there are two obstructions for a representation $\chi$ to be realized as the period character of an abelian differential which we shall refer in the following as \textit{Haupt's conditions}, see Section \S\ref{agv}. Realizing a representation as a character in a prescribed stratum turns out to be a more subtle problem because, in the realizing process, the orders of zeros can no longer be ignored. In the same spirit of \cite{KM2}, Le Fils has provided in \cite{fils} necessary and sufficient conditions for a representation $\chi$ to be a character in a given stratum. Around the same time, Bainbridge-Johnson-Judge-Park in \cite{BJJP} have provided, with an independent and alternative approach, necessary and sufficient conditions for a representation to be realized in a connected component of a prescribed stratum. 

\smallskip

\noindent For meromorphic differentials on compact Riemann surfaces, equivalently holomorphic differentials on punctured complex curves (\textit{i.e.} $n\ge1$), the image of the period mapping has been recently determined by Chenakkod-Faraco-Gupta in \cite[Theorem A]{CFG}. In this case no obstructions appear for realizing a representation $\chi$ as the period character of some meromorphic differential, \textit{i.e.} the period map \eqref{permap} turns out to be surjective. Moreover, they have described necessary and sufficient conditions for realizing $\chi$ as the period of some $(X,\omega)$ in a prescribed stratum $\mathcal{H}_g(m_1,\dots,m_k;-p_1,\dots,-p_n)$, see \cite[Theorems B, C, D]{CFG}. Sometimes, when necessary, we shall adopt “exponential” notation to denote multiple zeros or poles, \textit{e.g.} $\mathcal H_1(2,2;\,-2,-2)=\mathcal H_1(2^2;\,-2^2)$.

\smallskip 

\noindent The aim of the present paper is to extend the above study of the period map \eqref{permap} to connected components of the strata of meromorphic differentials whenever they are not connected. Our first result states as follows.

\begin{thma}\label{mainthm}
Let $\chi\colon\homolzz\longrightarrow\mathbb C$ be a non-trivial representation arising as the period character of some meromorphic genus-$g$ differential in a stratum $\mathcal{H}_g(m_1,\dots,m_k;-p_1,\dots,-p_n)$. Then $\chi$ can be realized in each of its connected components. 
\end{thma}

\noindent The trivial representation is the homomorphism $\chi\colon\homolzz\longrightarrow \mathbb C$ such that $\chi(\gamma)=0$ for all $\gamma\in \homolzz$. Our second result handles this special representation which is exceptional as follows.

\begin{thma}\label{thm:mainthm2}
Suppose the trivial representation arises as the period character of some genus-$g$ meromorphic differential in a stratum $\mathcal{H}_g(m_1,\dots,m_k;-p_1,\dots,-p_n)$. Then $\chi$ can be realized in each of its connected components with the only exceptions being the strata:
\begin{itemize}
    \item[1.] $\mathcal{H}_1(3,3;\,-3,-3)$,
    \smallskip
    \item[2.] $\mathcal H_g(2^{g+2};\,-2^3)\,$, for $\,g\ge1$,
    \smallskip
    \item[3.] $\mathcal H_g(2^{g+1};\,-4)\,$, for $\,g\ge1$.
\end{itemize}
\noindent Moreover in these exceptional strata, for $g=1$ these strata exhibit two connected components, one of which is primitive and the other is not. The trivial representation can only be realized in the non-primitive component. For $g\ge2$, the strata $\mathcal H_g(2^{g+2};\,-2^3)$ and $\mathcal H_g(2^{g+1};\,-4)$ have two connected components distinguished by the spin parity. The trivial representation can only be realized in the connected component with parity equal to $g\,(\textnormal{mod}\,2)$.
\end{thma}

\noindent Besides strata, there is another type of subspaces in $\Omega\mathcal{M}_{g,n}$ which is worth studying, called {\em isoperiodic leaves}, defined as the (non-empty) fibres of the period map in \eqref{permap}. As an immediate corollary of our main Theorems \ref{mainthm} and \ref{thm:mainthm2}, we obtain the following result. 

\begin{cora}
For a non-trivial representation its isoperiodic leaf $\mathcal{L}$ intersects each connected component of each stratum of meromorphic differentials. The same statement holds for the trivial representation except for the strata in Theorem~\ref{thm:mainthm2}. 
\end{cora}

\noindent It is interesting to study under which conditions an isoperiodic leaf is connected. For holomorphic differentials it is known that the isoperiodic leaves are generically connected up to a few exceptions, see \cite[Theorems 1.2 and 1.3]{CDF2}. For meromorphic differentials with two simple poles an analogous result has recently been obtained in \cite{CD}. It remains to study connected components of isoperiodic leaves for meromorphic differentials with general pole orders. An ongoing collaboration of the second named author with Guillaume Tahar and Yongquan Zhang aims to study the isoperiodic foliations for $\mathcal H_1(1,1;\,-2)$, see \cite{FTZ}. Another direction is to study period realization of {\em $k$-differentials} for $k > 1$, \textit{e.g.} $k=2$ is the case of quadratic differentials {\em i.e.} half-translation surfaces. We plan to treat these questions in future work.  

\subsection{Structure of the paper and strategy of the proof} This paper is organized as follows. In Section \S\ref{sec:tswp} we review basic concepts about differentials and translation surfaces, their strata with prescribed orders of zeros and poles, and the geometric invariants that can distinguish connected components of the strata. In Section \S\ref{sec:surgeries} we introduce several surgery operations that can be used to construct translation surfaces of higher genera and study how the concerned invariants change under these operations. In Section \S\ref{sec:mcga} we discuss a system of handle generators for the period domain in order to carry out inductive constructions. In Sections \S\ref{genusonemero}, we first prove Theorem \ref{mainthm} for surfaces of genus one. More precisely, for a non-trivial representation $\chi$ we provide a direct construction to realize a meromorphic genus-one differentials with period character $\chi$ and prescribed invariants. The proof involves a case-by-case discussion according to Table \ref{tab:indexes}, see Appendix \ref{appfdgo}. In Section \S\ref{sec:hgchyp}, we consider higher genus surfaces, \textit{i.e.} we shall suppose $g\ge2$. This section is entirely devoted to realize a representation as the period character of some hyperelliptic translation surface with poles and prescribed singularities. Next, in Section \S\ref{sec:hgcpar} we still consider higher genus surfaces and we aim to realize a given non-trivial representation as the period character of some translation structure with spin parity. The proof is based on an induction foundation having genus-one differentials as the base case. Therefore this section is mainly devoted to show how to run the inductive process. We finally consider the trivial representation on its own right. In Section \S\ref{sec:trirep} we shall prove Theorem \ref{thm:mainthm2} for the cases of genus one, hyperelliptic surfaces, surfaces with spin parity, and the trivial representation (which corresponds to exact differentials), respectively. Along the way we shall encounter and classify some exceptional strata in which the trivial representation cannot be realized for certain connected components. Finally in Appendix~\S\ref{appfdgo} we provide a flowchart to illustrate the relations of the various cases in the course of the proof. 

\smallskip

\noindent We remark that, comparing to \cite{CFG}, a series of new challenges that require distinct techniques appear for dealing with connected components of strata. Let us consider non-trivial representations in the first place. As in the holomorphic case, to realize a given representation $\chi$ as the period character of some meromorphic differential in a given stratum is a subtle problem because the order of singularities cannot be ignored. For a generic representation $\chi$, \textit{i.e.} $\textnormal{Im}(\chi)\not\subseteq \Z$, the problem essentially reduces to determine a suitable collection of (possibly unbounded) polygons that, once glued together along slits in a proper way, provide a translation surface in a stratum with a single zero of maximal order and prescribed orders of poles. Then, by breaking the sole zero into zeros of lower order, realizing a generic representation in a stratum with several zeros becomes a straightforward consequence of the former construction. Some obstructions appear in realizing a non-trivial \textit{integral representations} ($\textnormal{Im}(\chi)=\Z$) in a given stratum with \textit{all} simple poles -- these obstructions completely vanish in realizing an integral representation in a stratum with at least one pole of higher order, see \cite[Theorems D]{CFG}. However, whenever these obstructions are satisfied, the realization boils down once again to find an appropriate collection of polygons to glue together that provide the desired structure. For dealing with connected components, the gist of the idea is still the same but now the way these polygons are glued together also does matter. This is particularly evident in Section \S\ref{genusonemero} where, in order to realize a genus-one differential with prescribed rotation number, several copies of the standard differential $(\C,\,dz)$ need to be glued together by following a certain pattern. Different ways of gluing provide structures in the same stratum but lying in different connected components. This motivates the several constructions summarized in Appendix \ref{appfdgo}. A similar phenomenon can be seen in Section \S\ref{sec:hgchyp} in which we aim to realize a hyperelliptic translation surface with prescribed period character in a given stratum. Although realizing genus-one differential requires a deeper and more detailed argument, realizing higher genus surfaces with prescribed parity simplifies considerably due to the inductive foundation we have already alluded above, see Section \S\ref{sec:hgcpar}.

\smallskip

\noindent Let us finally discuss the trivial representation. Even the trivial representation can be realized in a stratum under certain necessary condition and, in this case, the realization relies on an inductive foundation on the genus. In other words, in \cite{CFG} the trivial representation is realized starting from a genus-zero differential with trivial periods and bubbling handle with trivial periods in order to obtain an exact differential in the desired stratum. To realize an exact differential in a given stratum with prescribed rotation number (if $g=1$), parity or hyperelliptic involution (if $g\ge2$) turns out to be an arduous challenge and not always possible. In fact, although a non-trivial representation can be realized in every connected component of every stratum in which it appears, the trivial representation may appear in a stratum without being realizable in each connected component of the same stratum, see Theorem \ref{thm:mainthm2}. The case of higher genus surfaces with prescribed parity relies on the realization of genus-one differentials with prescribed rotation number in a certain stratum. In other words, also in this case we use an inductive foundation having genus-one exact differential as the base case. Despite a geometric argument works for realizing an exact differential with prescribed invariants, it fails to show that the trivial representation cannot be realized in certain connected components of certain strata. In order to prove the second part of Theorem \ref{thm:mainthm2} we shall need to use an argument from the viewpoint of algebraic geometry. In fact, an exact differential $(X,\omega)$ yields a rational map $f\colon X\longrightarrow \cp$ ({\em i.e.} a branched cover of the sphere) and we shall use a monodromy argument associated to the cover to determine whether the trivial representation can be realized in certain connected components of the exceptional strata listed above.

\subsection{Acknowledgments}
This paper was initiated during the Richmond Geometry Festival in Summer 2021. We thank the conference organizers Marco Aldi, Allison Moore, and Nicola Tarasca for providing a wonderful platform of online communications. We also thank Quentin Gendron, Subhojoy Gupta, Guillaume Tahar, and Yongquan Zhang for helpful discussions on related topics. The first named author is partially supported by National Science Foundation Grant 2001040 and Simons Foundation Collaboration Grant 635235. The second named author is partially supported by GNSAGA of INDAM. When this project started, the second named author was affiliated with the MPIM Bonn and subsequently with the University of Bonn and is grateful to both institutions for the pleasant and stimulating environment in which he worked. Finally, he is indebted to Ursula Hamenst\"adt for the invaluable support he has received from her and which has helped him bring this current partnership to fruition.

\section{Translation surfaces with poles and meromorphic differentials}\label{sec:tswp}

\noindent We begin by recalling the notion of translation structure on a topological surface by providing a geometric and a complex-analytic perspective. 

\medskip

\noindent A \textit{translation structure} on a surface $S_{g,n}$ is a branched $(\C,\C)$-structure, \textit{i.e.} the datum of a maximal atlas where local charts in $\C$ have the form  $z\longmapsto z^k$, for $k\ge1$, and transition maps given by translations on their overlappings. Any such an atlas defines an underlying complex structure $X$ and the pullbacks of the $1$-form $dz$ on $\C$ via local charts globalize to a holomorphic differential $\omega$ on $X$. Vice versa, a holomorphic differential $\omega$ on a complex structure $X$ defines a singular Euclidean metric with isolated singularities corresponding to the  zeros of $\omega$. In a neighborhood of a point $P$ which is not a zero of $\omega$, a local coordinate is defined as
\begin{equation}
    z(Q)=\int_P^Q \omega 
\end{equation} 
in which  $\omega=dz$, and the coordinates of two overlapping neighborhoods differ by a translation $z\mapsto z+c$ for some $c\in\mathbb C$. Around a zero, say $P$ of order $k\ge1$, there exists a local coordinate $z$ such that $\omega=z^kdz$. The point $P$ is also called a \textit{branch point} because any local chart around it is locally a branched $k+1$ covering over $\C$ which is totally ramified at $P$.

\begin{defn}[Translation surfaces with poles]\label{tswp}
Let $\omega$ be a meromorphic differential on a compact Riemann surface $\overline{X}\in\mathcal{M}_g$. We define a \emph{translation surface with poles} to be the structure induced by $\omega$ on the surface $X = \overline{X}\setminus\Sigma$, where $\Sigma$ is the set of poles of $\omega$.
\end{defn}

\noindent Given a translation structure $(X,\omega)$ on a surface $S_{g,n}$, the local charts globalize to a holomorphic mapping $\text{dev}\colon\widetilde{S}_{g,n}\longrightarrow \C$ called the \textit{developing map}, where $\widetilde{S}_{g,n}$ is the universal cover of $S_{g,n}$. The translation structure on $S_{g,n}$ lifts to a translation structure $(\widetilde{X},\widetilde{\omega})$ and the developing map turns out to be locally univalent away from the zeros of $\widetilde{\omega}$. The developing map, in particular, satisfies an equivariant property with respect to a representation $\chi\colon\homolz\longrightarrow\C$ called \textit{holonomy} of the translation structure. The following Lemma establishes the relation between holonomy representations and period characters.

\begin{lem}
A representation $\chi\colon\shomolzn\longrightarrow \C$ is the period of some abelian differential $\omega\in\Omega(X)$ with respect to some complex structure $X$ on $S_{g,n}$ if and only if it is the holonomy of the translation structure on $S_{g,n}$ determined by $\omega$.
\end{lem}

\noindent This twofold nature of a representation $\chi$ permits us to tackle our problem by adopting a geometric approach. More precisely, in order to realize a representation $\chi$ on a prescribed connected component of some stratum of differentials, we shall realize it as the holonomy of some translation surface with poles $(X,\omega)$ with prescribed zeros and poles and, whenever they are defined, with prescribed spin structure or hyperelliptic involution. Some remarks are in order.

\begin{rmk}\label{gbcond}
Let $\omega$ be a meromorphic differential on a compact Riemann surface $\overline{X}\in\mathcal{M}_g$ which yields a translation surface $(X,\omega)$ with poles of finite orders $p_1,\dots,p_n$, $n\ge1$. Let $m_1,\dots,m_k$ be the orders of zeros of $\omega$. Then it is well-known that the following equality holds
\begin{equation}\label{gbeq}
    \sum_{i=1}^k m_i - \sum_{j=1}^n p_j=2g-2.
\end{equation}
\end{rmk}

\begin{rmk}\label{partrans}
Let $(X,\omega)$ be a translation surface, possibly with poles, and let us denote $X^*=X\setminus\{\text{zeros of }\omega\}$ and pick any point $x_0\in X^*$. Since the structure is flat, the parallel transport induced by the flat connection yields a homomorphism from $\pi_1(X^*,\,x_0)$ to $\textnormal{SO}(2,\mathbb R)\cong\mathbb S^1$ which acts on the tangent space of $x_0$. Since $\mathbb S^1$ is abelian, such a homomorphism factors through the homology group and hence we have a well-defined homomorphism $\text{PT}\colon\textnormal{H}_1(X^*,\,\mathbb Z)\to \textnormal{SO}(2,\mathbb R)\cong\mathbb S^1$. For translation surfaces this homomorphism turns out to be trivial in the sense that a parallel transport of a vector tangent to the Riemann surface $X$ along any closed path avoiding the zeros of $\omega$ brings the vector back to itself.
\end{rmk}

\begin{rmk}\label{foliations}
Let $x_0\in(X,\omega)$ be a non-singular point, \textit{i.e.} $x_0$ is not a zero for $\omega$. A given tangent direction at $x_0$ can be extended to all other non-singular points by means of the parallel transport. This yields a non-singular foliation which extends to a singular foliation with singularities at the branch points. Let $z=x+iy$ be a local coordinate at $x_0$. The \textit{horizontal foliation} is the oriented foliation determined by the positive real direction in the coordinate $z$. Notice that this is well-defined because different local coordinates differ by a translation. In the same fashion, the \textit{vertical foliation} is the oriented foliation determined by the positive purely imaginary direction in the coordinate $z$.
\end{rmk}

\subsection{Volume}\label{agv} We now discuss the notion of \textit{volume} which plays an important role in the theory. For our purposes, we shall mostly interested in the algebraic volume which is a topological invariant associated to a representation $\chi$.
\smallskip

\noindent Let us recall this notion in the holomorphic case. For a symplectic basis $\{\alpha_1,\beta_1,\dots,\alpha_g,\beta_g\}$ of $\shomolz$ we define the volume of a representation $\chi\colon\shomolz\longrightarrow \C$ as the quantity
\begin{equation}\label{algvol}
    \text{vol}(\,\chi)=\sum_{i=1}^g \Im\big(\,\overline{\chi(\alpha_i)}\,\chi(\beta_i)\big),
\end{equation} 
where $\Im(\overline{z}w)$ is the usual symplectic form on $\C$. As a consequence, this algebraic definition of volume of a character $\chi$ is invariant under precomposition with any automorphisms in $\textnormal{Aut}\,\big(\shomolz\big)\cong\spz$. The image of $\chi$, provided it has rank $2g$, turns out to be a polarized module.

\smallskip

\textit{Haupt's conditions.} As already alluded in the introduction, there are some obstructions for realizing $\chi$ as the period character of some holomorphic differential $\omega$ on a compact Riemann surface $X$ and both of these concern the volume. We shall recall them here for the reader's convenience. The first requirement is that the volume of $\chi$ has to be positive with respect to some symplectic basis of $\shomolz$. Indeed, one can show that this equals the area of the surface $X$ endowed with the singular Euclidean metric induced by $\omega$. There is a second obstruction that applies in the case $g\geq 2$ and only when the image of $\chi\colon\shomolz\to\C$ is a lattice, say $\Lambda$ in $\C$. In fact, one can show that in this special case $(X,\omega)$ arises from a branched cover of the flat torus $\C/\Lambda$. In particular, the following inequality must hold: $\textnormal{vol}(\chi)\ge 2\,\text{{Area}}(\mathbb C/\Lambda)$. Haupt's Theorem says that these are the only obstructions for realizing $\chi$ as the period of some holomorphic differential.

\smallskip

\begin{rmk} We provide here an alternative and more geometric definition of volume. Let $X$ be a compact Riemann surface and let $\omega\in\Omega(X)$ be a (possibly meromorphic) differential with period character $\chi$. Let $E$ be the complex line bundle over $X$ canonically associated to $\chi$ and let $f\colon X\longrightarrow E$ be a section. By lifting the map $f$ to $\widetilde{f}\colon \widetilde{X}\longrightarrow \widetilde{X}\times\mathbb C$ and projecting onto the second coordinate, the natural volume form $\frac{i}{2}dz\wedge d\overline{z}$ on $\mathbb C$ pulls back to a volume form over $X$. The quantity given by the integration of this form over $X$, that is 
\begin{equation}
\frac{i}{2}\int_X f^*\Big( \,dz\wedge d\overline{z}\,\Big)=\frac{i}{2}\int_X \omega\wedge\varpi, \,\,\, \text{ where } \omega=f^*dz \end{equation}
is the area of the singular Euclidean metric determined by $\omega$ on $X$. In particular, by means of Riemann's bilinear relations one can show that it agrees with the algebraic definition of volume as in \eqref{algvol}.
\end{rmk} 

\begin{defn}\label{algvoldef}
Let $\omega\in\Omega(X)$ be a holomorphic differential on a compact Riemann surface $X$ with period character $\chi\colon\shomolz\longrightarrow \C$. We  define the \textit{algebraic volume} of $\omega$ as the quantity $\textnormal{vol}(\,\chi)$ defined in formula \eqref{algvol}. Notice that the algebraic volume is well-defined since the volume of $\chi$ does not depend on the choice of the symplectic basis for $\shomolz$.
\end{defn}

\noindent We now extend the notion of volume to meromorphic differentials. In this case, the notion of volume relies on a choice of a splitting of $S_{g,n}$ as the connected sum of the closed surface $S_{g}$ and the $n$-punctured sphere. Let $\gamma$ be a simple closed separating curve in $S_{g,n}$ that bounds a sub-surface homeomorphic to $S_{g,1}$. There is a natural embedding $S_{g,1}\hookrightarrow S_{g,n}$ and thence an injection $\imath_g\colon\shomolz \longrightarrow \shomolzn$ because the curve $\gamma$ is trivial in homology. In the same fashion, $\gamma$ bounds a sub-surface homeomorphic to $S_{0,n+1}$ and the natural embedding $S_{0,n+1}\hookrightarrow S_{g,n}$ yields an injection $\imath_n\colon\textnormal{H}_1(S_{0,n},\,\mathbb Z)\to \shomolzn$. Therefore, given a splitting as above, any representation $\chi\colon\shomolzn\longrightarrow \C$ gives rise to two representations $\chi_g=\chi\circ\imath_g$ and $\chi_n=\chi\circ\imath_n$ that determine $\chi$ completely.

\begin{defn}\label{def:triend}
A representation $\chi$ is said to be of \textit{trivial-ends type} if $\chi_n$ is trivial, otherwise it is said to be of \textit{non-trivial-ends type}. The representation $\chi_n$ is always well-defined and in fact it does not depend on the choice of any splitting.
\end{defn}

\begin{rmk}\label{rmk:chigwelldef}
It is worth noticing that a representation $\chi_g$ generally depends on the embedding $S_{g,1}\hookrightarrow S_{g,n}$ and in fact on the curve $\gamma$ along which we split $S_{g,n}$. However, it is not hard to see that $\chi_g$ is uniquely determined if and only if $\chi$ is of trivial-ends type.
\end{rmk} 

\noindent We can now introduce the following

\begin{defn}\label{algvoldef2}
Let $\chi\colon\shomolzn\to \C$ be a representation of trivial-ends type. We  define the volume of $\chi$ as the volume of $\chi_g$ as defined in \eqref{algvol}. Let $\omega\in\Omega(X)$ be a holomorphic differential on $X\in\mathcal{M}_{g,n}$ with finite-order poles at the punctures and such that its period character $\chi$ is of trivial-ends type. Then we can define the \textit{algebraic volume} of $\omega$ as the volume of $\chi$. Notice that the algebraic volume does not depend on the choice of the splitting.
\end{defn}

\noindent In the case that a meromorphic differential admits poles with non-zero residues, \textit{i.e.} its period character has non-trivial $\chi_n$ part, then the algebraic volume as in Definition \ref{algvoldef2} does depend on the splitting. In fact if an absolute homology class goes across a pole with nonzero residue, its period will change by the amount of the residue. In other to define the algebraic volume for a generic meromorphic differential one has to consider homologically marked translation structures, \textit{i.e.} translation surfaces with an additional structure given by an isomorphism $m\colon\mathbb Z^{2g}\longrightarrow \shomolzn$. Notice that a marking naturally yields a splitting as above, however the same splitting is generally induced by different markings. In our constructions we mostly deal with translation surfaces corresponding to meromorphic differentials whose character has trivial $\chi_n$ part because we shall rely on those for the generic case. For this reason, we shall not need to introduce marked structures for our purposes.

\subsection{Further invariants on translation surfaces} In Section \S\ref{agv} we have introduced the volume as an algebraic invariant naturally attached to a representation $\chi$. In the present section we are going to introduce geometric invariants which we shall use to distinguish the connected components of strata whenever they fail to be connected. 

\subsubsection{Rotation number}\label{rotation} In this subsection $\overline{X}$ is a compact Riemann surface of genus one, \textit{i.e.} an elliptic curve $\mathbb C/\Lambda$ with $\Lambda$ is a rank two lattice in $\mathbb C$. We are going to introduce an invariant very specific to genus one translation surfaces with poles.

\smallskip

\noindent Let $X$ be a complex structure on $S_{1,n}$, where $n\ge1$, and let $\omega\in\Omega(X)$ be an abelian differential on $X$ with finite order poles at the punctures. Recall that away from the singularities of $\omega$ there is a well-defined horizontal direction and hence a non-singular horizontal foliation on $X\setminus \{\text{ zeros of }\omega\,\}$. Such a foliation extends to a singular horizontal foliation over the zeros of $\omega$. In fact, for a zero, say $P$, of order $k$ there are exactly $k+1$ horizontal directions leaving from $P$. Let $\gamma$ be a simple smooth curve parametrized by the arc length. Assume in addition that $\gamma$ avoids all the singularities of $\omega$. We shall define the index of $\gamma$ as a numerical invariant given by the comparison of the unit tangent field $\dot\gamma(t)$ and the unit vector field along $\gamma$ given by unit vectors tangent to the horizontal direction. More precisely, let us denote by $u(t)$ the unit vector at $\gamma(t)$ tangent to the leaf of the horizontal foliation through $\gamma(t)$, then the assignment 
\begin{equation}\label{comphorfol}
    t\longmapsto \theta(t)=\angle\big(\dot\gamma(t),\,u(t)\big)
\end{equation} defines a mapping $f_\gamma\colon\mathbb S^1\longrightarrow \mathbb S^1$.

\begin{defn}\label{index}
The index of $\gamma$ is defined as $\deg(f_\gamma)$ and denoted by $\textnormal{Ind}(\gamma)$. 
\end{defn}

\noindent The index $\textnormal{Ind}(\gamma)$ of a closed curve $\gamma$ measures the number of times the unit tangent vector field spins with respect to the direction given by the horizontal foliation. Since for translation surfaces the parallel transport yields a trivial homomorphism, see Remark \ref{partrans}, the total change of the angle is $2\pi\,\textnormal{Ind}(\gamma)$.

\medskip

\noindent We now allow perturbations of $\gamma$ in its homotopy class. The index of any curve remains unchanged if, while deforming it, we do not cross a singularity of $\omega$. On the other hand, whenever we cross a singularity of order $k$ the index of $\gamma$ changes by $\pm k$. As a consequence, the index of a curve is well-defined for homotopy classes in $\pi_1(\overline{X}\setminus\{\text{singularities of }\omega\})$. We can now define the rotation number.

\begin{defn}\label{rotnum}
Let $(X,\omega)\in \mathcal{H}_1(m_1,\dots,m_k; -p_1,\dots,-p_n)$ be a genus one translation surface with poles. Let $\{\alpha,\beta\}$ be a symplectic basis of $\textnormal{H}_1(\overline{X},\,\mathbb Z)$, the first homology group of $\overline{X}$. Let $\gamma_\alpha$ and $\gamma_\beta$ be the representatives of $\alpha$ and $\beta$ respectively. Assume they are both parametrized by the arc length and avoid the singularities of $\omega$. The \textit{rotation number} of $(X,\omega)$ is defined as
\begin{equation}
    \textnormal{rot}(X,\omega)=\gcd\big(\textnormal{Ind}(\gamma_\alpha),\,\textnormal{Ind}(\gamma_\beta),\,m_1,\dots,m_k,p_1,\dots,p_n\big).
\end{equation}
\end{defn}

\smallskip

\noindent In the next subsection we briefly introduce the rotation number in a more algebraic way that we shall use only in the final Section \S\ref{sec:trirep}. Notice that the well-known properties of $\gcd$ make the rotation number $\textnormal{rot}(X,\omega)$ well-defined for a given element $(X,\omega)\in\mathcal{H}_1(m_1,\dots,m_k; -p_1,\dots,-p_n)$. As we shall see below in Section \S\ref{mscc}, the rotation number is used to distinguish the connected components of a generic stratum of genus one-differentials  $\mathcal{H}_1(m_1,\dots,m_k; -p_1,\dots,-p_n)$. 

\smallskip

\subsubsection{Hyperelliptic involution}\label{sssec:hypinv} For a complex structure $X$ on $S_{g,n}$ we shall denote by $\overline{X}$ the compact Riemann surface obtained by filling the punctures with complex charts. Let us now introduce the following

\begin{defn}\label{hypdef}
A translation surface with poles $(X,\omega)$ is said to be \textit{hyperelliptic} if the Riemann surface $\overline{X}$ admits a hyperelliptic involution $\tau$ such that $\tau^*\omega=-\omega$.
\end{defn}

\noindent The mapping $\tau$ realizes an isometry between the singular Euclidean structures determined by the differentials $\omega$ and $-\omega$ on $\overline{X}$. In particular, the translation surfaces corresponding to these meromorphic differentials are equivalent as Euclidean structures, \textit{i.e.} geometric structures modelled on $(\mathbb C,\,\mathbb S^1\ltimes\mathbb C)$ but they are not equivalent as translation surfaces. 

\smallskip

\noindent For a holomorphic differential $\omega$ on a compact Riemann surface $X$, the translation surface $(X,\omega)$ is hyperelliptic if and only if the underlying Riemann surface $X$ is hyperelliptic because there are no non-zero holomorphic one-form on the sphere. For meromorphic differentials such a scenario is no longer true as shown by the following example.

\begin{ex}
Let us consider the translation structure with a single pole on the Riemann sphere $\cp$ determined by the meromorphic differential $dz$. Let $l_1,\,l_2,\,l_3 \subset \cp$ be three geodesic segments for the standard Euclidean structure and slit $\cp$ along them. The resulting slit surface is homeomorphic to a sphere with three open disjoint disks removed. Geometrically, the resulting surface is a pair of pants equipped with a Euclidean structure and piece-wise geodesic boundary. In fact, each boundary component comprises two geodesic segments, say $l_{ij}$ with $i=1,2,3$ and $j=1,2$, and two corner points each one of angle $2\pi$. We consider two copies of this latter surface and we glue them by identifying the boundary segments having the same label. The final surface turns out to be a compact Riemann surface $X$ of genus two equipped with a meromorphic differential $\omega$ with six zeros of order $1$ and two poles of order $2$. The Riemann surface $X$, being compact of genus two, admits a hyperelliptic involution $\tau$. We want to show with a direct computation that $(X,\omega)$ is not hyperelliptic.

\begin{figure}[!ht]
    \centering
    \begin{tikzpicture}[scale=0.6, every node/.style={scale=0.7}]
    \draw [thick] (-3.5,0) ellipse (5cm and 2cm);
    \draw [thick] (-2.7,0.1) to [out=330, in=210] (-0.3,0.1);
    \draw [thick] (-2.5,0) to [out=30, in=150] (-0.5,0);
    \draw [thick] (-6.7,0.1) to [out=330, in=210] (-4.3,0.1);
    \draw [thick] (-6.5,0) to [out=30, in=150] (-4.5,0);
    \draw [thin, blue] (-8.5,0) to [out=330, in=210] (-6.5,0);
    \draw [thin, blue, dashed] (-8.5,0) to [out=30, in=150] (-6.5,0);
    \draw [thin, blue] (-4.5,0) to [out=330, in=210] (-2.5,0);
    \draw [thin, blue, dashed] (-4.5,0) to [out=30, in=150] (-2.5,0);
    \draw [thin, blue] (-0.5,0) to [out=330, in=210] (1.5,0);
    \draw [thin, blue, dashed] (-0.5,0) to [out=30, in=150] (1.5,0);
    \draw [thick] (6.5,0) circle (2cm);
    \draw [thin, blue] (5.5,-0.5) to [out=330, in=210] (7.5,-0.5);
    \draw [thin, blue] (5.25,0.5) to [out=30, in=150] (6.25,0.5);
    \draw [thin, blue] (6.75,0.5) to [out=30, in=150] (7.75,0.5);
    \draw [-latex, thick] (2,0) to (4,0);
    \node at (3,0.375) {$2:1$};
\end{tikzpicture}

    \caption{A $2$-fold branched covering $\pi: X\longrightarrow \cp$.}
    \label{fig:my_label}
\end{figure}
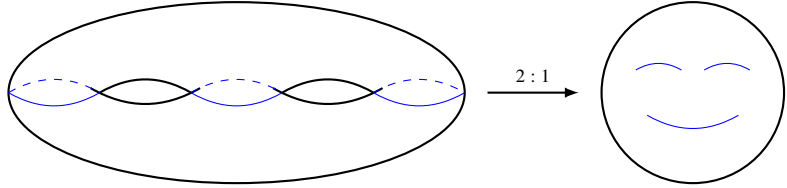

\noindent Let $\pi\colon X\longrightarrow \cp$ be the $2$-fold branched covering that naturally arises from our construction. The differential $\omega$ satisfies the equation $\omega=\pi^*(dz)$.
According to Definition \ref{hypdef} above, $(X,\omega)$ is hyperelliptic if and only if $\tau^*(\omega)=-\omega$. On the other hand, being of genus two, the hyperelliptic involution $\tau$ is an automorphism of $X$ and commutes with the projection $\pi\colon X\longrightarrow \cp$ in the sense that $\pi\circ \tau=\pi$. Therefore, $\omega = \pi^{*}(dz) = \tau^{*}(\omega)$ and hence 
\begin{equation}
    \tau^*(\omega)=-\omega \iff \omega=0
\end{equation}
which provides the desired contradiction.
\end{ex}

\subsubsection{Spin structure}\label{sssec:spinpar} The last geometric invariant we shall introduce is the spin structure. Let $X$ be a compact Riemann surface and let $P$ be a principal $\mathbb S^1$-bundle over $X$. A \textit{spin structure} on $X$ is a choice of a linear functional $\xi\colon\textnormal{H}_1(P,\, \mathbb Z_2)\to\mathbb Z_2$ having non-zero value on the cycle representing the fiber $\mathbb S^1$ of $P$. This is equivalent to a choice of a double covering $Q\longrightarrow P$ whose restriction to each fibre of $P$ is a $2$-fold covering of $\mathbb S^1$.

\smallskip

\noindent Recall that an element of $\textnormal{H}_1(P,\, \mathbb Z_2)$ can be regarded as a \textit{framed closed curve} in $X$, by which we mean a closed curve in $X$ and a smooth unit vector field along it. For a simple closed curve $\gamma$ in $X$ there is a preferred frame given its unit tangent vector field. We shall denote such a framed closed curve as $\vec{\gamma}$. This is well-defined in the sense that two homologous closed curves $\gamma,\,\delta\in \textnormal{H}_1(X,\, \mathbb Z_2)$ yield homologous framed closed curves $\vec{\gamma},\,\vec{\delta}$, see \cite[Section 3]{JO}.\\
The \textit{canonical lift} of a simple closed curve $\gamma$ is defined as $\widetilde{\gamma}=\vec{\gamma}+\textbf{z}$ where $\textbf{z}\in \textnormal{H}_1(P,\, \mathbb Z_2)$ is the homology class represented by the fibre $\mathbb S^1$. The assignment $\gamma\longmapsto \widetilde{\gamma}$ defines a mapping of sets $\textnormal{H}_1(X,\, \mathbb Z_2)\to \textnormal{H}_1(P,\, \mathbb Z_2)$ which fails to be a homomorphism, see \cite{JO}. Nevertheless, post-composition of such a map with a spin structure $\xi$ yields a quadratic form $\Omega_\xi\colon\textnormal{H}_1(X,\, \mathbb Z_2)\to \mathbb Z_2$ defined as 
\begin{equation}
    \Omega_\xi(\gamma)=\xi\big(\,\widetilde{\gamma}\,\big)=\langle \,\xi,\,\widetilde{\gamma}\,\rangle.
\end{equation}

\noindent According to \cite{AC} and \cite[Section 5]{JO} we shall introduce the following

\begin{defn}\label{parity}
For a symplectic basis $\{\alpha_1,\beta_1,\dots,\alpha_g,\beta_g\}$ of $\textnormal{H}_1(X,\, \mathbb Z_2)$ the Arf-invariant of $\Omega_\xi$ is defined by the formula
\begin{equation}\label{arfinv}
    \sum_{i=1}^g \Omega_\xi(\alpha_i)\Omega_\xi(\beta_i) \,\,(\text{mod}\,2).
\end{equation}
The \textit{parity} of a spin structure $\xi$ is defined as the Arf-invariant of $\Omega_\xi$. In particular, the parity of $\xi$ does not depend on the choice of the symplectic basis.
\end{defn}

\medskip 

\noindent Let $X$ be a complex structure on $S_{g,n}$, with $g\ge1$, and $\omega\in\Omega(X)$ be an abelian differential on $X$ with poles of finite order at the punctures. Recall that we can regard $\omega$ as a meromorphic differential on a compact Riemann surface $\overline{X}$. Whenever $(X,\omega)\in\mathcal{H}_g(2m_1,\dots,2m_k;-2p_1,\dots,-2p_n)$ then $\omega$ defines a natural spin structure on $\overline{X}$ as follows. 

\smallskip
\noindent Let $\gamma$ be a simple closed curve parametrized by the arc length and let $\vartheta$ be a unit vector field along $\gamma$. The couple $(\gamma, \vartheta)$ defines a framed closed curve, \textit{i.e.} an element of $\textnormal{H}_1(P,\,\mathbb Z_2)$. Recall that $\omega$ defines a non-singular horizontal foliation away from the singularities and hence the unit vector field $\vartheta$ can be compared to the unit vector field along $\gamma$ tangent to the horizontal foliation. In fact, as in Section \S\ref{rotation}, this can be done by means of a mapping $f_{(\gamma,\vartheta)}\colon\mathbb S^1\to \mathbb S^1$ defined as in the equation \eqref{comphorfol}. The spin structure induced by the meromorphic differential $\omega$ on $\overline{X}$ is then defined as the $\mathbb Z_2$-value linear functional
\begin{equation}\label{spinstrome}
    \xi\colon \textnormal{H}_1(P,\,\mathbb Z_2)\longrightarrow \mathbb Z_2,\,\, \text{ defined as }\,\, \xi(\gamma,\vartheta)=\deg(\,f_{(\gamma,\vartheta)}\,)\,\,(\text{mod}\,2).
\end{equation}

\noindent A direct application of \eqref{spinstrome} shows that, for the canonical lift $\widetilde{\gamma}$ of $\gamma$, the spin structure determined by $\omega$ satisfies the property
\begin{equation}\label{spindef}
    \xi(\,\widetilde{\gamma}\,)=\langle\,\xi,\,\widetilde{\gamma}\,\rangle=\langle\,\xi,\,\vec{\gamma}+\textbf{z}\,\rangle= \text{Ind}(\gamma)+1\,\,(\text{mod}\,2).
\end{equation}

\noindent As a consequence, we can apply formula \eqref{arfinv} to compute the parity $\varphi(\omega)$ of the spin structure determined by a meromorphic differential $\omega$. Given a symplectic basis $\{\alpha_1,\beta_1,\dots,\alpha_g,\beta_g\}$ of $\textnormal{H}_1(X,\, \mathbb Z_2)$, it follows that
\begin{equation}\label{eq:spinparity}
    \varphi(\omega)=\sum_{i=1}^g \big(\,\text{Ind}(\alpha_i)+1\big)\big(\,\text{Ind}(\beta_i)+1\big) \,\,(\text{mod}\,2).
\end{equation}

\noindent It is straightforward to check that $\varphi(\omega)$ does not depend on the choice of the symplectic basis. Finally, it follows from works of Atiyah in  \cite{AM} and Mumford in \cite{MD} that the spin parity is invariant under continuous deformations.

\begin{rmk}\label{twosimplepoles}
Let $(X,\omega)\in\mathcal{H}_g(2m_1,\dots,2m_k,-1,-1)$ be a genus-$g$ meromorphic differential with two simple poles. The meromorphic differential $\omega$ does not define any spin structure as in \eqref{spinstrome} because, for any curve $\gamma$, the index $\textnormal{Ind}(\gamma)$ is not longer well-defined as an element of $\mathbb Z_2$, see formula \eqref{spindef}. Nevertheless, we can still define an invariant that turns out to be the spin parity of some structure $(Y,\xi)$ in $\mathcal{H}_{g+1}(2m_1,\dots,2m_k)$. Recall that a neighborhood of a simple pole is an infinite cylinder. Around a puncture of $(X,\omega)$, we can find a simple closed geodesic curve, in fact there are infinitely many of such curves. Choose a waist geodesic curve on each cylinder and truncate $(X,\omega)$ along them. Since  $(X,\omega)$ has only two simple poles and their residues are opposite (and non-zero),  the curves above are isometric. We glue them and the resulting object is a translation structure on a closed surface of genus $g+1$. We define this latter as $(Y, \xi)$ and it lies in the stratum $\mathcal{H}_{g+1}(2m_1,\dots,2m_k)$ by construction. Therefore it makes sense to consider a \textit{parity of the spin structure} for differentials in $\mathcal{H}_g(2m_1,\dots,2m_k,-1,-1)$. See \cite[Section 5.3]{BC} for more details.
\end{rmk}

\subsection{Moduli spaces and connected components}\label{mscc} As alluded in the introduction, the moduli space $\Omega\mathcal{M}_{g,n}$ admits a natural stratification given by unordered partitions $(m_1,\dots,m_k;-p_1,\dots,-p_n)$ of $2g-2$, where negative integers are allowed only in the case of meromorphic differentials. In the present subsection we aim to recall for the reader's convenience the classifications of connected components of the strata both in the holomorphic and meromorphic cases.

\smallskip

\noindent Let us discuss first the case of holomorphic differentials which has been studied by Kontsevich-Zorich in \cite[Theorems 1 and 2]{KZ}. For low-genus surfaces, \textit{i.e.} $g=2,3$, we have the following

\begin{thmnn}[Kontsevich-Zorich]
The moduli space of abelian differentials on a curve of genus $g = 2$ contains two strata: $\mathcal{H}_2(1,1)$ and $\mathcal{H}_2(2)$. Each of them is connected and coincides with its hyperelliptic component.
Each of the strata $\mathcal{H}_3(2,2)$, $\mathcal{H}_3(4)$ of the moduli space of abelian differentials on a curve of genus $g = 3$ has two connected components: the hyperelliptic one, and one having odd spin structure. The other strata are connected for genus $g = 3$.
\end{thmnn}

\noindent For surfaces of genus $g\ge4$, the following classification holds.

\begin{thmnn}[Kontsevich-Zorich]
All connected components of any stratum of abelian differentials on a curve of genus $g\ge 4$ are described by the following list:
\begin{itemize}
    \item The stratum $\mathcal{H}_g(2g-2)$ has three connected components: the hyperelliptic one, and two other components corresponding to even and odd spin structures;
    \item the stratum $\mathcal{H}_g(2l,2l)$, $l\ge2$ has three connected components: the hyperelliptic one and two other components distinguished by the spin parity;
    \item all the other strata of the form $\mathcal{H}_g(2l_1,\dots,2l_n)$, where all $l_i \ge 1$, have two connected components corresponding to even and odd spin structures; 
    \item the strata  $\mathcal{H}_g(2l-1, 2l-1)$, $l\ge 2$, have two connected components: one comprises hyperelliptic structures and the other does not;  finally
    \item all the other strata of Abelian differentials on the curves of genus $g\ge4$ are nonempty and connected.
    \end{itemize}
\end{thmnn}

\smallskip

\noindent Let us discuss the case of meromorphic differentials and begin from the case of genus one meromorphic differentials. Let $\mathcal{H}_{1}(m_1,\dots,m_k;-p_1,\dots,-p_n)$ be a stratum in $\Omega\mathcal{M}_{1,n}$ and let $d$ be the number of positive divisors of $\gcd(m_1,\dots,m_k,p_1,\dots,p_n)$. Recall that a stratum is non-empty as soon as the Gauss-Bonnet condition holds for some non-negative integer $g$, see Remark \ref{gbcond}, and $\sum p_i>1$. In \cite[Theorem 1.1]{BC}, Boissy showed the following

\begin{thmnn}[Boissy]
Let $\mathcal{H}_{1}(m_1,\dots,m_k;-p_1,\dots,-p_n)$ be a non-empty stratum of genus one meromorphic differentials. Denote by $d$  the number of positive divisors of $\gcd(m_1,\dots,m_k,p_1,\dots,p_k)$. Then the number of connected components of this  stratum is:
\begin{itemize}
    \item $d-1$ connected components if $k=n=1$. In this case the stratum is $\mathcal{H}_1(m,-m)$ and the connected components are parametrized by the positive divisors of $m$ different from $m$ itself. 
    \item $d$ otherwise. In this case each connected component is parametrized by a positive divisor of $\gcd(m_1,\dots,m_k,p_1,\dots,p_k)$.
\end{itemize}
\end{thmnn}

\noindent Common divisors of the zero and pole orders in the above result are called {\em rotation numbers}.  
It is worth mentioning that connected components of the strata in genus one can also be classified by another invariant from the algebraic viewpoint, see \cite[Section \S3.2]{CC}. For a stratum $\mathcal H_1(m_1,\dots,m_k;\,-p_1,\dots,-p_n)$ of genus-one differentials, let $l$ be any positive divisor of $\gcd(m_1,\dots,m_k,p_1,\dots,p_n)$ -- with the only exception being $l=n$ for signatures of the form $(n;-n)$. We shall say that $(X,\omega)\in\mathcal H_1(m_1,\dots,m_k;\,-p_1,\dots,-p_n)$ has {\em torsion number} $l$ if 
\begin{equation}
    \sum_{j=1}^k \left(\frac{m_j}{l}\right)\,Z_j\,-\,\sum_{i=1}^n\left(\frac{p_i}{l}\right)\,P_i \,\sim \,0
\end{equation}
where $\{Z_i,\dots,Z_n\}$ and $\{P_1,\dots,P_n\}$ are respectively the set of zeros and poles of $\omega$. Here $\sim$ stands for linear equivalence which means there exists a meromorphic function on the underlying Riemann surface with zeros at $\{Z_1,\dots, Z_k\}$ and poles at $\{P_1,\dots,P_n\}$ with orders given by the coefficients in the relation. We shall use this characterisation in Section \S\ref{sec:trirep} to show a few Lemmas concerning the exceptional cases listed in Theorem \ref{thm:mainthm2}.

\begin{rmk}
The two notions of rotation numbers and torsion numbers coincide, see \cite[Section \S3.4]{Tahar} and 
\cite[Section \S3.4 and Proposition 3.13]{CG}.
\end{rmk}

\smallskip

\noindent For genus-$g$ meromorphic differentials with $g\ge2$, the classifications of connected components is similar to that of holomorphic differentials. Let 
$\mathcal{H}_{g}(m_1,\dots,m_k;-p_1,\dots,-p_n)$ be a stratum in $\Omega\mathcal{M}_{g,n}$. In \cite{BC}, Boissy has showed that a stratum admits a hyperelliptic component if and only if it is one of the following
\begin{equation}\label{hypcomp}
    \mathcal{H}_{g}(2m,-2p),\,\, \mathcal{H}_{g}(m, m,-2p),\,\, \mathcal{H}_{g}(2m,-p, -p),\,\, \mathcal{H}_{g}(m, m, -p, -p),
\end{equation}

\noindent otherwise there is no hyperelliptic component. According to Boissy, we introduce the following

\begin{defn}\label{def:kindofstrata}
For a stratum $\mathcal{H}_g(\kappa;-\nu)$ we shall say that the set of zeros and poles is 
\begin{itemize}
    \item of \textit{hyperelliptic type} if $\kappa=\{2m\}$ or $\{m,m\}$ and $\nu=\{2p\}$ or $\{p,p\}$, where $m$ and $p$ are positive integers;
    \item of \textit{even type} if $\kappa=(2m_1,\dots,2m_k)$ and $\nu=(2p_1,\dots,2p_n)$ or $\nu=(1,1)$.
\end{itemize}
We shall say that a translation surface is of even type if it belongs to a stratum of even type, namely the set of zeros and poles is of even type. Similarly, a translation surface is hyperelliptic if it belongs to the hyperelliptic component of the stratum.
\end{defn}

\noindent When necessary, we shall adopt the same terminology. We now state the following result due to Boissy concerning connected components of the strata of genus-$g$ meromorphic differentials with $g\ge2$.

\begin{thmnn}[Boissy]
Let $\mathcal{H}_g(\kappa,-\nu)=\mathcal{H}_{g}(m_1,\dots,m_k;-p_1,\dots,-p_n)$ be a non-empty stratum of genus-$g$ meromorphic differentials. We have the following.
\begin{itemize}
    \item If $p_1+\cdots+p_n$ is odd and greater than two then the stratum is connected.
    \smallskip
    \item If $\nu=\{2\}$ or $\nu=\{1,1\}$ \textit{and} $g=2$ then we distinguish two cases
    \begin{itemize}
        \item if the set of poles and zeros is of hyperelliptic type, then there are two connected components, one hyperelliptic, the other not; in this case, these two components are also distinguished by the parity of the spin structure,
        \smallskip
        \item otherwise the stratum is connected
    \end{itemize}
    \smallskip
    \item If $\nu$ is \textit{not} $\{2\},\,\{1,1\}$ (\textit{i.e.} $p_1+\cdots+p_n>2$) or $g\ge3$, then we distinguish two cases as follows
    \begin{itemize}
        \item if the set of poles and zeros is of hyperelliptic type, there is exactly one hyperelliptic connected component, and one or two non-hyperelliptic components that are described below. Otherwise, there is no hyperelliptic component.
        \smallskip
        \item if the set of poles and zeros is of even type, then the stratum contains exactly two non-hyperelliptic connected components that are distinguished by the parity of the spin structure. Otherwise the stratum contains exactly one non-hyperelliptic component.
    \end{itemize}
\end{itemize}
\end{thmnn}

\noindent This concludes the classification of the strata of meromorphic differentials. 

\medskip

\section{Surgeries and isoperiodic deformations}\label{sec:surgeries} 

\noindent In this section we recall a few surgeries we shall use in the sequel. These all basically consist of cutting  a given translation surfaces along one, or possibly more, geodesic segment(s) and gluing them back along those segments in order to get new translation structures. Different gluings will provide different translation structures. 

\subsection{Breaking a zero}\label{sec:zerobreak} The surgery we are going to describe has been introduced by Eskin-Masur-Zorich in \cite[Section 8.1]{EMZ} and it literally "breaks up" a branch point in two, or possibly more, branch points of lower orders. Complex-analytically this can be thought as the analogue to the classical Schiffer variations for Riemann surfaces, see \cite{NS}. The importance of this surgery for us is due to its fundamental property of preserving the topology of the underlying surface even if the overall geometry is changed. In particular, the new translation surface we obtain once the surgery is performed has the same period character as the original one. 

\begin{rmk} In the context of branched projective structures, such a surgery is also known as \emph{movement of branched points} and it has been originally introduced by Tan in \cite[Chapter 6]{TA} for showing the existence of a complex one-dimensional continuous family of deformations of a given structure.\end{rmk}

\noindent Let us now explain this surgery in more detail. Let $(X,\omega)$ be a translation surface possibly with poles. Breaking a zero is a procedure that takes place at the $\varepsilon$-neighbourhood of some zero of order $m$ of the differential on which it looks like the pull-back of the form $dz$ via a branched covering $z\mapsto z^{m+1}$. The differential is then modified by a surgery inside this $\varepsilon$-neighbourhood. Once this surgery is performed we obtain a new translation structure with two zeros of order $m_1$ and $m_2$ such that $m_1+m_2 = m$. Furthermore, the translation structure remains unchanged outside the $\varepsilon$-neighbourhood of such a  zero of order $m$. The idea is to consider the $\varepsilon$-neighbourhood of a zero of order $m$ as $m+1$ copies of a disc $D$ of radius $\varepsilon$ whose diameters are identified in a specified way. We can see this family of discs as a collection of $m+1$ upper half-discs and $m+1$ lower half-discs. See figure \ref{fig:zerolocmodel}.

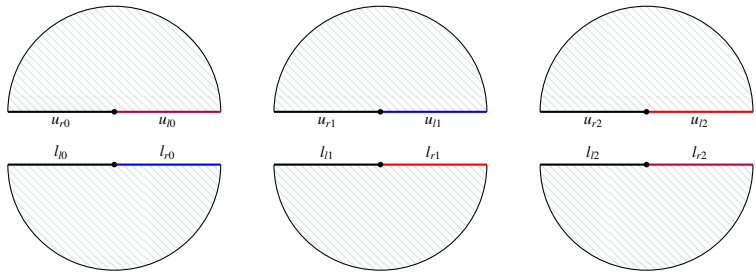
\begin{figure}[ht]
    \centering
    \begin{tikzpicture}[scale=0.7, every node/.style={scale=0.6}]
    \definecolor{pallido}{RGB}{221,227,227}
    \foreach \x [evaluate=\x as \coord using 4 + 5*\x] in {0, 1, 2} 
    {
    \draw [pattern=north west lines, pattern color=pallido] (\coord,0) arc [start angle = 0,end angle = 180,radius = 2];
    \draw [pattern=north west lines, pattern color=pallido] (\coord,-1) arc [start angle = 0,end angle = -180,radius = 2];
    }
    \foreach \x [evaluate=\x as \leftend using 5*\x] [evaluate=\x as \rightend using 2 + 5*\x] in {0,1,2} 
    {
    \draw [thick] (\leftend, 0) -- (\rightend, 0);
    \draw [thick] (\leftend, -1) -- (\rightend, -1);
    }
    \foreach \x [evaluate=\x as \leftlabel using 1 + 5*\x] [evaluate=\x as \rightlabel using 3 + 5*\x] in {0, 1, 2} 
    {
    \node [below] at (\leftlabel, 0) {$u_{r\x}$};
    \node [below] at (\rightlabel, 0) {$u_{l\x}$};
    \node [above] at (\leftlabel, -1) {$l_{l\x}$};
    \node [above] at (\rightlabel, -1) {$l_{r\x}$};
    }
    
\foreach \botindex / \topindex / \colr [evaluate=\botindex as \botleftend using 5*\botindex + 2] [evaluate=\botindex as \botrightend using 5*\botindex + 4] [evaluate=\topindex as \topleftend using 5*\topindex + 4] [evaluate=\topindex as \toprightend using 5*\topindex + 2] in {0/1/blue, 1/2/red, 2/0/purple} 
    {
    \draw [thick, \colr] (\topleftend, 0) -- (\toprightend, 0);
    \draw [thick, \colr] (\botleftend, -1) -- (\botrightend, -1);
    \fill (\toprightend, 0) circle (1.5pt);
    \fill (\botleftend, -1) circle (1.5pt);
    }
 
    \end{tikzpicture}
    \caption{An $\varepsilon$-neighbourhood of a zero of order $2$.}
    \label{fig:zerolocmodel}
\end{figure}

\noindent  We now break a zero of order $m$ into two zeros of order $m_1$ and $m_2$. To break a zero consists of identifying the diameters of the starting $m+1$ discs in a different way as follows. In order to do this we modify the labelling on the upper half disc indexed by $0$, the lower half disc indexed by $m_1$, and all upper and lower half discs with index more than $m_1$ accordingly. The modified labelling is shown below in Figure \ref{fig:splitlocmodel} for the case of splitting a zero of order $2$ into two zeros of order $1$.

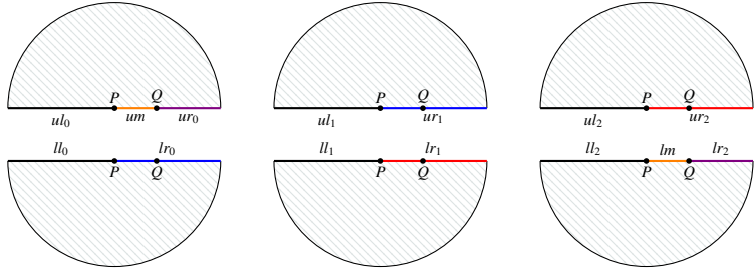
\begin{figure}[ht]
    \centering
    \begin{tikzpicture}[scale=0.7, every node/.style={scale=0.6}]
    \definecolor{pallido}{RGB}{221,227,227}
    \foreach \x [evaluate=\x as \coord using 4 + 5*\x] in {0, 1, 2} 
    {
    \draw [pattern=north west lines, pattern color=pallido] (\coord,0) arc [start angle = 0,end angle = 180,radius = 2];
    \draw [pattern=north west lines, pattern color=pallido] (\coord,-1) arc [start angle = 0,end angle = -180,radius = 2];
    }
    \foreach \x [evaluate=\x as \leftend using 5*\x] [evaluate=\x as \rightend using 2 + 5*\x] in {0, 1, 2} 
    {
    \draw [thick] (\leftend, 0) -- (\rightend, 0);
    \draw [thick] (\leftend, -1) -- (\rightend, -1);
    \node[above] at (\rightend, 0) {$P$};
    \node[below] at (\rightend, -1) {$P$};
    }
    \draw [thick, orange] (2,0) -- (2.8,0);
    \fill (2,0) circle (1.5pt);
    \draw [thick, orange] (12,-1) -- (12.8,-1);
    \fill (12,-1) circle (1.5pt);
    \node [below] at (2.4, 0) {$um$};
    \node [above] at (12.4, -1) {$lm$};
    \foreach \botindex / \topindex / \colr [evaluate=\botindex as \botleftend using 5*\botindex + 2] [evaluate=\botindex as \botrightend using 5*\botindex + 4] [evaluate=\topindex as \topleftend using 5*\topindex + 2] [evaluate=\topindex as \toprightend using 5*\topindex + 4] [evaluate=\topindex as \topsplpt using 5*\topindex + 2.8] [evaluate=\botindex as \botsplpt using 5*\botindex + 2.8] in {0/1/blue, 1/2/red} 
    {
    \draw [thick, \colr] (\topleftend, 0) -- (\toprightend, 0);
    \draw [thick, \colr] (\botleftend, -1) -- (\botrightend, -1);
    \fill (\topleftend, 0) circle (1.5pt);
    \fill (\botleftend, -1) circle (1.5pt);
    \fill (\topsplpt, 0) circle (1.5pt);
    \fill (\botsplpt, -1) circle (1.5pt);
    \node[above] at (\topsplpt, 0) {$Q$};
    \node[below] at (\botsplpt, -1) {$Q$}; 
    }
    \foreach \botindex / \topindex / \colr [evaluate=\botindex as \botleftend using 5*\botindex + 2.8] [evaluate=\botindex as \botrightend using 5*\botindex + 4] [evaluate=\topindex as \topleftend using 5*\topindex + 2.8] [evaluate=\topindex as \toprightend using 5*\topindex + 4] in {2/0/violet} 
    {
    \draw [thick, \colr] (\topleftend, 0) -- (\toprightend, 0);
    \draw [thick, \colr] (\botleftend, -1) -- (\botrightend, -1);
    \fill (\topleftend, 0) circle (1.5pt);
    \fill (\botleftend, -1) circle (1.5pt);
    \node[above] at (\topleftend, 0) {$Q$};
    \node[below] at (\botleftend, -1) {$Q$};
    }
    \foreach \x [evaluate=\x as \leftlabel using 1 + 5*\x] in {0, 1, 2} 
    {
    \node [below] at (\leftlabel, 0) {$ul_{\x}$};
    \node [above] at (\leftlabel, -1) {$ll_{\x}$};
    }
    \foreach \botindex / \topindex [evaluate=\botindex as \botlabel using 3+ 5*\botindex] [evaluate=\topindex as \toplabel using 3+ 5*\topindex] in {0/1, 1/2} 
    {
    \node [below] at (\toplabel, 0) {$ur_{\topindex}$};
    \node [above] at (\botlabel, -1) {$lr_{\botindex}$};
    }
    \foreach \botindex / \topindex [evaluate=\botindex as \botlabel using 3.4+ 5*\botindex] [evaluate=\topindex as \toplabel using 3.4+ 5*\topindex] in {2/0} 
    {
    \node [below] at (\toplabel, 0) {$ur_{\topindex}$};
    \node [above] at (\botlabel, -1) {$lr_{\botindex}$};
    }
    \end{tikzpicture}
    \caption{New labelling for breaking up a zero of order $2$ in two zeros of order $1$.}
    \label{fig:splitlocmodel}
\end{figure}

\noindent We now identify $ul_i$ with $ll_i$ and $lr_i$ with $ur_{i+1}$ as before with the added identification of $um$ with $lm$. This identification gives two singular points $A$ and $B$, where $A$ is a zero of the differential of order $k_1$ and $B$ is a zero of order $k_2$. We also get a geodesic line segment joining $A$ and $B$. Given $c \in \mathbb{C}\setminus\{0\}$ with length less than $2\varepsilon$, we can perform the surgery in such a way that the line segment joining $A$ and $B$ is $c$. It is also clear that such a deformation of the translation structure is only local. This procedure can be repeated multiple times to obtain zeros of orders $m_1, \dots, m_k$ from a single zero of order $m_1 + \cdots + m_k$. We shall frequently rely on this procedure of breaking a zero in the remaining part of this paper. The following Lemmas are easy to establish.

\begin{lem}
Let $(X,\omega)$ be a genus one differential with rotation number $r$. Let $(Y,\xi)$ be a genus one meromorphic differential in the stratum $\mathcal{H}_1(m_1,\dots,m_k;-p_1,\dots,-p_n)$ which is obtained from $(X,\omega)$ by breaking a zero. If $r$ divides $\gcd(m_1,\dots,m_k,p_1,\dots,p_n)$ then $(Y,\xi)$ has rotation number $r$.
\end{lem}

\begin{lem}
Breaking a zero does not alter the spin structure of a genus-$g$ differential for $g\ge2$.
\end{lem}

\begin{lem}\label{lem:breakhyp}
Let $(X,\omega)\in\mathcal{H}_g(2m;-\nu)$ be a hyperelliptic translation surface, where $\nu=\{2m-2g+2\}$, or $\{m-g+1,\,m-g+1\}$, and let $(Y,\xi)$ be the structure in $\mathcal{H}_g(m,m;-\nu)$ obtained from $(X,\omega)$ by breaking a zero. Then $(Y,\xi)$ is hyperelliptic.
\end{lem}

\smallskip

\subsection{Bubbling a handle with positive volume}\label{ssec:bubhandle} We now describe a second surgery introduced by Kontsevich and Zorich in \cite{KZ}. In their paper this surgery is defined for holomorphic differentials but, as already observed by Boissy in \cite{BC}, the same surgery can be defined for meromorphic differentials because this is a local surgery. Topologically, \textit{bubbling a handle with positive volume} consists of  adding a handle, say $\Sigma$ to a given surface. Let us see how this can be done metrically. Let $(X,\omega)\in\mathcal{H}_g(m_1,\dots,m_k;-p_1,\dots,-p_n)$ be a translation surface, possibly with finite order poles. Let $l\subset (X,\omega)$ be a geodesic segment with distinct endpoints. 

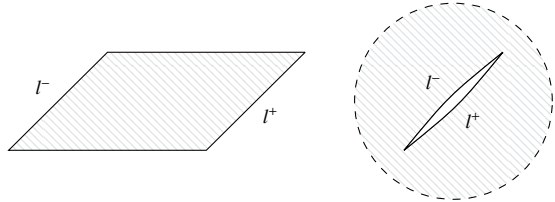
\begin{figure}[!ht] \label{fig:bubbhand}
\centering
\begin{tikzpicture}[scale=0.65, every node/.style={scale=0.75}]
\definecolor{pallido}{RGB}{221,227,227}
\draw[black, pattern=north west lines, pattern color=pallido] (-3,1) -- (-5, -1) -- (-9,-1) -- (-7, 1) -- (-3,1);
\node[above left] at (-8,0) {$l^-$};
\node[below right] at (-4,0) {$l^+$};
\draw[dashed, black, pattern=north west lines, pattern color=pallido] (0,0) circle (2);
\draw (-1,-1) .. controls (0.1, -0.1) .. (1,1);
\draw (-1,-1) .. controls (-0.1, 0.1) .. (1,1);
\node[above left] at (-0.1, 0.1) {$l^-$};
\node[below right] at (0.1, -0.1) {$l^+$}; 
\fill [white, thin, draw=black] (-1,-1) .. controls (0.1, -0.1) .. (1,1) .. controls (-0.1, 0.1) .. (-1,-1);
\end{tikzpicture}
\caption{Bubbling a handle with positive volume.}
\end{figure}

\noindent We slit the translation surface $(X,\omega)$ along $l$ and we label the side that has the surface on its left as $l^+$ and the other side as $l^-$. We then paste the extremal points of the geodesic segment and the resulting surface has two geodesic boundary components having (by construction) the same length. Let $\mathcal{P}\subset \mathbb C$ be a parallelogram such that the two opposite sides are both parallel to $l$ (via the developing map) -- the handle $\Sigma$ is obtained by gluing the opposite sides of $\mathcal{P}$. We can paste such a parallelogram to the slit $(X,\omega)$ and the resulting topological surface is homeomorphic to $S_{g+1,n}$. Metrically, we have a new translation surface $(Y,\xi)$. We can distinguish three mutually disjoint possibilities: 
\begin{itemize}
    \item[1.] Both of the extremal points of $l\subset (X,\omega)$ are regular, then $(Y,\xi)\in\mathcal{H}_g(2,m_1,\dots,m_k;-p_1,\dots,-p_n)$,
    \item[2.] One of the extremal points of $l\subset (X,\omega)$ is a zero of $\omega$ of order $m_i$, where $1\le i\le k$, then $(Y,\xi)$ belongs to the stratum $\mathcal{H}_g(m_1,\dots,m_i+2,\dots,m_k;-p_1,\dots,-p_n)$,
    \item[3.] Both of the extremal points are zeros of $\omega$ of orders $m_i,\,m_j$ respectively, then $(Y,\xi)$ belongs to the stratum $\mathcal{H}_g(m_1,\dots,\widehat{m}_i,\dots,\widehat{m}_j,\dots,m_i+m_j+2,\dots,m_k;-p_1,\dots,-p_n)$.
\end{itemize}

\noindent The following holds.

\begin{lem}\label{lem:spininv}
Let $(X,\omega)$ be a translation surface obtained from $(Y,\xi)\in\mathcal{H}_{g}\big(2m_1,\dots,2m_k;-2p_1,\dots,-2p_n\big)$ by bubbling a handle along a slit. Let $2\pi(l+1)$ be the angle around one of the extremal points of the slit; where $l=2m_i\ge2$ for some $i=1,\dots,k$ or $l=0$. Then, the spin structures determined by $\omega$ and $\xi$ are related as follows
\begin{equation}
    \varphi(\omega)-\varphi(\xi)=l \,\,\,(\textnormal{ mod }2\, ).
\end{equation}
In particular, bubbling a handle does not alter the spin parity.
\end{lem}

\begin{rmk}\label{rmk:spininv}
It is worth noticing that this Lemma differs from \cite[Lemma 11]{KZ} because Definition of bubbling is different in principle. More precisely, in the present paper a bubbling is performed along a geodesic segment that joins two points, possibly regular. In \cite{KZ}, however, the bubbling of a handle is performed along a saddle connection that joins two singular points obtained after breaking a zero. This preliminary operation is the reason of a possible alteration of the spin parity because, by breaking a zero, the resulting structure may no longer be of even type. More precisely, if we break a singular point on a structure of even type, the resulting singular points both have even order or odd order. In the former case bubbling a handle does not alter the parity but in the latter case it does. \cite[Lemma 11]{KZ} takes into account this possibility. Here we do not break any zero for bubbling a handle and hence the spin parity remains unaltered.
\end{rmk}

\begin{proof}[Proof of Lemma \ref{lem:spininv}]
By keeping in mind Remark \ref{rmk:spininv}, the result follows as in \cite[Lemma 11]{KZ}.
\end{proof}

\subsection{Bubbling a handle with non-positive volume.}\label{ssec:crophandle}

We have seen above how to glue a handle with positive volume on a given translation surface $(X,\omega)$. Here we briefly describe a way to glue a handle with non-positive volume and non-trivial periods; we shall describe a way to add handles with trivial periods in subsection \S\ref{ssec:gluingtrihand}. Topologically, this surgery deletes the interior of a parallelogram with distinct vertices on $(X,\omega)$ whose sides have lengths equal to the absolute periods of the handle we want to glue. This includes the case in which the parallelogram is degenerate, \textit{i.e.} with empty interior, and the surgery reduces to slit along a segment. We shall need the following Lemmas.

\begin{lem}\label{lem:eveninv}
If $(X,\omega)$ is a translation surface of even type and $(Y,\xi)$ is a translation surface obtained by bubbling a handle with non-positive volume, the $(Y,\xi)$ is of even type.
\end{lem}

\begin{proof}
Let $(X,\omega)$ be a translation surface of even type and let $\mathcal{P}\subset(X,\omega)$ be an embedded parallelogram. Since $(X,\omega)$ is a structure of even type, all vertices of $\mathcal{P}$ are either regular or singular of even order. Notice that a point of even order is of the form $(4m+2)\pi$, with $m=0$ for a regular point. Remove the interior of $\mathcal{P}$, identify the opposite edges, and let $(Y,\xi)$ be the resulting structure. The vertices of $\mathcal{P}$ are identified to a singular point of $\xi$ whose angle is $(4n+2)\pi$ for some $n$. Since the other zeros and all poles are untouched by this surgery, it follows that $(Y,\xi)$ is a structure of even type.
\end{proof}

\begin{lem}\label{lem:spininv2}
If $(X,\omega)$ is a translation surface of even type and $(Y,\xi)$ is a translation surface obtained by bubbling a handle with non-positive volume, then 
\begin{equation}
    \varphi(\omega)-\varphi(\xi)=0\,\,\,(\textnormal{ mod }2\,).
\end{equation}
In particular, bubbling a handle with non-positive volume does not alter the spin parity.
\end{lem}

\begin{proof}
Let $(X,\omega)$ be a translation surface and let $\mathcal{P}$ be an embedded parallelogram. Choose a collection of $2g$ oriented simple curves $\{\alpha_i,\,\beta_i\}_{1\le i\le g}$ on $(X,\omega)$ representing a symplectic basis for $\shomolzn$. Up to deforming the paths inside their isotopy classes we can make them stay away from some neighborhood, say $U$, of $\mathcal{P}$. In other words, we can assume that bubbling a handle with non-positive volume can be performed inside a neighborhood $U$. In particular, the surgery does not affect the initial collection of paths. We now complete the former collection of curves by adding two simple curves, say $\alpha_{g+1},\,\beta_{g+1}$ obtained after bubbling. These curves can be taken so that they both lie inside the neighborhood $U$. Since $(X,\omega)$ is a structure of even type, the vertices of $\mathcal{P}$ all have even orders. As a consequence, $\textnormal{Ind}(\alpha_{g+1})$ and $\textnormal{Ind}(\beta_{g+1})$ are odd positive integers, see Figure \ref{fig:parityneghandle}.

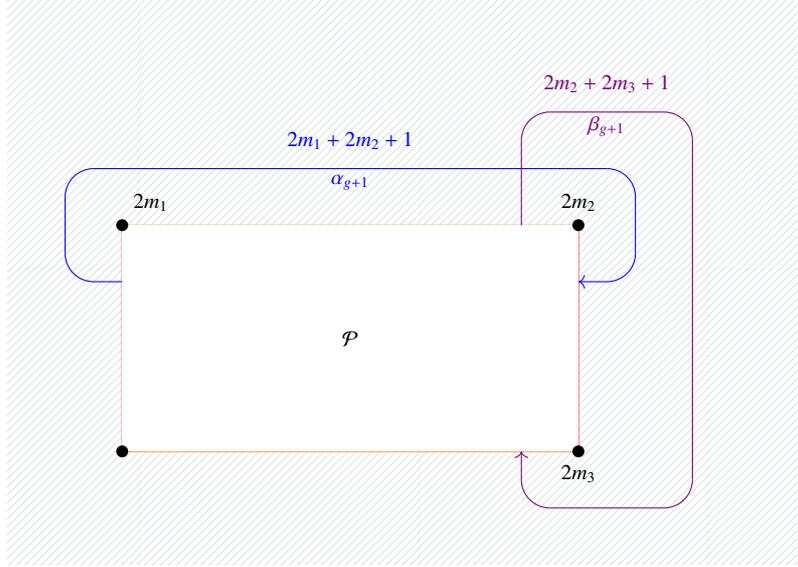
\begin{figure}[!ht] 
\centering
\begin{tikzpicture}[scale=1.5, every node/.style={scale=0.8}]
\definecolor{pallido}{RGB}{221,227,227}

    \pattern [pattern=north east lines, pattern color=pallido]
    (-3,3)--(4,3)--(4,-2)--(-3,-2)--(-3,3);
    
    \draw[thin, orange] (-2,1)--(2,1);
    \draw[thin, red] (-2,-1)--(-2,1);
    \draw[thin, orange] (-2,-1)--(2,-1);
    \draw[thin, red] (2,-1)--(2,1);
    
    \fill[white]
    (-2,1)--(2,1)--(2,-1)--(-2,-1)--(-2,1);
    
    \fill (2,1) circle (1.5pt);
    \fill (-2,1) circle (1.5pt);
    \fill (-2,-1) circle (1.5pt);
    \fill (2,-1) circle (1.5pt);
    
    \draw[thin, blue] (-2, 0.5)--(-2.25, 0.5);
    \draw [thin, blue] (-2.25,0.5) arc [start angle =270, end angle=180 , radius = 0.25];
    \draw[thin, blue] (-2.5,0.75)--(-2.5,1.25);
    \draw [thin, blue] (-2.5,1.25) arc [start angle =180, end angle=90 , radius = 0.25];
    \draw[thin, blue] (-2.25, 1.5)--(2.25, 1.5);
    \draw [thin, blue] (2.25,1.5) arc [start angle =90, end angle=0 , radius = 0.25];
    \draw[thin, blue] (2.5,0.75)--(2.5,1.25);
    \draw [thin, blue] (2.5,0.75) arc [start angle =0, end angle=-90 , radius = 0.25];
    \draw[thin, blue, ->] (2.25,0.5)--(2,0.5);
    
    \draw[thin, violet] (1.5,1)--(1.5,1.75);
    \draw[thin, violet] (1.5, 1.75) arc [start angle =180, end angle=90 , radius = 0.25];
    \draw[thin, violet] (1.75,2)--(2.75,2);
    \draw[thin, violet] (2.75, 2) arc [start angle =90, end angle=0 , radius = 0.25];
    \draw[thin, violet] (3,1.75)--(3,-1.25);
    \draw[thin, violet] (3, -1.25) arc [start angle =0, end angle=-90 , radius = 0.25];
    \draw[thin, violet] (2.75,-1.5)--(1.75,-1.5);
    \draw[thin, violet] (1.75, -1.5) arc [start angle =270, end angle=180 , radius = 0.25];
    \draw[thin, violet, <-] (1.5,-1)--(1.5,-1.25);
    
    \node at (0,0) {$\mathcal{P}$};
    \node at (2.25,2.25) {\textcolor{violet}{$2m_2+2m_3+1$}};
    \node at (0,1.75) {\textcolor{blue}{$2m_1+2m_2+1$}};
    \node at (2.25,1.875) {\textcolor{violet}{$\beta_{g+1}$}};
    \node at (0,1.375) {\textcolor{blue}{$\alpha_{g+1}$}};
    \node at (-1.75, 1.2) {$2m_1$};
    \node at (2, -1.2) {$2m_3$};
    \node at (2, 1.2) {$2m_2$};
    
\end{tikzpicture}
\caption{An embedded parallelogram $\mathcal{P}$ on a translation surface of even type. The vertices of $\mathcal{P}$ are points of orders $2m_i$, for $i=1,\dots,4$ and $m_i=0$ for regular points. A unit vector field along $\alpha_{g+1}$ winds $2m_1+2m_2+1$ times and a unit vector field along $\beta_{g+1}$ winds $2m_2+2m_3+1$ times. According to Definition \ref{index}, the indices of $\alpha_{g+1},\,\beta_{g+1}$ are equal to their winding numbers.}
\label{fig:parityneghandle}
\end{figure}

\noindent A direct computation shows that
\begin{align*}
    \varphi(\xi)&=\sum_{i=1}^{g+1} \big(\,\text{Ind}(\alpha_i)+1\big)\big(\,\text{Ind}(\beta_i)+1\big) \,\,(\text{mod}\,2)\\
    & = \sum_{i=1}^{g} \big(\,\text{Ind}(\alpha_i)+1\big)  \big(\,\text{Ind}(\beta_i)+1\big)+
    \big(\textnormal{Ind}(\alpha_{g+1})+1\big)\big( \textnormal{Ind}(\beta_{g+1}) +1\big)\,\,(\text{mod}\,2)\\
    &= \varphi(\omega) + \big(\textnormal{Ind}(\alpha_{g+1})+1\big)\big( \textnormal{Ind}(\beta_{g+1}) +1\big)\,\,(\text{mod}\,2)\\
    & =\varphi(\omega)\,\,(\text{mod}\,2)
\end{align*}
as desired.
\end{proof}

\smallskip

\subsection{Gluing surfaces along rays}\label{gluings} In what follows we shall need to glue translation surfaces along infinite rays. Recall, in fact, that our strategy to prove Theorem \ref{mainthm} is based on constructing translation surfaces with prescribed data (periods and geometric invariants) by gluing surfaces of lower-complexity. We begin by introducing the following

\begin{defn}[Gluing translation surfaces]\label{gluingsurfaces}
Let $(X_1,\,\omega_1)$ and $(X_2,\,\omega_2)$ be two translation surfaces, each with at least one pole. Let $r_i \subset (X_i,\,\omega_i)$, for $i=1,2$, be an embedded straight-line ray that starts from a singularity for $\omega_i$ or a regular point and ends in a pole. Assume that both $r_1$ and $r_2$ develop onto \textit{parallel} infinite rays on $\C$. Then we can define a translation surface $(Y,\,\xi)$ as follows: slit each ray $r_i$ and denote the resulting sides by $r_i^+$ and $r_i^-$; then identify $r_1^+$ with $r_2^-$ and $r_1^-$ and $r_2^+$ by a translation. If the surface $X_i$ is homeomorphic to $S_{g_i,\,k_i}$ for $i=1,2$, then the resulting surface $Y$ is homeomorphic to $S_{g_1+g_2,\, k_1+k_2-1}$ if the starting point of $r_i$ is not a pole for $\omega_i$, for $i=1,2$. In the case $r_i$ joins two poles, for $i=1,2$, then $Y$ is homeomorphic to $S_{g_1+g_2,\, k_1+k_2-2}$. Note that, in the former case, the starting points of the rays are identified to a branch point of $(Y,\,\xi)$ whose order is the sum of the corresponding orders of $(X_1,\,\omega_1)$ and $(X_2,\,\omega_2)$ respectively, and the other endpoints at infinity are identified to a higher order pole. 
\end{defn}

\noindent As a special case we have the following 

\begin{defn}[Bubbling a plane]\label{bubplane}
Let $(X,\,\omega)$ be a translation surface with poles and let $r\subset (X,\,\omega)$ be an infinite ray starting from a regular point or a zero of $\omega$. Let $\overline{r}\subset \mathbb C$ be the developing image of $r$. By \textit{bubbling a plane} along $r$ we mean the gluing of $(X,\,\omega)$ and $(\mathbb C,\,dz)$ along $r$ and $\overline{r}$ as described in Definition \ref{gluingsurfaces}. By repeating this surgery $m$ times is equivalent to glue $(X,\,\omega)$ to $(\mathbb C,\,z^{m-1}dz)$ along proper rays as described in Definition \ref{gluingsurfaces}.
\end{defn}

\noindent We can note that in the constructions above at most one new branch point is introduced arising from the extremal points of the rays after identification. Furthermore, bubbling a plane does not change the topology of the underlying surface; in particular the resulting structure after bubbling has the same periods as the former one. Another special case worth of interest for us is the following.

\begin{defn}[Gluing a cylinder along a bi-infinite ray]\label{gluecyl} 
Let $(X,\omega)$ be a translation surface with at least two poles. Let $r_1\subset (X,\omega)$ be an embedded bi-infinite geodesic ray joining two poles and let $v\in\C$ be the direction of $r$ once developed via the developing map. Let $C=\mathbb E^2/\langle z\mapsto z+w\rangle$ be an infinite cylinder with holonomy $w\in\C$. If $v\neq w$, then there exists an infinite geodesic line $r_2\subset C$ with direction $v$ after developing. We can glue $C$ to $(X,\omega)$ as follows: slit $(X,\omega)$ along $r_1$ and denote the resulting sides by $r_1^+$ and $r_1^-$. In a similar fashion, slit $C$ along $r_2$ and denote the resulting sides by $r_2^+$ and $r_2^-$. Let $(Y,\,\xi)$ be the translation surface obtained by gluing back $r_1^+$ with $r_2^-$ and $r_1^-$ with $r_2^+$. We shall say that $(Y,\,\xi)$ is obtained by \textit{gluing a cylinder to} $(X,\omega)$ \textit{along a bi-infinite ray}.
\end{defn}

\begin{rmk}
Notice that, according to Definition \ref{gluecyl}, the bi-infinite geodesic ray is allowed to pass through a branch point. Being the magnitude of the branch point greater than $2\pi$, unlike the situation at regular points, a geodesic segment with an endpoint at a branch point extends in infinitely many directions. In this case, we shall impose that the ray leaves the branch point with angle exactly $\pi$ on the left or on the right. Moreover, the residues of the resulting poles are given by the original residues $\pm$ the residues of the simple poles of the cylinder. 
\end{rmk}

\noindent The following results hold.

\begin{lem}[Invariance of the rotation number]\label{invrotnum}
Let $(X, \omega)$ be a genus one differential with at least two poles and rotation number $k$. Suppose there is a bi-infinite geodesic ray, say $r$, joining two poles. Then gluing a cylinder along $r$ as in Definition \ref{gluecyl} does not alter the rotation number.
\end{lem}

\begin{lem}[Invariance of the spin parity]\label{lem:invaspin}
Let $(X,\omega)$ be a translation surface with poles of even type and let $\varphi(\omega)$ be the parity of $\omega$. Let $(Y,\xi)$ be the translation surface obtained by bubbling $(\C,\,z^{2m-1}dz)$, for $m\ge1$, along a ray as in Definition \ref{gluingsurfaces}. Then $\varphi(\omega)=\varphi(\xi)$.
\end{lem}

\noindent The proofs of these Lemmas are easy to establish and left to the reader.

\smallskip

\section{Finding good systems of generators: actions on the representation space}\label{sec:mcga} 

\noindent For a given representation $\chi\colon\shomolzn\longrightarrow \C$, our proof of Theorem \ref{mainthm}, which we will develop from Section \S\ref{genusonemero} to Section \S\ref{sec:trirep}, relies on a direct construction of genus-$g$ differentials with prescribed invariants such as the rotation number for genus one differentials and the spin parity or the hyperellipticity for higher genus differentials. In order to perform our constructions, we need to consider a judicious \textit{system of handle generators} which is defined as follows.

\begin{defn}[Handle, handle-generators]\label{handle} On a surface $S_{g,n}$ of some positive genus $g>0$, a \textit{handle} is an embedded subsurface $\Sigma$ that is homeomorphic to $S_{1,1}$, and a \textit{handle-generator} is a simple closed curve that is one of the generators of $\text{H}_1(\Sigma, \mathbb{Z})\cong\shomolzoo$. A \textit{pair of handle-generators} for a handle will refer to a pair of simple closed curves $\{\alpha, \beta\}$ that generate $\text{H}_1(\Sigma, \mathbb{Z})$; in particular, $\alpha$ and $\beta$ intersect once. \end{defn} 

\begin{defn}[System of handle generators]\label{def:syshandle}
On a surface $S_{g,n}$ of some positive genus $g>0$, we consider a collection of pairwise disjoint $g$ handles $\Sigma_1,\dots,\Sigma_g$. A \textit{system of handle generators} is a collection of $g$ pairs of handle generators $\{\alpha_i,\,\beta_i\}_{1\le i\le g}$ such that $\{\alpha_i,\beta_i\}$ is a pair of handle generators for $\Sigma_i$.
\end{defn}

\noindent We can immediately notice that every system of handle generators yields a splitting, namely a simple closed separating curve $\gamma$ equal to the product of commutators $[\alpha_i,\beta_i]$. In fact, $S_{g,n}$ splits along $\gamma$ as the connected sum of a closed genus $g$ surface $S_g$ and a punctured sphere $S_{0,n}$. Conversely, once a splitting is defined, any representation $\chi\colon\shomolzn\to\C$ gives rise to a representation $\chi_g$ and a representation $\chi_n$ as defined in Section \S\ref{agv}. Recall that $\chi_n$ is always well-defined as it does not depend on the splitting whereas the representation $\chi_g$ does. In particular, $\chi_g$ is uniquely determined if and only if $\chi$ is of trivial-ends type, see Remark \ref{rmk:chigwelldef}.

\subsection{Mapping class group action}\label{ssec:mcgact}

\noindent In order to realize a representation $\chi$ as the period character of some translation surface with poles, in what follows it will be convenient to replace, if necessary, a given system of handle generators with a more suitable one. This replacement can be done by precomposing $\chi$ with an automorphism induced by a mapping class transformation, \textit{i.e.} an element of the mapping class group $\modul$. Thus, to prove our main Theorem \ref{mainthm} for a given $\chi$, it is sufficient to construct a genus $g$ meromorphic differential on $S_{g,n}$ for which the induced representation is $\chi \circ \phi$ for some $\phi \in \modul$. It is worth mentioning that this replacement is legitimized by the following
\begin{equation}
    \chi \in \textnormal{Per}\left( \mathcal{H}_g(m_1,\dots,m_k;-p_1,\dots,-p_n)\right) \,\,\Longleftrightarrow \,\, \chi\circ\phi \in \textnormal{Per}\left( \mathcal{H}_g(m_1,\dots,m_k;-p_1,\dots,-p_n)\right)
\end{equation} for any $\phi\in\modul$, and $\textnormal{Per}$ is the period mapping defined in \eqref{permap}. In fact, if $\phi$ is induced by a homeomorphism, say $f$, then the pullback of the translation structure by $f$ defines a new translation surface with poles in the same stratum and with period character $\chi\circ\phi$, see also \cite[Lemma 3.1]{fils} for the holomorphic case.

\smallskip

\subsection{Non-trivial systems of generators exist} We aim to list a few lemmas about the mapping class group action for finding systems of handle generators such that any element has a non-trivial period. Most of them has been already proved in \cite[Section \S 11]{CFG} and here we report a sketch of the proof for the reader's convenience. The following lemmas show that the image of every handle generator under $\chi$ can be assumed to be non-zero whenever $\chi$ is a non-trivial representation. 

\begin{lem}\label{lem:allhandholnonzero}
Let $\chi \in\text{\emph{Hom}}\left(\shomolzn, \,\mathbb{C}\right)$ be a representation and let $\{\alpha_i,\,\beta_i\}_{1\le i\le g}$ be a system of handle generators. Suppose that the corresponding $\chi_g$, determined by the induced splitting, is not trivial. Then, there exists $\phi \in \text{\emph{Mod}}(S_{g,n})$ such that $\chi \circ \phi(\alpha_i)$ and $\chi \circ \phi(\beta_i)$ are non-zero for all $1 \leq i \leq g$.
\end{lem}

\begin{proof}
Let $\chi$ be a non-trivial representation and let $\{\alpha_i,\,\beta_i\}_{1\le i\le g}$ be a system of handle generators. We can assume that $\chi(\alpha_1) \neq 0$. If $\chi(\beta_1)=0$, then replace $\beta_1$ with $\alpha_1\beta_1$ and observe that $\chi(\alpha_1\beta_1)\neq0$. Suppose there is an index $i$ such that $\chi(\alpha_i) = \chi(\beta_i) = 0$. We consider a mapping class $\phi$ such that 
\begin{equation}
   \phi(\alpha_1)=\alpha_1\alpha_i, \,\,\, \phi(\beta_1)=\beta_1, \,\,\, \phi(\alpha_i)=\alpha_i, \,\,\, \phi(\beta_i)=\beta_1\beta_i^{-1}, 
\end{equation}
and $\phi$ is the identity on the other handle generators. Finally, we replace the generator $\alpha_i$ with $\alpha_i\beta_1\beta_i^{-1}$. Observe that $\{\alpha_i\,\beta_1\beta_i^{-1},\, \beta_1\beta_i^{-1}\}$ is a pair of handle generators and $\chi(\alpha_i\,\beta_1\beta_i^{-1})=\chi(\beta_1\beta_i^{-1})\neq0$ by construction. By iterating this process finitely many times we get the desired result.
\end{proof}

\noindent Furthermore, for a representation $\chi$ we can also assume that the $\chi_g$-part induced by a splitting is non-trivial whenever the $\chi_n$-part is non-trivial. 

\begin{lem}\label{lem:handholnonzero}
Let $\chi \in\text{\emph{Hom}}\left(\shomolzn, \,\mathbb{C}\right)$ be a representation of non-trivial-ends type. Then, there exists a mapping class $\phi \in \text{\emph{Mod}}(S_{g,n})$ such that $(\chi \circ \phi)_g$ is non-trivial.
\end{lem}

\begin{proof}
Let $\chi$ be a non trivial-ends type representation and let $\{\alpha_i,\,\beta_i\}_{1\le i\le g}$ be a system of handle generators. Assume $\chi(\alpha_i) = \chi(\beta_i) = 0$ for all $1 \leq i \leq g$. Let $\gamma_i$ be a small loop around a puncture such that $\chi(\gamma_i) \neq 0$. Then we can find some handle generator $\alpha_j$ such that we have two curves $\alpha_j$ and $\alpha'_j$ satisfying $\alpha_j\alpha'_j = \gamma_i$ in $\shomolzn$. This implies that $\chi(\alpha'_j) \neq 0$. Thus, we can take $\phi$ to be an element of $\text{Mod}(S_{g,n})$, commonly known as a \textit{push transformation}, see \cite[Section \S4.2.1]{FM}, which takes the generator $\alpha_j$ to $\alpha'_j$ and leaves the other handle generators unchanged.
\end{proof}

\subsection{$\glplus$-action}\label{ssec:glact} We now consider an action on the representation space given by post-composition with elements of $\glplus$. This action can be combined with the mapping class group action, see Section \S\ref{ssec:mcgact} above, and it is easy to check that these actions commute. In the sequel it will be sometimes useful to consider this action in order to realize representations as the period of meromorphic differentials in a given connected component of a stratum. This is in fact legitimated by the fact
\begin{equation}
    \chi \in \textnormal{Per}\left( \mathcal{H}_g(m_1,\dots,m_k;-p_1,\dots,-p_n)\right) \,\,\Longleftrightarrow \,\, A\,\chi\in \textnormal{Per}\left( \mathcal{H}_g(m_1,\dots,m_k;-p_1,\dots,-p_n)\right)
\end{equation} for any $A\in\glplus$, and $\textnormal{Per}$ is the period mapping defined in \eqref{permap}. 

\begin{rmk}
Since $\glplus$ acts continuously on every stratum, it follows that the connected components, if any, are preserved under the action.  
\end{rmk}

\smallskip

\subsection{System of handle generators with positive volume}

\noindent Notice that a mapping class $\phi\in\modul$ does not need to preserve any splitting in general. Let us now consider again the notion of volume for meromorphic differentials already introduced in Section \S\ref{agv}. For a representation $\chi\colon\shomolzn\longrightarrow \C$, the volume of $\chi$, see Definition \ref{algvoldef2}, is well-defined if and only if it is of trivial-ends type. In other words, for any representation $\chi$ of trivial-ends type and for any mapping class $\phi\in\modul$ the equation 
\begin{equation}
    \textnormal{vol}(\chi_g)=\textnormal{vol}\left( (\chi\circ\phi)_g\right)
\end{equation}
where $\chi_g$ and $(\chi\circ\phi)_g$ are the representations induced by the systems of handle generators $\{\alpha_i,\,\beta_i\}_{1\le i\le g}$ and $\{\phi(\alpha_i),\,\phi(\beta_i)\}_{1\le i\le g}$ respectively. For a representation $\chi$ of non-trivial-ends type, the volume is no longer well-defined, and hence it is no longer an invariant, because the representation $\chi_g$ does depend on the splitting and therefore also $\text{vol}(\chi_g)$ depends on it. We shall take advantage of this caveat to prove the following result which will be exploited in the sequel. We begin with introducing the following

\begin{defn}\label{def:realcoll}
A representation $\chi$ is called \textit{real-collinear} if the image $\text{Im}(\chi)$ is contained in the $\mathbb R$-span of some $c\in\C^*$. Equivalently, $\chi$ is real-collinear if, up to replacing $\chi$ with $A\,\chi$ where $A\in\glplus$ if necessary, then $\textnormal{Im}(\chi)\subset \mathbb {R}$.
\end{defn}

\begin{lem}\label{lem:posvol}
Let $\chi \in\text{\emph{Hom}}\left(\shomolzn, \mathbb{C}\right)$ be a representation of non-trivial-ends type. If $\chi$ is not real-collinear then there exists a system of handle generators $\{\alpha_i,\,\beta_i\}_{1\le i\le g}$ such that $\textnormal{vol}(\chi_g)>0$.
\end{lem}

\begin{proof}
Let $\chi$ be a representation of non-trivial-ends type and let $\{\alpha_i,\,\beta_i\}_{1\le i\le g}$ be a system of handle generators. From Section \S\ref{agv}, recall that the curve $\gamma=[\alpha_1,\beta_1]\cdots[\alpha_g,\beta_g]$ bounds a sub-surface $\Sigma$ homeomorphic to $S_{g,1}$, the embedding $i\colon\Sigma\hookrightarrow S_{g,n}$ yields an injection $\imath_g\colon\text{H}_1(\Sigma,\,\mathbb{Z})\cong\shomolz\to\shomolzn$, and the representation $\chi_g$ is defined as $\chi\circ\imath_g$. In the case $\textnormal{vol}(\chi_g)>0$ there is nothing to prove and we are done, so let us suppose that $\textnormal{vol}(\chi_g)\le0$. Being $\chi$ is of non-trivial-ends type, there is a puncture, say $p$ with non-zero residue and let $\gamma$ be a simple closed curve around $p$. Clearly, $\chi(\gamma)=\chi_n(\gamma)\neq0$. Since $\chi$ is not real-collinear, we choose a pair of handle generators $\{\alpha,\beta\}$ in the given system such that $\chi(\beta)$ is not collinear with $\chi(\gamma)$. Replace it with the pair of generators $\{\alpha',\,\beta'\}$ where $\beta'=\beta$ and $\alpha'=\alpha\,\beta^{-n}\,\gamma\,(\beta\,\gamma)^n$ for $n\in\mathbb Z$. This is a simple closed non-separating curve and we can easily compute that
\begin{equation}
    \chi\left(\alpha\,\beta^{-n}\,\gamma\,(\beta\,\gamma)^n\,\right)=\chi(\alpha)+(n+1)\chi(\gamma).
\end{equation}
\noindent Up to relabelling all handles, we can assume $\{\alpha,\,\beta\}=\{\alpha_1,\,\beta_1\}$. The new system of handle generators will be $\{\alpha',\,\beta',\alpha_2,\beta_2,\dots,\alpha_g,\,\beta_g\}$. Such a system determines a different splitting, namely a different separating curve $\gamma'$ that bounds a sub-surface $\Sigma'$ homeomorphic to $S_{g,1}$. The embedding $i'\colon\Sigma'\hookrightarrow S_{g,n}$ yields a representation $\chi'_g=\chi\circ \imath'_g$, where $\imath'_g$ is the injection in homology induced by $i'$. The volume $\textnormal{vol}(\chi'_g)$ can be explicitly computed as follow
\begin{align}
    \textnormal{vol}(\chi'_g) & =  \Im\left(\,\overline{\chi(\alpha')}\,\chi(\beta')\,\right)+\sum_{i=2}^g \Im \left(\,\overline{\chi(\alpha_i)}\,\chi(\beta_i)\,\right)\\
     & = \Im\left(\,\overline{\chi(\alpha_1)+(n+1)\chi(\gamma)}\,\chi(\beta_1)\,\right)+\sum_{i=2}^g \Im \left(\,\overline{\chi(\alpha_i)}\,\chi(\beta_i)\,\right)\\
     &= (n+1)\Im\left(\,\overline{\chi(\gamma)}\,\chi(\beta_1)\,\right)+ \sum_{i=1}^g \Im \left(\,\overline{\chi(\alpha_i)}\,\chi(\beta_i)\,\right)\\
     &= (n+1)\Im\left(\,\overline{\chi(\gamma)}\,\chi(\beta_1)\,\right)+\textnormal{vol}(\chi_g).
\end{align}
Since $\textnormal{vol}(\chi_g)$ is constant and since $\chi(\gamma)$ and $\chi(\beta_1)$ are not collinear, by choosing $n\in\mathbb Z$ judiciously, we can make $\textnormal{vol}(\chi'_g)>0$ and hence obtain the desired result. \qedhere
\end{proof}

\noindent By combining Lemma \ref{lem:posvol} with Kapovich's results in \cite{KM2}, we can infer the following result on which we shall rely in the sequel.

\begin{cor}\label{cor:posvol}
Let $\chi \in\text{\emph{Hom}}\big(\shomolzn, \mathbb{C}\big)$ be a representation of non-trivial-ends type. If $\chi$ is not real-collinear then there exists a system of handle generators $\mathcal{G}=\{\alpha_i,\,\beta_i\}_{1\le i\le g}$ such that $\Im\big(\,\overline{\chi(\alpha_i)}\,\chi(\beta_i)\,\big)>0$ for any $i=1,\dots,g$.
\end{cor}

\smallskip

\subsection{Discrete and dense representations} In the sequel we shall distinguish three kinds of representations according to the following

\begin{defn}\label{defn:kindofreps}
A representation $\chi\colon\shomolzn\longrightarrow \C$ is said to be 
\begin{itemize}
    \item \textit{discrete} if the image of $\chi$ is a discrete subgroup of $\C$. Furthermore, we say that $\chi$ is discrete of rank \textit{one} if, up to replacing $\chi$ with $A\,\chi$ where $A\in\glplus$, then $\text{Im}(\chi)=\Z$. We say that $\chi$ is discrete of rank \textit{two} if, up to replacing $\chi$ with $A\,\chi$ where $A\in\glplus$, then $\text{Im}(\chi)=\Z\oplus i\,\Z$.
    \smallskip
    \item \textit{semi-discrete} if, up to replacing $\chi$ with $A\,\chi$ where $A\in\glplus$, then $\text{Im}(\chi)=U\oplus i\,\Z$, where $U$ is dense in $\mathbb R$.
    \smallskip
    \item \textit{dense} if the image of $\chi$ is dense in $\C$.
\end{itemize}
\end{defn}

\noindent We have the following Lemmas.

\begin{lem}\label{lem:disnormalform}
Let $\chi$ be a discrete representation of non-trivial-ends type. After replacing $\chi$ with the representation $A\,\chi$ where $A\in\glplus$ if necessary, then there exists a system of handle generators $\{\alpha_i,\,\beta_i\}_{1\le i\le g}$ such that

\begin{itemize}
    \item[1.] if $\chi$ is discrete of rank one then each handle generator has period one, \textit{i.e.} $\chi(\alpha_i)=\chi(\beta_i)=1$ for all $i=1,\dots,g$;
    \smallskip
    \item[2.] if $\chi$ is discrete of rank two, then the handle generators satisfy the following conditions:
    \begin{itemize}
    \item $\chi(\alpha_g)\in\Z_+$ and $\chi(\beta_g)=i$,
    \smallskip
    \item $0<\chi(\alpha_j)<\chi(\alpha_g)$ and $\chi(\alpha_j)=\chi(\beta_j)\in\Z_+$ for all $j=1,\dots,g-1$.
\end{itemize}
\end{itemize}
\end{lem}

\smallskip

\begin{proof}
Up to normalizing the representation with some $A\in\glplus$, we can assume $\text{Im}(\chi)\subset\Z\oplus i\,\Z$. The first case follows from \cite[Lemma 12.2]{CFG}. Let us consider the second case. By Lemma \ref{lem:posvol} there exists a system of handle generators $\{\alpha_i,\,\beta_i\}_{1\le i\le g}$ such that $\textnormal{vol}(\chi_g)>0$. Let us now consider the representation $\chi_g$ on its own right. Recall that we can regard this representation as $\chi_g\colon\shomolz\longrightarrow \C$. From \cite{KM2}, there exists a system of handle generators such that 
    \begin{itemize}
    \item $\chi(\alpha_g)\in\Z_+$ and $\chi(\beta_g)=i$,
    \smallskip
    \item $0<\chi(\alpha_j)<\chi(\alpha_g)$ and $\chi(\beta_j)=0$ for all $j=1,\dots,g-1$.
\end{itemize}
By replacing $\{\alpha_j,\,\beta_j\}$ with $\{\alpha_j,\,\alpha_j\,\beta_j\}$ for $j=1,\dots,g-1$ we get the desired result.
\end{proof}

\smallskip

\begin{lem}\label{lem:smallperiods}
Let $\chi\in\textnormal{Hom}\big(\shomolzn, \mathbb{C}\big)$ be a real-collinear representation of non-trivial-ends type. Let $\{\alpha_i,\,\beta_i\}_{1\le i\le g}$ be a system of handle generators. If $\chi$ is not discrete of rank one, there exists a mapping class $\phi\in\modul$ such that $\chi\circ\phi(\alpha_i)$ and $\chi\circ\phi(\beta_i)$ are both positive and arbitrarily small.
\end{lem}

\begin{proof}
Let $\chi$ be a representation of non-trivial-ends type such that $\textnormal{Im}(\chi)\subset \mathbb R$. Recall that any system of handle generators $\{\alpha_i,\,\beta_i\}_{1\le i\le g}$ induces a splitting and hence a representation $\chi_n$. Notice that, since $\chi$ is not discrete of rank one, at least one generator of $\shomolzn$ has an irrational period. If $\textnormal{Im}(\chi_n)\subset \mathbb Q$, then the result follows from \cite[Lemma 11.5]{CFG}; notice that this is always the case when $n=2$ (up to rescaling). Let us now assume $\textnormal{Im}(\chi_n)\not\subset \mathbb Q$ and $n\ge3$. Up to rescaling, we can assume one puncture, say $P_1$, has a rational residue, say $r_1$. Necessarily, there is a puncture $P_2$ with irrational residue $r_2$. We can assume that at least one handle generator has an  irrational period, otherwise we can change the system of handle generators above by replacing $\alpha_1$ with $\alpha_1\gamma_2$, where $\gamma_2$ is a loop around the puncture $P_2$. Notice that such a replacement can be performed with a mapping class transformation. Therefore we can assume that $\alpha_1$ has an irrational period. There exists $\phi_i\in\modul$ such that $\chi\circ\phi_i(\alpha_i)$ is also irrational. This mapping class can be explicitly written as
\begin{equation}
    \phi_i(\alpha_i)=\alpha_1\,\alpha_i, \quad \phi_i(\beta_1)=\beta_1\,\beta_i^{-1}, \quad 
    \phi_i(\delta)=\delta \quad \text{ for }\,\, \delta\notin\{\alpha_i,\beta_1\}.
\end{equation}
By composing all these $\phi_i$'s, the resulting mapping class $\phi$ provides a system of handle generators such that $\chi(\alpha_i)$ is irrational for any $i=1,\dots,g$. We now proceed as follows. If $\chi(\alpha_i)$ and $\chi(\beta_i)$ are linearly independent over $\Q$, we can  assume that both of them are positive after Dehn twists and we apply the Euclidean algorithm for making them arbitrarily small. In the case $\chi(\alpha_i)$ and $\chi(\beta_i)$ are linearly dependent over $\Q$ (hence $\chi(\beta_i)$ is also irrational), we replace $\beta_i$ with $\beta_i\gamma_1$. Since $\chi(\gamma_1)\in\mathbb Q$, then $\chi(\alpha_i)$ and $\chi(\beta_i\,\gamma_1)$ are linearly independent over $\Q$. Again, we can assume that both of them are positive after Dehn twists and we apply the Euclidean algorithm for making them arbitrarily small.
\end{proof}

\section{Meromorphic differentials of genus one}\label{genusonemero}

\noindent In this section we want to realize a given non-trivial representation $\chi\colon\shomolzon\longrightarrow \C$ as the period of some translation surface with poles in a prescribed stratum and with prescribed rotation number, see Definition \ref{rotnum}. More precisely we shall prove Theorem \ref{mainthm} for meromorphic differentials of genus one, namely we prove the following:

\begin{prop}\label{prop:genusone}
Let $\chi$ be a non-trivial representation and suppose it arises as the period character of some meromorphic differential in a stratum $\mathcal{H}_1(m_1,\dots,m_k;-p_1,\dots,-p_n)$\,. Then $\chi$ can be realized as the period character of some translation surfaces with poles in each connected component of the same stratum.
\end{prop}

\begin{rmk}\label{conscases}
Recall that for a stratum $\mathcal{H}_1(m_1,\dots,m_k;-p_1,\dots,-p_n)$, the possible values of the rotation number are given by the positive divisors of $\gcd\big(m_1,\dots,m_k;p_1,\dots,p_n \big)$, see Section \S\ref{mscc} and \cite[Chapter 4]{BC}. We can notice that each stratum of meromorphic differentials with at least one simple zero or pole is connected as the the greatest common divisor of the orders cannot be larger than one. Recall that the problem of realizing period characters in a given stratum (regardless of which connected component) has already been considered in \cite{CFG}. 
\end{rmk} 

\noindent In order to state our results, we introduce the following

\smallskip

\textit{Convention and terminology.} Recall that slitting a surface along an oriented geodesic segment $s$ is a topological surgery for which the interior of $s$ is replaced with two copies of itself. On the resulting surface, these two segments, form a piecewise geodesic boundary with two corner points, corresponding to the extremal points of $s$, each of which of angle $2\pi$. We shall denote by $s^+$ the piece of boundary which bounds the surface on its right with respect to the orientation induced by $s$. In a similar fashion, we denote by $s^-$ the piece of boundary which bounds the surface on its left with respect to the orientation induced by $s$. In the following constructions we shall need to slit and glue surfaces along geodesic segments. In order for this operation to be done we shall need to glue along segments which are parallel after developing, see Section \S\ref{gluings}. For any $c\in\C^*$, by \textit{slitting $\C$ along $c$} we shall mean a cut along any geodesic segment of length $|c|$ and slope equal to $\arg(c)$. As we shall need to consider several slit copies of $(\mathbb{C},\, dz)$, along different geodesic segments, we make use of an index $i$ to specify on which copy we  perform the slit. Finally, given a pair of handle generators $\alpha,\,\beta$ and a representation $\chi$, we shall use the Latin letter $a$ for any segment parallel and isometric to $\chi(\alpha)$ and the Latin letter $b$ for any segment parallel and isometric to $\chi(\beta)$. 

\medskip

\noindent The proof of Theorem \ref{mainthm} for  meromorphic differentials of genus one is divided as follows. In the next three sections \S\ref{genusoneordertwo}-\S\ref{resnotzero} we prove the main Theorem for strata of genus one meromorphic differentials with exactly one zero of maximal order. In the final section \S\ref{morezeros} we derive the most general case as a straightforward corollary of the other sections. The only missing case in the present section is the trivial representation that will be handled in Section \S\ref{sec:trirep}. Since this a case-by-case proof with several sub-cases, Appendix \S\ref{appfdgo} contains a flow diagram of the proof for the reader's convenience, see Table \ref{fig:flowdiagram}.

\subsection{Strata with poles of order two}\label{genusoneordertwo} In this section we shall prove the following lemma for the case where all poles have order exactly two and zero residue. We propose a proof in which the structures are explicitly constructed with prescribed periods and rotation number. As we shall see, all the other cases follow as simple variations of the constructions in this special case.

\begin{lem}\label{onepuncttorusp2}
Let $\chi\colon\shomolzon\longrightarrow \C$ be a non-trivial representation of trivial-ends type. If $\chi$ can be realized in the stratum $\mathcal{H}_1(2n;-2,\dots,-2)$ then it appears as the period character of a translation surface with poles in each connected component. 
\end{lem}

\noindent Let $\chi\colon\shomolzon\longrightarrow \C$ be a non-trivial representation and assume $\chi$ can be realized as the period character of a genus one meromorphic differential with poles in the stratum $\mathcal{H}_1(2n;-2,\dots,-2)$. Recall that any such a stratum has two connected components according to the rotation number with the only exception being the stratum $\mathcal{H}_1(2,-2)$ which is connected and hence already handled in \cite{CFG}. Therefore we assume without loss of generality that $n\ge2$.

\subsubsection{Realizing representations with rotation number one}\label{rot1} We begin with realizing $\chi$ as the period character of some genus one meromorphic differential with rotation number \textit{one}. We shall distinguish two cases according to the volume of the representation $\chi$.

\medskip

\paragraph{\textit{Positive volume}}\label{pk1} Let $\alpha,\beta$ be a pair of handle generators and let $\mathcal{P}\subset \mathbb{C}$ be the parallelogram defined by the chain
\begin{equation}
    P_0\mapsto P_0+\chi(\alpha)\mapsto P_0+\chi(\alpha)+\chi(\beta)=Q_0\mapsto P_0+\chi(\beta)\mapsto P_0,
\end{equation}

\noindent where $P_0\in\C$ is any point. According to our convention above, let us denote by $a_0^+$ (respectively $a_0^-$) the edge of $\mathcal{P}$ parallel to $\chi(\alpha)$ that bounds the parallelogram on its right (respectively left). In the same fashion, let us denote by $b_0^+$ (respectively $b_0^-$) the edge of $\mathcal{P}$ parallel to $\chi(\beta)$ that bounds the parallelogram on its right (resp. left). Recall that $(\mathbb{C},\, dz)$ is a genus zero translation surface with trivial period character, no zeros and one pole of order $2$ at the infinity with zero residue. We slit $(\mathbb{C},\,dz)$ along a segment $a$ and we denote the resulting segments as $a_n^+$ and $a_n^-$. Let $P_n$ and $Q_n=P_n+\chi(\alpha)$ be the extremal points of $a_n$. We next consider other $n-1$ copies of $(\mathbb{C},\, dz)$ and we slit each of them along a segment $b$. For any $i=1,\dots,n-1$, let $P_i$ and $Q_i=P_i+\chi(\beta)$ be the extremal points of $b_i$ and denote as $b_i^+$ and $b_i^-$ the resulting segments after slitting. We are now ready to glue these $n$ copies of $(\C,dz)$ and the parallelogram $\mathcal{P}$ all together. The desired structure is then obtained by identifying the segment $a_0^+$ with $a_n^-$, the segment $a_0^-$ with $a_n^+$, the segment $b_i^+$ with $b_{i+1}^-$ where $i=1,\dots,n-2$ and, finally, $b_0^-$ with $b_{n-1}^+$ and  $b_0^+$ with $b_1^-$, see Figure \ref{rotpk1}. 

\begin{figure}[ht]
    \centering
    \begin{tikzpicture}[scale=1.2, every node/.style={scale=0.8}]
    \definecolor{pallido}{RGB}{221,227,227}
    
    \pattern [pattern=north west lines, pattern color=pallido]
    (-4.5,-0.25)--(-0.5,-0.25)--(-0.5,2.25)--(-4.5,2.25)--(-4.55,-0.25);
    \draw [thick, violet, ->] (-3.25, 1) arc [start angle = 0, end angle = 345,radius = 0.25];
    \draw [thick, orange] (-3.5,1) to (-1.5,1);
    \draw [thick, orange, ->] (-3.5,1) to (-2.5,1);
    \draw [thick, black, dotted] (-1.5,1) to (-0.5,1);
    \fill (-3.5,1) circle (1.5pt);
    \fill (-1.5,1) circle (1.5pt);
    \node at (-2.5, 0.75) {$a_n^+$};
    \node at (-2.5, 1.25) {$a_n^-$};
    
    \pattern [pattern=north west lines, pattern color=pallido]
    (0,0)--(2,0)--(2,2)--(0,2)--(0,0);
    \draw [thick, orange] (0,0)--(2,0);
    \draw [thick, red] (2,0)--(2,2);
    \draw [thick, orange] (2,2)--(0,2);
    \draw [thick, red] (0,2)--(0,0);
    \draw [thick, orange, ->] (0,0) to (1,0);
    \draw [thick, red, ->>] (2,0) to (2,1);
    \draw [thick, orange, ->] (0,2) to (1,2);
    \draw [thick, red, ->>] (0,0) to (0,1);
    \fill (0,0) circle (1.5pt);
    \fill (0,2) circle (1.5pt);
    \fill (2,2) circle (1.5pt);
    \fill (2,0) circle (1.5pt);
    \draw [thick, violet, ->] (0.25,0) to (0.25, 2);
    \draw [thick, blue, ->] (0, 0.25) to (2, 0.25);
    \node at (1, -.25) {$a_0^-$};
    \node at (1, 2.25) {$a_0^+$};
    \node at (-0.25, 1) {$b_0^+$};
    \node at (2.25, 1) {$b_0^-$};
 
    \foreach \x [evaluate=\x as \coord using 2.5 + 3*\x] in {0, 1} 
    {
    \pattern [pattern=north west lines, pattern color=pallido]
    (\coord, -0.75)--(\coord+2.5, -0.75)--(\coord+2.5, 2.75)--(\coord, 2.75)--(\coord,-0.75);
    \draw [thick, blue, ->] (\coord+1.25, 0.25) arc [start angle = 90, end angle = -260,radius = 0.25];
    \draw [thick, red] (\coord+1.25, 0) to (\coord+1.25, 2);
    \draw [thick, red, ->>] (\coord+1.25,0) to (\coord+1.25,1);
    \fill (\coord+1.25,0) circle (1.5pt);
    \fill (\coord+1.25,2) circle (1.5pt);
    \node at (\coord+0.25, -0.5) {$\mathbb{C}$};   
    \draw [thick, black, dotted] (\coord+1.25, 2)--(\coord+1.25, 2.75);
    }
    
    \node at (3.5, 1) {$b_1^-$};
    \node at (4, 1) {$b_1^+$};
    \node at (6.375, 1) {$b_{n-1}^-$};
    \node at (7.125, 1) {$b_{n-1}^+$};
    
    \node at (5.25,1) {$\cdots$};
    \node at (1.5, 1.5) {$\mathcal{P}$};
    \node at (-4.25,0) {$\mathbb{C}$};
    \node at (1,.375) {$\alpha$};
    \node at (0.375,1) {$\beta$};
   
    \end{tikzpicture}
    \caption{Realizing a genus one translation surface with poles of order $2$, positive volume, and rotation number equal to $1$. In this case all poles are assumed to have zero residue.}
    \label{rotpk1}
\end{figure}
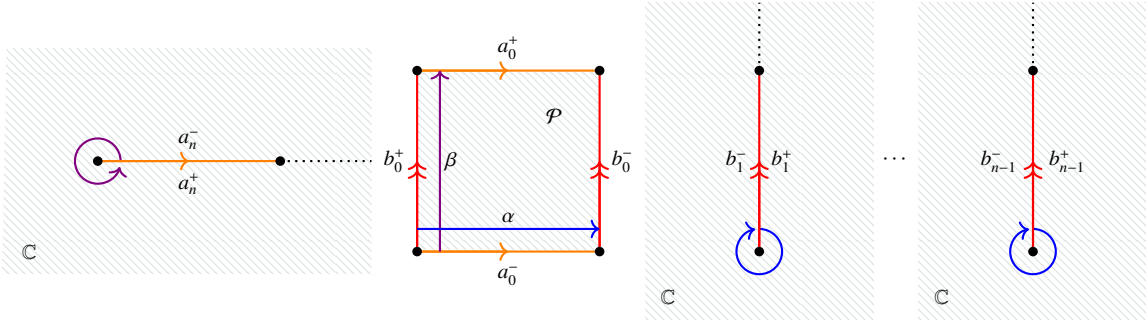

\noindent We will show that the resulting genus one meromorphic differential, say $(X,\omega)$, has rotation number one. We define $\alpha$ in the following way. For any $i=1,\dots,n-1$ we consider a small metric circle of radius $\varepsilon$ centered at $P_i$. After the cut and paste process described above, these $n-1$ circles define a smooth path with starting point on $b_0^-$ and ending point on $b_0^+$. These two points on $\mathcal{P}$ differ by $\chi(\alpha)$ and hence we can join them with a geodesic segment parallel to $a$. The resulting curve is a simple close curve on $(X,\omega)$, see Figure \ref{rotpk1}. It is easy to see that it has index $\text{Ind}(\alpha)=n-1$. In a similar way we can define $\beta$. We consider a small metric circle of radius $\varepsilon$ centered at $P_n$. After the cut and paste process, the extremal points of such an arc are on $a_0^\pm$ and differ by $\chi(\beta)$ on $\mathcal{P}$. We join them with a geodesic segment and the resulting curve is simple and closed in $(X,\omega)$. By construction, $\beta$ has index one. Therefore, according to the formula \eqref{rotnum} we have that $\gcd(\,n-1,\,1,\,2n,\,2\,)=1$ and hence the structure has rotation number one.

\medskip

\paragraph{\textit{Non-positive volume}}\label{nk1} We now assume that $\chi$ has non-positive volume. This case is similar to the case of positive volume. As above, let $\alpha,\beta$ be a set of handle generators and let $\mathcal{P}\subset \mathbb{C}$ be the closure of the \textit{exterior} of the parallelogram 
\begin{equation}
    P_0\mapsto P_0+\chi(\alpha)\mapsto P_0+\chi(\alpha)+\chi(\beta)=Q_0\mapsto P_0+\chi(\beta)\mapsto P_0,
\end{equation}
\noindent where $P_0\in\C$ is any point. Notice that $\mathcal{P}$ itself is a topological quadrilateral on $\cp$. We denote the sides of $\mathcal{P}$ as $a^{\pm}$ and $b_0^{\pm}$ according to our convention. We next consider $n-1$ copies of $(\mathbb{C},\, dz)$ and we slit each of them along a segment $b$. For any $i=1,\dots,n-1$, we denote as $b_i^+$ and $b_i^-$ the resulting segments after slitting and let $P_i$ and $Q_i=P_i+\chi(\beta)$ be the extremal points of $b_i$. The desired structure is then obtained by identifying the segment $a^+$ with $a^-$, the segment $b_{i}^-$ with $b_{i+1}^+$ for $i=1,\dots,n-2$, the segment $b_0^-$ with $b_{1}^+$ and the segment $b_0^+$ with $b_{n-1}^-$, see Figure \ref{rotnk1}. In this case it is still easy to observe that the final structure $(X,\omega)$ has rotation number one. We define $\alpha$ similarly as before. For $\beta$, we can choose any segment in $\mathcal{P}$ parallel to $b_0^+$ with extremal points $R$ and $R+\chi(\beta)$. We join $R$ (respectively $R+\chi(\beta)$) with any point in $a^-$ (respectively $a^+$) by means of an embedded arc. The resulting curve close up to a simple closed curve on $(X,\omega)$ and has index one. Therefore the structure $(X,\omega)$ has rotation number one.

\begin{figure}[ht]
    \centering
    \begin{tikzpicture}[scale=1.2, every node/.style={scale=0.8}]
    \definecolor{pallido}{RGB}{221,227,227}
   
    \pattern [pattern=north west lines, pattern color=pallido]
    (-1.5,-1)--(3,-1)--(3,3)--(-1.5,3)--(-1.5,-1);
    \fill [white] (0,0)--(2,0)--(2,2)--(0,2)--(0,0);
    \draw [thick, dotted, black] (2,0) -- (3,-1);
    \draw [thick, orange] (0,0)--(2,0);
    \draw [thick, red] (2,0)--(2,2);
    \draw [thick, orange] (2,2)--(0,2);
    \draw [thick, red] (0,2)--(0,0);
    \draw [thick, orange, ->] (0,0) to (1,0);
    \draw [thick, red, ->>] (2,2) to (2,1);
    \draw [thick, orange, ->] (0,2) to (1,2);
    \draw [thick, red, ->>] (0,2) to (0,1);
    \fill (0,0) circle (1.5pt);
    \fill (0,2) circle (1.5pt);
    \fill (2,2) circle (1.5pt);
    \fill (2,0) circle (1.5pt);
    \node at (1,2.25) {$a^-$};
    \node at (1,-0.25) {$a^+$};
    \node at (2.25,1) {$b_0^-$};
    \node at (-0.25,1) {$b_0^+$};
    
    \draw [thick, violet, ->] (0.25,2.0625) arc (50:310:0.75 and 1.375);
    
    \draw [thick, blue, <-] (2.0625,1.75) arc (-45:225:1.5 and 0.5);
 
    \foreach \x [evaluate=\x as \coord using 4 + 3.5*\x] in {0, 1} 
    {
    \pattern [pattern=north west lines, pattern color=pallido]
    (\coord, -1)--(\coord+2.5, -1)--(\coord+2.5, 3)--(\coord, 3)--(\coord,-1);
    \draw [thick, blue, ->] (\coord+1.25, 1.75) arc [start angle = 270, end angle = -80,radius = 0.25];
    \draw [thick, red] (\coord+1.25, 0) to (\coord+1.25, 2);
    \draw [thick, red, <<-] (\coord+1.25,1) to (\coord+1.25,2);
    \fill (\coord+1.25,0) circle (1.5pt);
    \fill (\coord+1.25,2) circle (1.5pt);
    \node at (\coord+0.25, -0.75) {$\mathbb{C}$};   
    \draw [thick, black, dotted] (\coord+1.25, 0)--(\coord+1.25, -1);
    }
    
    \node at (7,1) {$\cdots$};
    \node at (4.875, 1) {$b_1^+$};
    \node at (5.625, 1) {$b_1^-$};
    \node at (8.375, 1) {$b_{n-1}^+$};
    \node at (9.125, 1) {$b_{n-1}^-$};
    \node at (-1,-0.75) {$\mathcal{P}\subset\cp$};
    \node at (1,2.75) {$\alpha$};
    \node at (-0.75,1) {$\beta$};
    \end{tikzpicture}
    \caption{Realizing a genus one translation surface with poles of order $2$, non-positive volume, and rotation number equal to $1$. In this case all poles are assumed to have zero residue.}
    \label{rotnk1}
\end{figure}
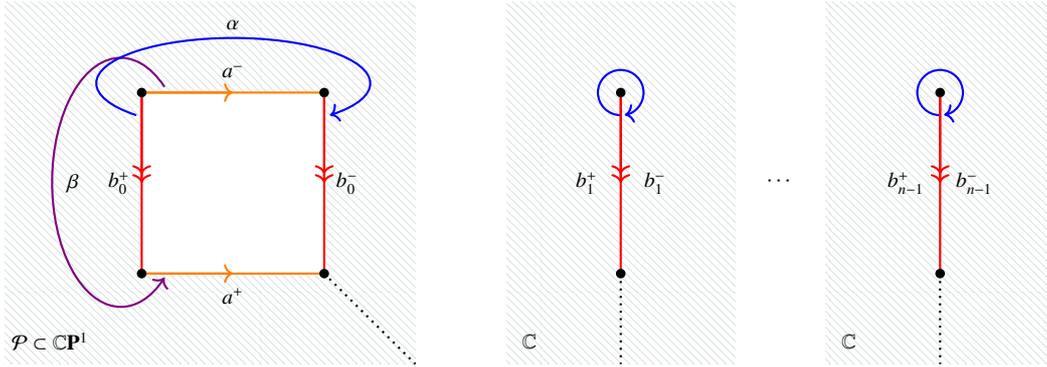

\noindent 

\smallskip

\subsubsection{Realizing representations with rotation number two}\label{rot2} We now realize a non-trivial representation $\chi$ as the period character of some genus one meromorphic differential with rotation number \textit{two}. Recall that poles are supposed to be of order $2$ with zero residue. We distinguish four cases according to the volume of $\chi$ and the parity of $n$ (the number of punctures). All constructions are quite similar to those realized in Section \S\ref{rot1} and in fact they differ mainly in the way we glue copies of $(\C, dz)$.

\medskip

\paragraph{\textit{Positive volume and even number of punctures.}}\label{pvep} Let $\alpha,\beta$ be a set of handle generators and let $\mathcal{P}\subset \mathbb{C}$ be the parallelogram defined by the chain
\begin{equation}
    P_0\mapsto P_0+\chi(\alpha)\mapsto P_0+\chi(\alpha)+\chi(\beta)=Q_0\mapsto P_0+\chi(\beta)\mapsto P_0,
\end{equation}
\noindent where $P_0\in\C$ is any point. Let $a^{\pm}$ denote the edges of $\mathcal{P}$, parallel to $\chi(\alpha)$ that bounds the parallelogram on its right (respectively left). Similarly, let $b_0^+$ (respectively $b_0^-$) denote the edge of $\mathcal{P}$ parallel to $\chi(\beta)$ that bounds the parallelogram on its right (resp. left). We next consider $n$ copies of $(\mathbb{C},\, dz)$ and we slit each of them along a segment $b$. For any $i=1,\dots,n$, we denote as $b_i^+$ and $b_i^-$ the resulting segments after slitting and let $P_i$ and $Q_i=P_i+\chi(\beta)$ be the extremal points of $b_i$. We then glue these $n$ copies of $(\C,dz)$ and the parallelogram $\mathcal{P}$ all together. The desired structure is then obtained by identifying the segment $a^+$ with $a^-$, the segment $b_i^+$ with $b_{i+1}^-$ for $i=1,\dots,n-1$ and the segment $b_0^-$ with $b_{n-1}^+$ and the segment $b_0^+$ with $b_1^-$, see Figure \ref{rotpek2}.

\begin{figure}[ht]
    \centering
    \begin{tikzpicture}[scale=1.25, every node/.style={scale=0.8}]
    \definecolor{pallido}{RGB}{221,227,227}
    
    \pattern [pattern=north west lines, pattern color=pallido]
    (0,0)--(2,0)--(2,2)--(0,2)--(0,0);
    \draw [thick, orange] (0,0)--(2,0);
    \draw [thick, red] (2,0)--(2,2);
    \draw [thick, orange] (2,2)--(0,2);
    \draw [thick, red] (0,2)--(0,0);
    \draw [thick, orange, ->] (0,0) to (1,0);
    \draw [thick, red, ->>] (2,0) to (2,1);
    \draw [thick, orange, ->] (0,2) to (1,2);
    \draw [thick, red, ->>] (0,0) to (0,1);
    \fill (0,0) circle (1.5pt);
    \fill (0,2) circle (1.5pt);
    \fill (2,2) circle (1.5pt);
    \fill (2,0) circle (1.5pt);
    \draw [thick, violet, ->] (0.25,0) to (0.25, 2);
    \draw [thick, blue, ->] (0, 0.25) to (2, 0.25);
    \node at (1,-0.25) {$a^-$};
    \node at (1,2.25) {$a^+$};
    \node at (2.25,1) {$b_0^-$};
    \node at (-0.25,1) {$b_0^+$};
 
    \foreach \x [evaluate=\x as \coord using 3 + 3*\x] in {0, 1} 
    {
    \pattern [pattern=north west lines, pattern color=pallido]
    (\coord, -0.75)--(\coord+2.5, -0.75)--(\coord+2.5, 2.75)--(\coord, 2.75)--(\coord,-0.75);
    \draw [thick, blue, ->] (\coord+1.25, 0.25) arc [start angle = 90, end angle = -260,radius = 0.25];
    \draw [thick, red] (\coord+1.25, 0) to (\coord+1.25, 2);
    \draw [thick, red, ->>] (\coord+1.25,0) to (\coord+1.25,1);
    \fill (\coord+1.25,0) circle (1.5pt);
    \fill (\coord+1.25,2) circle (1.5pt);
    \node at (\coord+0.25, -0.5) {$\mathbb{C}$};  
    \draw [thick, black, dotted] (\coord+1.25, 2)--(\coord+1.25, 2.75);
    }
    
    \node at (5.75,1) {$\cdots$};
    \node at (1.5, 1.5) {$\mathcal{P}$};
    \node at (1,.375) {$\alpha$};
    \node at (0.375,1) {$\beta$};
    \node at (4, 1) {$b_1^-$};
    \node at (4.5, 1) {$b_1^+$};
    \node at (6.875, 1) {$b_{n-1}^-$};
    \node at (7.625, 1) {$b_{n-1}^+$};
   
    \end{tikzpicture}
    \caption{Realizing a genus one translation surface with poles of order $2$, positive volume, and rotation number equal to $2$. In this case there is an \textit{even} number of punctures corresponding to poles with zero residue.}
    \label{rotpek2}
\end{figure}
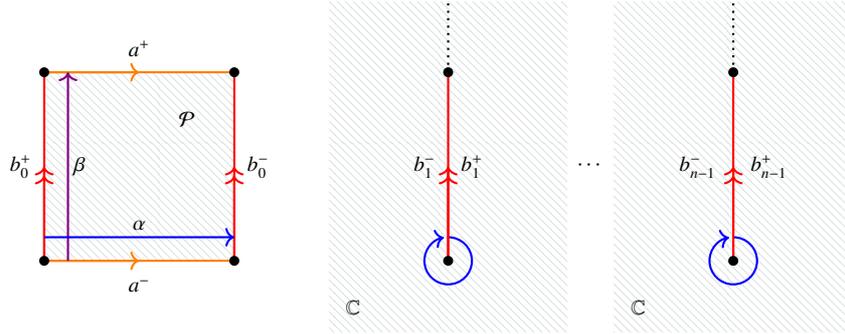

\noindent We define the curves $\alpha,\,\beta$ exactly as we have done above for the case discussed in paragraph \S\ref{pk1}. It is easy to check that  $\text{Ind}(\alpha)=n$ by construction and $\text{Ind}(\beta)=0$ because $\beta$ is geodesic. Since in this paragraph $n$ is assumed to be even we have that $\gcd(n,0,2n,2)=2$ and hence $(X,\omega)$ has rotation number $2$.

\medskip

\paragraph{\textit{Positive volume and odd number of punctures.}}\label{pvop} Let $\alpha,\beta$ be a set of handle generators. Following \cite[Section 3.3.2]{BC}, we shall consider here two half-planes defined as follows. We first consider the broken line defined by the ray corresponding to $\mathbb{R}^-$, two edges corresponding to the chain of vectors $\chi(\alpha)$ and $\chi(\beta)$ in this order with starting point at the origin and, finally, a horizontal ray $r_1$ from the right end of $\chi(\beta)$ and parallel to $\mathbb R$. Define $H_1$ as the half-plane bounded on the left by this broken line. Similarly, define $H_2$ as the half-plane bounded on the right by the broken line defined by the ray corresponding to $\mathbb{R}^-$, two edges corresponding to the chain of vectors $\chi(\beta)$ and $\chi(\alpha)$ in this order with starting point at the origin and, finally, a horizontal ray $r_2$ from the right end of $\chi(\alpha)$ and parallel to $\mathbb R$, see Figure \ref{rotpok2}. We denote by $a_0^-,\,b_0^-$ the edges of the boundary of $H_1$ corresponding to $\chi(\alpha)$ and $\chi(\beta)$ respectively. Similarly, we denote by $a_0^+,\,b_0^+$ the edges of the boundary of $H_2$ corresponding to $\chi(\alpha)$ and $\chi(\beta)$ respectively.

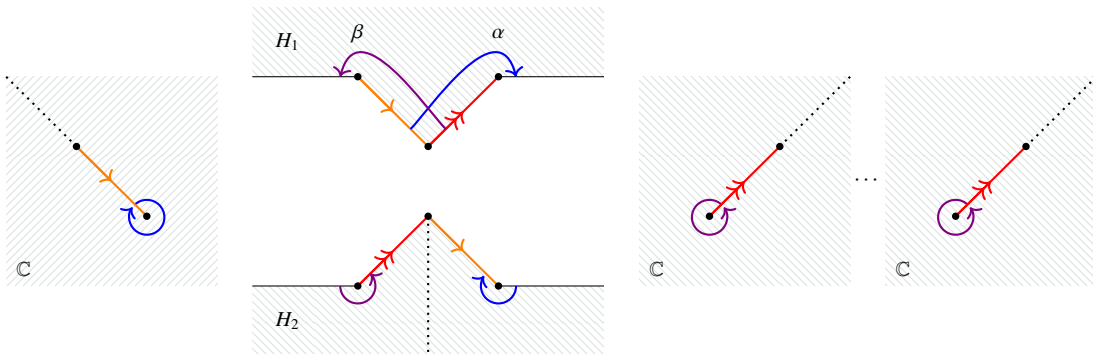
\begin{figure}[ht]
    \centering
    \begin{tikzpicture}[scale=0.925, every node/.style={scale=0.8}]
    \definecolor{pallido}{RGB}{221,227,227}
    
    \pattern [pattern=north west lines, pattern color=pallido]
    (0.5,1.5)--(2,1.5)--(3,0.5)--(4,1.5)--(5.5,1.5)--(5.5,2.5)--(0.5,2.5)--(0.5,1.5);
    \draw [thin, black] (0.5,1.5)--(2,1.5);
    \draw [thick, orange] (2,1.5)--(3,0.5);
    \draw [thick, orange, ->] (2,1.5)--(2.5,1);
    \draw [thick, red] (3,0.5)--(4,1.5);
    \draw [thick, red, ->>] (3,0.5)--(3.5,1);
    \draw [thin, black] (4,1.5)--(5.5,1.5);
    \fill (2,1.5) circle (1.5pt);
    \fill (3,0.5) circle (1.5pt);
    \fill (4,1.5) circle (1.5pt);
    \draw [thick, blue, ->] (2.75, 0.75) .. controls (3, 1) and (4,2.5) .. (4.25, 1.5);
    \draw [thick, violet, ->] (3.25, 0.75) .. controls (3, 1) and (2,2.5) .. (1.75, 1.5);
    \node at (1,2) {$H_1$};
    
    \pattern [pattern=north west lines, pattern color=pallido]
    (0.5,-1.5)--(2,-1.5)--(3,-0.5)--(4,-1.5)--(5.5,-1.5)--(5.5,-2.5)--(0.5,-2.5)--(0.5,-1.5);
    \draw [thin, black] (0.5,-1.5)--(2,-1.5);
    \draw [thick, red] (2,-1.5)--(3,-0.5);
    \draw [thick, red, ->>] (2,-1.5)--(2.5,-1);
    \draw [thick, orange] (3,-0.5)--(4,-1.5);
    \draw [thick, orange,->] (3, -0.5)--(3.5,-1);
    \draw [thin, black] (4,-1.5)--(5.5,-1.5);
    \fill (2,-1.5) circle (1.5pt);
    \fill (3,-0.5) circle (1.5pt);
    \fill (4,-1.5) circle (1.5pt);
    \draw [thick, black, dotted] (3,-0.5)--(3, -2.5);
    \draw [thick, blue, ->] (4.25, -1.5) arc [start angle = 0, end angle = -215, radius = 0.25];
    \draw [thick, violet, ->] (1.75, -1.5) arc [start angle = 180, end angle = 405, radius = 0.25];
    \node at (1,-2) {$H_2$};
    
    \pattern [pattern=north east lines, pattern color=pallido]
    (-3,-1.5)--(0,-1.5)--(0,1.5)--(-3,1.5)--(-3,-1.5);
    \draw [thick, blue, ->] (-1.17677669529, -0.33833883) arc [start angle = 135, end angle = -215, radius = 0.25];
    \draw [thick, orange] (-2,0.5)--(-1,-0.5);
    \draw [thick, orange, ->] (-2,0.5)--(-1.5,0);
    \fill (-2,0.5) circle (1.5pt);
    \fill (-1,-0.5) circle (1.5pt);
    \draw [thick, black, dotted] (-2, 0.5)--(-3, 1.5);
    \node at (-2.75, -1.25) {$\mathbb{C}$};  
    
    \foreach \x [evaluate=\x as \coord using 6 + 3.5*\x] in {0, 1} 
    {
    \pattern [pattern=north west lines, pattern color=pallido]
    (\coord, -1.5)--(\coord+3, -1.5)--(\coord+3, 1.5)--(\coord, 1.5)--(\coord,-1.5);
    \draw [thick, violet, ->] (\coord+1.17677669529, -0.33833883) arc [start angle = 45, end angle = 395, radius = 0.25];
    \draw [thick, red] (\coord+1, -0.5) to (\coord+2, 0.5);
    \draw [thick, red, ->>] (\coord+1, -0.5) to (\coord+1.5,0);
    \fill (\coord+1,-0.5) circle (1.5pt);
    \fill (\coord+2,0.5) circle (1.5pt);
    \node at (\coord+0.25, -1.25) {$\mathbb{C}$};  
    \draw [thick, black, dotted] (\coord+2, 0.5)--(\coord+3, 1.5);
    }
    
    \node at (9.25,0) {$\cdots$};
    \node at (4,2.125) {$\alpha$};
    \node at (2,2.125) {$\beta$};

    \end{tikzpicture}
    \caption{Realizing a genus one translation surface with poles of order $2$, positive volume, and rotation number equal to $2$. In this case there is an \textit{odd} number of punctures corresponding to poles with zero residue.}
    \label{rotpok2}
\end{figure}

\smallskip 

\noindent We consider a copy of $(\C,\,dz)$ and we slit it along a segment $a$ and we denote the resulting segments $a_n^{\pm}$. Next we consider other $n-2$ copies of $(\C,\,dz)$ and we slit each of them along a segment $b$ and we denote the resulting segments $b_i^{\pm}$, for $i=2,\dots,n-1$. The sign $\pm$ are according to our convention. We now glue the edges in the usual way as above and we also glue the ray $r_1$ with $r_2$. Notice that this identification does not affect the set of periods since these two rays differ by a translation $\chi(\alpha)+\chi(\beta)$. The final surface is a genus one meromorphic differential. It remains to show that it has rotation number $2$. We can choose a pair of handle generators as shown in Figure \ref{rotpok2}. We can notice that the curve $\alpha$ has index $2$ by construction whereas the curve $\beta$ has index $n-1$. Observe that $n-1$ is even because $n$ is assumed to be odd. Therefore we can conclude that  $\gcd(2,\,n-1,\,2n,\,-2)=2$ as desired.

\medskip

\paragraph{\textit{Non-positive volume and odd number of punctures.}}\label{npvonp} This case is very similar to the case described in paragraph \S\ref{nk1}. Notice that $n\ge3$ in the present case. Given a set of handle generators $\alpha,\beta$ we define $\mathcal{P}\subset \mathbb{C}$ to be the closure of the \textit{exterior} of the parallelogram defined by the chain
\begin{equation}
    P_0\mapsto P_0+\chi(\alpha)\mapsto P_0+\chi(\alpha)+\chi(\beta)=Q_0\mapsto P_0+\chi(\beta)\mapsto P_0,
\end{equation}
\noindent where $P_0\in\C$ is any point and we label the edges $a_0^\pm$ and $b_0^\pm$ as usual. Consider $n-1$ copies of $(\mathbb C,\, dz)$ and slit $n-2$ of these copies along a segment $b$ and the remaining one along a segment $a$, see Figure \ref{rotnok2}. This is the only difference with respect to the case in paragraph \S\ref{nk1}. Then we proceed as usual in order to get the desired genus one meromorphic differential.

\begin{figure}[ht]
    \centering
    \begin{tikzpicture}[scale=1, every node/.style={scale=0.8}]
    \definecolor{pallido}{RGB}{221,227,227}
   
    \pattern [pattern=north west lines, pattern color=pallido]
    (-6,-0.25)--(-2,-0.25)--(-2,2.25)--(-6,2.25)--(-6,-0.25);
    \draw [thick, violet, ->] (-4.75, 1) arc [start angle = 0, end angle =345 , radius = 0.25];
    \draw [thick, orange] (-5,1)--(-3,1);
    \draw [thick, orange, ->] (-5,1)--(-4,1);
    \fill (-5,1) circle (1.5pt);
    \fill (-3,1) circle (1.5pt);
    \node at (-4, 0.75) {$a_{n-1}^+$};
    \node at (-4, 1.25) {$a_{n-1}^-$};

    \pattern [pattern=north west lines, pattern color=pallido]
    (-1.5,-1)--(3,-1)--(3,3)--(-1.5,3)--(-1.5,-1);
    \fill [white] (0,0)--(2,0)--(2,2)--(0,2)--(0,0);
    \draw [thick, orange] (0,0)--(2,0);
    \draw [thick, red] (2,0)--(2,2);
    \draw [thick, orange] (2,2)--(0,2);
    \draw [thick, red] (0,2)--(0,0);
    \draw [thick, orange, ->] (0,0) to (1,0);
    \draw [thick, red, ->>] (2,2) to (2,1);
    \draw [thick, orange, ->] (0,2) to (1,2);
    \draw [thick, red, ->>] (0,2) to (0,1);
    \fill (0,0) circle (1.5pt);
    \fill (0,2) circle (1.5pt);
    \fill (2,2) circle (1.5pt);
    \fill (2,0) circle (1.5pt);
    \node at (1, -0.25) {$a_0^+$};
    \node at (1, +2.25) {$a_0^-$};
    \node at (-0.25, 1) {$b_0^+$};
    \node at (+2.25, 1) {$b_0^-$};
    
    \draw [thick, violet, ->] (0.25,2.0625) arc (50:310:0.75 and 1.375);
    
    \draw [thick, blue, <-] (2.0625,1.75) arc (-45:225:1.5 and 0.5);
 
    \foreach \x [evaluate=\x as \coord using 3.5 + 3*\x] in {0, 1} 
    {
    \pattern [pattern=north west lines, pattern color=pallido]
    (\coord, -1)--(\coord+2.5, -1)--(\coord+2.5, 3)--(\coord, 3)--(\coord,-1);
    \draw [thick, blue, ->] (\coord+1.25, 1.75) arc [start angle = 270, end angle = -80,radius = 0.25];
    \draw [thick, red] (\coord+1.25, 0) to (\coord+1.25, 2);
    \draw [thick, red, <<-] (\coord+1.25,1) to (\coord+1.25,2);
    \fill (\coord+1.25,0) circle (1.5pt);
    \fill (\coord+1.25,2) circle (1.5pt);
    \node at (\coord+0.25, -0.5) {$\mathbb{C}$};   
    }
    
    \node at (4.375, 1) {$b_1^+$};
    \node at (5.125, 1) {$b_1^-$};
    \node at (7.375, 1) {$b_{n-2}^+$};
    \node at ( 8.25, 1) {$b_{n-2}^-$};
    
    \node at (6.25,1) {$\cdots$};
    \node at (1, 1) {$\mathcal{P}$};
    \node at (-1, -0.5) {$\mathbb{C}$};
    \node at (1, 2.75) {$\alpha$};
    \node at (-1.125,1) {$\beta$};
    \end{tikzpicture}
    \caption{Realizing a genus one translation surface with poles of order $2$, negative volume, and rotation number equal to $2$. In this case there is an \textit{odd} number of punctures corresponding to poles with zero residue.}
    \label{rotnok2}
\end{figure}
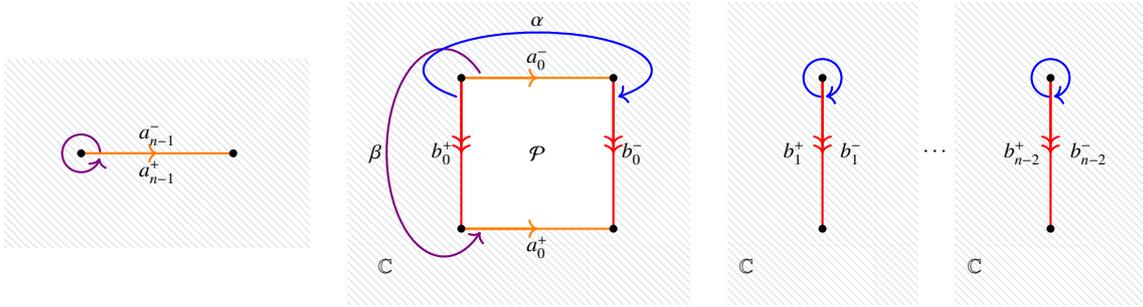

\noindent We just need to check that this structure has rotation number equal to two. We choose $\alpha,\,\beta$ as in Figure \ref{rotnok2}. The curve $\alpha$ has index $n-1$ which is even because $n$ is supposed to be odd. Since the curve $\beta$ has index $2$ by construction, it follows that $\gcd(n-1,\,2,\,2n,\,2)=2$ as desired.

\medskip

\paragraph{\textit{Non-positive volume and even number of punctures.}}\label{nvep} We finally consider the case in which $\chi$ has non-positive volume and the number of puncture is even. For this case we need a slightly different construction. Let $P$ be any point on $(\C,\,dz)$, define $Q=P+\chi(\alpha)-\chi(\beta)$, and let $c$ be the geodesic segment joining the points $P$ and $Q$. Let $T_1$ be the triangle with vertices $P$, $P+\chi(\alpha)$ and $Q$. Notice that, by construction, the sides of $T_1$ are $a,b,c$. We define $\mathcal{T}_1$ as the closure in $\cp$ of the exterior of $T_1$. Notice that $\mathcal{T}_1$ is a triangle in $\cp$. We next consider another copy of $(\C,\,dz)$ and define $T_2$ as the triangle with vertices $P$, $P-\chi(\beta)$ and $Q$. In this case the sides of $T_2$ are still $a,b,c$. Let $\mathcal{T}_2$ be the triangle in $\cp$ given by the closure of the exterior of $T_2$. We now glue the triangles $\mathcal{T}_1$ and $\mathcal{T}_2$ along the edge $c$. The resulting space is a topological parallelogram $\mathcal{P}=\mathcal{T}_1\cup\mathcal{T}_2$ with two edges parallel to $a$ and two edges parallel to $b$ by construction. Such a parallelogram $\mathcal{P}$ has in its interior two special points corresponding to the points at infinity of the two copies of $\cp$. We set $a_0^+$ (resp. $a_0^-$) the edge of $\mathcal{P}$ parallel to $a$ and that bounds the quadrilateral on its right (resp. left). In the same fashion $b_0^+$ (resp. $b_0^-$) is the edge of $\mathcal{P}$ parallel to $b$ and that bounds the quadrilateral on its right (resp. left). See Figure \ref{quadnek2}. 

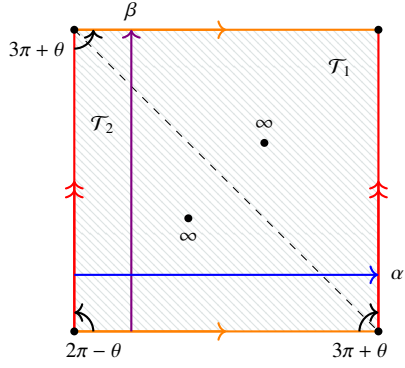
\begin{figure}[ht]
    \centering
    \begin{tikzpicture}[scale=1, every node/.style={scale=0.8}]
    \definecolor{pallido}{RGB}{221,227,227}
    \definecolor{pallido2}{RGB}{221,227,240}
    
    \pattern [pattern=north west lines, pattern color=pallido]
    (0,0)--(4,0)--(4,4)--(0,4)--(0,0);
    \draw[thick, orange] (0,0)--(4,0);
    \draw[thick, orange] (0,4)--(4,4);
    \draw[thick, red] (0,0)--(0,4);
    \draw[thick, red] (4,0)--(4,4);
    \draw[thick, orange, ->] (0,0)--(2,0);
    \draw[thick, orange, ->] (0,4)--(2,4);
    \draw[thick, red, ->>] (0,0)--(0,2);
    \draw[thick, red, ->>] (4,0)--(4,2);
    
    \fill (0,0) circle (1.5pt);
    \fill (0,4) circle (1.5pt);
    \fill (4,4) circle (1.5pt);
    \fill (4,0) circle (1.5pt);
    \draw [thin, dashed] (0,4)--(4,0);
    
    \fill (1.5,1.5) circle (1.5pt);
    \node at (1.5,1.25) {$\infty$};
    \fill (2.5,2.5) circle (1.5pt);
    \node at (2.5,2.75) {$\infty$};
    
    \draw[thick, violet, ->] (0.75,0)--(0.75,4);
    \draw[thick, blue, ->] (0,0.75)--(4,0.75);    
    
    \draw [thick, ->] (0.25, 0) arc [start angle = 0, end angle = 90,radius = 0.25];
    \draw [thick, ->] (0, 3.75) arc [start angle = -90, end angle = 0,radius = 0.25];
    \draw [thick, ->] (3.75, 0) arc [start angle = 180, end angle = 90,radius = 0.25];
    \node at (0.25,-0.25) {$2\pi-\theta$};
    \node at (-0.5,3.75) {$3\pi+\theta$};
    \node at (3.75,-0.25) {$3\pi+\theta$};
    
    \node at (0.75,4.25) {$\beta$};
    \node at (4.25, 0.75) {$\alpha$};
    \node at (3.5,3.5) {$\mathcal{T}_1$};
    \node at (0.375,2.75) {$\mathcal{T}_2$};
    
    \end{tikzpicture}
    \caption{The topological quadrilateral $\mathcal{Q}$ obtained by gluing the triangles $\mathcal{T}_1$ and $\mathcal{T}_2$ along the edge $c$. There are two opposite inner angles of magnitude $2\pi-\theta$ and the other two have magnitude $3\pi+\theta$, where $0<\theta<\pi$. The picture shows also how to choose the curves $\alpha$ and $\beta$. The curve $\alpha$ might be prolonged in the case the number of punctures is higher than $2$, see Figure \ref{rotnek2} below (there $\mathcal{Q}$ is drawn in a slightly different way). }
    \label{quadnek2}
\end{figure}

\noindent We consider $n-2$ copies of $(\C,\, dz)$ each of which slits along a segment $b$, leaving from $P_i$ to $Q_i=P_i+\chi(\beta)$. We denote the sides as $b_i^{\pm}$ according to our convention for $i=1,\dots,n-2$. We finally glue all these structures together as follows. The edges $a_0^+$ and $a_0^-$ are  identified together. Then we glue $b_i^-$ with $b_{i+1}^+$ for $i=1,\dots,n-3$, the edge $b_{n-2}^-$ with $b_0^+$ and, finally, $b_0^-$ with $b_1^+$. The resulting structure is a genus one meromorphic differential. See Figure \ref{rotnek2}.

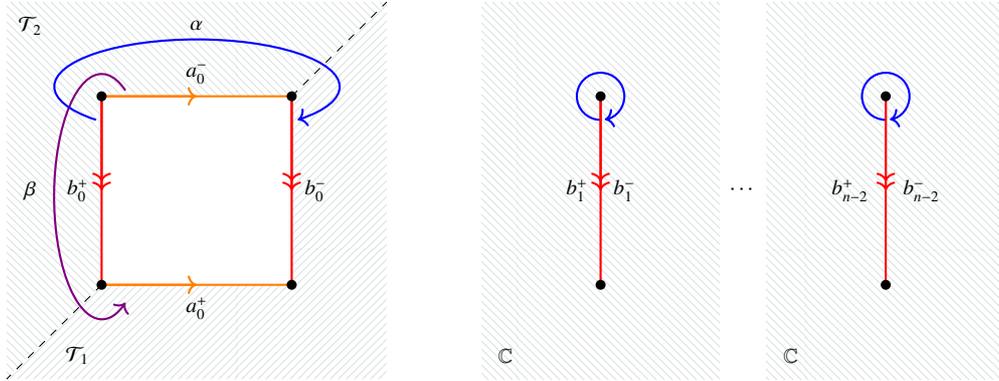
\begin{figure}[!ht]
    \centering
    \begin{tikzpicture}[scale=1.25, every node/.style={scale=0.8}]
    \definecolor{pallido}{RGB}{221,227,227}
    \definecolor{pallido2}{RGB}{221,227,240}
   
    \pattern [pattern=north west lines, pattern color=pallido]
    (-1,-1)--(0,0)--(0,2)--(2,2)--(3,3)--(-1,3)--(-1,-1);
    \pattern [pattern=north east lines, pattern color=pallido]
    (-1,-1)--(0,0)--(2,0)--(2,2)--(3,3)--(3,-1)--(-1,-1);

    \fill [white] (0,0)--(2,0)--(2,2)--(0,2)--(0,0);
    \draw [thick, orange] (0,0)--(2,0);
    \draw [thick, red] (2,0)--(2,2);
    \draw [thick, orange] (2,2)--(0,2);
    \draw [thick, red] (0,2)--(0,0);
    \draw [thick, orange, ->] (0,0) to (1,0);
    \draw [thick, red, ->>] (2,2) to (2,1);
    \draw [thick, orange, ->] (0,2) to (1,2);
    \draw [thick, red, ->>] (0,2) to (0,1);
    \fill (0,0) circle (1.5pt);
    \fill (0,2) circle (1.5pt);
    \fill (2,2) circle (1.5pt);
    \fill (2,0) circle (1.5pt);
    \draw [thin, dashed] (0,0)--(-1,-1);
    \draw [thin, dashed] (2,2)--(3,3);
    \node at (1, -0.25) {$a_0^+$};
    \node at (1, +2.25) {$a_0^-$};
    \node at (-0.25, 1) {$b_0^+$};
    \node at (+2.25, 1) {$b_0^-$};
    
    \draw [thick, violet, ->] (0.25,2.0625) arc (60:300:0.5 and 1.3);
    
    \draw [thick, blue, <-] (2.0625,1.75) arc (-45:225:1.5 and 0.5);
 
    \foreach \x [evaluate=\x as \coord using 4 + 3*\x] in {0, 1} 
    {
    \pattern [pattern=north west lines, pattern color=pallido]
    (\coord, -1)--(\coord+2.5, -1)--(\coord+2.5, 3)--(\coord, 3)--(\coord,-1);
    \draw [thick, blue, ->] (\coord+1.25, 1.75) arc [start angle = 270, end angle = -80,radius = 0.25];
    \draw [thick, red] (\coord+1.25, 0) to (\coord+1.25, 2);
    \draw [thick, red, <<-] (\coord+1.25,1) to (\coord+1.25,2);
    \fill (\coord+1.25,0) circle (1.5pt);
    \fill (\coord+1.25,2) circle (1.5pt);
    \node at (\coord+0.25, -0.75) {$\mathbb{C}$};   
    }
    
    \node at (5, 1) {$b_1^+$};
    \node at (5.5, 1) {$b_1^-$};
    \node at (7.875, 1) {$b_{n-2}^+$};
    \node at (8.625, 1) {$b_{n-2}^-$};

    \node at (6.75,1) {$\cdots$};
    \node at (-0.25,-0.75) {$\mathcal{T}_1$};
    \node at (-0.75,2.75) {$\mathcal{T}_2$};
    \node at (1,2.75) {$\alpha$};
    \node at (-0.75,1) {$\beta$};
    \end{tikzpicture}
    \caption{Realizing a genus one translation surface with poles of order $2$, negative volume, and rotation number equal to $2$. In this case there is an \textit{even} number of punctures corresponding to poles with zero residue.}
    \label{rotnek2}
\end{figure}

\noindent It remains to show that such a genus one meromorphic differential has rotation number $2$. We define the curve $\alpha$ as already done in the previous paragraphs, \textit{e.g.} \S\ref{pk1}. For any $i=1,\dots,n-2$ we consider a small metric circle of radius $\varepsilon$ around $P_i$. After the cut and paste process, these $n-2$ curves define a path leaving from a point, say $R\in b_0^+$ to $R+\chi(\alpha)\in b_0^-$. We then join these latter points with a smooth path in $\mathcal{Q}$ in order to get the desired curve $\alpha$. Such a curve has index $n$ because it turns once around any point $P_i$ and then turns by an angle $4\pi$ inside $\mathcal{Q}$ (more precisely the tangent vector to the path turns by an angle $4\pi$). The curve $\beta$ is defined by joining a point $R'$ to $R'+\chi(\beta)$. Such a curve has index two because it also turns by an angle $4\pi$ inside $\mathcal{Q}$. These curves are drawn in Figures \ref{quadnek2} and \ref{rotnek2}. Therefore, it follows that $\gcd(n,\,2,\,2n,\,2)=2$ because $n$ is assumed to be even. This concludes the proof of Lemma \ref{onepuncttorusp2}. \qed\qedhere

\subsection{Strata with poles of order higher than two}\label{genusoneorderhigherthantwo} We next consider strata $\mathcal{H}_1(np;-p,\dots,-p)$ of genus one meromorphic differentials where poles are allowed to have orders greater than two. In this section we still assume that all residues are equal to zero. 

\begin{lem}\label{onepuncttorushigherp}
Let $\chi\colon\shomolzon\longrightarrow \C$ be a non-trivial representation of trivial-ends type. If $\chi$ can be realized in the stratum $\mathcal{H}_1(np;-p,\dots,-p)$ then it appears as the period character of a translation surface with poles in each of its connected components. 
\end{lem}

\noindent Before moving to the proof of this Lemma we observe that the case $p=2$ is handled by Lemma \ref{onepuncttorusp2}.  Therefore we can assume $p\ge3$ in the proof of Lemma \ref{onepuncttorushigherp}.

\smallskip

\noindent  Let $\chi\colon\shomolzon\longrightarrow \C$ be a non-trivial representation and assume that $\chi$ can be realized as the period character of a genus one meromorphic differential with poles in a stratum $\mathcal{H}_1(np;-p,\dots,-p)$. The greatest common divisor of the orders of singularities is equal to $p$. Therefore we want to realize $\chi$ as the period of a genus one meromorphic differential with prescribed rotation number $r$ where $r$ divides $p$. The cases $n=1,2$ are special and they work differently from the generic case $n\ge3$. Moreover, in all cases, we shall need to  distinguish two sub-cases according to the sign of the volume of $\chi_g$ induced by $\chi$. Recall that, since $\chi$ is assumed to be of trivial-ends type, the volume of $\chi$ is well-defined and it does not depend on any splitting, see Definition \ref{algvoldef2}. Therefore, we divide the proof in four subsections according to the cases above. As we shall see, all these cases appear as variations of the preceding constructions.

\smallskip 

\subsubsection{Positive volume and at most two punctures}\label{pv2p} Let $\chi$ be a representation with positive volume. Let $\alpha,\,\beta$ be a pair of handle generators and let $\mathcal{P}$ be the parallelogram defined by
\begin{equation}
    P_0\mapsto P_0+\chi(\alpha)\mapsto P_0+\chi(\alpha)+\chi(\beta)=Q_0\mapsto P_0+\chi(\beta)\mapsto P_0,
\end{equation}
\noindent where $P_0\in\C$ is any point. We introduce the following genus zero meromorphic differential, say $(X_1,\omega_1)$. Consider a copy of $(\C,\,dz)$ and let $b_1$ be the geodesic segment $P_1$ to $Q_1=P_1+\chi(\beta)$. Let $k\in\{1,\dots,p-1\}$ be a positive integer and consider $k-1$ distinct rays starting from $P_1$ and intersecting $b_1$ only at $P_1$. Next we consider $p-k-1$ distinct rays starting from $Q_1$ intersecting $b_1$ only at $Q_1$. Finally, along every ray we bubble a copy of $(\C,\,dz)$. The resulting structure is a genus zero differential with two zeros of orders $k-1$ and $p-k-1$ and a pole of order $p$. 

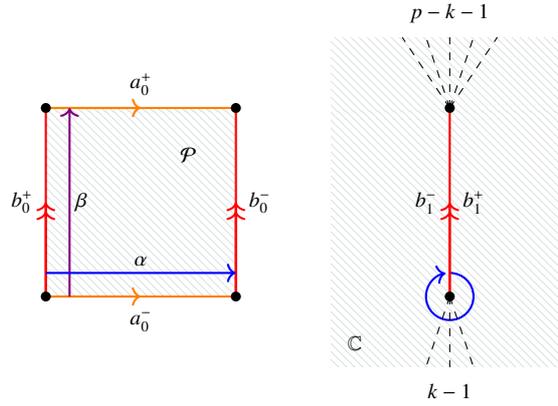
\begin{figure}[ht]
    \centering
    \begin{tikzpicture}[scale=1.25, every node/.style={scale=0.8}]
    \definecolor{pallido}{RGB}{221,227,227}
    
    \pattern [pattern=north west lines, pattern color=pallido]
    (0,0)--(2,0)--(2,2)--(0,2)--(0,0);
    \draw [thick, orange] (0,0)--(2,0);
    \draw [thick, red] (2,0)--(2,2);
    \draw [thick, orange] (2,2)--(0,2);
    \draw [thick, red] (0,2)--(0,0);
    \draw [thick, orange, ->] (0,0) to (1,0);
    \draw [thick, red, ->>] (2,0) to (2,1);
    \draw [thick, orange, ->] (0,2) to (1,2);
    \draw [thick, red, ->>] (0,0) to (0,1);
    \fill (0,0) circle (1.5pt);
    \fill (0,2) circle (1.5pt);
    \fill (2,2) circle (1.5pt);
    \fill (2,0) circle (1.5pt);
    \draw [thick, violet, ->] (0.25,0) to (0.25, 2);
    \draw [thick, blue, ->] (0, 0.25) to (2, 0.25);
    \node at (2.25,1) {$b_0^-$};
    \node at (-0.25,1) {$b_0^+$};
    \node at (1,-0.25) {$a_0^-$};
    \node at (1,2.25) {$a_0^+$};
 
    \foreach \x [evaluate=\x as \coord using 3 + 3*\x] in {0} 
    {
    \pattern [pattern=north west lines, pattern color=pallido]
    (\coord, -0.75)--(\coord+2.5, -0.75)--(\coord+2.5, 2.75)--(\coord, 2.75)--(\coord,-0.75);
    \draw [thick, blue, ->] (\coord+1.25, 0.25) arc [start angle = 90, end angle = -260,radius = 0.25];
    \draw [thick, red] (\coord+1.25, 0) to (\coord+1.25, 2);
    \draw [thick, red, ->>] (\coord+1.25,0) to (\coord+1.25,1);
    \fill (\coord+1.25,0) circle (1.5pt);
    \fill (\coord+1.25,2) circle (1.5pt);
    \node at (\coord+0.25, -0.5) {$\mathbb{C}$};  
    \draw [thin, black, dashed] (\coord+1.25, 2)--(\coord+1, 2.75);
    \draw [thin, black, dashed] (\coord+1.25, 2)--(\coord+1.25, 2.75);
    \draw [thin, black, dashed] (\coord+1.25, 2)--(\coord+1.5, 2.75);
    }
    
    \foreach \x [evaluate=\x as \coord using 3 + 3*\x] in {0} 
    {
    \draw [thin, black, dashed] (\coord+1.25, 0)--(\coord+1, -0.75);
    \draw [thin, black, dashed] (\coord+1.25, 0)--(\coord+1.25, -0.75);
    \draw [thin, black, dashed] (\coord+1.25, 0)--(\coord+1.5, -0.75);
    \draw [thin, black, dashed] (\coord+1.25, 2)--(\coord+0.75, 2.75);
    \draw [thin, black, dashed] (\coord+1.25, 2)--(\coord+1.75, 2.75);
    }
    
    \node at (1.5, 1.5) {$\mathcal{P}$};
    \node at (1,.375) {$\alpha$};
    \node at (0.375,1) {$\beta$};
    \node at (4.25, -1) {$k-1$};
    \node at (4.25, 3) {$p-k-1$};
    \node at (4,1) {$b_1^-$};
    \node at (4.5,1) {$b_1^+$};
   
    \end{tikzpicture}
    \caption{Realizing a genus one translation surface with one pole of order $p$, positive volume, and rotation number equal to $\gcd(k,p)$. In this case the pole has (necessarily) zero residue.}
    \label{rotpkhtris}
\end{figure}

\noindent We glue these structures as follows. We slit $(X_1,\,\omega_1)$ along $b_1$ and denote the resulting edges by $b_1^\pm$, according to our convention. Identify $b_j^+$ with $b_{j+1}^-$, where $j$ is taken in $\{0,1\}$. Finally identify $a_0^+$ with $a_0^-$. The final structure, say $(Y,\xi)$, is a genus one meromorphic differential with a single zero of order $p$ and one pole of order $p$. It remains to show that such a structure has rotation number $\gcd(k,p)$. Consider a small metric circle of radius $\varepsilon$ centered at $P_1$. After the cut and paste process, these circles form an oriented smooth path with starting point $R\in b_0^-$ and ending point $R'\in b_0^+$. The points $R$ and $R'$ differ by $\chi(\alpha)$. Define $\alpha$ as the smooth oriented curve obtained by joining these points. The unit tangent vector along $\alpha$ turns by an angle of $2k\pi$. We can define $\beta$ as any waist curve in the cylinder obtained by gluing the parallelogram $\mathcal{P}$. Such a curve can be chosen to be geodesic and hence its index is zero. Therefore, $(Y,\xi)$ has rotation number equal to $\gcd(k,0,p,p)=\gcd(k,p)$ as desired. Notice that $\alpha$ cannot have index equal to $p$ and hence the rotation number cannot be $p$ as expected.

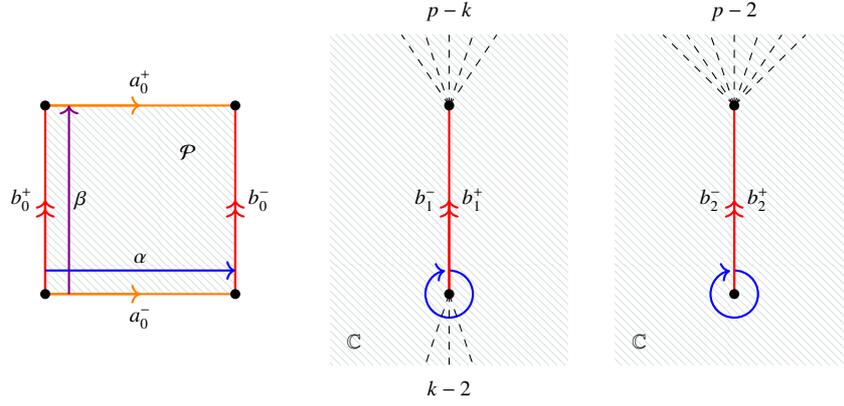
\begin{figure}[ht]
    \centering
    \begin{tikzpicture}[scale=1.25, every node/.style={scale=0.8}]
    \definecolor{pallido}{RGB}{221,227,227}
    
    \pattern [pattern=north west lines, pattern color=pallido]
    (0,0)--(2,0)--(2,2)--(0,2)--(0,0);
    \draw [thick, orange] (0,0)--(2,0);
    \draw [thick, red] (2,0)--(2,2);
    \draw [thick, orange] (2,2)--(0,2);
    \draw [thick, red] (0,2)--(0,0);
    \draw [thick, orange, ->] (0,0) to (1,0);
    \draw [thick, red, ->>] (2,0) to (2,1);
    \draw [thick, orange, ->] (0,2) to (1,2);
    \draw [thick, red, ->>] (0,0) to (0,1);
    \fill (0,0) circle (1.5pt);
    \fill (0,2) circle (1.5pt);
    \fill (2,2) circle (1.5pt);
    \fill (2,0) circle (1.5pt);
    \draw [thick, violet, ->] (0.25,0) to (0.25, 2);
    \draw [thick, blue, ->] (0, 0.25) to (2, 0.25);
    \node at (2.25,1) {$b_0^-$};
    \node at (-0.25,1) {$b_0^+$};
    \node at (1,-0.25) {$a_0^-$};
    \node at (1,2.25) {$a_0^+$};
 
    \foreach \x [evaluate=\x as \coord using 3 + 3*\x] in {0, 1} 
    {
    \pattern [pattern=north west lines, pattern color=pallido]
    (\coord, -0.75)--(\coord+2.5, -0.75)--(\coord+2.5, 2.75)--(\coord, 2.75)--(\coord,-0.75);
    \draw [thick, blue, ->] (\coord+1.25, 0.25) arc [start angle = 90, end angle = -260,radius = 0.25];
    \draw [thick, red] (\coord+1.25, 0) to (\coord+1.25, 2);
    \draw [thick, red, ->>] (\coord+1.25,0) to (\coord+1.25,1);
    \fill (\coord+1.25,0) circle (1.5pt);
    \fill (\coord+1.25,2) circle (1.5pt);
    \node at (\coord+0.25, -0.5) {$\mathbb{C}$};  
    \draw [thin, black, dashed] (\coord+1.25, 2)--(\coord+1, 2.75);
    \draw [thin, black, dashed] (\coord+1.25, 2)--(\coord+1.25, 2.75);
    \draw [thin, black, dashed] (\coord+1.25, 2)--(\coord+1.5, 2.75);
    }
    
    \foreach \x [evaluate=\x as \coord using 3 + 3*\x] in {0} 
    {
    \draw [thin, black, dashed] (\coord+1.25, 0)--(\coord+1, -0.75);
    \draw [thin, black, dashed] (\coord+1.25, 0)--(\coord+1.25, -0.75);
    \draw [thin, black, dashed] (\coord+1.25, 0)--(\coord+1.5, -0.75);
    \draw [thin, black, dashed] (\coord+1.25, 2)--(\coord+0.75, 2.75);
    \draw [thin, black, dashed] (\coord+1.25, 2)--(\coord+1.75, 2.75);
    }
    
    \foreach \x [evaluate=\x as \coord using 3 + 3*\x] in {1} 
    {
    \draw [thin, black, dashed] (\coord+1.25, 2)--(\coord+0.75, 2.75);
    \draw [thin, black, dashed] (\coord+1.25, 2)--(\coord+1.75, 2.75);
    \draw [thin, black, dashed] (\coord+1.25, 2)--(\coord+0.5, 2.75);
    \draw [thin, black, dashed] (\coord+1.25, 2)--(\coord+2, 2.75);
    }

    \node at (1.5, 1.5) {$\mathcal{P}$};
    \node at (1,.375) {$\alpha$};
    \node at (0.375,1) {$\beta$};
    \node at (4.25, -1) {$k-2$};
    \node at (7.25,3) {$p-2$};
    \node at (4.25, 3) {$p-k$};
    \node at (4,1) {$b_1^-$};
    \node at (4.5,1) {$b_1^+$};
    \node at (7,1) {$b_2^-$};
    \node at (7.5,1) {$b_2^+$};
   
    \end{tikzpicture}
    \caption{Realizing a genus one translation surface with two poles of order $p$, positive volume, and rotation number equal to $\gcd(k,p)$. In this case the number of punctures is exactly two and both poles have zero residue.}
    \label{rotpkhbis}
\end{figure}

\noindent The case with two punctures follows after a small modification of the previous construction. We consider a second copy of $(\C,\,dz)$ and let $b_2$ be the geodesic joining the points $P_2$ and $Q_2=P_2+\chi(\beta)$. In this case we consider $p-2$ distinct rays leaving from $Q_2$ and intersecting $b_2$ only at $Q_2$. We then bubble a copy $(\C,\,dz)$ along every ray. Equivalently, we consider only one ray and we bubble along it a copy of $(\C,\, z^{p-2}dz)$. The resulting structure is a genus zero meromorphic differential, say $(X_2,\omega_2)$ with a single zero of order $p-2$ and a pole of order $p$. We glue these structures as follows. We slit both $(X_i,\,\omega_i)$ along $b_i$ and denote the resulting edges by $b_i^\pm$, according to our convention. Identify $b_j^+$ with $b_{j+1}^-$, where $j$ is taken in $\{0,1,2\}$. The final structure, say $(Y,\xi)$, is a genus one meromorphic differential with a single zero of order $2p$ and two poles of order $p$. The same argument above shows that $(Y,\xi)$ has rotation number $\gcd(k,p)$ as desired, where $k\in \{2,\ldots,p\}$. See Figure \ref{rotpkhbis}.

\subsubsection{Positive volume and more than two punctures}\label{pvmp} Assume that the representation $\chi$ has positive volume and at least three punctures. This case is nothing but a variation of the case discussed in the  paragraph \S\ref{pvop}. In order to get the desired structure with prescribed rotation number $\gcd(k,p)$, we shall consider $n-1$ copies of $(\mathbb C,\, dz)$ in this case and we properly modify them to obtain a genus zero meromorphic differential with a single pole of order $p$. We begin with these structures.

\smallskip

\noindent First of all we fix a positive integer $k\in\{2,\dots,p\}$. Notice that, if $k-1\ge n-2$ then there exist integers $d\ge1$ and $0\le l < n-2$ such that $\,k-1=d(n-2)+l\,$. In the case $k-1< n-2$ we can always find a positive integer $t$ such that $\,tk-1=n-2+l\,$ and $0\le l< n-2$ ( here $d=1$ ). Let us consider $(\C,\,dz)$ and a pair of points $P,Q$ such that $Q=P+\chi(\beta)$. We bubble $d$ copies of $(\C,\,dz)$ along rays leaving from $P$ and $p-d-2$ copies of $(\C,\,dz)$ along rays leaving from $Q$. All rays leaving from $P$ (resp. from $Q$) are taken in such a way they do not contain $Q$ (resp. $P$): none of them contains the geodesic segment joining $P$ and $Q$. After bubbling, the resulting structure, say $(X,\omega)$, is a genus zero meromorphic differential with two zeros of orders $d$ and $p-d-2$ and a single pole of order $p$. We consider $l$ copies $(X_1,\omega_1),\dots,(X_l,\omega_l)$ of this structure. Consider again $(\C,\,dz)$ and the pair of points $P,Q$ such that $Q=P+\chi(\beta)$. Define $(Y,\xi)$ as the genus zero differential obtained by bubbling $d-1$ copies of $(\C,\,dz)$ along rays leaving from $P$ and $p-d-1$ copies of $(\C,\,dz)$ along rays leaving from $Q$. Even in this case the rays are taken so that none of them contains the geodesic segment joining $P$ and $Q$. We consider $n-l-2$ copies $(Y_1, \xi_1),\dots,(Y_{n-l-2},\xi_{n-l-2})$ of this structure. Finally consider another copy of $(\C,\,dz)$ and a segment $a$ with extremal points $P_{n-1}$ and $Q_{n-1}=P_{n-1}+\chi(\alpha)$. We bubble $k-2$ copies of $(\C,\,dz)$ along rays leaving from $P_{n-1}$ that do not contain $Q_{n-1}$ and we bubble $p-k$ copies of $(\C,\,dz)$ along rays leaving from $Q_{n-1}$ that does not contain $P_{n-1}$. This genus zero meromorphic differential, say $(Z,\,\eta)$, has two zeros of order $k-2$ and $p-k$ and a single pole of order $p$.

\begin{figure}[ht]
    \centering
    \begin{tikzpicture}[scale=1.2, every node/.style={scale=0.8}]
    \definecolor{pallido}{RGB}{221,227,227}
    
    \pattern [pattern=north west lines, pattern color=pallido]
    (1,1.5)--(2,1.5)--(2,0.5)--(4,1.5)--(5,1.5)--(5,2.5)--(1,2.5)--(1,1.5);
    \draw [thin, black] (1,1.5)--(2,1.5);
    \draw [thick, red] (2,1.5)--(2,0.5);
    \draw [thick, red, ->>] (2,0.5)--(2,1);
    \draw [thick, orange] (2,0.5)--(4,1.5);
    \draw [thick, orange, ->] (2,0.5)--(3,1);
    \draw [thin, black] (4,1.5)--(5,1.5);
    \fill (2,1.5) circle (1.5pt);
    \fill (2,0.5) circle (1.5pt);
    \fill (4,1.5) circle (1.5pt);
    \draw [thick, blue, ->] (2, 0.75) .. controls (3, 1.5) and (4,2.5) .. (4.25, 1.5);
    \draw [thick, violet, ->] (2.5, 0.75) .. controls (2.5, 1) and (2,2.5) .. (1.75, 1.5);
    \node at (1.25,2.25) {$H_1$};
    \node at (4, 2.25) {$\alpha$};
    \node at (2, 2.25) {$\beta$};
    
    \pattern [pattern=north west lines, pattern color=pallido]
    (1,0)--(2,0)--(4,1)--(4,0)--(5,0)--(5,-1)--(1,-1)--(1,0);
    \draw [thin, black] (1,0)--(2,0);
    \draw [thick, orange] (2,0)--(4,1);
    \draw [thick, orange, ->] (2,0)--(3,0.5);
    \draw [thick, red] (4,1)--(4,0);
    \draw [thick, red,->>] (4, 0)--(4,0.5);
    \draw [thin, black] (4,0)--(5,0);
    \fill (2,0) circle (1.5pt);
    \fill (4,1) circle (1.5pt);
    \fill (4,0) circle (1.5pt);
    \draw [thin, black, dashed] (4,1)--(2.5, -1);
    \draw [thin, black, dashed] (4,1)--(2, -1);
    \draw [thin, black, dashed] (4,1)--(1.5, -1);
    \draw [thick, blue, ->] (4.25, 0) arc [start angle = 0, end angle = -255, radius = 0.25];
    \draw [thick, violet, ->] (1.75, 0) arc [start angle = 180, end angle = 375, radius = 0.25];
    \node at (2,-1.25) {$p-2$};
    \node at (4.75,-0.75) {$H_2$};
    
    \pattern [pattern=north west lines, pattern color=pallido]
    (1,-1.5)--(5,-1.5)--(5,-3.5)--(1,-3.5)--(1,-1.5);
    \draw [thick, orange] (2,-3)--(4,-2);
    \fill (2,-3) circle (1.5pt);
    \fill (4,-2) circle (1.5pt);
    \draw [thick, violet, ->] (1.75, -3) arc [start angle = 180, end angle = 375, radius = 0.25];
    \draw [thick, violet, ] (1.75, -3) arc [start angle = 180, end angle = 25, radius = 0.25];
    \draw [thin, black, dashed] (4,-2)--(5, -2);
    \draw [thin, black, dashed] (4,-2)--(5, -1.75);
    \draw [thin, black, dashed] (4,-2)--(5, -2.25);
    \draw [thin, black, dashed] (2,-3)--(1, -2.5);
    \draw [thin, black, dashed] (2,-3)--(1, -3.5);
    \draw [thin, black, dashed] (2,-3)--(1, -3);
    \draw [thin, black, dashed] (2,-3)--(1, -2.75);
    \draw [thin, black, dashed] (2,-3)--(1, -3.25);

    \node at (5.5, -2) {$p-k$};
    \node at (1.5, -2.2) {$k-2$};
    \node at (4.5,-3) {$\mathbb{C}$};

    \foreach \x [evaluate=\x as \coord using 6 + 3*\x] in {0} 
    {
    \pattern [pattern=north west lines, pattern color=pallido]
    (\coord, -1)--(\coord+2.5, -1)--(\coord+2.5, 2.5)--(\coord, 2.5)--(\coord,-1);
    \draw [thick, blue, ->] (\coord+1.25, 0) arc [start angle = 90, end angle = -260,radius = 0.25];
    \draw [thick, red] (\coord+1.25, -0.25) to (\coord+1.25, 1.75);
    \draw [thick, red, ->>] (\coord+1.25,-0.25) to (\coord+1.25,0.75);
    \fill (\coord+1.25,-0.25) circle (1.5pt);
    \fill (\coord+1.25,1.75) circle (1.5pt);
    \node at (\coord+0.25, -0.75) {$\mathbb{C}$};  
    \draw [thin, black, dashed] (\coord+1.25, 1.75)--(\coord+1, 2.5);
    \draw [thin, black, dashed] (\coord+1.25, 1.75)--(\coord+1.5, 2.5);
    \draw [thin, black, dashed] (\coord+1.25, -0.25)--(\coord+0.5, -1);
    \draw [thin, black, dashed] (\coord+1.25, -0.25)--(\coord+1, -1);
    \draw [thin, black, dashed] (\coord+1.25, -0.25)--(\coord+1.5, -1);
    \draw [thin, black, dashed] (\coord+1.25, -0.25)--(\coord+2, -1);
    \node at (\coord+1.25, 2.75) {$p-d-2$};
    \node at (\coord+1.25, -1.25) {$d$};
    \node at (\coord-0.5,0.75) {$\cdots$};
    \node at (\coord+3,0.75) {$\cdots$};
    }
    
    \foreach \x [evaluate=\x as \coord using 10 + 3*\x] in {0} 
    {
    \pattern [pattern=north west lines, pattern color=pallido]
    (\coord, -1)--(\coord+2.5, -1)--(\coord+2.5, 2.5)--(\coord, 2.5)--(\coord,-1);
    \draw [thick, blue, ->] (\coord+1.25, 0) arc [start angle = 90, end angle = -260,radius = 0.25];
    \draw [thick, red] (\coord+1.25, -0.25) to (\coord+1.25, 1.75);
    \draw [thick, red, ->>] (\coord+1.25,-0.25) to (\coord+1.25,0.75);
    \fill (\coord+1.25,-0.25) circle (1.5pt);
    \fill (\coord+1.25,1.75) circle (1.5pt);
    \node at (\coord+0.25, -0.75) {$\mathbb{C}$};  
    \draw [thin, black, dashed] (\coord+1.25, 1.75)--(\coord+0.75, 2.5);
    \draw [thin, black, dashed] (\coord+1.25, 1.75)--(\coord+1.25, 2.5);
    \draw [thin, black, dashed] (\coord+1.25, 1.75)--(\coord+1.75, 2.5);
    \draw [thin, black, dashed] (\coord+1.25, -0.25)--(\coord+0.75, -1);
    \draw [thin, black, dashed] (\coord+1.25, -0.25)--(\coord+1.25, -1);
    \draw [thin, black, dashed] (\coord+1.25, -0.25)--(\coord+1.75, -1);
    \node at (\coord+1.25, 2.75) {$p-d-1$};
    \node at (\coord+1.25, -1.25) {$d-1$};
    \node at (\coord-0.5,0.75) {$\cdots$};
    \node at (\coord+3,0.75) {$\cdots$};
    }
    
    \node at (7, 0.75) {$b_{i}^-$};
    \node at (7.5, 0.75) {$b_{i}^+$};
    \node at (11, 0.75) {$b_{j}^-$};
    \node at (11.5, 0.75) {$b_{j}^+$};
    
    \end{tikzpicture}
    \caption{Realizing a genus one translation surface with poles of order $p$, positive volume, and rotation number equal to $\gcd(k,p)$. In this case the number of punctures is supposed to be at least three and all the poles have zero residue.}
    \label{rotpkh}
\end{figure}
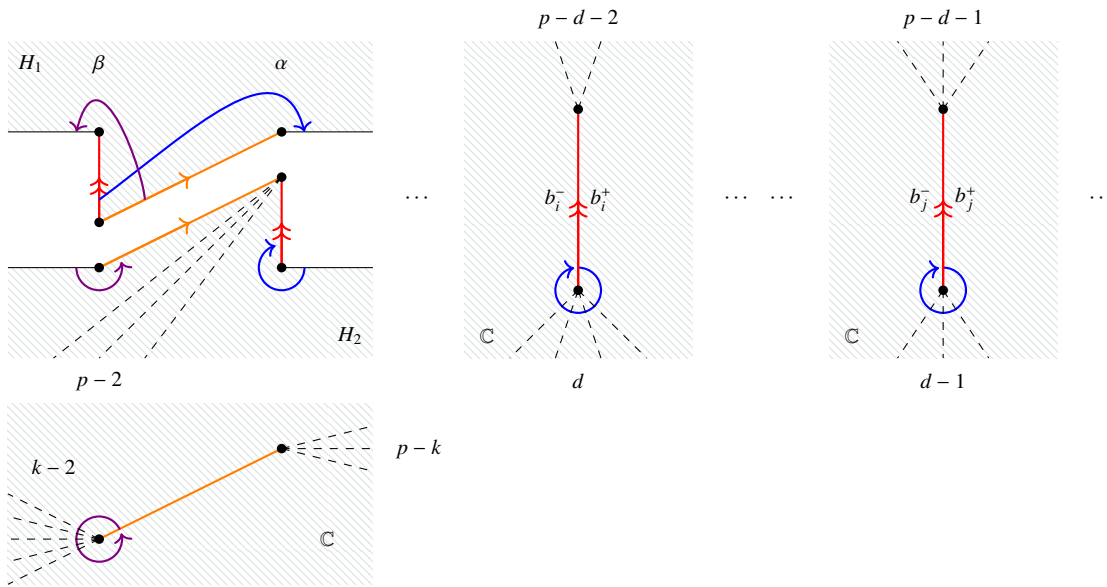

\noindent Let $\alpha,\,\beta$ be a pair of handle generators and let $a,\,b\in\C$ the corresponding images via $\chi$. Let $H_1$ and $H_2$ be the half-planes as defined in \S\ref{pvop}. We denote by $a_0^-,\,b_0^-$ the edges of the boundary of $H_1$ corresponding to $\chi(\alpha)$ and $\chi(\beta)$ respectively. Similarly, we denote by $a_0^+,\,b_0^+$ the edges of the boundary of $H_2$ corresponding to $\chi(\alpha)$ and $\chi(\beta)$ respectively. Let $Q$ be the endpoint of $\chi(\alpha)$ (and hence the starting point of $-\chi(\beta)$) on the boundary of $H_2$. Notice that there are $p-2$ rays, say $r_1,\dots,r_{p-2}$, leaving from $Q$ and pointing at the infinity.

\smallskip

\noindent It remains to glue all the structures defined so far together. First of all we slit $(X_i,\omega_i)$ for $i=1,\dots,l$ along the saddle connection $b$ joining the two zeros of $\omega_i$. Denote the resulting edges by $b_i^{\pm}$. In the same fashion, we slit $(Y_j,\xi_j)$, for $j=1,\dots,n-l-2$, along the saddle connection joining the zeros of $\xi_j$. We finally slit $(Z,\eta)$ along $a$ and we label the edges as $a_n^{\pm}$. After slitting, we identify $a_0^+$ with $a_n^-$ and $a_0^-$ with $a_n^+$. Similarly, we identify $b_i^+$ with $b_{i+1}^-$, where $i=1,\dots, l-1$, and $b_j^+$ with $b_{j+1}^-$, where $j=1,\dots,n-l-3$. The remaining identifications are: the edge $b_0^+$ with $b_1^-$, the edge $b_l^+$ with $b_1^-$, the edge $b_{n-l-2}^-$ with $b_0^+$. The resulting structure is a genus one meromorphic differential. By bubbling $p-2$ copies of $(\C,\,dz)$, each one along a ray leaving from $Q$, we get the desired structure with rotation number $\gcd(k,p)$. In fact, we can define $\alpha$ and $\beta$ as in Figure \ref{rotpkh}. The curve $\alpha$ has index 
\begin{equation}
    \text{Ind}(\alpha)=
    \begin{cases}
    k \quad \text{ if } k-1\ge n-2, \text{ or}\\
    t k \quad \text{ if } k-1<n-2
    \end{cases}
\end{equation}
\noindent by construction. It is straightforward to check that the curve $\beta$ has index $k$. Therefore the rotation number is equal to $\gcd(tk,k,np,p)=\gcd(k,p)$ as desired, where $t=1$ if $k-1\ge n-2$.

\subsubsection{Non-positive volume and one puncture}\label{npvop} This is an easy case to deal with. Let $\alpha,\,\beta$ be a pair of handle generators and let $a,b$ be their images via $\chi$. Let $\mathcal{P}$ be the exterior of the parallelogram (possibly degenerate) defined by the chain
\begin{equation}
    P\mapsto P+a\mapsto P+a+b=Q\mapsto P+b\mapsto P,
\end{equation}
with sides $a^\pm,\,b^\pm$ according to our convention. Let $r_P$ be a geodesic ray joining $P$ and the pole and glue along $k-2$ copies of the genus zero differential $(\C,\,dz)$ as in Definition \ref{gluingsurfaces}. Similarly, let $r_Q$ be a geodesic ray joining $Q$ and the pole and glue along $p-k$ copies of the genus zero differential $(\C,\,dz)$. We finally glue $a^+$ with $a^-$ and $b^+$ with $b^-$. The resulting surface is a genus one differential in $\mathcal{H}_1(p,-p)$. It remains to show that it has rotation number $\gcd(k,p)$. This can be easily verified by choosing $\alpha$ and $\beta$ as in \S\ref{npvonp} above. See Figure \ref{rotnhop}.

\begin{figure}[ht]
    \centering
    \begin{tikzpicture}[scale=1.2, every node/.style={scale=0.8}]
    \definecolor{pallido}{RGB}{221,227,227}

    \pattern [pattern=north west lines, pattern color=pallido]
    (-1.5,-1)--(3,-1)--(3,3)--(-1.5,3)--(-1.5,-1);
    \fill [white] (0,0)--(2,0)--(2,2)--(0,2)--(0,0);
    \draw [thin, dashed, black] (2,0) -- (3,-1);
    \draw [thin, dashed, black] (2,0) -- (3,-0.5);
    \draw [thin, dashed, black] (2,0) -- (2.5,-1);
    
    \draw [thin, dashed, black] (0,2) -- (-1.5,3);
    \draw [thin, dashed, black] (0,2) -- (-1.5,2.75);
    \draw [thin, dashed, black] (0,2) -- (-1.5,2.5);
    \draw [thin, dashed, black] (0,2) -- (-1.5,2.25);
    \draw [thick, orange] (0,0)--(2,0);
    \draw [thick, red] (2,0)--(2,2);
    \draw [thick, orange] (2,2)--(0,2);
    \draw [thick, red] (0,2)--(0,0);
    \draw [thick, orange, ->] (0,0) to (1,0);
    \draw [thick, red, ->>] (2,2) to (2,1);
    \draw [thick, orange, ->] (0,2) to (1,2);
    \draw [thick, red, ->>] (0,2) to (0,1);
    \fill (0,0) circle (1.5pt);
    \fill (0,2) circle (1.5pt);
    \fill (2,2) circle (1.5pt);
    \fill (2,0) circle (1.5pt);
    
    \draw [thick, violet, ->] (0.25,2.0625) arc (50:310:0.75 and 1.375);
    
    \draw [thick, blue, <-] (2.0625,1.75) arc (-45:225:1.5 and 0.5);

    \node at (-0.875,-0.75) {$\mathcal{P}\subset\cp$};
    \node at (1,2.375) {$\alpha$};
    \node at (-0.75,1) {$\beta$};
    \node at (3.5, -0.75) {$p-k-1$};
    \node at (-2, 2.5) {$k-1$};
    
    \end{tikzpicture}
    \caption{Realizing a genus one translation surface with one single pole of order $p$, non-positive volume, and rotation number equal to $\gcd(k,p)$.}
    \label{rotnhop}
\end{figure}
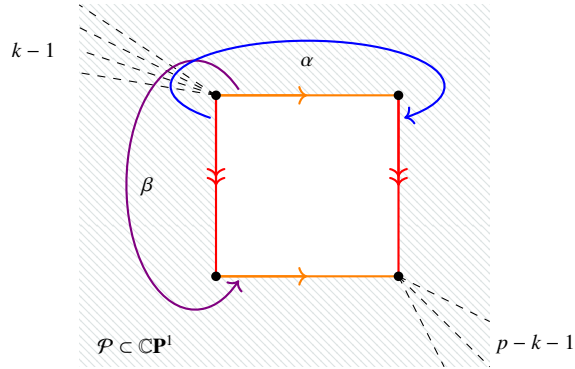

\subsubsection{Non-positive volume and two punctures}\label{npv2p} In this case we consider the quadrilateral $\mathcal{Q}$ already defined in paragraph \S\ref{nvep}. Recall that $\mathcal{Q}$ is obtained by gluing two triangles, $\mathcal{T}_1$ with vertices $P,\,P+\chi(\alpha),\,Q$ and $\mathcal{T}_2$ with vertices $P,\,P-\chi(\beta),\,Q$, where $Q=P+\chi(\alpha)-\chi(\beta)$. There are two pairs of opposite sides one of which differ by a translation $\chi(\alpha)$ and the other by a translation $\chi(\beta)$. See paragraph \S\ref{nvep} for more details about this construction.

\begin{figure}[ht]
    \centering
    \begin{tikzpicture}[scale=1, every node/.style={scale=0.8}]
    \definecolor{pallido}{RGB}{221,227,227}
    \definecolor{pallido2}{RGB}{221,227,240}
    
    \pattern [pattern=north west lines, pattern color=pallido]
    (0,0)--(4,0)--(4,4)--(0,4)--(0,0);
    \draw[thick, orange] (0,0)--(4,0);
    \draw[thick, orange] (0,4)--(4,4);
    \draw[thick, red] (0,0)--(0,4);
    \draw[thick, red] (4,0)--(4,4);
    \draw[thick, orange, ->] (0,0)--(2,0);
    \draw[thick, orange, ->] (0,4)--(2,4);
    \draw[thick, red, ->>] (0,0)--(0,2);
    \draw[thick, red, ->>] (4,0)--(4,2);
    
    \fill (0,0) circle (1.5pt);
    \fill (0,4) circle (1.5pt);
    \fill (4,4) circle (1.5pt);
    \fill (4,0) circle (1.5pt);
    \draw [thin, dashed] (0,4)--(4,0);
    \draw [thick, dotted, bend right] (0,4)--(1.5,1.5);
    \draw [thick, dotted, bend left] (0,0)--(1.5,1.5);
    \draw [thick, dotted] (4,4)--(2.5,2.5);
    
    \fill (1.5,1.5) circle (1.5pt);
    \node at (1.5,1.25) {$\infty$};
    \fill (2.5,2.5) circle (1.5pt);
    \node at (2.5,2.75) {$\infty$};
    
    \draw[thick, violet, ->] (0.5,0)--(0.5,4);
    \draw[thick, blue, ->] (0,0.5)--(4,0.5);    
    
    \node at (0.5,4.25) {$\beta$};
    \node at (4.25, 0.5) {$\alpha$};
    \node at (-0.25,-0.25) {$P-\chi(\beta)$};
    \node at (4.25, -0.25) {$Q$};
    \node at (4.25,4.25) {$P+\chi(\alpha)$};
    \node at (-0.25, 4.25) {$P$};
    
    \pattern [pattern=north east lines, pattern color=pallido]
    (6.5,-0.5)--(8,1)--(8,3)--(10,3)--(11.5,4.5)--(6.5,4.5)--(6.5,-0.5);
    \pattern [pattern=north west lines, pattern color=pallido2]
    (6.5,-0.5)--(11.5,-0.5)--(11.5,4.5)--(10,3)--(10,1)--(8,1)--(6.5,-0.5);

    \fill [white] (8,1)--(10,1)--(10,3)--(8,3)--(8,1);
    \draw [thin, dashed, black] (8,3) -- (8.25, 4.5);
    \draw [thin, dashed, black] (8,3) -- (8.75, 4.5);
    \draw [thin, dashed, black] (8,3) -- (7.25, 4.5);
    \draw [thin, dashed, black] (8,3) -- (7.75, 4.5);
    
    \draw [thin, dashed, black] (8,1) -- (6.5, 1);
    \draw [thin, dashed, black] (8,1) -- (6.5, 1.25);
    \draw [thin, dashed, black] (8,1) -- (6.5, 0.75);
    
    \draw [thin, dashed, black] (10,1) -- (11.5, 1.5);
    \draw [thin, dashed, black] (10,1) -- (11.5, 1.25);
    \draw [thin, dashed, black] (10,1) -- (11.5, 1);
    \draw [thin, dashed, black] (10,1) -- (11.5, 0.75);
    \draw [thin, dashed, black] (10,1) -- (11.5, 0.5);
    
    \draw [thick, orange] (8,1)--(10,1);
    \draw [thick, red] (8,1)--(8,3);
    \draw [thick, orange] (8,3)--(10,3);
    \draw [thick, red] (10,1)--(10,3);
    \draw [thick, orange, ->] (8,1) to (9,1);
    \draw [thick, red, ->>] (8,3) to (8,2);
    \draw [thick, orange, ->] (8,3) to (9,3);
    \draw [thick, red, ->>] (10,3) to (10,2);
    \fill (8,1) circle (1.5pt);
    \fill (8,3) circle (1.5pt);
    \fill (10,1) circle (1.5pt);
    \fill (10,3) circle (1.5pt);
    \draw [thin, dashed] (8,1)--(6.5,-0.55);
    \draw [thin, dashed] (10,3)--(11.5, 4.5);

    \node at (8, 4.75) {$k-2$};
    \node at (12, 1) {$p-2$};
    \node at (6, 1) {$p-k$};
    \node at (11,0) {$\mathcal{T}_1$};
    \node at (7, 3.75) {$\mathcal{T}_2$};
    \node at (7.375, 3) {$P-\chi(\beta)$};
    \node at (10.375, 3) {$Q$};
    \node at (9.75, 0.75) {$P+\chi(\alpha)$};
    \node at (8, 0.75) {$P$};
    
    \end{tikzpicture}
    \caption{The quadrilateral $\mathcal{Q}$ from two different perspectives. On the left-hand-side we can see the curves $\alpha$ and $\beta$. Any dotted line represents a ray from a vertex of $\mathcal{Q}$ to a point at infinity. These rays are better represented on the right-hand-side. }
    \label{quadnekmore}
\end{figure}
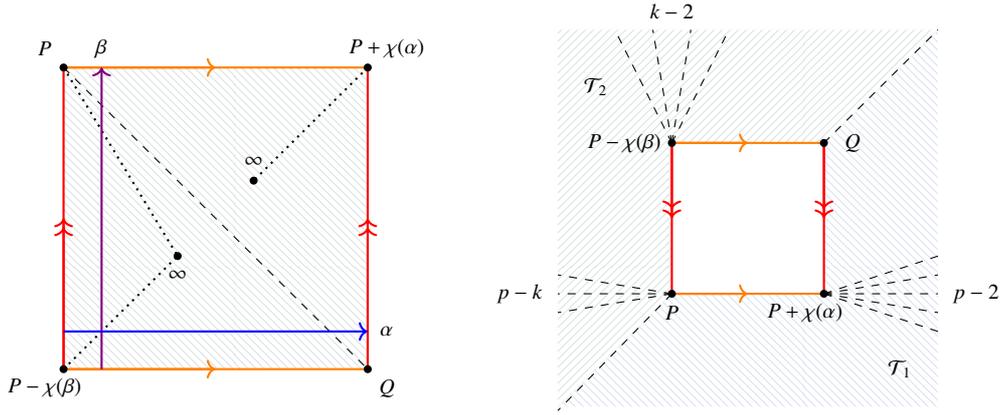

\noindent Given a positive integer $k\in\{2,\dots,p\}$, we consider $k-2$ rays leaving from $P-\chi(\beta)$ and $p-k$ rays leaving from $P$ such that they all point toward the pole contained in $\mathcal{T}_2$. Next we consider other $p-2$ rays all leaving from $P+\chi(\alpha)$ and pointing toward the pole contained in $\mathcal{T}_1$. By gluing the opposite sides of $\mathcal{Q}$ by using the translations $\chi(\alpha)$ and $\chi(\beta)$ we obtain a genus one meromorphic differential, say $(X,\omega)$ with a single zero of order $2p$ and two poles of order $p$. It remains to show that $(X,\omega)$ has rotation number $\gcd(k,p)$. We can choose the curves $\alpha$ and $\beta$ as in \S\ref{nvep}, see Figure \ref{quadnekmore}. By construction, $\alpha$ has index $k$ and $\beta$ has index $p$. Therefore $\gcd(k,p,2p,p)=\gcd(k,p)$ as desired.

\subsubsection{Non-positive volume and more than two punctures}\label{npvmp} This is the last case to handle in this section and it is similar to the case discussed in \S\ref{pvmp}. Recall that, for a positive integer $k\in\{2,\dots,p\}$ we have introduced two integers $d,l\in\Z^+$ and used them to define the genus zero differentials $(X,\omega)$, $(Y,\xi)$ and $(Z, \eta)$, see Section \S\ref{pvmp} for these constructions. In this case we still consider $l$ copies, say $(X_1,\omega_1),\dots,(X_l,\omega_l)$ of $(X,\omega)$. Recall that $l$ can be zero. We then consider $n-l-2$ copies of $(Y,\xi)$ that we denote as $(Y_1,\xi_1),\dots,(Y_{n-l-2},\xi_{n-l-2})$ and a unique copy of $(Z,\eta)$. 

\smallskip

\noindent Let $\alpha,\,\beta$ be a pair of handle generators and let $\mathcal{P}$ be the closure of the exterior parallelogram in $\C$ defined by the chain
\begin{equation}
    P_0\mapsto P_0+\chi(\alpha)\mapsto P_0+\chi(\alpha)+\chi(\beta)=Q_0\mapsto P_0+\chi(\beta)\mapsto P_0,
\end{equation}
\noindent where $P_0\in\C$ is any point. Notice that we can find $p-2$ rays, say $r_1,\dots,r_{p-2}$, leaving from $Q$ pointing at the infinity, see Figure \ref{rotnh}. As in Section \S\ref{pvmp} we glue all these structures together in the usual way. The resulting structure is a genus one meromorphic differential. By bubbling a copy of $(\C,\,dz)$ along $r_i$, for any $i=1,\dots,p-2$, we obtain a genus one meromorphic differential with one zero of order $np$ and $n$ poles of order $p$.

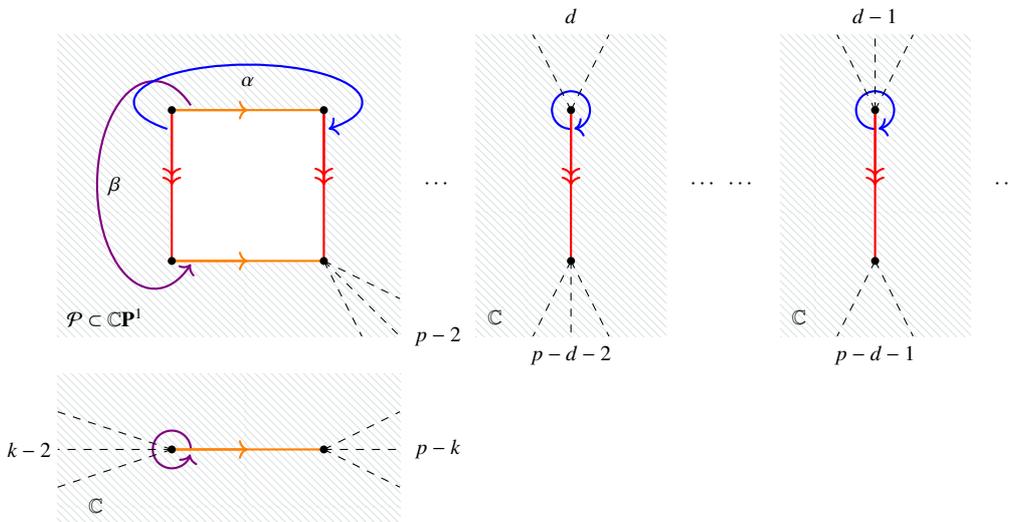
\begin{figure}[ht]
    \centering
    \begin{tikzpicture}[scale=1, every node/.style={scale=0.8}]
    \definecolor{pallido}{RGB}{221,227,227}
   
    \pattern [pattern=north west lines, pattern color=pallido]
    (-1.5,-1.5)--(3,-1.5)--(3,-3.5)--(-1.5,-3.5)--(-1.5,-1.5);
    \draw [thick, violet, ->] (0.25, -2.5) arc [start angle = 0, end angle =345 , radius = 0.25];
    \draw [thick, orange] (0,-2.5)--(2,-2.5);
    \draw [thick, orange, ->] (0,-2.5)--(1,-2.5);
    \draw [thin, black, dashed] (0,-2.5)--(-1.5,-2);
    \draw [thin, black, dashed] (0,-2.5)--(-1.5,-2.5);
    \draw [thin, black, dashed] (0,-2.5)--(-1.5,-3);
    \draw [thin, black, dashed] (2,-2.5)--(3,-2);
    \draw [thin, black, dashed] (2,-2.5)--(3,-2.5);
    \draw [thin, black, dashed] (2,-2.5)--(3,-3);
    \fill (0,-2.5) circle (1.5pt);
    \fill (2,-2.5) circle (1.5pt);
   
    \pattern [pattern=north west lines, pattern color=pallido]
    (-1.5,-1)--(3,-1)--(3,3)--(-1.5,3)--(-1.5,-1);
    \fill [white] (0,0)--(2,0)--(2,2)--(0,2)--(0,0);
    \draw [thin, dashed, black] (2,0) -- (3,-1);
    \draw [thin, dashed, black] (2,0) -- (3,-0.5);
    \draw [thin, dashed, black] (2,0) -- (2.5,-1);
    \draw [thick, orange] (0,0)--(2,0);
    \draw [thick, red] (2,0)--(2,2);
    \draw [thick, orange] (2,2)--(0,2);
    \draw [thick, red] (0,2)--(0,0);
    \draw [thick, orange, ->] (0,0) to (1,0);
    \draw [thick, red, ->>] (2,2) to (2,1);
    \draw [thick, orange, ->] (0,2) to (1,2);
    \draw [thick, red, ->>] (0,2) to (0,1);
    \fill (0,0) circle (1.5pt);
    \fill (0,2) circle (1.5pt);
    \fill (2,2) circle (1.5pt);
    \fill (2,0) circle (1.5pt);
    
    \draw [thick, violet, ->] (0.25,2.0625) arc (50:310:0.75 and 1.375);
    
    \draw [thick, blue, <-] (2.0625,1.75) arc (-45:225:1.5 and 0.5);
 
    \foreach \x [evaluate=\x as \coord using 4 + 2*\x] in {0} 
    {
    \pattern [pattern=north west lines, pattern color=pallido]
    (\coord, -1)--(\coord+2.5, -1)--(\coord+2.5, 3)--(\coord, 3)--(\coord,-1);
    \draw [thick, blue, ->] (\coord+1.25, 1.75) arc [start angle = 270, end angle = -80,radius = 0.25];
    \draw [thick, red] (\coord+1.25, 0) to (\coord+1.25, 2);
    \draw [thick, red, <<-] (\coord+1.25,1) to (\coord+1.25,2);
    \fill (\coord+1.25,0) circle (1.5pt);
    \fill (\coord+1.25,2) circle (1.5pt);
    \draw [thin, black, dashed] (\coord+1.25, 2)--(\coord+0.75, 3);
    \draw [thin, black, dashed] (\coord+1.25, 2)--(\coord+1.75, 3);
    \draw [thin, black, dashed] (\coord+1.25, 0)--(\coord+0.75, -1);
    \draw [thin, black, dashed] (\coord+1.25, 0)--(\coord+1.75, -1);
    \draw [thin, black, dashed] (\coord+1.25, 0)--(\coord+1.25, -1);
    \node at (\coord-0.5, 1) {$\cdots$};
    \node at (\coord+3, 1) {$\cdots$};
    \node at (\coord+0.25, -0.75) {$\mathbb{C}$};   
    \node at (\coord+1.25, 3.25) {$d$};
    \node at (\coord+1.25, -1.25) {$p-d-2$};
    }
    
    \foreach \x [evaluate=\x as \coord using 4 + 2*\x] in {2} 
    {
    \pattern [pattern=north west lines, pattern color=pallido]
    (\coord, -1)--(\coord+2.5, -1)--(\coord+2.5, 3)--(\coord, 3)--(\coord,-1);
    \draw [thick, blue, ->] (\coord+1.25, 1.75) arc [start angle = 270, end angle = -80,radius = 0.25];
    \draw [thick, red] (\coord+1.25, 0) to (\coord+1.25, 2);
    \draw [thick, red, <<-] (\coord+1.25,1) to (\coord+1.25,2);
    \fill (\coord+1.25,0) circle (1.5pt);
    \fill (\coord+1.25,2) circle (1.5pt);
    \node at (\coord+0.25, -0.75) {$\mathbb{C}$}; 
    \node at (\coord-0.5, 1) {$\cdots$};
    \node at (\coord+3, 1) {$\cdots$};
    \draw [thin, black, dashed] (\coord+1.25, 2)--(\coord+0.75, 3);
    \draw [thin, black, dashed] (\coord+1.25, 2)--(\coord+1.25, 3);
    \draw [thin, black, dashed] (\coord+1.25, 2)--(\coord+1.75, 3);
    \draw [thin, black, dashed] (\coord+1.25, 0)--(\coord+0.75, -1);
    \draw [thin, black, dashed] (\coord+1.25, 0)--(\coord+1.75, -1);
    \node at (\coord+1.25, 3.25) {$d-1$};
    \node at (\coord+1.25, -1.25) {$p-d-1$};
    }
    
    \node at (-0.875,-0.75) {$\mathcal{P}\subset\cp$};
    \node at (-1,-3.25) {$\mathbb{C}$};
    \node at (1,2.375) {$\alpha$};
    \node at (-0.75,1) {$\beta$};
    \node at (-1.875, -2.5) {$k-2$};
    \node at (3.5, -2.5) {$p-k$};
    \node at (3.5, -1) {$p-2$};

    \end{tikzpicture}
    \caption{Realizing a genus one translation surface with poles of order $p$, negative volume, and rotation number equal to $\gcd(k,p)$. In this case the number of punctures is supposed to be at least three and all the poles have zero residue.}
    \label{rotnh}
\end{figure}

\noindent It just remains to show that this latter structure has the desired rotation number $\gcd(k,p)$. If we choose the curves $\alpha$ and $\beta$ as shown in Figure \ref{rotnh}, we can observe that
\begin{equation}
    \text{Ind}(\alpha)=
    \begin{cases}
    k \quad \text{ if } k-1\ge n-2, \text{ or}\\
    tk \quad \text{ if } k-1<n-2,
    \end{cases}
\end{equation} 
\noindent and $\text{Ind}(\beta)=k$ by construction. Therefore $\gcd(tk,k,np,p)=\gcd(k,p)$ as desired. This  concludes the proof of Lemma \ref{onepuncttorushigherp}. \qed\qedhere

\smallskip

\begin{rmk}
It is also possible to realize a genus one meromorphic differential with poles of order $p$, a single zero of maximal order and rotation number one by bubbling sufficiently many copies of $(\C,\,dz)$ along rays on a genus one differential with poles of order two obtained as in Section \S\ref{rot1}. Suppose $(X,\omega)$ is a such a structure. We can find $n$ rays, all leaving from the unique zero of $\omega$ and such that each one joins a pole of $\omega$. Different rays join different poles. Then we can bubble $p-2$ copies of $(\C,\, dz)$ along every ray. The resulting structure turns out to be a genus one differential with $n$ poles each of order $p$. If these rays are properly chosen, it is also possible to preserve the rotation number. In Figures \ref{rotpk1} and \ref{rotnk1} the dotted lines are possible candidates. Bubbling along them one gets a genus one meromorphic differential with poles of order $p$ and, since the indices of the curves $\alpha$ and $\beta$ are not affected by any bubbling, the rotation number does not change. In a similar fashion in certain cases it is possible to realize a genus one meromorphic differential with poles of order $p$ and rotation number two. In fact, it is possible to modify the structures obtained in Figures \ref{rotpek2}, \ref{rotpok2} and \ref{rotnok2} by bubbling $p-2$ copies of $(\C,\, dz)$ along each dotted line. The indices of $\alpha$ and $\beta$ are not affected in these cases and therefore the resulting structures have rotation number two. 
\end{rmk}

\subsection{Strata with poles with non-zero residue}\label{resnotzero} We consider the case of characters $\chi$ of non-trivial-ends type. Notice that this automatically implies $n\ge2$. Here we shall consider $S_{1,n}$ as the surface obtained by gluing the surfaces $S_{1,2}$ and $S_{0,n}$ along a ray joining two punctures as in Definition \ref{gluingsurfaces} with the only difference that for the moment no geometry is involved, see Figure \ref{subsurface}.

\begin{figure}[!ht]
    \centering
    \begin{tikzpicture}[scale=1, every node/.style={scale=0.875}]
    \definecolor{pallido}{RGB}{221,227,227}

    \fill[pattern=north west lines, pattern color=pallido] (-5.5,0) ellipse (3cm and 2cm); 
    \fill[white] (-6.5,0) to [out=330, in=210] (-4.5,0) to [out=150, in=30] (-6.5,0);
    \draw [thick, blue, dashed] (-8.5,0) to [out=30, in=150] (-6.5,0);
    \draw [thick, red] (-5.5,0) ellipse (1.5cm and 0.75cm);
    \draw [thick] (-5.5,0) ellipse (3cm and 2cm);
    \draw [thick] (-6.7,0.1) to [out=330, in=210] (-4.3,0.1);
    \draw [thick] (-6.5,0) to [out=30, in=150] (-4.5,0);
    \draw [thick, blue] (-8.5,0) to [out=330, in=210] (-6.5,0);

    \node at (-3.5,-0.75) {$\times$};
    \draw [thick, violet, ->] (-3.125, -0.75) arc [start angle = 0, end angle = 360,radius = 0.375];
    \node at (-3.5,0.75) {$\times$};
    \draw [thick, violet, <-] (-3.125, 0.75) arc [start angle = 0, end angle = 360,radius = 0.375];
    \node at (-8.75,0) {\textcolor{blue}{$\alpha$}};
    \node at (-5.5,-1) {\textcolor{red}{$\beta$}};
    \node at (-3.125, -0.375) {\textcolor{violet}{$\delta_2$}};
    \node at (-3.125,   0.375) {\textcolor{violet}{$\delta_1$}};
    \draw [thick] (0.5,0) ellipse (2cm and 2cm);
    \fill[pattern=north west lines, pattern color=pallido] (0.5,0) ellipse (2cm and 2cm);
    \node at (0.5,-1.5) {$\times$};
    \draw [thick, violet, <-] (0.875, -1.5) arc [start angle = 0, end angle = 360,radius = 0.375];
    \node at (0.5,1.5) {$\times$};
    \draw [thick, violet, <-] (0.875, 1.5) arc [start angle = 0, end angle = 360,radius = 0.375];
    \node at (2,0) {$\times$};
    \draw [thick, violet, <-] (2.375, 0) arc [start angle = 0, end angle = 360,radius = 0.375];
    \node at (-1,0) {$\times$};
    \draw [thick, violet, <-] (-0.625, 0) arc [start angle = 0, end angle = 360,radius = 0.375];
    \node at (1.25, -1.5) {\textcolor{violet}{$\gamma_3$}};
    \node at (1.25,  1.5) {\textcolor{violet}{$\gamma_1$}};
    \node at (1.375, 0) {\textcolor{violet}{$\gamma_2$}};
    \node at (-0.325, 0) {\textcolor{violet}{$\gamma_n$}};
    \draw[thick, gray] (2,0) to (0.5,1.5);
    \draw[thick, gray] (-3.5, -0.75) to (-3.5, 0.75);
    \node at (-0.25, -0.75) {$\ddots$};
    \end{tikzpicture}
    \caption{The surfaces $S_{1,2}$ and $S_{0,n}$. By slitting and pasting along the infinite rays coloured in gray, the resulting surface is homeomorphic to $S_{1,n}$.}
    \label{subsurface}
\end{figure}
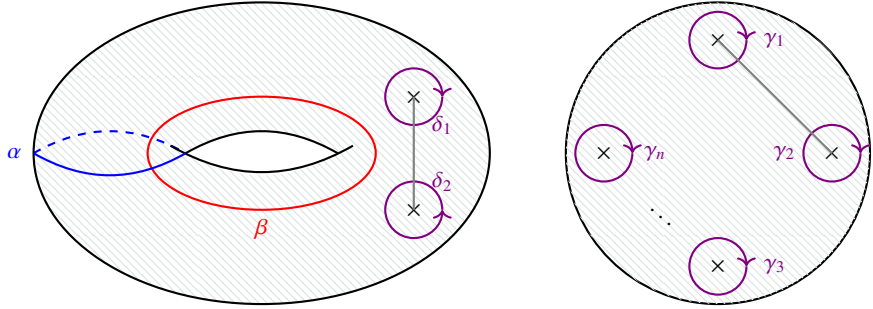

\smallskip

\noindent The idea behind the proof of this case is as follow. Starting from the representation $\chi$ we shall introduce a new representation $\rho\colon\text{H}_1(S_{1,\,2},\,\Z)\longrightarrow \C$ which, in principle, does not need to be the restriction of $\chi$ to any subsurface of $S_{1,\,n}$ homeomorphic to $S_{1,2}$. For any $p\ge2$ and for any $k\in\Z$ dividing $p$ we can realize $\rho$ in the stratum $\mathcal{H}_1(2p;-p,-p)$ as the holonomy of a genus one differential with rotation number $k$ as done in Sections \S\ref{genusoneordertwo} and \S\ref{genusoneorderhigherthantwo}. Then we properly modify this structure in order to get the desired one. To this latter, in the case $n\ge3$, we shall finally glue a genus zero meromorphic differential in order to obtain the desired structure on $S_{1,\,n}$ with period character $\chi$ and prescribed rotation number. 

\smallskip

\noindent To define an auxiliary representation $\rho$, we provide a more general definition which we need to use later on.

\begin{defn}[Auxiliary representation]\label{defsupprep}
 Let $\chi\colon\shomolzn\longrightarrow \C$ be a representation of non-trivial-ends type and $\imath\colon S_{h,\,m}\hookrightarrow S_{g,n}$ be an embedding with $1\le h\le g$ and $1\le m \le n$. Let $\{\alpha_i,\beta_i\}_{1\le i\le h}$ be a system of handle generators for $ \text{H}_1(S_{h,\,m},\,\Z)$. We define an \emph{auxiliary representation} $\rho\colon \text{H}_1(S_{h,\,m},\,\Z)\longrightarrow \C$ as follows:
\begin{equation}\label{supprep}
    \rho(\alpha_i)=\chi\circ\imath(\alpha_i), \quad \rho(\beta_i)=\chi\circ\imath(\beta_i), \quad \rho(\delta_1)=\cdots=\rho(\delta_m)=0
\end{equation} where $\delta_1,\dots,\delta_m$ are simple closed curves each of which  encloses a puncture on $S_{h,\,m}$ and  oriented in such a way that $[\alpha_1,\beta_1]\cdots[\alpha_h,\beta_h]=\delta_1\cdots\delta_m$. If $h=g$ and $m=n$, then the embedding $\imath$ can be taken as the identity map. 
\end{defn}

\noindent Before proceeding, once again, we recall for the reader's convenience that no puncture is assumed to be a simple pole, see Remark \ref{conscases}. Moreover, in this section we assume that all poles have the same order $p\ge2$. We shall distinguish three cases according to the following subsections.

\smallskip

\subsubsection{Two poles with non-zero residue}\label{2pntr} We suppose $n=2$ and hence $\chi\colon\text{H}_1(S_{1,2},\,\Z)\longrightarrow \C$ is a representation  such that 
\begin{equation}
    \chi(\delta_1)=\chi(\delta_2^{-1})=w\in\C^*,
\end{equation} 

\noindent where $\delta_1,\,\delta_2$ are non-homotopic simple closed curves both enclosing a single puncture of $S_{1,2}$. We introduce an auxiliary representation $\rho\colon\text{H}_1(S_{1,2},\,\Z)\longrightarrow \C$, as in Definition \ref{defsupprep}, defined as 
\begin{equation}\label{eq:auxreponetwo}
    \rho(\alpha)=\chi(\alpha),\,\qquad\,\rho(\beta)=\chi(\beta),\,\qquad\, \rho(\delta_1)=\rho(\delta_2)=0.
\end{equation}

\noindent Let $(X,\omega)$ be the translation surface with period character $\rho$ and rotation number $k$ realized as in Section \S\ref{genusonemero}; more precisely as in \S\ref{genusoneordertwo} if $p=2$ or as in  \S\ref{pv2p} and \S\ref{npv2p} if $p>2$. The following holds.

\begin{lem}\label{exray}
There is a bi-infinite geodesic ray joining the poles such that its direction after developing, say $v\in\C$, is different from $w$.
\end{lem}

\noindent Suppose Lemma \ref{exray} holds; then we can find a bi-infinite geodesic ray $r\subset (X,\omega)$ joining the two poles and such that its direction after developing is $v\neq w$. Let $C_w=\mathbb{E}^2/\langle z\mapsto z+w\rangle$ and let $\pi_w\colon\mathbb{E}^2\to C_w$ be the covering projection. Since $v\neq w$, we can glue $C_w$ along the bi-infinite ray $r$ as in Definition \ref{gluecyl}. In fact, any straight line with direction $v$ in $\mathbb E^2$ projects to a bi-infinite geodesic line which wraps around $C_w$. Let $\overline{r}\subset C_w$ be a bi-infinite geodesic ray with direction $v$ after developing. Slit $(X,\omega)$ along $r$ and let $r^\pm$ be the resulting edges. Similarly, slit $C_w$ along $\overline{r}$ and call the resulting edges $\overline{r}^\pm$. Notice that $C_w\setminus \overline{r}$ is an infinite open strip whose closure is obtained by adding the edges $\overline{r}^\pm$. We can glue this latter closed strip to the slit  $(X,\omega)$ by identifying $r^+$ with $\overline{r}^-$ and $r^-$ with $\overline{r}^+$. The resulting translation surface $(Y,\xi)$ is still a genus one differential. The differential $\xi$ has a single zero of order $2p$ and two poles of order $p$ with residue $\pm w$ by construction. Therefore $(Y,\xi)$ has period character $\chi$ and rotation number $k$ as a consequence of Lemma \ref{invrotnum}. It   remains to show Lemma \ref{exray}.

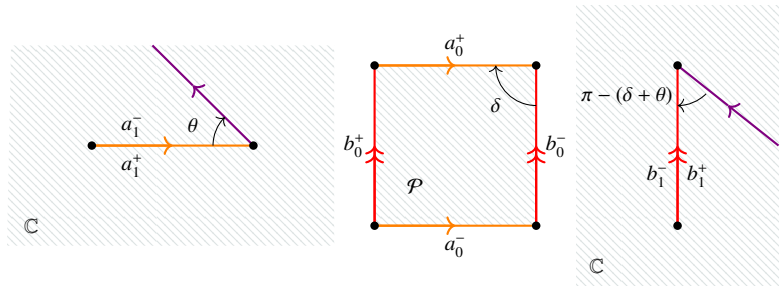
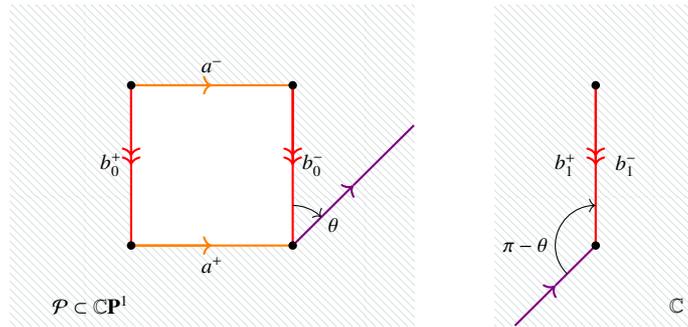
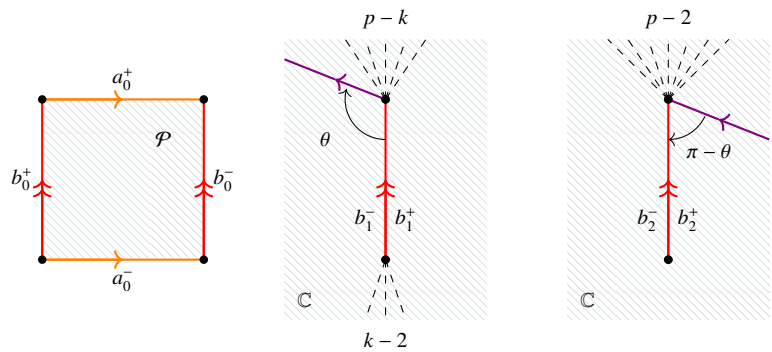
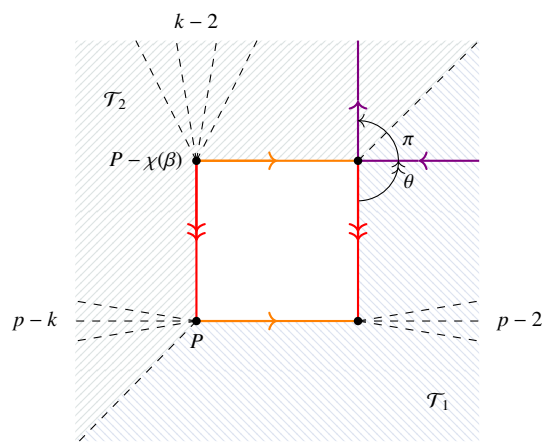
\begin{figure}
    \centering
    \subfloat[][Case 1 of Lemma \ref{exray}. \label{exray1}]
    {
    \begin{tikzpicture}[scale=1.0625, every node/.style={scale=0.8}]
    \definecolor{pallido}{RGB}{221,227,227}
    
    \pattern [pattern=north west lines, pattern color=pallido]
    (-4.5,-0.25)--(-0.5,-0.25)--(-0.5,2.25)--(-4.5,2.25)--(-4.55,-0.25);
    \draw [thick, orange] (-3.5,1) to (-1.5,1);
    \draw [thick, orange, ->] (-3.5,1) to (-2.5,1);
    \draw [thick, violet] (-2.25, 1.75) to (-2.75, 2.25);
    \draw [thick, violet, ->] (-1.5, 1) to (-2.25, 1.75);
    \draw [thin, black, ->] (-2, 1) arc [start angle = 180, end angle = 135,radius = 0.5];
    \fill (-3.5,1) circle (1.5pt);
    \fill (-1.5,1) circle (1.5pt);
    \node at (-3, 0.75) {$a_1^+$};
    \node at (-3, 1.25) {$a_1^-$};
    \node at (-2.25, 1.25) {$\theta$};
    
    \pattern [pattern=north west lines, pattern color=pallido]
    (0,0)--(2,0)--(2,2)--(0,2)--(0,0);
    \draw [thick, orange] (0,0)--(2,0);
    \draw [thick, red] (2,0)--(2,2);
    \draw [thick, orange] (2,2)--(0,2);
    \draw [thick, red] (0,2)--(0,0);
    \draw [thick, orange, ->] (0,0) to (1,0);
    \draw [thick, red, ->>] (2,0) to (2,1);
    \draw [thick, orange, ->] (0,2) to (1,2);
    \draw [thick, red, ->>] (0,0) to (0,1);
    \fill (0,0) circle (1.5pt);
    \fill (0,2) circle (1.5pt);
    \fill (2,2) circle (1.5pt);
    \fill (2,0) circle (1.5pt);
    \draw [thin, black, ->] (2, 1.5) arc [start angle = 270, end angle = 180,radius = 0.5];
    \node at (1, -.25) {$a_0^-$};
    \node at (1, 2.25) {$a_0^+$};
    \node at (-0.25, 1) {$b_0^+$};
    \node at (2.25, 1) {$b_0^-$};
 
    \foreach \x [evaluate=\x as \coord using 2.5 + 3*\x] in {0} 
    {
    \pattern [pattern=north west lines, pattern color=pallido]
    (\coord, -0.75)--(\coord+2.5, -0.75)--(\coord+2.5, 2.75)--(\coord, 2.75)--(\coord,-0.75);
    \draw [thick, red] (\coord+1.25, 0) to (\coord+1.25, 2);
    \draw [thick, red, ->>] (\coord+1.25,0) to (\coord+1.25,1);
    \draw[thick, violet] (\coord+1.25, 2) to (\coord+1.875, 1.5);
    \draw[thick, violet, <-] (\coord+1.875, 1.5) to (\coord+2.5,1);
    \draw [thin, black, <-] (\coord+1.25, 1.5) arc [start angle = 270, end angle = 315,radius = 0.5];
    \node at (\coord+0.625, 1.625) {$\pi-(\delta+\theta)$};
    \fill (\coord+1.25,0) circle (1.5pt);
    \fill (\coord+1.25,2) circle (1.5pt);
    \node at (\coord+0.25, -0.5) {$\mathbb{C}$};  
    }
    
    \node at (1.5, 1.5) {$\delta$};
    \node at (3.5, 0.625) {$b_1^-$};
    \node at (4, 0.625) {$b_1^+$};
    \node at (0.5, 0.5) {$\mathcal{P}$};
    \node at (-4.25,0) {$\mathbb{C}$};
    \end{tikzpicture}
    }\\
    \subfloat[][Case 2 of Lemma \ref{exray}. \label{exray2}]
    {
    \begin{tikzpicture}[scale=1.0625, every node/.style={scale=0.8}]
    \definecolor{pallido}{RGB}{221,227,227}
   
    \pattern [pattern=north west lines, pattern color=pallido]
    (-1.5,-1)--(3.5,-1)--(3.5,3)--(-1.5,3)--(-1.5,-1);
    \fill [white] (0,0)--(2,0)--(2,2)--(0,2)--(0,0);
    \draw [thick, orange] (0,0)--(2,0);
    \draw [thick, red] (2,0)--(2,2);
    \draw [thick, orange] (2,2)--(0,2);
    \draw [thick, red] (0,2)--(0,0);
    \draw [thick, orange, ->] (0,0) to (1,0);
    \draw [thick, red, ->>] (2,2) to (2,1);
    \draw [thick, orange, ->] (0,2) to (1,2);
    \draw [thick, red, ->>] (0,2) to (0,1);
    \draw [thick, violet, ->] (2, 0) to (2.75, 0.75);
    \draw [thick, violet] (2.75, 0.75) to (3.5, 1.5);
    \fill (0,0) circle (1.5pt);
    \fill (0,2) circle (1.5pt);
    \fill (2,2) circle (1.5pt);
    \fill (2,0) circle (1.5pt);
    \node at (1,2.25) {$a^-$};
    \node at (1,-0.25) {$a^+$};
    \node at (2.25,1) {$b_0^-$};
    \node at (-0.25,1) {$b_0^+$};
    \draw [thin, black, ->] (2, 0.5) arc [start angle = 90, end angle = 45,radius = 0.5];
    \node at (2.5, 0.25) {$\theta$};

    \foreach \x [evaluate=\x as \coord using 4.5 + 3.5*\x] in {0} 
    {
    \pattern [pattern=north west lines, pattern color=pallido]
    (\coord, -1)--(\coord+2.5, -1)--(\coord+2.5, 3)--(\coord, 3)--(\coord,-1);
    \draw [thick, red] (\coord+1.25, 0) to (\coord+1.25, 2);
    \draw [thick, red, <<-] (\coord+1.25,1) to (\coord+1.25,2);
    \draw [thick, violet] (\coord+1.25, 0) to (\coord+0.75, -0.5);
    \draw [thick, violet, <-]     (\coord+0.75, -0.5) to (\coord+0.25, -1);
    \draw [thin, black, <-] (\coord+1.25, 0.5) arc [start angle = 90, end angle = 225,radius = 0.5];
    \node at (\coord+0.375, 0) {$\pi-\theta$};
    \fill (\coord+1.25,0) circle (1.5pt);
    \fill (\coord+1.25,2) circle (1.5pt);
    \node at (\coord+2.25, -0.75) {$\mathbb{C}$};   

    }
    
    \node at (5.375, 1) {$b_1^+$};
    \node at (6.125, 1) {$b_1^-$};
    \node at (-0.5,-0.75) {$\mathcal{P}\subset\cp$};
    \end{tikzpicture}
    }\\
    \subfloat[][Case 3 of Lemma \ref{exray}. \label{exray3}]
    {
    \begin{tikzpicture}[scale=1.0625, every node/.style={scale=0.8}]
    \definecolor{pallido}{RGB}{221,227,227}
    
    \pattern [pattern=north west lines, pattern color=pallido]
    (0,0)--(2,0)--(2,2)--(0,2)--(0,0);
    \draw [thick, orange] (0,0)--(2,0);
    \draw [thick, red] (2,0)--(2,2);
    \draw [thick, orange] (2,2)--(0,2);
    \draw [thick, red] (0,2)--(0,0);
    \draw [thick, orange, ->] (0,0) to (1,0);
    \draw [thick, red, ->>] (2,0) to (2,1);
    \draw [thick, orange, ->] (0,2) to (1,2);
    \draw [thick, red, ->>] (0,0) to (0,1);
    \fill (0,0) circle (1.5pt);
    \fill (0,2) circle (1.5pt);
    \fill (2,2) circle (1.5pt);
    \fill (2,0) circle (1.5pt);
    \node at (2.25,1) {$b_0^-$};
    \node at (-0.25,1) {$b_0^+$};
    \node at (1,-0.25) {$a_0^-$};
    \node at (1,2.25) {$a_0^+$};
 
    \foreach \x [evaluate=\x as \coord using 3 + 3.5*\x] in {0, 1} 
    {
    \pattern [pattern=north west lines, pattern color=pallido]
    (\coord, -0.75)--(\coord+2.5, -0.75)--(\coord+2.5, 2.75)--(\coord, 2.75)--(\coord,-0.75);
    \draw [thick, red] (\coord+1.25, 0) to (\coord+1.25, 2);
    \draw [thick, red, ->>] (\coord+1.25,0) to (\coord+1.25,1);
    \fill (\coord+1.25,0) circle (1.5pt);
    \fill (\coord+1.25,2) circle (1.5pt);
    \node at (\coord+0.25, -0.5) {$\mathbb{C}$};  
    \draw [thin, black, dashed] (\coord+1.25, 2)--(\coord+1, 2.75);
    \draw [thin, black, dashed] (\coord+1.25, 2)--(\coord+1.25, 2.75);
    \draw [thin, black, dashed] (\coord+1.25, 2)--(\coord+1.5, 2.75);
    }
    
    \foreach \x [evaluate=\x as \coord using 3 + 3.5*\x] in {0} 
    {
    \draw [thin, black, dashed] (\coord+1.25, 0)--(\coord+1, -0.75);
    \draw [thin, black, dashed] (\coord+1.25, 0)--(\coord+1.25, -0.75);
    \draw [thin, black, dashed] (\coord+1.25, 0)--(\coord+1.5, -0.75);
    \draw [thin, black, dashed] (\coord+1.25, 2)--(\coord+0.75, 2.75);
    \draw [thin, black, dashed] (\coord+1.25, 2)--(\coord+1.75, 2.75);
    \draw[thick, violet, ->] (\coord+1.25, 2) to (\coord+0.625, 2.25);
    \draw[thick, violet]     (\coord+0.625, 2.25) to (\coord, 2.5);
    \draw [thin, black, ->] (\coord+1.25, 1.5) arc [start angle = 270, end angle = 165,radius = 0.5];
    \node at (\coord+0.5, 1.5) {$\theta$};
    \fill (\coord+1.25,0) circle (1.5pt);
    \fill (\coord+1.25,2) circle (1.5pt);
    }
    
    \foreach \x [evaluate=\x as \coord using 3 + 3.5*\x] in {1} 
    {
    \draw [thin, black, dashed] (\coord+1.25, 2)--(\coord+0.75, 2.75);
    \draw [thin, black, dashed] (\coord+1.25, 2)--(\coord+1.75, 2.75);
    \draw [thin, black, dashed] (\coord+1.25, 2)--(\coord+0.5, 2.75);
    \draw [thin, black, dashed] (\coord+1.25, 2)--(\coord+2, 2.75);
    \draw[thick, violet] (\coord+1.25, 2) to (\coord+1.875, 1.75);
    \draw[thick, violet, <-] (\coord+1.875, 1.75) to (\coord+2.5, 1.5);
    \draw [thin, black, <-] (\coord+1.25, 1.5) arc [start angle = 270, end angle = 335,radius = 0.5];
    \node at (\coord+1.75, 1.35) {$\pi-\theta$};
    \fill (\coord+1.25,0) circle (1.5pt);
    \fill (\coord+1.25,2) circle (1.5pt);
    }

    \node at (1.5, 1.5) {$\mathcal{P}$};
    \node at (4.25, -1) {$k-2$};
    \node at (7.75,3) {$p-2$};
    \node at (4.25, 3) {$p-k$};
    \node at (4,0.5) {$b_1^-$};
    \node at (4.5,0.5) {$b_1^+$};
    \node at (7.5,0.5) {$b_2^-$};
    \node at (8,0.5) {$b_2^+$};
    \end{tikzpicture}
    }\\
    \subfloat[][Case 4 of Lemma \ref{exray}. \label{exray4}]
    {
    \begin{tikzpicture}[scale=1.0625, every node/.style={scale=0.8}]
    \definecolor{pallido}{RGB}{221,227,227}
    \definecolor{pallido2}{RGB}{221,227,240}
    
    \pattern [pattern=north east lines, pattern color=pallido]
    (6.5,-0.5)--(8,1)--(8,3)--(10,3)--(11.5,4.5)--(6.5,4.5)--(6.5,-0.5);
    \pattern [pattern=north west lines, pattern color=pallido2]
    (6.5,-0.5)--(11.5,-0.5)--(11.5,4.5)--(10,3)--(10,1)--(8,1)--(6.5,-0.5);
    
    \draw[thick, violet] (10,3) to (10.75, 3);
    \draw[thick, violet, <-] (10.75,3) to (11.5, 3);
    \draw[thick, violet, ->] (10,3) to (10, 3.75);
    \draw[thick, violet] (10,3.75) to (10, 4.5);
    \draw [thin, black, ->] (10.5, 3) arc [start angle = 0, end angle = 90,radius = 0.5];
    \draw [thin, black, ->>] (10, 2.5) arc [start angle = 270, end angle = 360,radius = 0.5];
    \node at (10.625, 3.25) {$\pi$};
    \node at (10.625, 2.75) {$\theta$};

    \fill [white] (8,1)--(10,1)--(10,3)--(8,3)--(8,1);
    \draw [thin, dashed, black] (8,3) -- (8.25, 4.5);
    \draw [thin, dashed, black] (8,3) -- (8.75, 4.5);
    \draw [thin, dashed, black] (8,3) -- (7.25, 4.5);
    \draw [thin, dashed, black] (8,3) -- (7.75, 4.5);
    
    \draw [thin, dashed, black] (8,1) -- (6.5, 1);
    \draw [thin, dashed, black] (8,1) -- (6.5, 1.25);
    \draw [thin, dashed, black] (8,1) -- (6.5, 0.75);
    
    \draw [thin, dashed, black] (10,1) -- (11.5, 1.25);
    \draw [thin, dashed, black] (10,1) -- (11.5, 1);
    \draw [thin, dashed, black] (10,1) -- (11.5, 0.75);
    
    \draw [thick, orange] (8,1)--(10,1);
    \draw [thick, red] (8,1)--(8,3);
    \draw [thick, orange] (8,3)--(10,3);
    \draw [thick, red] (10,1)--(10,3);
    \draw [thick, orange, ->] (8,1) to (9,1);
    \draw [thick, red, ->>] (8,3) to (8,2);
    \draw [thick, orange, ->] (8,3) to (9,3);
    \draw [thick, red, ->>] (10,3) to (10,2);
    \fill (8,1) circle (1.5pt);
    \fill (8,3) circle (1.5pt);
    \fill (10,1) circle (1.5pt);
    \fill (10,3) circle (1.5pt);
    \draw [thin, dashed] (8,1)--(6.5,-0.55);
    \draw [thin, dashed] (10,3)--(11.5, 4.5);

    \node at (8, 4.75) {$k-2$};
    \node at (12, 1) {$p-2$};
    \node at (6, 1) {$p-k$};
    \node at (11,0) {$\mathcal{T}_1$};
    \node at (7, 3.75) {$\mathcal{T}_2$};
    \node at (7.375, 3) {$P-\chi(\beta)$};
    \node at (8, 0.75) {$P$};
    \end{tikzpicture}
    }\\
    \caption{This picture shows how to find a bi-infinite ray as claimed in Lemma \ref{exray}. Once the pieces are all glued together the purple rays determine a bi-infinite geodesic ray joining the poles of the resulting translation surface. Notice the cases $p=2$ and $k=2$ are subsumed in \ref{exray3} and \ref{exray4} according to the sign of the volume.}
    \label{fig:rays}
\end{figure}

\begin{proof}[Proof of Lemma \ref{exray}]
Let $\chi\colon\text{H}_1(S_{1,2},\,\Z)\longrightarrow \C$ be a representation, let $\{\alpha,\,\beta\}$ be a pair of handle generators, and let $a=\chi(\alpha)$ and $b=\chi(\beta)$ be their images. We provide a direct proof case by case.  For the following cases we rely on the constructions developed in Section \S\ref{rot1}.

\SetLabelAlign{center}{\null\hfill\text{#1}\hfill\null}
\begin{itemize}[leftmargin=2em, labelwidth=1em, align=center, itemsep=\parskip]

    \item \textit{Case 1: positive volume with $p=2$ and $k=1$.} In this case a genus one differential with period $\chi$ is obtained by gluing together the quadrilateral $\mathcal{P}$ with edges $a_0^\pm$ and $b_0^\pm$ (according to our convention) and two genus zero differentials, say $(X_i,\omega_i)$ for $i=1,2$. Notice that $\mathcal{P}$ cannot be degenerate because the volume of $\chi$ is positive. From \S\ref{pk1}, recall that $(X_1,\omega_1)$ is slit along a geodesic segment joining $P_1$ with $Q_1=P_1+a$ and recall that $(X_2,\omega_2)$ is slit along a geodesic segment joining $P_2$ with $Q_2=P_2+b$. Let $\delta$ be the oriented angle between $b_0^-$ and $a_0^+$. Let $r_1$ be a geodesic ray on $(X_1,\omega_1)$ leaving from $Q_1$ with angle $0<\theta<\pi-\delta$ with respect to $a^-$. Finally, let $r_2$ be the geodesic ray leaving from $Q_2$ with angle $\pi-(\delta+\theta)$ with respect to $b^+$. Once $\mathcal{P}$ is glued with $(X_1,\omega_1)$ and $(X_2,\omega_2)$ as described in \S\ref{pk1}, the rays $r_1$ and $r_2$ form a bi-infinite geodesic ray say $r$, on the final surface passing through the branch point, obtained after identification, and such that its developed image is a geodesic straight line in $\mathbb{E}^2$. The construction has been done in such a way that $r$ leaves the branch point with angle $\pi$ on the left. Notice that a similar construction can be done so that $r$ leaves the branch point with angle $\pi$ on the right. See Figure \ref{exray1}.\\
    
    \item \textit{Case 2: non-positive volume with $p=2$ and $k=1$.} In this case a genus one differential with period $\chi$ is obtained by removing the interior of the quadrilateral $\mathcal{P}$ with edges $a_0^\pm$ and $b_0^\pm$ from a copy $(\C,\, dz)$ and then glue another copy of $(\C,\, dz)$ slit along a geodesic segment joining $P$ with $Q=P+b$, see \S\ref{nk1}. Here $\mathcal{P}$ could be degenerate. Let $r_1$ be a geodesic ray on  $\C\setminus \text{int}\big(\mathcal{P}\big)$ leaving from $P+a+b$ with angle $0<\theta<\pi$ with respect to $b_0^-$. Then define $r_2$ as the geodesic ray leaving from $Q$ with angle $\pi-\theta$ with respect to $b_1^+$. After gluing as described in \ref{nk1}, the rays $r_1$ and $r_2$ form a bi-infinite geodesic ray on the final surface, passing through the branch point obtained after identification and such that its developed image is a geodesic line in $\mathbb{E}^2$. The construction has been done in such a way that $r$ leaves the branch point with angle $\pi$ on the left. A similar construction can be done so that $r$ leaves the branch point with angle $\pi$ on the right. See Figure \ref{exray2}.
\end{itemize}

\noindent For the remaining cases we rely on the constructions made in \S\ref{pv2p} and \S\ref{npv2p}. Recall that both of these constructions extend those made in  \S\ref{pvep} and \S\ref{nvep} for the case $p=2$ and rotation number two. Therefore we can handle all these cases together.

\SetLabelAlign{center}{\null\hfill\text{#1}\hfill\null}
\begin{itemize}[leftmargin=2em, labelwidth=1em, align=center, itemsep=\parskip]
    \item \textit{Case 3: positive volume and $p>2$ or $p=2$ with $k=2$.} In the case $p=2$ and $k=2$ or $p>2$, a genus one differential with period $\chi$ is obtained by gluing together the quadrilateral $\mathcal{P}$ with edges $a_0^\pm$ and $b_0^\pm$ (according to our convention) and two genus zero differentials, say $(X_i,\omega_i)$ for $i=1,2$, both slit along a geodesic segment joining $P_i$ with $Q_i=P_i+b$. Recall that if $p=2$, then $(X_i,\omega_i)$ is a copy of $(\C,\,dz)$ for $i=1,2$, see \S\ref{pvep}. In the case $p>2$, then $(X_1,\omega_1)$ is a genus zero differential with two zeros of orders $k-2$ and $p-k$ at $P_1$ and $Q_1$ respectively and one pole of order $p$; whereas $(X_2,\omega_2)$ is a genus zero differential with one zero of order $p-2$ at $Q_2$ and one pole of order $p$. In both cases, we shall define $r_1$ as the geodesic ray on $(X_1,\omega_1)$ leaving from $Q_1$ with angle $0<\theta<\pi$ with respect to $b_1^-$. Then we define $r_2$ as the geodesic ray leaving from $Q_2$ with angle $\pi-\theta$ with respect to $b_2^+$. Once $\mathcal{P}$ is glued with $(X_1,\omega_1)$ and $(X_2,\omega_2)$ as described in \S\ref{pvep}, the rays $r_1$ and $r_2$ form a bi-infinite geodesic ray on the final surface, passing through the branch point obtained after identification and such that its developed image is a geodesic straight line in $\mathbb{E}^2$. The construction has been done in such a way that $r$ leaves the branch point with angle $\pi$ on the left. A similar construction can be done so that $r$ leaves the branch point with angle $\pi$ on the right. See Figure \ref{exray3}.\\
    
    \item \textit{Case 4: non-positive volume and $p>2$ or $p=2$ with $k=2$.} In this last case we realize a genus one differential exactly as we have done in \S\ref{nvep} if $p=2$ or as in \S\ref{npv2p} if $p>2$. By adopting the notation used there, we can observe that the angle at $Q$ is $\frac{3\pi}{2}$ in both constructions. There is enough room for finding a bi-infinite ray passing through the branch point such that one of the two angles is $\pi$. Moreover we can choose this bi-infinite ray such that the angle with respect to $b^-$ is $0<\theta<\pi$. The construction has been done in such a way that $r$ leaves the branch point with angle $\pi$ on the right. A similar construction can be done so that $r$ leaves the branch point with angle $\pi$ on the left. See Figure \ref{exray4}.
\end{itemize}
We finally observe that in all constructions the direction of the bi-infinite ray $r$ depends on the angle $\theta$ which can be chosen with certain flexibility. In particular, $\theta$ can be chosen in such a way that $r$ has direction $v$ after developing with $v\neq w$ as desired. This concludes the proof of Lemma \ref{exray}.
\end{proof}

\subsubsection{All poles with non-zero residue}\label{apntr} Assume $n\ge3$ and now we deal with a representation of non-trivial-ends type  $\chi\colon\shomolzon\longrightarrow \C$. Let $\gamma_i\subseteq S_{1,n}$ denote a simple closed curve enclosing the $i$-th puncture and define $w_i=\chi(\gamma_i)$. Since no puncture has zero residue, it follows that $w_i\in\C^*$ for any $i=1,\dots,n$. Recall that $w_1+\cdots+w_n=0$ as a consequence of the Residue Theorem. Before proceeding with this case, we introduce some notations and generalities.
\smallskip

\noindent Let us split $S_{1,n}$ as in Figure \ref{subsurface} and focus on the surface $S_{0,n} \cong \mathbb{S}^2\setminus \{P_1,\dots,P_{n} \}$ of such a splitting. Let $\gamma_i$ be the simple closed curve  enclosing the puncture $P_{i}$ on $S_{0,n}$ for $i=1,\dots,n$. For the reader's convenience we recall that the $\chi_n$-part of $\chi$ is a representation defined as follows:
\begin{equation}
\chi_n\colon\shomolznps\longrightarrow\C,\,\,\,\,\,\,\chi_{n}(\gamma_i)=w_i \text{ for } i=1,\dots,n.
\end{equation}

\noindent Notice that, since $w_1+\cdots+w_n=0$, the representation $\chi_{n}$ is well-defined. We now need to extend our earlier Definition \ref{def:realcoll} as follows: 

\begin{defn}\label{ratchar}
A real-collinear representation $\chi\colon\shomolzn\longrightarrow\C$ is called \textit{rational} if the image $\text{Im}(\chi)$ is contained in the $\mathbb Q$-span of some $c\in\C^*$. A real-collinear representation $\chi$ is not rational otherwise.
\end{defn}

\begin{defn}[Reordering property]\label{sortprop}
Let $w_1,\dots,w_n\in\mathbb R^*$ be non-zero real numbers with zero sum, namely they satisfy $w_1+\cdots+w_n=0$. We shall say that $\{w_1,\dots,w_n\}$ satisfy the \textit{reordering property} if there is a permutation $\sigma\in\mathfrak{S}_n$ such that the reals $w_{\sigma(i)}>0$ for any $1\le \sigma(i)\le h$, where $h$ is an integer smaller than $n$, the reals $w_{\sigma(i)}<0$ for any $h+1\le \sigma(i) \le n$, and the following condition
\begin{equation}\label{orderingcond}
        \sum_{\sigma(i)=1}^s w_{\sigma(i)} =-\sum_{\sigma(j)=t+1}^n w_{\sigma(j)}
    \end{equation} holds only for $s=t=h$. Notice that this definition can naturally extend to sets of real-collinear complex numbers with zero sum.
\end{defn}

\noindent This definition is slightly different from those used in \cite{CFG}. As we shall see, the reordering property plays an important role in this section. The following Lemma holds and the proof can be found in \cite[Appendix B]{CFG}.

\begin{lem} \label{lem:irrmult}
Let $W=\{w_1, \ldots, w_{n}\}\subset\mathbb R^*$ be a set of non-zero real numbers with zero sum. Suppose there exists a pair of numbers in $W$ with irrational ratio. Then $W$ satisfies the reordering property.
\end{lem}

\noindent From \cite{CFG}, we also need to invoke the following proposition about  holonomy representations of genus zero differentials with prescribed singularities.

\begin{prop}\label{hgc:mainprop}
Let $\chi_n\colon \shomolznps \to \mathbb{C}$ be a non-trivial representation. Let $p_1, p_2, \ldots, p_n\in\Z^+$ be positive integers satisfying the following properties:
\begin{itemize}
    \item Either
    \begin{itemize}
        \item $\chi_n$ is not real-collinear, \textit{i.e.} $\textnormal{Im}(\chi_{n})$ is not contained in the $\mathbb R$-span of some $c\in \C^*$,
        \item $\chi_n$ is real-collinear but not rational, \textit{i.e.} $\textnormal{Im}(\chi_{n})$ is contained in the $\mathbb R$-span but not in the $\mathbb Q$-span of some $c\in \C^*$,
        \item at least one of $p_1, p_2, \ldots ,p_n$ is different from 1, and
    \end{itemize}
    \item $p_i \geq 2$ whenever $\chi(\gamma_i) =0$. 
\end{itemize}
Then $\chi$ appears as the holonomy of a translation structure on $S_{0,n}$ determined by a meromorphic differential on $\cp$ with a single zero of order $m=p_1+\cdots+p_n-2$ and a pole of order $p_i$ at the puncture enclosed by the curve $\gamma_i$, for each $1\leq i\leq n$. 
\end{prop}

\noindent We shall need to consider two sub-cases depending on whether $\chi_n$ is rational. We begin with the following case:
\medskip

\paragraph{\textit{The representation $\chi_{n}$ is not rational}}\label{chinonrat} Let $\alpha,\beta$ be a pair of handle generators for $\text{H}_1(S_{1,2},\,\mathbb{Z})$. By means of an auxiliary representation $\rho$ defined as in \eqref{eq:auxreponetwo}, see also Definition \ref{defsupprep}, we can realize a genus one differential $(X_1,\omega_1)\in\mathcal{H}_1(2p;-p,-p)$ with rotation number $k$ and period character defined as
\begin{equation}
    \rho(\alpha)=\chi(\alpha),\quad \rho(\beta)=\chi(\beta), \quad \rho(\delta_1)=\rho(\delta_2^{-1})=0. 
\end{equation}

\noindent We apply Proposition \ref{hgc:mainprop} to our representation $\chi_n$ and realize this latter as the period character of a genus zero differential $(X_2,\omega_2)\in\mathcal{H}_0\big((n-2)p;-1,-1,-p,\dots,-p\big)$. 

\smallskip

\noindent In order to realize the desired surface, we shall glue them along a bi-infinite geodesic ray. Recall that, according to Lemma \ref{exray}, we can always find a bi-infinite geodesic ray on $(X_1,\omega_1)$ with prescribed direction (hence slope) once developed such that:
\begin{itemize}
    \item[\textit{(i)}] it passes through the zero of $\omega_1$ only once with angle $\pi$ on its \textit{right}, and
    \item[\textit{(ii)}] it joins the poles of $\omega_1$.
\end{itemize}

\noindent Let $r\subset (X_1,\omega_1)$ be a ray as above with direction $v$ once developed. What remains to do is to find a proper bi-infinite ray on $(X_2,\omega_2)$ with the same direction and passing through the unique zero of $\omega_2$ with angle $\pi$ on its \textit{left}. Here $\text{Im}(\chi_n)$ is assumed to be not contained in the $\mathbb{Q}$-span of any $c\in\C^*$ and hence, according to \cite[Proof of Proposition 10.1]{CFG} there are two possible situations:

\SetLabelAlign{center}{\null\hfill\text{#1}\hfill\null}
\begin{itemize}[leftmargin=2.5em, labelwidth=2em, align=center, itemsep=\parskip]
    \item [1.] The first one arises if $\text{Im}(\chi_n)$ is not contained in the $\mathbb{R}$-span of any $c\in\C^*$. Whenever this is the case, we can reorder the punctures in such a way that the points $\{\arg(w_i)\}\subset \mathbb{S}^1$, with $w_i=\chi(\gamma_i)$, are cyclically ordered. With respect to this ordering, the vectors $w_i$ yield a convex polygon $\mathcal{P}$ with $n$ sides. Denote the edges of $\mathcal{P}$ as $e_i^-$, for $i=1,\dots,n$. To any side we can glue a half-infinite strip $\mathcal{S}_i$ bounded by the edge $e_i^+$ and two infinite parallel rays $r_i^{\pm}$ oriented so that $r_i^+=r_i^-+w_i$. The quotient space $\mathcal{P}\cup \bigcup \mathcal{S}_i$ obtained by identifying the vertices of $\mathcal{P}$ and the rays $r_i^{\pm}$, for $i=1,\dots,n$, turns out to be an $n$-punctured sphere with a translation structure with all simple poles and a single zero of maximal order. In this case it is an easy matter to see that for almost any slope $\theta\in\mathbb{S}^1$ (hence almost any directions), there is a bi-infinite geodesic ray with slope $\theta$. Therefore, we can find a bi-infinite geodesic ray, say $\overline{r}$, with direction $v$ and such that it passes through the zero of $\omega_2$ by leaving an angle of magnitude $\pi$ on its left. Such a ray joins two adjacent poles say $P_j$ and $P_{j+1}$. We define $(X_2,\omega_2)$ as the translation structure obtained by bubbling $p$ copies of the standard differential $(\C,\,dz)$ to all poles with the only exceptions being $P_j$ and $P_{j+1}$. Finally, glue $(X_1,\omega_1)$ and $(X_2,\omega_2)$ as in Definition \ref{gluingsurfaces} by slitting them along the rays $r$ and $\overline r$ respectively. The resulting surface is homeomorphic to $S_{g,n}$ and carries a translation surface with poles $(Y,\xi)\in\mathcal{H}_1\big(np;-p,\dots,-p \big)$ with rotation number $k$ by construction.
    \smallskip
    
    \item [2.] The second possibility arises in the case $\chi_n$ is real-collinear but not rational. In order to recall the construction, we assume for a moment that $\arg(w_i)\in\{0,\pi\}$, \textit{i.e.} $w_i\in\mathbb R$, because the  general case of $\arg(w_i)\in\{\delta,\delta+\pi\}$ follows by rotating the following construction by $\delta$. Up to relabelling the punctures, there is a positive integer, say $h$ less than $n$, such that $\{w_1,\dots,w_h\}$ are all positive real numbers and $\{w_{h+1},\dots,w_n\}$ are all negative real numbers. Since $\text{Im}(\chi_n)$ is not contained in the $\mathbb{Q}$-span of any $c\in\C^*$ it is possible to find a pair of reals in $\{w_1,\dots,w_n\}$ with irrational ratio; hence Lemma \ref{lem:irrmult} applies and the reals $w_1,\dots,w_n$ satisfy the reordering property. Let $\zeta_1=0$ and let us define $\zeta_{l+1}=\zeta_{l}+w_l$. Consider the infinite strip $\{ z \in \mathbb{C} \,\, \vert \,\, 0 < \Re(z) < \zeta_{h+1} \}$. In this infinite strip, we make half-infinite vertical slits pointing upwards at the points $\zeta_2, \ldots, \zeta_{h}$ and half-infinite vertical slits pointing downwards at the points $\zeta_{h+2}, \ldots, \zeta_{n}$, see Figure \ref{fig:qspan}. By gluing the rays $r_i^+$ and $r_i^-$ for $i=1,\dots,n$ we obtain a translation surface $(Z, \eta)$ with all simple poles and a single zero of maximal order.
    
    \begin{figure}[!ht]
     \centering
     \begin{tikzpicture}[scale=1, every node/.style={scale=0.75}]
        \definecolor{pallido}{RGB}{221,227,227}
         \fill[pattern=north west lines, pattern color=pallido] (0,3) -- (0,-3) -- (3,-3) -- (3, 3);
         \fill[pattern=north west lines, pattern color=pallido] (5,3) -- (5,-3) -- (8.5,-3) -- (8.5, 3);
         
         \draw[thin, blue] (8.5,0) to (7,-1.5);
         \draw[thin, blue, ->] (7,-1.5) to (7.75,-0.75);
         \draw[thin, blue] (8.5,-1.5) to (7,-3);
         \draw[thin, blue, ->] (7,-3) to (7.75,-2.25);
         \draw[thin, blue] (6,0) to (8.5,2.5);
         \draw[thin, blue, ->] (6,0) to (7.25, 1.25);
         \draw[thin, blue] (6,2.5) to (6.5,3);
         \draw [thin, black, <-] (8.5,0.375) arc [start angle = 90, end angle = 225,radius = 0.375];
         \draw [thin, black, ->] (6,0.375) arc [start angle = 90, end angle = 45,radius = 0.375];

         \fill[white] ($(1,0) + (92:3)$)--(1,0)-- ++(88:3);
         \fill[white] ($(2,0) + (-92:3)$)--(2,0)-- ++(-88:3);
         \fill[white] ($(6,0) + (92:3)$)--(6,0)-- ++(88:3);
         \fill[white] ($(7,0) + (-92:3)$)--(7,0)-- ++(-88:3);
         \draw[decoration={markings, mark=at position 0.6 with {\arrow{latex}}}, postaction={decorate}] (0,0) -- (0,3);
         \draw[decoration={markings, mark=at position 0.6 with {\arrow{latex}}}, postaction={decorate}] (0,0) -- (0,-3);
         \draw[decoration={markings, mark=at position 0.6 with {\arrow{latex}}}, postaction={decorate}] (1,0) -- ++(92:3);
         \draw[decoration={markings, mark=at position 0.6 with {\arrow{latex}}}, postaction={decorate}] (1,0) -- ++(88:3);
         \draw[decoration={markings, mark=at position 0.6 with {\arrow{latex}}}, postaction={decorate}] (2,0) -- ++(-88:3);
         \draw[decoration={markings, mark=at position 0.6 with {\arrow{latex}}}, postaction={decorate}] (2,0) -- ++(-92:3);
         \node at (4,0) {$\ldots$};
         \draw[decoration={markings, mark=at position 0.6 with {\arrow{latex}}}, postaction={decorate}] (6,0) -- ++(92:3);
         \draw[decoration={markings, mark=at position 0.6 with {\arrow{latex}}}, postaction={decorate}] (6,0) -- ++(88:3);
         \draw[decoration={markings, mark=at position 0.6 with {\arrow{latex}}}, postaction={decorate}] (7,0) -- ++(-88:3);
         \draw[decoration={markings, mark=at position 0.6 with {\arrow{latex}}}, postaction={decorate}] (7,0) -- ++(-92:3);
         \draw[decoration={markings, mark=at position 0.6 with {\arrow{latex}}}, postaction={decorate}] (8.5,0) -- (8.5,-3);
        \draw[decoration={markings, mark=at position 0.6 with {\arrow{latex}}}, postaction={decorate}] (8.5,0) -- (8.5,3);
         \fill (0,0) circle (1pt);
         \fill (1,0) circle (1pt);
         \fill (2,0) circle (1pt);
         \fill (6,0) circle (1pt);
         \fill (7,0) circle (1pt);
         \fill (8.5,0) circle (1pt);
         \node[right] at (0, 1.5) {$r_1^+$};
         \node[left] at ($(1,0) + (92:1.5)$) {$r_1^-$};
         \node[right] at ($(1,0) + (88:1.5)$) {$r_2^+$};
         \node[left] at ($(6,0) + (92:1.5)$) {$r_{h-1}^-$};
         \node[right] at ($(6,0) + (88:1.5)$) {$r_{h}^+$};
        \node[left] at (8.5, 1.5) {$r_{h}^-$};
        \node[right] at (0, -1.5) {$r_n^-$};
         \node[left] at ($(2,0) + (-92:1.5)$) {$r_n^+$};
         \node[right] at ($(2,0) + (-88:1.5)$) {$r_{n-1}^-$};
         \node[left] at ($(7,0) + (-92:1.5)$) {$r_{h+2}^+$};
         \node[right] at ($(7,0) + (-88:1.5)$) {$r_{h+1}^-$};
        \node[left] at (8.5, -1.5) {$r_{h+1}^+$};
        \node[left] at (0,0) {$\zeta_1$};
        \node[below] at (1,0) {$\zeta_2$};
        \node[below] at (6,0) {$\zeta_{h}$};
        \node[right] at (8.5,0) {$\zeta_{h+1}$};
        \node[above] at (7,0) {$\zeta_{h+2}$};
        \node[above] at (2,0) {$\zeta_n$};
        \node[left] at (8.25,0.375) {$\pi-\theta$};
        \node[above] at (6.25,0.5) {$\theta$};
     \end{tikzpicture}
     \caption{The picture shows the construction in the case $\text{Im}(\chi_n)$ is contained in the $\mathbb{R}$-span of some $c \in \mathbb{C}^*$ but not in the $\mathbb Q$-span. The blue line represents a ray with slope $\frac{\pi}{2}-\theta$ along which we can glue a genus one differential.} \label{fig:qspan}
    \end{figure}
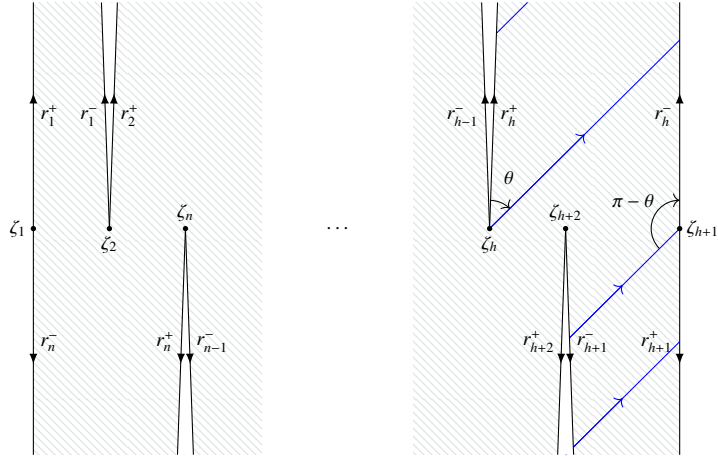

    \noindent In this case, for any $\theta\neq0$ we can find a bi-infinite geodesic ray, say $\overline r$ joining two punctures and passing through the branch point by leaving an angle $\pi$ on its left. In fact, we can consider a ray leaving from $\zeta_h$ with angle $\theta$ with respect to $r_h^+$. Such a ray points towards the puncture $P_h$. Then consider a ray leaving from $\zeta_{h+1}$ with angle $\pi-\theta$ with respect to $r_h^-$ pointing towards the puncture $P_{h+1}$. These rays determine the desired bi-infinite geodesic ray, say $\overline r$, on $(Z,\eta)$. In fact, by construction it joins two punctures and leaves an angle of magnitude $\pi$ on its left at the unique branch point of $(Z,\eta)$. We define $(X_2,\omega_2)$ as the structure obtained by bubbling $p$ copies of the differential $(\C,\,dz)$ to all poles with the only exceptions being $P_h$ and $P_{h+1}$. As above, we glue $(X_1,\omega_1)$ and $(X_2,\omega_2)$ as in Definition \ref{gluingsurfaces} by slitting them along $r$ and $\overline r$ respectively. The final surface is homeomorphic to $S_{g,n}$ and carries a translation structure with poles in the stratum $\mathcal{H}_1\big(np;-p,\dots,-p \big)$ with rotation number $k$ by construction.
\end{itemize}

\noindent For more details about these constructions the reader can consult the proofs of \cite[Propositions 6.1 and 10.1]{CFG}.

\smallskip

\paragraph{\textit{The representation $\chi_{n}$ is rational}}\label{chirat} It remains to deal with the case of $\chi_n$ being rational. The strategy developed in paragraph \S\ref{chinonrat} does not always apply because for a rational representation (Definition \ref{ratchar}) the reordering property (Definition \ref{sortprop}) fails in general. Whenever a rational representation $\chi_n$ satisfies the reordering property, then the construction developed for a non-rational real-collinear representation as in paragraph \S\ref{chinonrat} applies \textit{mutatis mutandis}. Therefore, in what follows, we shall restrict ourselves to rational representations for which the reordering property fails. In this very special situation, we shall adopt a strategy which is a blend of our construction so far with \cite[Proof of Proposition 10.1 - Case 3]{CFG}. In particular, we do not rely on the splitting introduced at the beginning of  \S\ref{resnotzero}.

\smallskip

\noindent Let $\chi\colon\shomolzon\longrightarrow \C$ be a non-trivial representation with rational $\chi_n$-part. Let $\Sigma\subset S_{1,n}$ be any handle, let $\alpha,\beta$ be a pair of handle generators for $\text{H}_1(\Sigma,\,\mathbb{Z})\subset \shomolzon$,  and define $a,b$ as their respective images via $\chi$. The restriction of $\chi$ to $\Sigma$ is a representation of trivial-end type, say $\chi_1$, and hence the notion of volume for $\chi_1$ is well-defined. According to the sign of $\text{vol}(\chi_1)$, we shall distinguish two constructions. We shall treat the positive volume case in detail; the non-positive volume one will follow after a simple modification of the first case. Let $\mathcal{P}$ be the parallelogram defined by the chain 
\begin{equation}\label{parchirat}
    P_0\mapsto P_0+\chi(\alpha)\mapsto P_0+\chi(\alpha)+\chi(\beta)\mapsto P_0+\chi(\beta)\mapsto P_0,
\end{equation}
where $P_0\in\C$ is any point. As usual, we label the sides of $\mathcal{P}$ with $a^{\pm}_0$ and $b_0^{\pm}$ according to our convention.

\smallskip 

\noindent We now consider the $\chi_n$-part of $\chi$. The representation $\chi_n\colon\shomolznps\longrightarrow \C$ is a non-trivial representation and we denote $\chi_n(\gamma_i)=w_i\in\C^*$ for any $i$, where $\gamma_i$ is a simple closed curve around the $i$-th puncture. 
We assume $W=\{w_1,\dots,w_n\}\subset \C^*$ is a set of real-collinear complex numbers. Up to permutation of the labels of the punctures, we can assume that all numbers in $\{w_1,\dots,w_h\}$ have the same argument $\delta\in ]-\pi,\,\pi]$ and all numbers in $\{w_{h+1},\dots,w_n\}$ have the same argument $\pi+\delta$ for some $1\le h\le n-1$. Assume that $W$ does \textit{not} satisfy the reordering property. According to Definition \ref{sortprop} there are two positive integers $1\le s < h$ and $h+1\le t < n$ such that 
\begin{equation}\label{orderingcond2}
        \sum_{i=1}^s w_i =-\sum_{j=1}^{n-t} w_{t+j}=-\sum_{j=t+1}^n w_{j}.
\end{equation}
\noindent The indices $s,t$ yield the following two partitions: $W_1^\delta\cup W_2^\delta=\{w_1,\dots, w_s\}\cup\{w_{s+1},\dots,w_h\}$ of $\{w_1,\dots,w_h\}$ and the partition $W_2^{\pi+\delta}\cup W_1^{\pi+\delta}=\{w_{h+1},\dots,w_{t}\}\cup\{w_{t+1},\dots,w_n\}$ of $\{w_{h+1},\dots,w_n\}$. We can observe that the collections $W_1=W_1^\delta\cup W_1^{\pi+\delta}$ and $W_2=W_2^\delta\cup W_2^{\pi+\delta}$ both have zero sum by construction. Generally, the indices $s,t$ are not uniquely determined and hence there can be different partitions. On the other hand, the following construction does not depend on the choice of the partition wherever there are more than one. Finally, we can assume without loss of generality that $\delta$ has magnitude such that $-\arg(b)<\delta\le \arg(b)$. Let us proceed with our construction. 

\smallskip 

\noindent Let $(\C,\,dz)$ be a copy of the standard genus zero meromorphic differential and slit it along a geodesic segment $b_1$ with extremal points $P_1$ and $Q_1=P_1+\chi(\beta)$. Consider the collection $W_1$ above. According to our notation, we set $\zeta_1=Q_1$ and thus define the point $\zeta_{l+1}=\zeta_l+w_{l}$, where $w_{l}\in W_1$ ($l=1,\dots, s,t+1,\dots,n$). Define $e=\overline{\zeta_1\,\zeta_{s+1}}$ and notice that $e\cap b_1=\zeta_1$ because $\delta\neq-\arg(b)$. We next slit the structure $(\C,\,dz)$ along the edge $e$ and label the resulting sides as $e^{\pm}$, where the sign is taken according to our convention. We partition the edge $e^+$ as follow: We define $e_l^+=\overline{\zeta_l\,\zeta_{l+1}}$ for $l=1,\dots, s$. Notice that these segments are pairwise adjacent or disjoint. Moreover, by construction, $e^+=e_1^+\cup\cdots\cup e_s^+$. In the same fashion, we partition the edge $e^-$ as follows: We define $e_{l}^-=\overline{\zeta_{l}\,\zeta_{l+1}}$ for $l=t+1,\dots, n$. By construction, $e^-=e_{t+1}^-\cup\cdots\cup e_{n}^-$. We eventually slit $(\C,\,dz)$ along the ray $\overline r_1$ starting from $\zeta_1$, orthogonal to $e$ with oriented angle $\frac{\pi}{2}$ with respect to $e$. Finally, we introduce $n+s-t$ half-strips as follows: For any $l=1,\dots,s$, we define $\mathcal{S}_l$ as an infinite half-strip bounded by the geodesic segment $e_l^-$ and by two infinite rays $r_l^{\pm}$ pointing in the direction $\delta+\frac{\pi}{2}$; and for any $l=t+1,\dots,n$, the strip $\mathcal{S}_l$ is bounded by the geodesic segment $e_l^+$ and by two infinite rays $r_l^{\pm}$ pointing in the direction $\delta-\frac{\pi}{2}$. 

\smallskip

\noindent We next consider another copy of $(\C,\,dz)$; we slit it along a geodesic segment $b_2$ with extremal points $P_2$ and $Q_2=P_2+\chi(\beta)$. Here we consider the collection $W_2=\{w_{s+1},\dots, w_t\}$ above and then we similarly proceed as above: We set $\zeta_{s+1}=Q_2$ and thus define the point $\zeta_{l+1}=\zeta_l+w_{l}$, where $w_{l}\in W_2$. Define $\overline e=\overline{\zeta_{s+1}\,\zeta_{h+1}}$ and notice that $\overline e\cap b_2=\zeta_{s+1}$ because $\delta\neq-\arg(b)$. We then slit the structure $(\C,\,dz)$ along the edge $\overline e$ and label the resulting sides as $\overline e^{\pm}$. As above the edge $\overline e^+$ is partitioned by segments $e_l^+=\overline{\zeta_l\,\zeta_{l+1}}$ for $l=s+1,\dots, h$. By construction $\overline e^+=e_{s+1}^+\cup\cdots\cup e_h^+$. In the same fashion, the edge $e^-$ is partitioned by segments $e_{l}^-=\overline{\zeta_{l}\,\zeta_{l+1}}$ for $l=h+1,\dots, t$. It follows by construction that $\overline e^-=e_{h+1}^-\cup\cdots\cup e_{t}^-$. We eventually slit $(\C,\,dz)$ along $\overline r_{s+1}$, the ray starting from $\zeta_{s+1}$, orthogonal to $\overline e$ with oriented angle $\frac{\pi}{2}$ with respect to $\overline e$. We introduce $t-s$ half-strips as follows: For any $l=s+1,\dots,h$, the infinite half-strip $\mathcal{S}_l$ is bounded by the geodesic segment $e_l^-$ and by two infinite rays $r_l^{\pm}$ pointing toward the direction $\delta+\frac{\pi}{2}$; and for any $l=h+1,\dots,t$, let $\mathcal{S}_l$ be an infinite half-strip bounded by the geodesic segment $e_l^+$ and by two infinite rays $r_l^{\pm}$ pointing in the direction $\delta-\frac{\pi}{2}$. 

\smallskip 

\noindent We can finally glue all the pieces together and then obtain the desired structure. We begin by some usual identifications, namely we identify the pair of edges $b_j^-$ and $b_{j+1}^+$, for $j=0,1,2$ and $b_3^+=b_0^+$, and then the pair of edges $a^+$ and $a^-$. We next proceed with gluing the strips $\mathcal{S}_l$, for $l=1,\dots,n$, in the following way. For any $l\notin\{1,s+1$\}, we paste the strip $\mathcal{S}_l$ by identifying $e_l^+$ with $e_l^-$ and then the rays $r_l^+$ and $r_l^-$ together. In the case $l\in\{1,s+1\}$ we paste the strip $\mathcal{S}_l$ in a different way, by first identifying the edges $e_l^-$ with $e_l^-$ and then identifying $r_l^+$ with $\overline r_l^-$ and $r_l^-$ with $\overline r_l^+$, see Figure \ref{fig:chiratposvol}. 

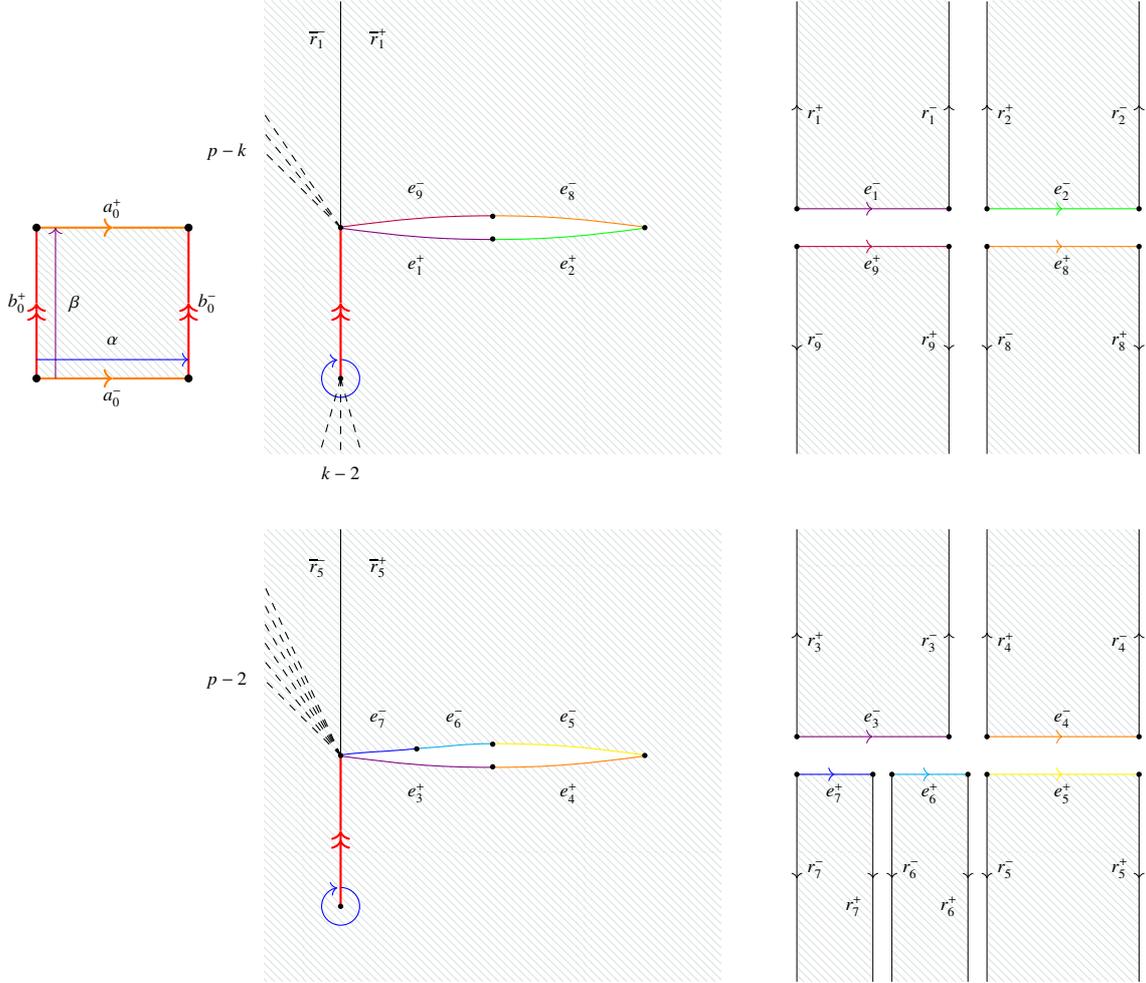
\begin{figure}[!ht] 
\centering
\begin{tikzpicture}[scale=1, every node/.style={scale=0.7}]
\definecolor{pallido}{RGB}{221,227,227}

    \pattern [pattern=north west lines, pattern color=pallido]
    (-2,2) -- (0,2) -- (0,4) -- (-2,4) -- (-2,2);
    \draw [thick, orange] (-1,2)--(0,2);
    \draw [thick, red] (0,3)--(0,4);
    \draw [thick, orange] (-1,4)--(0,4);
    \draw [thick, red] (-2,3)--(-2,4);
    \draw [thick, orange, ->] (-2,2) to (-1,2);
    \draw [thick, red, ->>] (0,2) to (0,3);
    \draw [thick, orange, ->] (-2,4) to (-1,4);
    \draw [thick, red, ->>] (-2,2) to (-2,3);
    \draw [thin, violet, ->] (-1.75,2) to (-1.75,4);
    \draw [thin, blue, ->] (-2,2.25) to (0,2.25);
    \fill (-2,2) circle (1.5pt);
    \fill (0,2) circle (1.5pt);
    \fill (0,4) circle (1.5pt);
    \fill (-2,4) circle (1.5pt);

    \node at (-1, 2.5) {$\alpha$};
    \node at (-1.5, 3) {$\beta$};
    \node at (-1, 1.75) {$a_0^-$};
    \node at (-1, 4.25) {$a_0^+$};
    \node at (-2.25, 3) {$b_0^+$};
    \node at (0.25, 3) {$b_0^-$};

    \foreach \y [evaluate=\y as \coord using  \y] in {-6, 1} 
    {
    \pattern [dashed, pattern=north west lines, pattern color=pallido] (1,\coord) -- (7, \coord) -- (7, \coord+6) -- (1, \coord+6)-- (1,\coord);
    \draw [thin, blue, ->] (2, \coord+1.25) arc [start angle = 90, end angle = -260,radius = 0.25];
    \draw[thin, black] (2, \coord+3) to (2, \coord+6);
    \draw[thick, red] (2, \coord+3) to (2, \coord+2);
    \draw[thick, red, <<-] (2, \coord+2) to (2, \coord+1);
    \draw [thin, black, dashed] (2, \coord+3)--(1, \coord+4.5);
    \draw [thin, black, dashed] (2, \coord+3)--(1, \coord+4.25);
    \draw [thin, black, dashed] (2, \coord+3)--(1, \coord+4);
    }

    \foreach \y [evaluate=\y as \coord using  \y] in {1} 
    {

    \draw[thick, purple] (2, \coord+3) to [out=7.5, in=180] (4, \coord+3.15);
    \draw[thick, orange] (4, \coord+3.15) to [out=0, in=172.5] (6, \coord+3);
    \draw[thick, green] (6, \coord+3) to [out=187.5, in=0] (4, \coord+2.85);
    \draw[thick, violet] (2, \coord+3) to [out=352.5, in=180] (4, \coord+2.85);
    \fill[white] (2, \coord+3) to [out=7.5, in=180] (4, \coord+3.15) to  [out=0, in=172.5] (6, \coord+3) to [out=187.5, in=0] (4, \coord+2.85) to  [out=180, in=352.5] (2, \coord+3);
    
    \node at (3, \coord+3.5) {$e_9^-$};
    \node at (5, \coord+3.5) {$e_8^-$};
    \node at (3, \coord+2.5) {$e_1^+$};
    \node at (5, \coord+2.5) {$e_2^+$};
    
    \draw [thin, black, dashed] (2, \coord+1)--(1.725, \coord);
    \draw [thin, black, dashed] (2, \coord+1)--(2, \coord);
    \draw [thin, black, dashed] (2, \coord+1)--(2.275, \coord);
    
    \fill (4,\coord+3.15) circle (1pt);
    \fill (4,\coord+2.85) circle (1pt);
    \fill (6,\coord+3) circle (1pt);
    \fill (2,\coord+3) circle (1pt);
    \fill (2,\coord+1) circle (1pt);

    \pattern [dashed, pattern=north west lines, pattern color=pallido] (8,\coord+3.25) -- (10, \coord+3.25) -- (10, \coord+6) -- (8, \coord+6)-- (8,\coord+3.25);
    \pattern [dashed, pattern=north west lines, pattern color=pallido] (10.5,\coord+3.25) -- (12.5, \coord+3.25) -- (12.5, \coord+6) -- (10.5, \coord+6)-- (10.5,\coord+3.25);
  
    \draw[thin, violet, ->] (8, \coord+3.25) -- (9, \coord+3.25);
    \draw[thin, violet] (9, \coord+3.25) -- (10, \coord+3.25);
    \draw[thin, green, ->] (10.5, \coord+3.25) -- (11.5, \coord+3.25);
    \draw[thin, green] (11.5, \coord+3.25) -- (12.5, \coord+3.25);
    \draw[black, ultra thin, ->] (8, \coord+3.25) -- (8, \coord+4.625);
    \draw[black, ultra thin] (8, \coord+4.625) -- (8, \coord+6);
    \draw[black, ultra thin, ->] (10, \coord+3.25) -- (10, \coord+4.625);
    \draw[black, ultra thin] (10, \coord+4.625) -- (10, \coord+6);
    \draw[black, ultra thin, ->] (10.5, \coord+3.25) -- (10.5, \coord+4.625);
    \draw[black, ultra thin] (10.5, \coord+4.625) -- (10.5, \coord+6);
    \draw[black, ultra thin, ->] (12.5, \coord+3.25) -- (12.5, \coord+4.625);
    \draw[black, ultra thin] (12.5, \coord+4.625) -- (12.5, \coord+6);

    \fill (8,\coord+3.25) circle (1pt);
    \fill (10,\coord+3.25) circle (1pt);
    \fill (10.5,\coord+3.25) circle (1pt);
    \fill (12.5,\coord+3.25) circle (1pt);
    
    \pattern [dashed, pattern=north west lines, pattern color=pallido] (8,\coord+2.75) -- (10, \coord+2.75) -- (10, \coord) -- (8, \coord)-- (8,\coord+2.75);
    \pattern [dashed, pattern=north west lines, pattern color=pallido] (10.5,\coord+2.75) -- (12.5, \coord+2.75) -- (12.5, \coord) -- (10.5, \coord)-- (10.5,\coord+2.7);
  
    \draw[thin, purple, ->] (8, \coord+2.75) -- (9, \coord+2.75);
    \draw[thin, purple] (9, \coord+2.75) -- (10, \coord+2.75);
    \draw[thin, orange, ->] (10.5, \coord+2.75) -- (11.5, \coord+2.75);
    \draw[thin, orange] (11.5, \coord+2.75) -- (12.5, \coord+2.75);
    \draw[black, ultra thin, ->] (8, \coord+2.75) -- (8, \coord+1.375);
    \draw[black, ultra thin] (8, \coord+1.375) -- (8, \coord);
    \draw[black, ultra thin, ->] (10, \coord+2.75) -- (10, \coord+1.375);
    \draw[black, ultra thin] (10, \coord+1.375) -- (10, \coord);
    \draw[black, ultra thin, ->] (10.5, \coord+2.75) -- (10.5, \coord+1.375);
    \draw[black, ultra thin] (10.5, \coord+1.375) -- (10.5, \coord); 
    \draw[black, ultra thin, ->] (12.5, \coord+2.75) -- (12.5, \coord+1.375);
    \draw[black, ultra thin] (12.5, \coord+1.375) -- (12.5, \coord);
    
    \node at (1.7, \coord+5.5) {$\overline r_1^-$};
    \node at (2.5, \coord+5.5) {$\overline r_1^+$};
    
    \node at (8.25, \coord+4.5) {$r_1^+$};
    \node at (9.75, \coord+4.5) {$r_1^-$};
    \node at (10.75, \coord+4.5) {$r_2^+$};
    \node at (12.25, \coord+4.5) {$r_2^-$};
    
    \node at (9, \coord+3.5) {$e_1^-$};
    \node at (9, \coord+2.5) {$e_9^+$};
    \node at (11.5, \coord+3.5) {$e_2^-$};
    \node at (11.5, \coord+2.5) {$e_8^+$};
    
    \node at (8.25, \coord+1.5) {$r_9^-$};
    \node at (9.75, \coord+1.5) {$r_9^+$};
    \node at (10.75, \coord+1.5) {$r_8^-$};
    \node at (12.25, \coord+1.5) {$r_8^+$};
    
    \node at (2, \coord-0.25) {$k-2$};
    \node at (0.5, \coord+4) {$p-k$};
    
    \fill (8,\coord+2.75) circle (1pt);
    \fill (10,\coord+2.75) circle (1pt);
    \fill (10.5,\coord+2.75) circle (1pt);
    \fill (12.5,\coord+2.75) circle (1pt);
    
    }
    
    \foreach \y [evaluate=\y as \coord using  \y] in {-6} 
    {
    
    \draw [thin, black, dashed] (2, \coord+3)--(1, \coord+4.75);
    \draw [thin, black, dashed] (2, \coord+3)--(1, \coord+5);
    \draw [thin, black, dashed] (2, \coord+3)--(1, \coord+5.25);
    
    \draw[thick, blue] (2, \coord+3) to [out=7.5, in=186.5] (3, \coord+3.085);
    \draw[thick, cyan] (3, \coord+3.085) to [out=3.5 , in=180] (4, \coord+3.15);
    \draw[thick, yellow] (4, \coord+3.15) to [out=0, in=172.5] (6, \coord+3);
    \draw[thick, orange] (6, \coord+3) to [out=187.5, in=0] (4, \coord+2.85);
    \draw[thick, violet] (2, \coord+3) to [out=352.5, in=180] (4, \coord+2.85);
    \fill[white] (2, \coord+3) to [out=7.5, in=186.5] (3, \coord+3.085) to [out=3.5 , in=180] (4, \coord+3.15) to [out=0, in=172.5] (6, \coord+3) (6, \coord+3) to [out=187.5, in=0] (4, \coord+2.85) to  [out=180, in=352.5] (2, \coord+3);
    
    \node at (3, \coord+2.5) {$e_3^+$};
    \node at (5, \coord+2.5) {$e_4^+$};
    \node at (2.5, \coord+3.5) {$e_7^-$};
    \node at (3.5, \coord+3.5) {$e_6^-$};
    \node at (5, \coord+3.5) {$e_5^-$};
    
    \pattern [dashed, pattern=north west lines, pattern color=pallido] (8,\coord+3.25) -- (10, \coord+3.25) -- (10, \coord+6) -- (8, \coord+6)-- (8,\coord+3.25);
    \pattern [dashed, pattern=north west lines, pattern color=pallido] (10.5,\coord+3.25) -- (12.5, \coord+3.25) -- (12.5, \coord+6) -- (10.5, \coord+6)-- (10.5,\coord+3.25);
  
    \draw[thin, violet, ->] (8, \coord+3.25) -- (9, \coord+3.25);
    \draw[thin, violet] (9, \coord+3.25) -- (10, \coord+3.25);
    \draw[thin, orange, ->] (10.5, \coord+3.25) -- (11.5, \coord+3.25);
    \draw[thin, orange] (11.5, \coord+3.25) -- (12.5, \coord+3.25);
    \draw[black, ultra thin, ->] (8, \coord+3.25) -- (8, \coord+4.625);
    \draw[black, ultra thin] (8, \coord+4.625) -- (8, \coord+6);
    \draw[black, ultra thin, ->] (10, \coord+3.25) -- (10, \coord+4.625);
    \draw[black, ultra thin] (10, \coord+4.625) -- (10, \coord+6);
    \draw[black, ultra thin, ->] (10.5, \coord+3.25) -- (10.5, \coord+4.625);
    \draw[black, ultra thin] (10.5, \coord+4.625) -- (10.5, \coord+6);
    \draw[black, ultra thin, ->] (12.5, \coord+3.25) -- (12.5, \coord+4.625);
    \draw[black, ultra thin] (12.5, \coord+4.625) -- (12.5, \coord+6);
    
    \fill (8,\coord+3.25) circle (1pt);
    \fill (10,\coord+3.25) circle (1pt);
    \fill (10.5,\coord+3.25) circle (1pt);
    \fill (12.5,\coord+3.25) circle (1pt);
    
    \pattern [dashed, pattern=north west lines, pattern color=pallido] (8,\coord+2.75) -- (9, \coord+2.75) -- (9, \coord) -- (8, \coord)-- (8,\coord+2.75);
    \pattern [dashed, pattern=north west lines, pattern color=pallido] (9.25,\coord+2.75) -- (10.25, \coord+2.75) -- (10.25, \coord) -- (9.25, \coord)-- (9.25,\coord+2.75);
    \pattern [dashed, pattern=north west lines, pattern color=pallido] (10.5,\coord+2.75) -- (12.5, \coord+2.75) -- (12.5, \coord) -- (10.5, \coord)-- (10.5,\coord+2.7);
  
    \draw[thin, blue, ->] (8, \coord+2.75) -- (8.5, \coord+2.75);
    \draw[thin, blue] (8.5, \coord+2.75) -- (9, \coord+2.75);
    \draw[thin, cyan, ->] (9.25, \coord+2.75) -- (9.75, \coord+2.75);
    \draw[thin, cyan] (9.75, \coord+2.75) -- (10.25, \coord+2.75);
    \draw[thin, yellow, ->] (10.5, \coord+2.75) -- (11.5, \coord+2.75);
    \draw[thin, yellow] (11.5, \coord+2.75) -- (12.5, \coord+2.75);
    \draw[black, ultra thin, ->] (8, \coord+2.75) -- (8, \coord+1.375);
    \draw[black, ultra thin] (8, \coord+1.375) -- (8, \coord);
    \draw[black, ultra thin, ->] (9, \coord+2.75) -- (9, \coord+1.375);
    \draw[black, ultra thin] (9, \coord+1.375) -- (9, \coord);
    \draw[black, ultra thin, ->] (9.25, \coord+2.75) -- (9.25, \coord+1.375);
    \draw[black, ultra thin] (9.25, \coord+1.375) -- (9.25, \coord);
    \draw[black, ultra thin, ->] (10.25, \coord+2.75) -- (10.25, \coord+1.375);
    \draw[black, ultra thin] (10.25, \coord+1.375) -- (10.25, \coord);
    \draw[black, ultra thin, ->] (10.5, \coord+2.75) -- (10.5, \coord+1.375);
    \draw[black, ultra thin] (10.5, \coord+1.375) -- (10.5, \coord);
    \draw[black, ultra thin, ->] (12.5, \coord+2.75) -- (12.5, \coord+1.375);
    \draw[black, ultra thin] (12.5, \coord+1.375) -- (12.5, \coord);

    \node at (1.7, \coord+5.5) {$\overline r_5^-$};
    \node at (2.5, \coord+5.5) {$\overline r_5^+$};
    
    \node at (8.25, \coord+4.5) {$r_3^+$};
    \node at (9.75, \coord+4.5) {$r_3^-$};
    \node at (10.75, \coord+4.5) {$r_4^+$};
    \node at (12.25, \coord+4.5) {$r_4^-$};
    
    \node at (9, \coord+3.5) {$e_3^-$};
    \node at (11.5, \coord+3.5) {$e_4^-$};
    \node at (11.5, \coord+2.5) {$e_5^+$};
    \node at (9.75, \coord+2.5) {$e_6^+$};
    \node at (8.5, \coord+2.5) {$e_7^+$};

    \node at (8.25, \coord+1.5) {$r_7^-$};
    \node at (8.75, \coord+1) {$r_7^+$};
    \node at (9.5, \coord+1.5) {$r_6^-$};
    \node at (10, \coord+1) {$r_6^+$};
    \node at (10.75, \coord+1.5) {$r_5^-$};
    \node at (12.25, \coord+1.5) {$r_5^+$};

    \node at (0.5, \coord+4) {$p-2$};
    \fill (4,\coord+3.15) circle (1pt);
    \fill (3,\coord+3.088) circle (1pt);
    \fill (4,\coord+2.85) circle (1pt);
    \fill (6,\coord+3) circle (1pt);
    \fill (2,\coord+3) circle (1pt);
    \fill (2,\coord+1) circle (1pt);

    \fill (8,\coord+2.75) circle (1pt);
    \fill (9,\coord+2.75) circle (1pt);
    \fill (9.25,\coord+2.75) circle (1pt);
    \fill (10.25,\coord+2.75) circle (1pt);
    \fill (10.5,\coord+2.75) circle (1pt);
    \fill (12.5,\coord+2.75) circle (1pt);

    }

\end{tikzpicture}
\caption{An example to illustrate how to realize a representation $\chi$ of non-trivial-ends type with rational $\chi_n$-part as the holonomy of a translation surface with poles in $\mathcal{H}_1(np; -p,\dots,-p)$ with prescribed rotation number $k$. In this picture $n=9$ with $h=4$, $s=2$, and $t=7$.}
\label{fig:chiratposvol}
\end{figure}

\noindent The resulting surface is a genus one differential $(Y,\xi)$ with one single zero, two double poles and $n-2$ simple poles. It remains to bubble sufficiently many copies of $(\C,\,dz)$ in order to get the desired pole orders and rotation number. Once identified together, the rays $r_l^+$ and $r_l^-$ determine an infinite ray $\tilde r_l\subset(Y,\xi)$ joining the unique zero of $\xi$ and a simple pole for any $l\notin\{1,s+1\}$. We bubble along any such a ray a copy of the genus zero differential $(\C,\, z^{p-2}dz)$. Equivalently, we bubble $p-1$ copies of $(\C,\,dz)$. All simple poles are turned to higher order poles of order $p$ with non-zero residue. By construction, we can always find an infinite ray leaving from $Q_2$ and pointing towards the infinity. We bubble along such a ray a copy of $(\C,\, z^{p-2}dz)$. Finally, we can find an infinite ray leaving from $P_1$ and another leaving from $Q_1$ which are disjoint and both point toward the infinity. We bubble a copy of $(\C,\, z^{k-2}dz)$ along the first ray and a copy of $(\C,\, z^{p-k}dz)$ along the second ray. The final surface is a genus one differential $(X,\omega)\in\mathcal{H}_1(np; -p,\dots,-p)$ where each pole has non-zero residue. By choosing $\alpha$ and $\beta$ as in \S\ref{pv2p} one can show that $(X,\omega)$ has rotation number equal to $\gcd(k,p)$ as desired.

\smallskip

\noindent So far we have focused on the case that the representation $\chi_1$ has positive volume. The non-positive volume case works in the same fashion with the only exception being that the interior of the parallelogram $\mathcal{P}$ defined by the above chain \eqref{parchirat} is cut out from the first copy of $(\C,\,dz)$ considered above. The rest of the construction works \textit{mutatis mutandis} for the case of non-positive volume, see Figure \ref{fig:chiratnegvol}. Alternatively, since $\chi$ is of non-trivial-ends type, Lemma \ref{lem:posvol} applies and one can find another pair of handle generators such that the respective volume is positive.

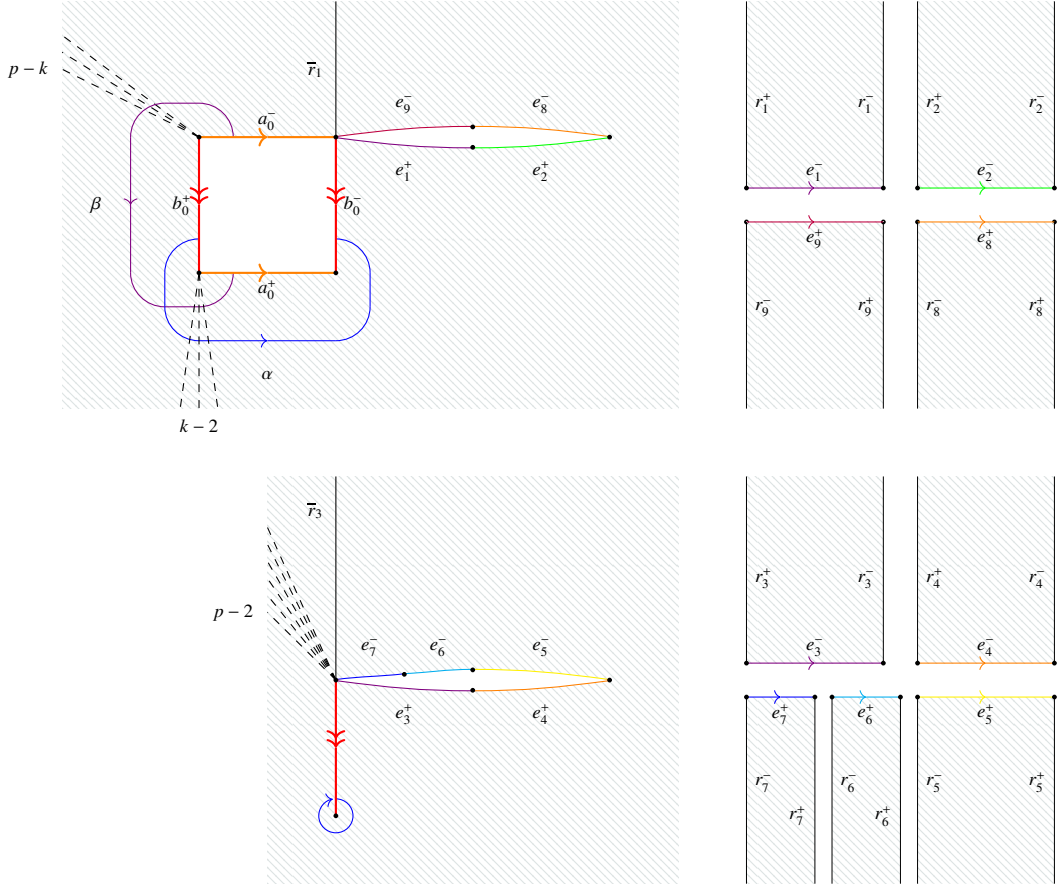
\begin{figure}[!ht] 
\centering
\begin{tikzpicture}[scale=0.9, every node/.style={scale=0.7}]
\definecolor{pallido}{RGB}{221,227,227}

    \foreach \y [evaluate=\y as \coord using  \y] in {-6, 1} 
    {
    \pattern [dashed, pattern=north west lines, pattern color=pallido] (1,\coord-1) -- (7, \coord-1) -- (7, \coord+5) -- (1, \coord+5)-- (1,\coord-1);
    }

    \foreach \y [evaluate=\y as \coord using  \y] in {1} 
    {
    
    \pattern [dashed, pattern=north west lines, pattern color=pallido] (-2,\coord-1) -- (1, \coord-1) -- (1, \coord+5) -- (-2, \coord+5)-- (-2,\coord-1);
    
    \draw[thin, black] (2, \coord+3) to (2, \coord+5);
    \draw [thin, black, dashed] (0, \coord+3)--(-2, \coord+4.5);
    \draw [thin, black, dashed] (0, \coord+3)--(-2, \coord+4.25);
    \draw [thin, black, dashed] (0, \coord+3)--(-2, \coord+4);
    
    \draw [thin, violet] (0.5,\coord+3) arc [start angle = 0, end angle =90 , radius = 0.5];
    \draw [thin, violet] (0, \coord+3.5) -- (-0.5, \coord+3.5);
    \draw [thin, violet] (-0.5, \coord+3.5) arc [start angle = 90, end angle =180 , radius = 0.5];
    \draw [thin, violet, ->] (-1,\coord+3)--(-1,\coord+2);
    \draw [thin, violet] (-1,\coord+2)--(-1,\coord+1);
    \draw [thin, violet] (-1,\coord+1) arc [start angle = 180, end angle=270 , radius = 0.5];
    \draw [thin, violet] (-0.5,\coord+0.5)--(0,\coord+0.5);
    \draw [thin, violet] (0,\coord+0.5) arc [start angle = 270, end angle = 360 , radius = 0.5];
    
    \draw [thin, blue] (0,\coord+1.5) arc [start angle = 90, end angle =180 , radius = 0.5];
    \draw [thin, blue] (-0.5, \coord+1) -- (-0.5, \coord+0.5);
    \draw [thin, blue] (-0.5, \coord+0.5) arc [start angle = 180, end angle =270 , radius = 0.5];
    \draw [thin, blue, ->] (0,\coord)--(1,\coord);
    \draw [thin, blue] (1,\coord)--(2,\coord);
    \draw [thin, blue] (2,\coord) arc [start angle = 270, end angle=360 , radius = 0.5];
    \draw [thin, blue] (2.5,\coord+0.5)--(2.5,\coord+1);
    \draw [thin, blue] (2.5,\coord+1) arc [start angle = 0, end angle = 90 , radius = 0.5];

    \draw[thick, purple] (2, \coord+3) to [out=7.5, in=180] (4, \coord+3.15);
    \draw[thick, orange] (4, \coord+3.15) to [out=0, in=172.5] (6, \coord+3);
    \draw[thick, green] (6, \coord+3) to [out=187.5, in=0] (4, \coord+2.85);
    \draw[thick, violet] (2, \coord+3) to [out=352.5, in=180] (4, \coord+2.85);
    \fill[white] (2, \coord+3) to [out=7.5, in=180] (4, \coord+3.15) to  [out=0, in=172.5] (6, \coord+3) to [out=187.5, in=0] (4, \coord+2.85) to  [out=180, in=352.5] (2, \coord+3) to  (2, \coord+3);
    \fill[white] (2, \coord+1) -- (2, \coord+3) -- (0, \coord+3) -- (0, \coord+1) -- (2, \coord+1);
    
    \draw[thick, orange] (2, \coord+3) to (1, \coord+3);
    \draw[thick, orange, <-] (1, \coord+3) to (0, \coord+3);
    \draw[thick, orange] (2, \coord+1) to (1, \coord+1);
    \draw[thick, orange, <-] (1, \coord+1) to (0, \coord+1);
    
    \draw[thick, red, ->>] (2, \coord+3) to (2, \coord+2);
    \draw[thick, red] (2, \coord+2) to (2, \coord+1);
    \draw[thick, red, ->>] (0, \coord+3) to (0, \coord+2);
    \draw[thick, red] (0, \coord+2) to (0, \coord+1);
    
    \draw [thin, black, dashed] (0, \coord+1)--(-0.275, \coord-1);
    \draw [thin, black, dashed] (0, \coord+1)--(0, \coord-1);
    \draw [thin, black, dashed] (0, \coord+1)--(0.275, \coord-1);
    
    \fill (4,\coord+3.15) circle (1pt);
    \fill (4,\coord+2.85) circle (1pt);
    \fill (6,\coord+3) circle (1pt);
    \fill (2,\coord+3) circle (1pt);
    \fill (2,\coord+1) circle (1pt);
    \fill (0,\coord+3) circle (1pt);
    \fill (0,\coord+1) circle (1pt);

    \node at (3, \coord+3.5) {$e_9^-$};
    \node at (5, \coord+3.5) {$e_8^-$};
    \node at (3, \coord+2.5) {$e_1^+$};
    \node at (5, \coord+2.5) {$e_2^+$};

    \pattern [dashed, pattern=north west lines, pattern color=pallido] (8,\coord+2.25) -- (10, \coord+2.25) -- (10, \coord+5) -- (8, \coord+5)-- (8,\coord+2.25);
    \pattern [dashed, pattern=north west lines, pattern color=pallido] (10.5,\coord+2.25) -- (12.5, \coord+2.25) -- (12.5, \coord+5) -- (10.5, \coord+5)-- (10.5,\coord+2.25);
  
    \draw[thin, violet, ->] (8, \coord+2.25) -- (9, \coord+2.25);
    \draw[thin, violet]     (9, \coord+2.25) -- (10, \coord+2.25);
    \draw[thin, green, ->] (10.5, \coord+2.25) -- (11.5, \coord+2.25);
    \draw[thin, green] (11.5, \coord+2.25) -- (12.5, \coord+2.25);
    \draw[black, ultra thin] (8, \coord+2.25) -- (8, \coord+5);
    \draw[black, ultra thin] (10, \coord+2.25) -- (10, \coord+5);
    \draw[black, ultra thin] (10.5, \coord+2.25) -- (10.5, \coord+5);
    \draw[black, ultra thin] (12.5, \coord+2.25) -- (12.5, \coord+5);
    
    \fill (8,\coord+1.75) circle (1pt);
    \fill (10,\coord+1.75) circle (1pt);
    \fill (10.5,\coord+1.75) circle (1pt);
    \fill (12.5,\coord+1.75) circle (1pt);
    
    \pattern [dashed, pattern=north west lines, pattern color=pallido] (8,\coord+1.75) -- (10, \coord+1.75) -- (10, \coord-1) -- (8, \coord-1)-- (8,\coord+1.75);
    \pattern [dashed, pattern=north west lines, pattern color=pallido] (10.5,\coord+1.75) -- (12.5, \coord+1.75) -- (12.5, \coord-1) -- (10.5, \coord-1)-- (10.5,\coord+1.7);
  
    \draw[thin, purple, ->] (8, \coord+1.75) -- (9, \coord+1.75);
    \draw[thin, purple] (9, \coord+1.75) -- (10, \coord+1.75);
    \draw[thin, orange, ->] (10.5, \coord+1.75) -- (11.5, \coord+1.75);
    \draw[thin, orange] (11.5, \coord+1.75) -- (12.5, \coord+1.75);
    \draw[black, ultra thin] (8, \coord+1.75) -- (8, \coord-1);
    \draw[black, ultra thin] (10, \coord+1.75) -- (10, \coord-1);
    \draw[black, ultra thin] (10.5, \coord+1.75) -- (10.5, \coord-1);
    \draw[black, ultra thin] (12.5, \coord+1.75) -- (12.5, \coord-1);
    
    \node at (1.7, \coord+4) {$\overline r_1$};
    
    \node at (8.25, \coord+3.5) {$r_1^+$};
    \node at (9.75, \coord+3.5) {$r_1^-$};
    \node at (10.75, \coord+3.5) {$r_2^+$};
    \node at (12.25, \coord+3.5) {$r_2^-$};
    
    \node at (9, \coord+2.5) {$e_1^-$};
    \node at (9, \coord+1.5) {$e_9^+$};
    \node at (11.5, \coord+2.5) {$e_2^-$};
    \node at (11.5, \coord+1.5) {$e_8^+$};
    
    \node at (8.25, \coord+0.5) {$r_9^-$};
    \node at (9.75, \coord+0.5) {$r_9^+$};
    \node at (10.75, \coord+0.5) {$r_8^-$};
    \node at (12.25, \coord+0.5) {$r_8^+$};
    
    \node at (0, \coord-1.25) {$k-2$};
    \node at (-2.5, \coord+4) {$p-k$};
    
    \node at (-1.5, \coord+2) {$\beta$};
    \node at (1, \coord-0.5) {$\alpha$}; 
    \node at (1, 1.75) {$a_0^+$}; 
    \node at (1, 4.25) {$a_0^-$}; 
    \node at (-0.25, 3) {$b_0^+$}; 
    \node at (2.25, 3) {$b_0^-$};
    
    \fill (8,\coord+2.25) circle (1pt);
    \fill (10,\coord+2.25) circle (1pt);
    \fill (10.5,\coord+2.25) circle (1pt);
    \fill (12.5,\coord+2.25) circle (1pt);
    
    }
    
    \foreach \y [evaluate=\y as \coord using  \y] in {-6} 
    {
    
    \draw[thin, black] (2, \coord+2) to (2, \coord+5);
    \draw[thin, blue, ->] (2, \coord+0.25) arc [start angle = 90, end angle = -260,radius = 0.25];
    
    \draw [thin, black, dashed] (2, \coord+2)--(1, \coord+3.5);
    \draw [thin, black, dashed] (2, \coord+2)--(1, \coord+3.25);
    \draw [thin, black, dashed] (2, \coord+2)--(1, \coord+3);
    \draw [thin, black, dashed] (2, \coord+2)--(1, \coord+3.75);
    \draw [thin, black, dashed] (2, \coord+2)--(1, \coord+4);
    \draw [thin, black, dashed] (2, \coord+2)--(1, \coord+4.25);
    
    \draw[thick, blue] (2, \coord+2) to [out=7.5, in=186.5] (3, \coord+2.085);
    \draw[thick, cyan] (3, \coord+2.085) to [out=3.5 , in=180] (4, \coord+2.15);
    \draw[thick, yellow] (4, \coord+2.15) to [out=0, in=172.5] (6, \coord+2);
    \draw[thick, orange] (6, \coord+2) to [out=187.5, in=0] (4, \coord+1.85);
    \draw[thick, violet] (2, \coord+2) to [out=352.5, in=180] (4, \coord+1.85);
    \fill[white] (2, \coord+2) to [out=7.5, in=186.5] (3, \coord+2.085) to [out=3.5 , in=180] (4, \coord+2.15) to [out=0, in=172.5] (6, \coord+2) (6, \coord+2) to [out=187.5, in=0] (4, \coord+1.85) to  [out=180, in=352.5] (2, \coord+2) to  (2, \coord+2);

    \pattern [dashed, pattern=north west lines, pattern color=pallido] (8,\coord+2.25) -- (10, \coord+2.25) -- (10, \coord+5) -- (8, \coord+5)-- (8,\coord+2.25);
    \pattern [dashed, pattern=north west lines, pattern color=pallido] (10.5,\coord+2.25) -- (12.5, \coord+2.25) -- (12.5, \coord+5) -- (10.5, \coord+5)-- (10.5,\coord+2.25);
  
    \draw[thin, violet, ->] (8, \coord+2.25) -- (9, \coord+2.25);
    \draw[thin, violet] (9, \coord+2.25) -- (10, \coord+2.25);
    \draw[thin, orange, ->] (10.5, \coord+2.25) -- (11.5, \coord+2.25);
    \draw[thin, orange] (11.5, \coord+2.25) -- (12.5, \coord+2.25);
    \draw[black, ultra thin] (8, \coord+2.25) -- (8, \coord+5);
    \draw[black, ultra thin] (10, \coord+2.25) -- (10, \coord+5);
    \draw[black, ultra thin] (10.5, \coord+2.25) -- (10.5, \coord+5);
    \draw[black, ultra thin] (12.5, \coord+2.25) -- (12.5, \coord+5);
    \draw[thick, red, ->>] (2, \coord+2) to (2, \coord+1);
    \draw[thick, red] (2, \coord+1) to (2, \coord+0);
    
    \fill (8,\coord+2.25) circle (1pt);
    \fill (10,\coord+2.25) circle (1pt);
    \fill (10.5,\coord+2.25) circle (1pt);
    \fill (12.5,\coord+2.25) circle (1pt);
    
    \node at (3, \coord+1.5) {$e_3^+$};
    \node at (5, \coord+1.5) {$e_4^+$};
    \node at (2.5, \coord+2.5) {$e_7^-$};
    \node at (3.5, \coord+2.5) {$e_6^-$};
    \node at (5, \coord+2.5) {$e_5^-$};
    
    \pattern [dashed, pattern=north west lines, pattern color=pallido] (8,\coord+1.75) -- (9, \coord+1.75) -- (9, \coord-1) -- (8, \coord-1)-- (8,\coord+1.75);
    \pattern [dashed, pattern=north west lines, pattern color=pallido] (9.25,\coord+1.75) -- (10.25, \coord+1.75) -- (10.25, \coord-1) -- (9.25, \coord-1)-- (9.25,\coord+1.75);
    \pattern [dashed, pattern=north west lines, pattern color=pallido] (10.5,\coord+1.75) -- (12.5, \coord+1.75) -- (12.5, \coord-1) -- (10.5, \coord-1)-- (10.5,\coord+1.7);
  
    \draw[thin, blue, ->] (8, \coord+1.75) -- (8.5, \coord+1.75);
    \draw[thin, blue] (8.5, \coord+1.75) -- (9, \coord+1.75);
    \draw[thin, cyan, ->] (9.25, \coord+1.75) -- (9.75, \coord+1.75);
    \draw[thin, cyan] (9.75, \coord+1.75) -- (10.25, \coord+1.75);
    \draw[thin, yellow, ->] (10.5, \coord+1.75) -- (11.5, \coord+1.75);
    \draw[thin, yellow] (11.5, \coord+1.75) -- (12.5, \coord+1.75);
    \draw[black, ultra thin] (8, \coord+1.75) -- (8, \coord-1);
    \draw[black, ultra thin] (9, \coord+1.75) -- (9, \coord-1);
    \draw[black, ultra thin] (9.25, \coord+1.75) -- (9.25, \coord-1);
    \draw[black, ultra thin] (10.25, \coord+1.75) -- (10.25, \coord-1);
    \draw[black, ultra thin] (10.5, \coord+1.75) -- (10.5, \coord-1);
    \draw[black, ultra thin] (12.5, \coord+1.75) -- (12.5, \coord-1);
   
    \node at (1.7, \coord+4.5) {$\overline r_3$};
    
    \node at (8.25, \coord+3.5) {$r_3^+$};
    \node at (9.75, \coord+3.5) {$r_3^-$};
    \node at (10.75, \coord+3.5) {$r_4^+$};
    \node at (12.25, \coord+3.5) {$r_4^-$};
    
    \node at (9, \coord+2.5) {$e_3^-$};
    \node at (11.5, \coord+2.5) {$e_4^-$};
    \node at (11.5, \coord+1.5) {$e_5^+$};
    \node at (9.75, \coord+1.5) {$e_6^+$};
    \node at (8.5, \coord+1.5) {$e_7^+$};

    \node at (8.25, \coord+0.5) {$r_7^-$};
    \node at (8.75, \coord) {$r_7^+$};
    \node at (9.5, \coord+0.5) {$r_6^-$};
    \node at (10, \coord) {$r_6^+$};
    \node at (10.75, \coord+0.5) {$r_5^-$};
    \node at (12.25, \coord+0.5) {$r_5^+$};

    \node at (0.5, \coord+3) {$p-2$};
    \fill (4,\coord+2.15) circle (1pt);
    \fill (3,\coord+2.085) circle (1pt);
    \fill (4,\coord+1.85) circle (1pt);
    \fill (6,\coord+2) circle (1pt);
    \fill (2,\coord+2) circle (1pt);
    \fill (2,\coord+0) circle (1pt);

    \fill (8,\coord+1.75) circle (1pt);
    \fill (9,\coord+1.75) circle (1pt);
    \fill (9.25,\coord+1.75) circle (1pt);
    \fill (10.25,\coord+1.75) circle (1pt);
    \fill (10.5,\coord+1.75) circle (1pt);
    \fill (12.5,\coord+1.75) circle (1pt);

    }
\end{tikzpicture}
\caption{An example to illustrate how to realize a representation $\chi$ of non-trivial-ends type with negative $\chi_1$-part and with rational $\chi_n$-part as the holonomy of a translation surface with poles in $\mathcal{H}_1(np; -p,\dots,-p)$ with prescribed rotation number $k$. In this picture $n=9$ with $h=4$, $s=2$, and $t=7$.}
\label{fig:chiratnegvol}
\end{figure}

\begin{rmk}\label{excaserat}
There is an exceptional case not covered by the construction above, which is a representation of non-trivial-ends type $\chi\colon\shomolzon\longrightarrow \C$ as the holonomy of some genus one differential in the connected component of the stratum $\mathcal{H}_1(2n, -2,\dots,-2)$ of translation surfaces with rotation number one. Recall that here we assume that $\chi$ has rational $\chi_n$-part and $\text{Im}(\chi_n)$ does not satisfy the reordering property. Nevertheless, a slight modification of the previous construction permits to realize $\chi$ even in this special case. In short, the modification consists in pasting all the strips $\mathcal{S}_l$ in one copy of $(\C,\,dz)$. This can be simply done as follows. We define $e=\overline{\zeta_1\,\zeta_{h+1}}$ and let $e^{\pm}$ be the edges we obtain by slitting $(\C,\,dz)$ along $e$. Then we partition $e^+=e_1^+\cup\cdots\cup e_h^+$ and $e^-=e_{h+1}^-\cup\cdots\cup e_n^-$. The rest of the construction is essentially the same.
\end{rmk}

\subsubsection{At least one pole has zero residue}\label{optr} Let $\chi\colon\shomolzon\longrightarrow \C$ be a representation of non-trivial-ends type. As above, let $\gamma_i$ denote a simple closed curve enclosing the $i$-th puncture. In this subsection we assume that $\chi(\gamma_i)\neq0$ for $i=1,\dots,m<n$ and $\chi(\gamma_i)=0$ for $i=m+1,\dots,n$. The representation $\chi$ naturally yields a new representation $\overline \chi\colon\text{H}_1(S_{1,m},\,\Z)\longrightarrow \C$ of non-trivial-ends type obtained by "filling" the punctures with labelling $i=m+1,\dots,n$. 

\smallskip

\noindent Let $\chi_m\colon\text{H}_1(S_{0,m},\,\mathbb Z)\longrightarrow \C$ be the $\chi_m$-part of the representation $\overline \chi$. We distinguish two cases as follows. If the representation $\chi_m$ is not real-collinear, see Definition \ref{ratchar}, or it is real-collinear and satisfies the reordering property, see Definition \ref{sortprop}, we need to introduce an auxiliary representation as in Definition \ref{defsupprep}:
\begin{equation}
    \rho(\alpha)=\chi(\alpha), \quad \rho(\beta)=\chi(\beta), \quad \rho(\delta_1)=\cdots=\rho(\delta_{n-m+1})=0,
\end{equation}

\noindent where $\{\alpha,\beta\}$ is a pair of handle generators of $\text{H}_1(S_{1,\,n-m+1},\,\mathbb Z)\subset \text{H}_1(S_{1,\,n},\,\mathbb Z)$. According to the constructions developed in Sections \S\ref{genusoneordertwo} and \S\ref{genusoneorderhigherthantwo}, we can realize $\rho$ as the holonomy of some translation structure $(X_1,\omega_1)$ in the stratum $\mathcal{H}_1\big( (n-m+1)p; -p,\dots,-p \big)$ with $n-m+1$ poles of order $p$ with zero residue and prescribed rotation number $k$. Notice that, from our constructions it is always possible to find an infinite ray $r$ with fixed direction, say $v$, starting from the zero of $\omega_1$ and pointing toward a pole. For instance, such a ray can be taken as any ray leaving from any point labelled with "$Q$" in our constructions and pointing toward a direction $v$. Next we can realize $\chi_m$ as the holonomy of some translation surface $(X_2,\omega_2)\in\mathcal{H}_0\big(m-2;-1,\dots,-1\big)$. Up to changing the direction of $v$ a little if needed, we can find an infinite ray, say $\overline r$, leaving from the zero of $\omega_2$ and pointing toward one of the poles with direction $v$. We glue $(X_1,\omega_1)$ and $(X_2,\omega_2)$ by slitting and glue the rays $r$ and $\overline r$ as described in Definition \ref{gluingsurfaces}. Notice that the resulting translation structure is homeomorphic to $S_{1,n}$ and it has one zero, $n-m$ poles of order $p$ with zero residue, one pole of order $p$ with non-zero residue, and $m-1$ simple poles. We can finally find $m-1$ rays joining the zero of the resulting translation structure and the simple poles. We bubble along each one of these rays a copy of $(\C,\,z^{p-2}dz)$. This final translation structure lies in $\mathcal{H}_0\left( np; -p,\dots,-p \right)$. By construction it has period character $\chi$ and rotation number $k$ as desired.

\begin{figure}[!ht] 
\centering
\begin{tikzpicture}[scale=0.9, every node/.style={scale=0.7}]
\definecolor{pallido}{RGB}{221,227,227}

    \pattern [pattern=north west lines, pattern color=pallido]
    (-2,2) -- (0,2) -- (0,4) -- (-2,4) -- (-2,2);
    \draw [thick, orange] (-1,2)--(0,2);
    \draw [thick, red] (0,3)--(0,4);
    \draw [thick, orange] (-1,4)--(0,4);
    \draw [thick, red] (-2,3)--(-2,4);
    \draw [thick, orange, ->] (-2,2) to (-1,2);
    \draw [thick, red, ->>] (0,2) to (0,3);
    \draw [thick, orange, ->] (-2,4) to (-1,4);
    \draw [thick, red, ->>] (-2,2) to (-2,3);
    \draw [thin, violet, ->] (-1.75,2) to (-1.75,4);
    \draw [thin, blue, ->] (-2,2.25) to (0,2.25);
    \fill (-2,2) circle (1.5pt);
    \fill (0,2) circle (1.5pt);
    \fill (0,4) circle (1.5pt);
    \fill (-2,4) circle (1.5pt);

    \node at (-1, 2.5) {$\alpha$};
    \node at (-1.5, 3) {$\beta$};
    \node at (-1, 1.75) {$a_0^-$};
    \node at (-1, 4.25) {$a_0^+$};
    \node at (-2.25, 3) {$b_0^+$};
    \node at (0.25, 3) {$b_0^-$};

    \foreach \y [evaluate=\y as \coord using  \y] in {-11, 1} 
    {
    \pattern [dashed, pattern=north west lines, pattern color=pallido] (1,\coord) -- (7, \coord) -- (7, \coord+6) -- (1, \coord+6)-- (1,\coord);
    \draw[thin, blue, ->] (2, \coord+1.25) arc [start angle = 90, end angle = -260,radius = 0.25];
    \draw[thin, black] (2, \coord+3) to (2, \coord+6);
    \draw[thick, red] (2, \coord+3) to (2, \coord+2);
    \draw[thick, red, <<-] (2, \coord+2) to (2, \coord+1);
    \draw [thin, black, dashed] (2, \coord+3)--(1, \coord+4.5);
    \draw [thin, black, dashed] (2, \coord+3)--(1, \coord+4.25);
    \draw [thin, black, dashed] (2, \coord+3)--(1, \coord+4);
    }

    \foreach \y [evaluate=\y as \coord using  \y] in {1} 
    {

    \draw[thin, black] (2, \coord+3) to [out=10, in=170] (6, \coord+3);
    \draw[thick, green] (4, \coord+2.75) to [out=0, in=195] (6, \coord+3);
    \draw[thick, violet] (2, \coord+3) to [out=345, in=180] (4, \coord+2.75);
    \fill[white] (2, \coord+3) to [out=10, in=170] (6, \coord+3) to [out=195, in=0] (4, \coord+2.75) to [out=180, in=345] (2, \coord+3);

    \draw [thin, black, dashed] (2, \coord+1)--(1.725, \coord);
    \draw [thin, black, dashed] (2, \coord+1)--(2, \coord);
    \draw [thin, black, dashed] (2, \coord+1)--(2.275, \coord);
    
    \fill (4,\coord+2.75) circle (0.75pt);
    \fill (6,\coord+3) circle (0.75pt);
    \fill (2,\coord+3) circle (0.75pt);
    \fill (2,\coord+1) circle (0.75pt);

    \pattern [dashed, pattern=north west lines, pattern color=pallido] (8,\coord+3.25) -- (10, \coord+3.25) -- (10, \coord+6) -- (8, \coord+6)-- (8,\coord+3.25);
    \pattern [dashed, pattern=north west lines, pattern color=pallido] (10.5,\coord+3.25) -- (12.5, \coord+3.25) -- (12.5, \coord+6) -- (10.5, \coord+6)-- (10.5,\coord+3.25);
  
    \draw[thin, violet, ->] (8, \coord+3.25) -- (9, \coord+3.25);
    \draw[thin, violet] (9, \coord+3.25) -- (10, \coord+3.25);
    \draw[thin, green, ->] (10.5, \coord+3.25) -- (11.5, \coord+3.25);
    \draw[thin, green] (11.5, \coord+3.25) -- (12.5, \coord+3.25);
    \draw[black, ultra thin] (8, \coord+3.25) -- (8, \coord+6);
    \draw[black, ultra thin] (10, \coord+3.25) -- (10, \coord+6);
    \draw[black, ultra thin] (10.5, \coord+3.25) -- (10.5, \coord+6);
    \draw[black, ultra thin] (12.5, \coord+3.25) -- (12.5, \coord+6);

    \fill (8,\coord+3.25) circle (1pt);
    \fill (10,\coord+3.25) circle (1pt);
    \fill (10.5,\coord+3.25) circle (1pt);
    \fill (12.5,\coord+3.25) circle (1pt);
    
    \node at (1.7, \coord+5.5) {$\overline r_1$};
    
    \node at (8.25, \coord+4.5) {$r_1$};
    \node at (9.75, \coord+4.5) {$r_1^-$};
    \node at (10.75, \coord+4.5) {$r_2^+$};
    \node at (12.25, \coord+4.5) {$r_2^-$};
    \node at (9, \coord+3.5) {$e_1^-$};
    \node at (11.5, \coord+3.5) {$e_2^-$};
    \node at (2, \coord-0.25) {$k-2$};
    \node at (0.5, \coord+4) {$p-k$};
    
    \node at (4, \coord+3.5) {$l_{1}^-$};
    \node at (3, \coord+2.5) {$e_1^+$};
    \node at (5, \coord+2.5) {$e_2^+$};
    
    }
    
    \foreach \y [evaluate=\y as \coord using  \y] in {-4} 
    {
    
    \pattern [dashed, pattern=north west lines, pattern color=pallido] (-3,\coord+4) -- (1.5, \coord+4) -- (1.5, \coord+0) -- (-3, \coord+0)-- (-3,\coord+4);
    \node at (2 ,\coord+2) {$\cdots$};
    \pattern [dashed, pattern=north west lines, pattern color=pallido] (2.5,\coord+4) -- (7, \coord+4) -- (7, \coord+0) -- (2.5, \coord+0)-- (2.5,\coord+4);
    
    \node at (-2.5, \coord+0.5) {$\mathbb C$};
    \node at (3, \coord+0.5) {$\mathbb C$};
    
    \draw[thin, black] (-2.75, \coord+2) to [out=10, in=170] (1.25, \coord+2);
    \draw[thin, black] (-2.75, \coord+2) to [out=350, in=190] (1.25, \coord+2);
    \fill[white] (-2.75, \coord+2) to [out=10, in=170] (1.25, \coord+2) to [out=190, in=350] (-2.75, \coord+2);
    
    \draw [thin, black, dashed] (-2.75, \coord+2)--(-2.75, \coord+4);
    \draw [thin, black, dashed] (-2.75, \coord+2)--(-2.5, \coord+4);
    \draw [thin, black, dashed] (-2.75, \coord+2)--(-2.25, \coord+4);
    
    \node at (-2.5, \coord+4.25) {$p-2$};
    
    \draw[thin, black] (2.75, \coord+2) to [out=350, in=190] (6.75, \coord+2);
    \draw[thick, orange] (6.75, \coord+2) to [out=170, in=0] (4.75, \coord+2.2);
    \draw[thick, purple] (4.75, \coord+2.2) to [out=180, in=10] (2.75, \coord+2);
    \fill[white] (2.75, \coord+2) to [out=350, in=190] (6.75, \coord+2) to [out=170, in=0] (4.75, \coord+2.2) to [out=180, in=10] (2.75, \coord+2);
    
    \draw [thin, black, dashed] (2.75, \coord+2)--(2.75, \coord+4);
    \draw [thin, black, dashed] (2.75, \coord+2)--(3, \coord+4);
    \draw [thin, black, dashed] (2.75, \coord+2)--(3.25, \coord+4);
    
    \node at (3, \coord+4.25) {$p-2$};
    
    \fill (-2.75, \coord+2) circle (1pt);
    \fill (1.25, \coord+2) circle (1pt);
    \fill (2.75, \coord+2) circle (1pt);
    \fill (4.75, \coord+2.2) circle (1pt);
    \fill (6.75, \coord+2) circle (1pt);
    
    \pattern [dashed, pattern=north west lines, pattern color=pallido] (8,\coord+2) -- (10, \coord+2) -- (10, \coord) -- (8, \coord)-- (8,\coord+2);
    \pattern [dashed, pattern=north west lines, pattern color=pallido] (10.5,\coord+2) -- (12.5, \coord+2) -- (12.5, \coord) -- (10.5, \coord)-- (10.5,\coord+2);
  
    \draw[thin, purple, ->] (8, \coord+2) -- (9, \coord+2);
    \draw[thin, purple] (9, \coord+2) -- (10, \coord+2);
    \draw[thin, orange, ->] (10.5, \coord+2) -- (11.5, \coord+2);
    \draw[thin, orange] (11.5, \coord+2) -- (12.5, \coord+2);
    \draw[black, ultra thin] (8, \coord+2) -- (8, \coord);
    \draw[black, ultra thin] (10, \coord+2) -- (10, \coord);
    \draw[black, ultra thin] (10.5, \coord+2) -- (10.5, \coord);
    \draw[black, ultra thin] (12.5, \coord+2) -- (12.5, \coord);
    
    \fill (8,\coord+2) circle (1pt);
    \fill (10,\coord+2) circle (1pt);
    \fill (10.5,\coord+2) circle (1pt);
    \fill (12.5,\coord+2) circle (1pt);
    
    \node at (9, \coord+1.75) {$e_9^+$};
    \node at (11.5, \coord+1.75) {$e_8^+$};
        
    \node at (8.25, \coord+0.75) {$r_9^-$};
    \node at (9.75, \coord+0.75) {$r_9^+$};
    \node at (10.75, \coord+0.75) {$r_8^-$};
    \node at (12.25, \coord+0.75) {$r_8^+$};
    
    \node at (-0.75, \coord+1.5) {$l_1^+$};
    \node at (-0.75, \coord+2.5) {$l_2^-$};
    
    \node at (4.75, \coord+1.5) {$l_{n-m}^+$};
    \node at (3.75, \coord+2.5) {$e_9^-$};
    \node at (5.75, \coord+2.5) {$e_8^-$};
    
    }
    
    \foreach \y [evaluate=\y as \coord using  \y] in {-11} 
    {
    
    \draw [thin, black, dashed] (2, \coord+3)--(1, \coord+4.75);
    \draw [thin, black, dashed] (2, \coord+3)--(1, \coord+5);
    \draw [thin, black, dashed] (2, \coord+3)--(1, \coord+5.25);
    
    \draw[thick, blue] (2, \coord+3) to  (3, \coord+3.15);
    \draw[thick, cyan] (3, \coord+3.15) to [out=8.5, in=180] (4, \coord+3.25);
    \draw[thick, yellow] (4, \coord+3.25) to [out=0, in=170] (6, \coord+3);
    \draw[thick, orange] (6, \coord+3) to [out=190, in=0] (4, \coord+2.75);
    \draw[thick, violet] (2, \coord+3) to [out=350, in=180] (4, \coord+2.75);
    \fill[white] (2, \coord+3) to  (3, \coord+3.15) to [out=8.5, in=180] (4, \coord+3.25) to [out=0, in=170] (6, \coord+3) to [out=190, in=0] (4, \coord+2.75) to [out=180, in=350] (2, \coord+3) ;
    
    \node at (3, \coord+2.5) {$e_3^+$};
    \node at (5, \coord+2.5) {$e_4^+$};
    \node at (5, \coord+3.5) {$e_5^-$};
    \node at (3.5, \coord+3.5) {$e_6^-$};
    \node at (2.5, \coord+3.5) {$e_7^-$};

    \pattern [dashed, pattern=north west lines, pattern color=pallido] (8,\coord+3.25) -- (10, \coord+3.25) -- (10, \coord+6) -- (8, \coord+6)-- (8,\coord+3.25);
    \pattern [dashed, pattern=north west lines, pattern color=pallido] (10.5,\coord+3.25) -- (12.5, \coord+3.25) -- (12.5, \coord+6) -- (10.5, \coord+6)-- (10.5,\coord+3.25);
  
    \draw[thin, violet, ->] (8, \coord+3.25) -- (9, \coord+3.25);
    \draw[thin, violet] (9, \coord+3.25) -- (10, \coord+3.25);
    \draw[thin, orange, ->] (10.5, \coord+3.25) -- (11.5, \coord+3.25);
    \draw[thin, orange] (11.5, \coord+3.25) -- (12.5, \coord+3.25);
    \draw[black, ultra thin] (8, \coord+3.25) -- (8, \coord+6);
    \draw[black, ultra thin] (10, \coord+3.25) -- (10, \coord+6);
    \draw[black, ultra thin] (10.5, \coord+3.25) -- (10.5, \coord+6);
    \draw[black, ultra thin] (12.5, \coord+3.25) -- (12.5, \coord+6);
    
    \fill (8,\coord+3.25) circle (1pt);
    \fill (10,\coord+3.25) circle (1pt);
    \fill (10.5,\coord+3.25) circle (1pt);
    \fill (12.5,\coord+3.25) circle (1pt);
    
    \pattern [dashed, pattern=north west lines, pattern color=pallido] (8,\coord+2.75) -- (9, \coord+2.75) -- (9, \coord) -- (8, \coord)-- (8,\coord+2.75);
    \pattern [dashed, pattern=north west lines, pattern color=pallido] (9.25,\coord+2.75) -- (10.25, \coord+2.75) -- (10.25, \coord) -- (9.25, \coord)-- (9.25,\coord+2.75);
    \pattern [dashed, pattern=north west lines, pattern color=pallido] (10.5,\coord+2.75) -- (12.5, \coord+2.75) -- (12.5, \coord) -- (10.5, \coord)-- (10.5,\coord+2.7);
  
    \draw[thin, blue, ->] (8, \coord+2.75) -- (8.5, \coord+2.75);
    \draw[thin, blue] (8.5, \coord+2.75) -- (9, \coord+2.75);
    \draw[thin, cyan, ->] (9.25, \coord+2.75) -- (9.75, \coord+2.75);
    \draw[thin, cyan] (9.75, \coord+2.75) -- (10.25, \coord+2.75);
    \draw[thin, yellow, ->] (10.5, \coord+2.75) -- (11.5, \coord+2.75);
    \draw[thin, yellow] (11.5, \coord+2.75) -- (12.5, \coord+2.75);
    \draw[black, ultra thin] (8, \coord+2.75) -- (8, \coord);
    \draw[black, ultra thin] (9, \coord+2.75) -- (9, \coord);
    \draw[black, ultra thin] (9.25, \coord+2.75) -- (9.25, \coord);
    \draw[black, ultra thin] (10.25, \coord+2.75) -- (10.25, \coord);
    \draw[black, ultra thin] (10.5, \coord+2.75) -- (10.5, \coord);
    \draw[black, ultra thin] (12.5, \coord+2.75) -- (12.5, \coord);
   
    \node at (1.7, \coord+5.5) {$\overline r_3$};
    
    \node at (8.25, \coord+4.5) {$r_3^+$};
    \node at (9.75, \coord+4.5) {$r_3^-$};
    \node at (10.75, \coord+4.5) {$r_4^+$};
    \node at (12.25, \coord+4.5) {$r_4^-$};
    
    \node at (9, \coord+3.5) {$e_3^-$};
    \node at (11.5, \coord+3.5) {$e_4^-$};
    \node at (11.5, \coord+2.5) {$e_5^+$};
    \node at (9.75, \coord+2.5) {$e_6^+$};
    \node at (8.5, \coord+2.5) {$e_7^+$};

    \node at (8.25, \coord+1.5) {$r_7^-$};
    \node at (8.75, \coord+1) {$r_7^+$};
    \node at (9.5, \coord+1.5) {$r_6^-$};
    \node at (10, \coord+1) {$r_6^+$};
    \node at (10.75, \coord+1.5) {$r_5^-$};
    \node at (12.25, \coord+1.5) {$r_5^+$};

    \node at (0.5, \coord+4) {$p-2$};
    \fill (4,\coord+3.25) circle (1pt);
    \fill (3,\coord+3.15) circle (1pt);
    \fill (4,\coord+2.75) circle (1pt);
    \fill (6,\coord+3) circle (1pt);
    \fill (2,\coord+3) circle (1pt);
    \fill (2,\coord+1) circle (1pt);

    \fill (8,\coord+2.75) circle (1pt);
    \fill (9,\coord+2.75) circle (1pt);
    \fill (9.25,\coord+2.75) circle (1pt);
    \fill (10.25,\coord+2.75) circle (1pt);
    \fill (10.5,\coord+2.75) circle (1pt);
    \fill (12.5,\coord+2.75) circle (1pt);

    }
\end{tikzpicture}
\caption{How to realize a representation $\chi$ as the period character of some translation structure in $\mathcal{H}_1(np;-p,\dots,-p)$ when the $\chi_n$-part is rational and some (not all) punctures have zero residue.}
\label{fig:chiratpunctwithzerores}
\end{figure}
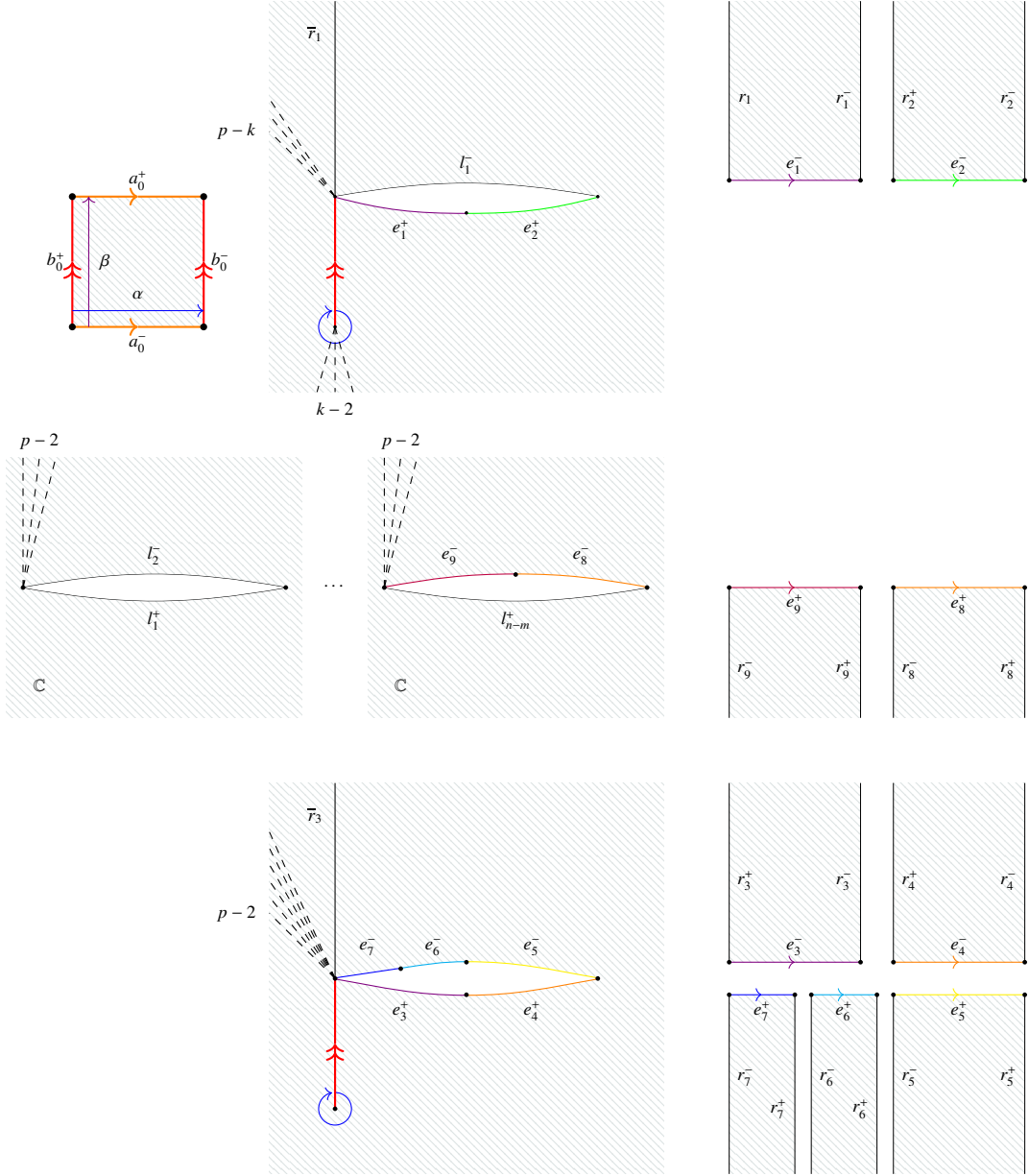

\noindent In the case that $\chi_m$ is rational and it does not satisfy the reordering property, we need to modify a little our construction in paragraph \S\ref{chirat} as shown in Figure \ref{fig:chiratpunctwithzerores}. We adopt the same notation as above with the only difference being that $n$ is now replaced by $m$. Let $w_i=\chi(\gamma_i)$ for $i=1,\dots,m$ and consider the collection of $W=\{w_1,\dots,w_m\}\subset \C^*$. Since $\chi_m$ does not satisfy the reordering property, there are two integers $s<t$ such that $W=W_1\cup W_2=\{w_1,\dots,w_s, w_{t+1},\dots,w_m\}\cup\{ w_{s+1},\dots,w_t\}$. We introduce the same collections of (possibly unbounded) polygons that comprise the following pieces. A parallelogram $\mathcal{P}$ determined by the chain \eqref{parchirat}. A copy of $(\C,\,dz)$ slit along the edges $b=\chi(\beta)$ and $e$ and along an infinite ray $\overline r_1$. A second copy of $(\C,\,dz)$ slit along the edges $b=\chi(\beta)$ and $\overline e$ and along an infinite ray $\overline r_{s+1}$ and, finally, $m$ strips $\mathcal{S}_l$ for $l=1,\dots,m$. As in paragraph \S\ref{chirat}, we slit $e$ and we partition the edge $e^+$ into $s$ sub-segments each of length $|w_l|$ for $l=1,\dots,s$. Unlike above, here we denote $e^-=l_1^-$. We next introduce $n-m$ new pieces, each of which is a copy of $(\C,\,dz)$. We slit them along $e$, \textit{i.e.} a segment congruent and with the same direction as $e$ and we denote the resulting edges $l_i^+$ and $l_{i+1}^-$ for $i=1,\dots,n-m$. Finally, we partition $l_{n-m+1}^-$ into $n-t$ segments, say $e_l^-$, each of which of length $|w_l|$ for $l=t+1,\dots,n$. Now we paste all the pieces as done above by identifying the edges with the same label (and opposite sign) and eventually bubbling copies of $(\C,\,dz)$ in order to have all poles of order $p$. The resulting translation surface lies in $\mathcal{H}_1(np;-p,\dots,-p)$. In particular, bubbling copies of $(\C,\,dz)$ properly, we can realize a structure so that its rotation number is $k$ as desired. The case of non-positive volume works \textit{mutatis mutandis}. 

\begin{rmk}
Since our construction relies on the discussion in paragraph \S\ref{chirat}, similarly in this case there is an exceptional case not covered by the construction above, which is to realize a representation of non-trivial-ends type $\chi\colon\shomolzon\longrightarrow \C$ as the holonomy of some genus one differential in the connected component of the stratum $\mathcal{H}_1(2n, -2,\dots,-2)$ of translation surfaces with rotation number one. However, this issue can be solved by using the same construction mentioned in Remark \ref{excaserat} by adding $n-m$ copies of $(\C,\,z^{p-2}dz)$ each of which is slit along $e$.
\end{rmk}

\subsection{General cases}\label{morezeros} We finally consider strata of genus one differentials with poles of different orders and possibly multiple zeros. The following statements are now corollaries of the Lemmas proved in the previous Sections \S\ref{genusoneordertwo} - \S\ref{resnotzero}.

\begin{cor}\label{genpol}
Let $\chi\colon\shomolzon\longrightarrow \C$ be any non-trivial representation. If $\chi$ can be realized in the stratum $\mathcal{H}_1(m;-p_1,\dots,-p_n)$ then it appears as the period character of a translation surface with poles in each connected component. 
\end{cor}

\begin{proof}
We begin with noticing that $m=p_1+\cdots+p_n$ and hence the following condition holds 
\begin{equation}
    \gcd(m,p_1,\dots,p_n)=\gcd(p_1,\dots,p_n)=p.
\end{equation}
\noindent For any $k$ dividing $p$, we can realize the representation $\chi$ as the holonomy of some translation structure $(X,\omega)$ in the stratum $\mathcal{H}_1(np; -p,\dots,-p)$ with rotation number $k$ as done in the previous sections. According to our constructions, we can always find $n$ rays $r_i$ joining the zero of $\omega$ with the puncture $P_i$. By bubbling $p_i-p$ copies of $(\C,\,dz)$ along the ray $r_i$ we can realize a translation surface $(Y,\xi)\in\mathcal{H}_1(m;-p_1,\dots,-p_n)$. In particular, we can choose these rays in such a way that the rotation number remains unaffected after any bubbling -- for instance, by using the notation in Sections \S\ref{genusoneordertwo} - \S\ref{genusoneorderhigherthantwo}, choose the ray $r_i$ as any ray leaving the point $Q_i$, see the dotted lines in Figures \ref{rotpk1}, \ref{rotnk1}, \ref{rotpek2}, and \ref{rotpok2}. This completes the proof.
\end{proof}

\begin{cor}\label{multzeros}
Let $\chi\colon\shomolzon\longrightarrow \C$ be any non-trivial representation. If $\chi$ can be realized in the stratum $\mathcal{H}_1(m_1,\dots,m_k;-p_1,\dots,-p_n)$ then it appears as the period character of a translation surface with poles in each connected component. 
\end{cor}

\begin{proof}
It is sufficient to observe that if $d=\gcd(m_1,\dots,m_k,p_1,\dots,p_n)$ then $d$ divides $m=m_1+\cdots+m_k$. Therefore we can realize $\chi$ as the period character of a translation surface with a single zero of order $m$ in the  stratum $\mathcal{H}_1(m;-p_1,\dots,-p_n)$. Then we can break the single zero as described in Section \S\ref{sec:zerobreak} to get the desired structure.
\end{proof}

\section{Higher genus meromorphic differentials with hyperelliptic structure}\label{sec:hgchyp}

\noindent We begin to prove Theorem \ref{mainthm} for surfaces of genus at least two and we shall complete the proof in Sections \S\ref{sec:hgcpar} and \S\ref{sec:trirep}. Given a non-trivial representation $\chi\colon\shomolzn\longrightarrow \C$, in the present section we shall determine whether $\chi$ appears as the period character of a hyperelliptic translation surface with poles, \textit{i.e.} a translation surface admitting a special symmetry of order two already introduced in Section \S\ref{sssec:hypinv}, see Definition \ref{hypdef}. More precisely, our aim is to prove the following

\begin{prop}\label{prop:hgdiffhyp}
Let $\chi$ be a non-trivial representation and suppose it arises as the period character of some meromorphic genus $g$ differential in a stratum admitting a hyperelliptic component. Then $\chi$ can be realized as the period character of some hyperelliptic translation surfaces with poles in the same stratum.
\end{prop}

\noindent According to Boissy, see \cite[Proposition 5.3]{BC} and Section \S\ref{mscc} above, a stratum admits a connected component of hyperelliptic translation surfaces if and only if it is one of the following
\begin{equation}
    \mathcal{H}_{g}(2m;-2p),\,\,\, \mathcal{H}_{g}(m, m;-2p),\,\,\, \mathcal{H}_{g}(2m;-p, -p),\,\,\, \mathcal{H}_{g}(m, m; -p, -p),
\end{equation}

\noindent for some $1\le p\le m$. Therefore, in what follows we assume that $\chi$ can be realized as the period character of some translation surface in one of those strata.  We shall consider two cases according to whether the representation $\chi$ is or is not of trivial-ends type. In Sections \S\ref{ssec:hypzeres}, \S\ref{ssec:hypnotzeres} and \S\ref{ssec:hypnotzeressimp} we prove Proposition \ref{prop:hgdiffhyp} for strata of meromorphic differentials with exactly one zero of maximal order and we derive the general case by breaking zeros, see Section \S\ref{sec:zerobreak}, as similarly done previously for genus one differentials. Finally, in Section \S\ref{ssec:nothypgenus2} we shall derive Theorem \ref{mainthm} for strata of genus two meromorphic differentials listed above.

\subsection{Trivial representation}\label{ssec:trirephyp} We study whether the trivial representation can appear as the period character of some hyperelliptic translation surfaces. In \cite[Theorem B]{CFG} the authors provided necessary and sufficient conditions for realizing the trivial representation in a given stratum $\mathcal{H}_g(m_1,\dots,m_k;-p_1,\dots,-p_n)$. One of these conditions relates the order of zeros with the order of poles as follows
\begin{equation}\label{eq:triconstr}
    m_j\le \sum_{i=1}^n p_i-n-1.
\end{equation}
In what follow, especially in Section \S\ref{sec:trirep}, we shall make a strong use of this constraint and we shall refer to it as \textit{Hurwitz type inequality}. For $\mathcal{H}_g(2m;-2p)$, this formula implies that $2m\le 2p-2$ which never holds for positive genus surfaces; in fact Remark \ref{gbcond} implies $g\le0$. For the same reason, the trivial representation cannot be realized in  strata of hyperelliptic type as $\mathcal{H}_g(2m;-p,-p)$. However, a few exceptions appear for strata $\mathcal{H}_g(m,m;-2p)$ and $\mathcal{H}_g(m,m;-p,-p)$. We shall consider them carefully in Section \S\ref{sec:trirep} below. The following Lemma holds.

\begin{lem}
The trivial representation cannot be realized as the period character of any hyperelliptic translation surface with poles and a single zero.
\end{lem}

\smallskip

\subsection{Poles with zero residue}\label{ssec:hypzeres} In the present subsection we shall consider representations of trivial-end type, see Definition \ref{def:triend}, and we distinguish two cases depending on whether $n=1$ or $n=2$.

\subsubsection{One higher order pole}\label{sssec:hyponehop} According to Section \S\ref{ssec:trirephyp}, a representation $\chi\colon\shomolzo\longrightarrow \C$ can be realized in a stratum $\mathcal{H}_g(2m;-2p)$, for any $p\ge1$, as long as $\chi$ is non-trivial. Therefore, our aim here is to realize any such a representation as the period character of some hyperelliptic translation surface with poles. Let $\{\alpha_1,\beta_1,\dots,\alpha_g,\beta_g\}$ be a system of handle generators for $\shomolzo$, see Definition \ref{def:syshandle}. Given a non-trivial representation $\chi\colon\shomolzo\longrightarrow \C$, we define
\begin{equation}
    \chi(\alpha_i)=a_i \,\,\text{ and }\,\, \chi(\beta_i)=b_i.
\end{equation}

\noindent Since $\chi$ is non-trivial, Lemma \ref{lem:allhandholnonzero} applies, and we can assume that both $a_i$ and $b_i$ are non-zero for all $i=1,\dots,g$. Up to replacing $\{\alpha_i,\beta_i\}$ with either $\{\alpha_i^{-1},\beta_i^{-1}\}$, $\{\beta_i,\alpha_i^{-1}\}$, or $\{\beta_i^{-1},\alpha_i\}$ if necessary -- notice that all these changes can be performed with a mapping class $\phi\in\modul$ -- we can also assume that both $\arg(a_i),\arg(b_i)\in \left[-\frac\pi2,\,\frac\pi2\,\right[$, so all vectors point rightwards. Let $P_0\in\C$ be any point and consider the chain, say $\mathcal{C}_1$, of segments defined as:

\begin{equation}\label{eq:chainpolzer}
    P_0\mapsto P_0+a_1=P_1\mapsto P_0+a_1+b_1=P_2\mapsto P_0+a_1+b_1+a_2=P_3\mapsto\cdots\mapsto P_0+\sum_{i=1}^g (a_i+b_i)=P_{2g}.
\end{equation}

\noindent Let $r_1$ be a half-ray leaving from $P_0$ and parallel to $\mathbb R^-=\{x\in\mathbb R\,|\, x<0\}$, hence pointing leftwards. Similarly, we define $r_2$ as the half-ray leaving from $P_{2g}$ and parallel to $\mathbb R^+$, hence pointing rightwards. We define $H_1$ as the (broken) half-plane bounded by the half-rays $r_1$ and $r_2$ and the chain $\mathcal C_1$ on their left, \textit{i.e.} $H_1$ is bounded by $r_1^-$, $r_2^-$ and $\mathcal{C}_1^-$, see Figure \ref{fig:hyponepole}, where the sign is used according to the usual convention. 

\smallskip

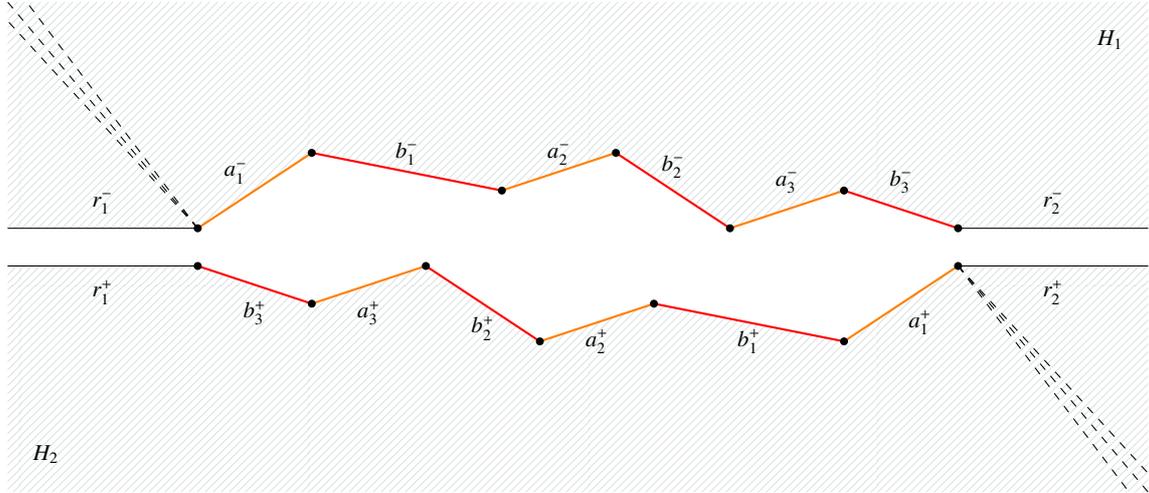
\begin{figure}[!ht] 
\centering
\begin{tikzpicture}[scale=1, every node/.style={scale=0.85}]
\definecolor{pallido}{RGB}{221,227,227}

    \pattern [pattern=north east lines, pattern color=pallido]
    (-7.5, 0)--(-5,0)--(-3.5,1)--(-1,0.5)--(0.5,1)--(2,0)--(3.5,0.5)--(5,0)--(7.5,0)--(7.5,3)--(-7.5, 3)--(-7.5, 0);
    \draw [thin, black](-7.5, 0) -- (-5,0);
    \draw [thick, orange] (-5,0)--(-3.5,1);
    \draw [thick, red] (-3.5,1)--(-1,0.5);
    \draw [thick, orange] (-1,0.5)--(0.5,1);
    \draw [thick, red] (0.5,1)--(2,0);
    \draw [thick, orange] (2,0)--(3.5,0.5);
    \draw [thick, red] (3.5,0.5)--(5,0);
    \draw [thin, black](5, 0) -- (7.5,0);
    \fill (-5,0) circle (1.5pt);
    \fill (-3.5,1) circle (1.5pt);
    \fill (-1,0.5) circle (1.5pt);
    \fill (0.5,1) circle (1.5pt);
    \fill (2,0) circle (1.5pt);
    \fill (3.5,0.5) circle (1.5pt);
    \fill (5,0) circle (1.5pt);

    \pattern [pattern=north east lines, pattern color=pallido]
    (-7.5, -0.5)--(-5,-0.5)--(-3.5,-1)--(-2,-0.5)--(-0.5,-1.5)--(1,-1)--(3.5,-1.5)--(5,-0.5)--(7.5, -0.5)--(7.5, -3.5)--(-7.5, -3.5)--(-7.5, -2);
    \draw [thin, black](-7.5, -0.5) -- (-5, -0.5);
    \draw [thick, red] (-5,-0.5)--(-3.5,-1);
    \draw [thick, orange] (-3.5,-1)--(-2,-0.5);
    \draw [thick, red] (-2,-0.5)--(-0.5,-1.5);
    \draw [thick, orange] (-0.5,-1.5)--(1,-1);
    \draw [thick, red] (1,-1)--(3.5,-1.5);
    \draw [thick, orange] (3.5,-1.5)--(5,-0.5);
    \draw [thin, black](5, -0.5) -- (7.5, -0.5);
    \fill (-5,-0.5) circle (1.5pt);
    \fill (-3.5,-1) circle (1.5pt);
    \fill (-2,-0.5) circle (1.5pt);
    \fill (-0.5,-1.5) circle (1.5pt);
    \fill (1,-1) circle (1.5pt);
    \fill (3.5,-1.5) circle (1.5pt);
    \fill (5,-0.5) circle (1.5pt);
    
    \draw [dashed, black] (-5,0)--(-7.5,2.75);
    \draw [dashed, black] (-5,0)--(-7.5,3);
    \draw [dashed, black] (-5,0)--(-7.25,3);
    
    \draw [dashed, black] (5,-0.5)--(7.5,-3.25);
    \draw [dashed, black] (5,-0.5)--(7.5,-3.5);
    \draw [dashed, black] (5,-0.5)--(7.25,-3.5);
    
    \node at (7, 2.5) {$H_1$};
    \node at (-7, -3) {$H_2$};
    
    \node at (-4.5, 0.75) {$a_1^-$};
    \node at (4.5, -1.25) {$a_1^+$};
    \node at (-2.25, 1) {$b_1^-$};
    \node at (2.25, -1.5) {$b_1^+$};
    \node at (-0.25, 1) {$a_2^-$};
    \node at (0.25, -1.5) {$a_2^+$};
    
    \node at (1.25, 0.825) {$b_2^-$};
    \node at (-1.25, -1.325) {$b_2^+$};
    \node at (2.75, 0.625) {$a_3^-$};
    \node at (-2.75, -1.125) {$a_3^+$};
    \node at (4.25, 0.625) {$b_3^-$};
    \node at (-4.25, -1.125) {$b_3^+$};
    
    \node at (-6.25, 0.35) {$r_1^-$};
    \node at (6.25, 0.35) {$r_2^-$};
    \node at (6.25, -0.85) {$r_2^+$};
    \node at (-6.25, -0.85) {$r_1^+$};
    
\end{tikzpicture}
\caption{Realization of a hyperelliptic genus $g$ meromorphic differential with a single zero of order $2m=2g+2p-2$, a single pole of order $2p$, and prescribed periods. The figure depicts the case $g=3$.}
\label{fig:hyponepole}
\end{figure}

\noindent Let us now define a second chain, say $\mathcal{C}_2$ as follows: 
\begin{equation}\label{eq:chainpolzerr}
    P_0\mapsto P_0+b_g=P_1'\mapsto P_0+b_g+a_g=P_2'\mapsto P_0+b_g+a_g+b_{g-1}=P_3'\mapsto\cdots\mapsto P_0+\sum_{i=1}^g (a_i+b_i)=P_{2g}.
\end{equation}
Notice that it differs from $\mathcal{C}_1$ by a rotation of order two about the midpoint of the segment $\overline{P_0\,P_{2g}}$. Moreover, the half-rays $r_1$ and $r_2$ are swapped by this rotation. Define $H_2$ as the (broken) half-plane bounded by the rays $r_1$ and $r_2$ and the chain $\mathcal{C}_2$ on their right, \textit{i.e.} $H_2$ is bounded by the rays $r_1^+$ and $r_2^+$ and $\mathcal{C}_2^+$. 

\smallskip

\noindent Next, we identify the rays $r_1^+$ and $r_1^-$ together. In the same fashion, we also identify the rays $r_2^+$ and $r_2^-$ together. The resulting is a $4g$-gon bounding a topological immersed disk on the Riemann sphere $\cp$ punctured at the infinity. We finally identify the edges with the same label, \textit{i.e.} we identify $a_i^+$ with $a_i^-$ and $b_i^+$ with $b_i^-$ for all $i=1,\dots,g$. The resulting space is a topological genus $g$ surface with one puncture equipped with a translation structure with one branch point of order $2g$ and one pole of order $2$ and period character $\chi$. By construction, such a structure admits a non-trivial automorphism of order two, namely an involution induced by the rotation of order two that exchanges $C_1$ with $C_2$. Therefore the resulting structure is hyperelliptic as desired, and lies in the stratum $\mathcal{H}_g(2g;-2)$. 

\smallskip 

\noindent In order to get a structure in the stratum $\mathcal{H}_g(2m;-2p)$ we need to modify the broken half-planes as follows. Consider on $H_1$ a ray leaving from $P_0$ with angle $0<\theta<\frac{\pi}{2}$ with respect to $r_1^-$ and bubble $p-1$ copies of $(\C,\,dz)$, see Definition \ref{bubplane}. Similarly, in $H_2$ consider a ray leaving from $P_{2g}$ with angle $\theta$ with respect to $r_2^+$ and then bubble $p-1$ copies of $(\C,\,dz)$. Then proceed as above. The resulting space is again a genus $g$ surface equipped with a translation structure with period character $\chi$ in the stratum $\mathcal{H}_g(2m;-2p)$ by construction, where $m=g+p-1$.

\smallskip

\noindent By breaking the zero of the structure obtained as above as two zeros both of order $m$, Lemma \ref{lem:breakhyp} implies that the resulting translation surface lies in the hyperelliptic component of $\mathcal{H}_g(m,m;-2p)$ and it has period character $\chi$ as desired. 

\subsubsection{Two higher order poles}\label{sssec:hyptwohop} Let $\chi\colon\textnormal{H}_1(S_{g,\,2},\,\Z)\longrightarrow \C$ be a representation of trivial-ends type. We aim to realize $\chi$ as the period character of some hyperelliptic translation surface in the stratum $\mathcal{H}_g(2m;-p,-p)$, where $p>1$. Again, according to Section \S\ref{ssec:trirephyp}, we can assume $\chi$ to be non-trivial because the trivial representation cannot be realized in any strata of the form $\mathcal{H}_g(2m;-p,-p)$ and $\mathcal{H}_g(m,m;-p,-p)$.

\smallskip

\noindent This case follows from a modification of the construction developed in Section \S\ref{sssec:hyponehop}. We shall adopt the same notation as above. Given a non-trivial representation $\chi$ as above, we define
\begin{equation}
    \chi(\alpha_i)=a_i \,\,\text{ and }\,\, \chi(\beta_i)=b_i.
\end{equation}
Once again, since $\chi$ is non-trivial, Lemma \ref{lem:allhandholnonzero} applies, and hence we can assume that both $a_i$ and $b_i$ are non-zero and $\arg(a_i),\arg(b_i)\in \left[-\frac\pi2,\,\frac\pi2\,\right[$ for all $i=1,\dots,g$. Given any point $P_0\in\C$, we define the broken half-plane $H_1$ exactly as in Section \S\ref{sssec:hyponehop}, \textit{i.e.} a half-plane bounded by two half-rays $r_1$ and $r_2$ and a chain of segments, say $\mathcal{C}_1$, defined as in \eqref{eq:chainpolzer}, joining $P_0$ with $P_{2g}=P_0+\sum (a_i+b_i)$. Notice that $r_1\cup\mathcal{C}_1\cup r_2$ bounds $H_1$ on its left.

\smallskip

\noindent In the same fashion, we define the broken half-plane $H_2$ exactly as in Section \S\ref{sssec:hyponehop}, \textit{i.e.} a half-plane bounded by two half-rays $r_1$ and $r_2$ and a chain of segments, say $\mathcal{C}_2$, defined as in \eqref{eq:chainpolzerr}, joining the points $P_0$ and $P_{2g}$. Notice that, in this case, $r_1\cup\mathcal{C}_2\cup r_2$ bounds $H_2$ on its right.

\smallskip

\noindent Next we single out two broken half-planes, say $H_3$ and $H_4$, defined as follows. Consider first the unique segment, say $e$, joining $P_0$ and $P_{2g}$. Recall that $\arg(a_i),\arg(b_i)\in \left[-\frac\pi2,\,\frac\pi2\,\right[$, so $\Re(P_{2g})\ge\Re(P_0)$ where $\Re(\cdot)$ denotes the real part. We extend $e$ with the half-ray $s_1$ leaving from $P_0$ and parallel to $\mathbb R^-$ and the half-ray $s_2$ leaving from $P_{2g}$ and parallel to $\mathbb R^+$. The chain $s_1\cup e\cup s_2$ is embedded in $\C$; in fact it bounds an embedded triangle with one vertex at the infinity and splits the complex plane in two half-planes. Let $H_3$ be the broken half-plane on the right of $s_1\cup e\cup s_2$ and let $H_4$ be the broken half-plane on the left. By design the following identification hold $s_1=r_1$ and $s_2=r_2$.

\smallskip

\noindent We finally glue these broken half-planes as follows, see Figure \ref{fig:hyptwopole}. The edges $e^+\subset H_3$ and $e^-\subset H_4$ are identified. Then, for $j=1,2$, we identify $r_j^-\subset H_1$ with $s_j^+\subset H_3$. Similarly, for $j=1,2$, we identify the rays $r_j^+\subset H_2$ and $s_j^-\subset H_4$. Finally, we identify $a_i^+$ with $a_i^-$ and $b_i^+$ with $b_i^-$ for all $i=1,\dots,g$. The resulting object is a translation surface with a single zero of order $2g+2$ and two poles both of order $2$ and zero residue. The resulting structure admits a hyperelliptic involution in the sense of Definition \ref{hypdef} that swaps $H_1\cup H_3$ with $H_2\cup H_4$.

\smallskip

\noindent In order to get a structure in the stratum $\mathcal{H}_g(2m;-p,-p)$ we can modify the broken half-planes as in Section \S\ref{sssec:hyponehop}. On $H_1$ we consider a ray leaving from $P_0$ with angle $0<\theta<\frac{\pi}{2}$ with respect to $r_1^-$ and bubble along it $p-1$ copies of $(\C,\,dz)$. Similarly, on $H_2$ we consider a ray leaving from $P_{2g}$ with angle $\theta$ with respect to $r_2^+$ and bubble $p-1$ copies of $(\C,\,dz)$. Then proceed as above. The resulting object is a genus $g$ surface equipped with a translation structure with period character $\chi$ in the stratum $\mathcal{H}_g(2m;-p,-p)$ by construction, where $m=g+p-1$.

\smallskip

\begin{figure}[!ht] 
\centering
\begin{tikzpicture}[scale=1, every node/.style={scale=0.85}]
\definecolor{pallido}{RGB}{221,227,227}

    \pattern [pattern=north east lines, pattern color=pallido]
    (-7.5, 0)--(-5,0)--(-3.5,1)--(-1,0.5)--(0.5,1)--(2,0)--(3.5,0.5)--(5,0)--(7.5,0)--(7.5,3)--(-7.5, 3)--(-7.5, 0);
    \draw [thin, black](-7.5, 0) -- (-5,0);
    \draw [thick, orange] (-5,0)--(-3.5,1);
    \draw [thick, red] (-3.5,1)--(-1,0.5);
    \draw [thick, orange] (-1,0.5)--(0.5,1);
    \draw [thick, red] (0.5,1)--(2,0);
    \draw [thick, orange] (2,0)--(3.5,0.5);
    \draw [thick, red] (3.5,0.5)--(5,0);
    \draw [thin, black](5, 0) -- (7.5,0);
    \fill (-5,0) circle (1.5pt);
    \fill (-3.5,1) circle (1.5pt);
    \fill (-1,0.5) circle (1.5pt);
    \fill (0.5,1) circle (1.5pt);
    \fill (2,0) circle (1.5pt);
    \fill (3.5,0.5) circle (1.5pt);
    \fill (5,0) circle (1.5pt);

    \pattern [pattern=north east lines, pattern color=pallido]
    (-7.5, -0.5)--(-5,-0.5)--(5, -0.5)--(7.5, -0.5)--(7.5, -2)--(-7.5, -2)--(-7.5, -0.5);
    \draw [thin, black](-7.5, -0.5) -- (-5, -0.5);
    \draw [thin, blue] (-5,-0.5)--(5,-0.5);
    \draw [thin, black](5, -0.5) -- (7.5, -0.5);
    \fill (-5,-0.5) circle (1.5pt);
    \fill (5, -0.5) circle (1.5pt);
    
    \pattern [pattern=north east lines, pattern color=pallido]
    (-7.5, -4)--(-5, -4)--(5, -4)--(7.5, -4)--(7.5, -2.5)--(-7.5, -2.5)--(-7.5, -4);
    \draw [thin, black] (-7.5, -4) -- (-5, -4);
    \draw [thin, blue] (-5, -4)--(5, -4);
    \draw [thin, black] (5, -4) -- (7.5, -4);

    \fill (-5,-4) circle (1.5pt);
    \fill (5, -4) circle (1.5pt);

    \pattern [pattern=north east lines, pattern color=pallido]
    (-7.5, -4.5)--(-5,-4.5)--(-3.5,-5)--(-2,-4.5)--(-0.5,-5.5)--(1,-5)--(3.5,-5.5)--(5,-4.5)--(7.5, -4.5)--(7.5, -7.5)--(-7.5, -7.5)--(-7.5, -6);
    \draw [thin, black](-7.5, -4.5) -- (-5, -4.5);
    \draw [thick, red] (-5,-4.5)--(-3.5,-5);
    \draw [thick, orange] (-3.5,-5)--(-2,-4.5);
    \draw [thick, red] (-2,-4.5)--(-0.5,-5.5);
    \draw [thick, orange] (-0.5,-5.5)--(1,-5);
    \draw [thick, red] (1,-5)--(3.5,-5.5);
    \draw [thick, orange] (3.5,-5.5)--(5,-4.5);
    \draw [thin, black](5, -4.5) -- (7.5, -4.5);
    \fill (-5,-4.5) circle (1.5pt);
    \fill (-3.5,-5) circle (1.5pt);
    \fill (-2,-4.5) circle (1.5pt);
    \fill (-0.5,-5.5) circle (1.5pt);
    \fill (1,-5) circle (1.5pt);
    \fill (3.5,-5.5) circle (1.5pt);
    \fill (5,-4.5) circle (1.5pt);
    
    \draw [dashed, black] (-5,0)--(-7.5,2.75);
    \draw [dashed, black] (-5,0)--(-7.5,3);
    \draw [dashed, black] (-5,0)--(-7.25,3);
    
    \draw [dashed, black] (5,-4.5)--(7.5, -7.25);
    \draw [dashed, black] (5,-4.5)--(7.5,  -7.5);
    \draw [dashed, black] (5,-4.5)--(7.25, -7.5);
    
    \node at (7, 2.5) {$H_1$};
    \node at (-7, -1.5) {$H_3$};
    \node at (7, -3) {$H_4$};
    \node at (-7, -7) {$H_2$};
    
    \node at (-4.5, 0.75) {$a_1^-$};
    \node at (4.5, -5.25) {$a_1^+$};
    \node at (-2.25, 1) {$b_1^-$};
    \node at (2.25, -5.5) {$b_1^+$};
    \node at (-0.25, 1) {$a_2^-$};
    \node at (0.25, -5.5) {$a_2^+$};
    
    \node at (1.25, 0.825) {$b_2^-$};
    \node at (-1.25, -5.325) {$b_2^+$};
    \node at (2.75, 0.625) {$a_3^-$};
    \node at (-2.75, -5.125) {$a_3^+$};
    \node at (4.25, 0.625) {$b_3^-$};
    \node at (-4.25, -5.125) {$b_3^+$};
    
    \node at (-6.25, 0.35) {$r_1^-$};
    \node at (6.25, 0.35) {$r_2^-$};
    \node at (6.25, -0.75) {$s_2^+$};
    \node at (-6.25, -0.75) {$s_1^+$};
    
    \node at (0, -0.75) {$e^+$};
    \node at (0, -3.65) {$e^-$};
    
    \node at (-6.25, -3.65) {$s_1^-$};
    \node at (6.25, -3.65) {$s_2^-$};
    \node at (6.25, -4.75) {$r_2^+$};
    \node at (-6.25, -4.75) {$r_1^+$};
    
\end{tikzpicture}
\caption{Realization of a hyperelliptic genus $g$ meromorphic differential with a single zero of order $2m=2g+2p-2$, two poles each of order $p\ge 2$, and prescribed periods. The figure depicts the case $g=3$.}
\label{fig:hyptwopole}
\end{figure}

\noindent Again, by breaking the zero of the structure obtained as above as two zeros both of order $m$, Lemma \ref{lem:breakhyp} implies that the resulting translation surface lies in the hyperelliptic component of $\mathcal{H}_g(m,m;-p,-p)$ and it has period character $\chi$ as desired. 

\smallskip

\subsection{Higher order poles with non-zero residue}\label{ssec:hypnotzeres} Suppose $p\ge2$ and let $\chi\colon\textnormal{H}_1(S_{g,\,2},\,\Z)\longrightarrow \C$ be a representation of \textit{non}-trivial-ends type. We aim to realize $\chi$ as the period character of some hyperelliptic translation surface in the stratum $\mathcal{H}_g(2m;-p,-p)$. In this case we realize a hyperelliptic translation surface with poles by extending the construction developed in \S\ref{sssec:hyptwohop}. We have already seen above that there always exists a system of handle generators such that $a_i,b_i\in\C^*$ and $\arg(a_i),\,\arg(b_i)\in\left[-\frac\pi2,\,\frac\pi2\,\right[$. Let $\gamma$ be a simple loop around a puncture, say $P$, since $\chi$ is of non-trivial-ends type we have that $\chi(\gamma)=w\in\C^*$. Notice that we can assume $\arg(w)\in \left[-\frac\pi2,\,\frac\pi2\,\right[$ if necessary. In fact, if $Q$ is the second puncture and $\delta$ is a simple loop around it then $\chi(\delta)=-w$. We shall distinguish two cases depending on whether $w$ is parallel to $v=\sum_i\,(a_i+b_i)$. Notice that, according to our choices, $\arg(v)\in\left[-\frac\pi2,\,\frac\pi2\,\right[$. We begin with the general case handled in the following subsection.

\subsubsection{Generic case: $v$ and $w$ are not parallel} Given a representation $\chi$ of non-trivial-ends type, we can define four, possibly broken, half-planes $H_1,H_2,H_3,H_4$ as done in Section \S\ref{sssec:hyponehop}. Since we need to change a little the notation we briefly summarize it as follows, see Figure \ref{fig:hyptwopolenonzerores}. Let $P_0\in\C$ be any point and set as above the point $P_{2g}=P_0+v$. Let $\theta=\arg(v)$ and let $R_\theta$ be the rotation of angle $\theta$ about the origin on $\C$. In this subsection we shall agree that all rays leaving from $P_0$ are parallel to $R_\theta\cdot\mathbb R^-$ and, similarly, all rays leaving from $P_{2g}$ are parallel to $R_\theta\cdot\mathbb R^+$. Alternatively, we can replace the representation $\chi$ with $A\cdot\chi$ so that $\arg(v)=0$, where $A\in\glplus$. The regions are:

\begin{itemize}
    \item[1.] $H_1$ is a broken half-plane bounded by chain of edges defined as in \eqref{eq:chainpolzer} joining $P_0$ and $P_{2g}$ and two rays, say $r_1^-$ and $r_2^-$, respectively leaving from $P_0$ and $P_{2g}$;
    \smallskip
    \item[2.] $H_2$ is a half-plane bounded by a geodesic edge $s_0^+$ joining $P_0$ and $P_{2g}$ and two rays, say $r_3^+$ and $r_2^+$, respectively leaving from $P_0$ and $P_{2g}$;
    \smallskip
    \item[3.] $H_3$ is a half-plane bounded by a geodesic edge $r_0^-$ joining $P_0$ and $P_{2g}$ and two rays, say $r_5^-$ and $r_4^-$, respectively leaving from $P_0$ and $P_{2g}$; finally
    \smallskip
    \item[4.] $H_4$ is a broken half-plane bounded by a chain of edges as in \eqref{eq:chainpolzerr} joining $P_0$ and $P_{2g}$ and two rays, say $r_5^+$ and $r_6^+$, respectively leaving from $P_0$ and $P_{2g}$.
\end{itemize}

\begin{figure}[!ht] 
\centering
\begin{tikzpicture}[scale=1, every node/.style={scale=0.85}]
\definecolor{pallido}{RGB}{221,227,227}

    \pattern [pattern=north east lines, pattern color=pallido]
    (-7.5, 0)--(-5,0)--(-3.5,1)--(-1,0.5)--(0.5,1)--(2,0)--(3.5,0.5)--(5,0)--(7.5,0)--(7.5,3)--(-7.5, 3)--(-7.5, 0);
    \draw [thin, black](-7.5, 0) -- (-5,0);
    
    \draw [thick, orange] (-5,0)--(-3.5,1);
    \draw [thick, red] (-3.5,1)--(-1,0.5);
    
    \draw [thick, orange] (-1,0.5)--(0.5,1);
    \draw [thick, red] (0.5,1)--(2,0);
    
    \draw [thick, orange] (2,0)--(3.5,0.5);
    \draw [thick, red] (3.5,0.5)--(5,0);
    
    \draw [thin, black](5, 0) -- (7.5,0);
    \fill (-5,0) circle (1.5pt);
    \fill (-3.5,1) circle (1.5pt);
    \fill (-1,0.5) circle (1.5pt);
    \fill (0.5,1) circle (1.5pt);
    \fill (2,0) circle (1.5pt);
    \fill (3.5,0.5) circle (1.5pt);
    \fill (5,0) circle (1.5pt);

    \pattern [pattern=north east lines, pattern color=pallido]
    (-7.5, -0.5)--(-5,-0.5)--(5, -0.5)--(7.5, -0.5)--(7.5, -2)--(-7.5, -2)--(-7.5, -0.5);
    \draw [thin, black](-7.5, -0.5) -- (-5, -0.5);
    \draw [thin, blue] (-5,-0.5)--(5,-0.5);
    \draw [thin, black](5, -0.5) -- (7.5, -0.5);
    \fill (-5,-0.5) circle (1.5pt);
    \fill (5, -0.5) circle (1.5pt);
    
    \pattern [pattern=north east lines, pattern color=pallido]
    (-7.5,-2.5)--(7.5,-2.5)--(7.5,-4.5)--(-7.5,-4.5)--(-7.5,-2.5);
    \draw [thin, black] (-7.5, -2.5) -- (-5, -2.5);
    \draw [thin, purple] (-5, -2.5)--(5, -2.5);
    \draw [thin, black] (5, -2.5) -- (7.5, -2.5);
    \draw [thin, black] (-7.5, -4.5) -- (-5, -4.5);
    \draw [thin, blue] (-5, -4.5)--(5, -4.5);
    \draw [thin, black] (5, -4.5) -- (7.5, -4.5);
    \draw [thin, black, ->] (4,-4.375)--(4,-2.625);
    \fill (-5,-2.5) circle (1.5pt);
    \fill (5, -2.5) circle (1.5pt);
    \fill (-5,-4.5) circle (1.5pt);
    \fill (5, -4.5) circle (1.5pt);

    \pattern [pattern=north east lines, pattern color=pallido]
    (-7.5, -6.5)--(-5, -6.5)--(5, -6.5)--(7.5, -6.5)--(7.5, -5)--(-7.5, -5)--(-7.5, -6.5);
    \draw [thin, black] (-7.5, -6.5) -- (-5, -6.5);
    \draw [thin, purple] (-5, -6.5)--(5, -6.5);
    \draw [thin, black] (5, -6.5) -- (7.5, -6.5);
    \fill (-5,-6.5) circle (1.5pt);
    \fill (5, -6.5) circle (1.5pt);

    \pattern [pattern=north east lines, pattern color=pallido]
    (-7.5, -7)--(-5,-7)--(-3.5,-7.5)--(-2,-7)--(-0.5,-8)--(1,-7.5)--(3.5,-8)--(5,-7)--(7.5, -7)--(7.5, -10)--(-7.5, -10)--(-7.5, -8.5);
    \draw [thin, black](-7.5, -7) -- (-5, -7);
    \draw [thin, red] (-5,-7)--(-3.5,-7.5);
    \draw [thin, orange] (-3.5,-7.5)--(-2,-7);
    \draw [thin, red] (-2,-7)--(-0.5,-8);
    \draw [thin, orange] (-0.5,-8)--(1,-7.5);
    \draw [thin, red] (1,-7.5)--(3.5,-8);
    \draw [thin, orange] (3.5,-8)--(5,-7);
    \draw [thin, black](5, -7) -- (7.5, -7);
    \fill (-5,-7) circle (1.5pt);
    \fill (-3.5,-7.5) circle (1.5pt);
    \fill (-2,-7) circle (1.5pt);
    \fill (-0.5,-8) circle (1.5pt);
    \fill (1,-7.5) circle (1.5pt);
    \fill (3.5,-8) circle (1.5pt);
    \fill (5,-7) circle (1.5pt);
    
    \draw [dashed, black] (-5,0)--(-7.5,2.75);
    \draw [dashed, black] (-5,0)--(-7.5,3);
    \draw [dashed, black] (-5,0)--(-7.25,3);
    
    \draw [dashed, black] (5,-7)--(7.5, -9.75);
    \draw [dashed, black] (5,-7)--(7.5,  -10);
    \draw [dashed, black] (5,-7)--(7.25, -10);
    
    \node at (7, 2.5) {$H_1$};
    \node at (-7, -1.5) {$H_2$};
    \node at (-3.25, -3.5) {$C$};
    \node at (3.75, -3.5) {$w$};
    \node at (7, -5.5) {$H_3$};
    \node at (-7, -9.5) {$H_4$};
    
    \node at (-4.5, 0.75) {$a_1^-$};
    \node at (4.5, -7.75) {$a_1^+$};
    \node at (-2.25, 1) {$b_1^-$};
    \node at (2.25, -8) {$b_1^+$};
    \node at (-0.25, 1) {$a_2^-$};
    \node at (0.25, -8) {$a_2^+$};
    
    \node at (1.25, 0.825) {$b_2^-$};
    \node at (-1.25, -7.825) {$b_2^+$};
    \node at (2.75, 0.625) {$a_3^-$};
    \node at (-2.75, -7.625) {$a_3^+$};
    \node at (4.25, 0.625) {$b_3^-$};
    \node at (-4.25, -7.625) {$b_3^+$};
    
    \node at (-6.25, 0.35) {$r_1^-$};
    \node at (6.25, 0.35) {$r_2^-$};
    \node at (6.25, -0.75) {$r_2^+$};
    \node at (-6.25, -2.75) {$r_1^+$};
    
    \node at (0, -2.75) {$r_0^+$};
    \node at (0, -6.15) {$r_0^-$};
    \node at (0, -0.75) {$s_0^+$};
    \node at (0, -4.25) {$s_0^-$};
    
    \node at (-6.25, -4.25) {$r_3^-$};
    \node at (6.25, -6.15) {$r_4^-$};
    \node at (6.25, -2.75) {$r_4^+$};
    \node at (-6.25, -0.75) {$r_3^+$};
    
    \node at (-6.25, -6.15) {$r_5^-$};
    \node at (-6.25, -7.25) {$r_5^+$};
    \node at (6.25, -4.25) {$r_6^-$};
    \node at (6.25, -7.25) {$r_6^+$};
    
\end{tikzpicture}
\caption{Realization of a hyperelliptic genus $g$ meromorphic differential with two zeros each of order $m=g+p-1$, two poles each of order $p\ge 2$ with non-zero residue, and prescribed periods. The figure depicts the case $g=3$.}
\label{fig:hyptwopolenonzerores}
\end{figure}
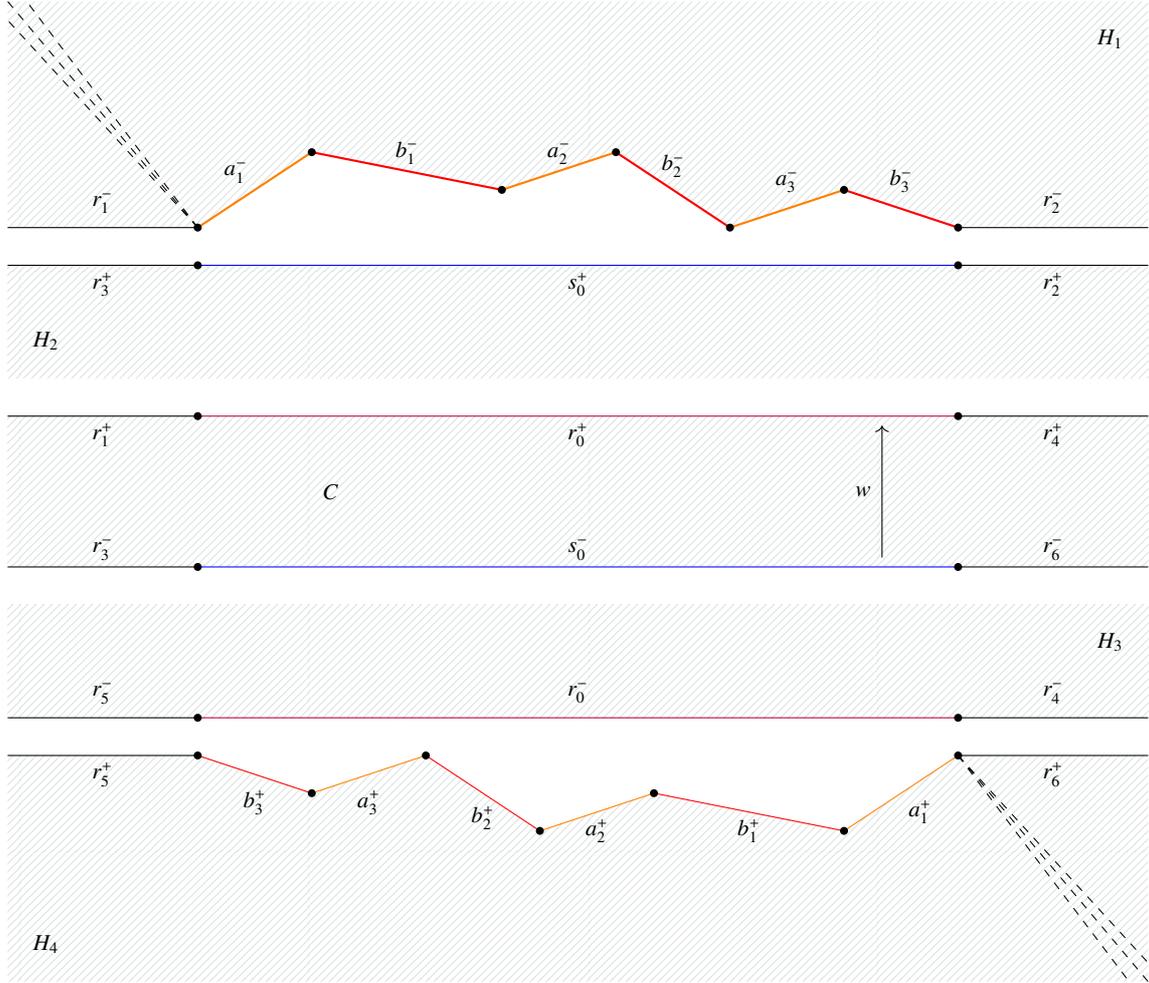

\noindent We next realize a cylinder $C$ with period $w$ as follows. We join the points $P_0$ and $P_{2g}$ by an edge, say $s_0^-$. Extend $s_0^-$ with two rays, say $r_3^-$ and $r_6^-$ leaving from $P_0$ and $P_{2g}$ respectively. Notice that $s_0^-$ has slope $\theta$ and therefore $r_3^-\,\cup s_0^-\,\cup r_6^-$ is a straight line, say $l_1$ by design. Define $r_1^+=r_3^-\,+\,w$ and $r_4^+=r_6^-\,+\,w$. Since $w$ is not parallel to $v$, the chain $r_1^+\,\cup\,r_0^+\,\cup\,r_4^+$ is a parallel straight line, say $l_2$. We define $C$ as the strip bounded by the lines $l_1$ and $l_2$. By construction, the cylinder obtained by identifying them has period $w$. 

\smallskip

\noindent All the necessary pieces are now introduced and defined. We glue all rays with the same label together, \textit{i.e.} we glue $r_i^+$ with $r_i^-$ for $i=1,\dots,6$. Then glue the edges $r_0^+$ and $r_0^-$ as well as the edges $s_0^+$ and $s_0^-$ together. Finally, glue the edges $a_i^+$ and $b_i^+$ with $a_i^-$ and $b_i^-$ respectively. The resulting object is a translation surface with two poles of order $2$ and a single zero of order $2g+2$. By construction, this translation surface admits a hyperelliptic involution that fixes the cylinder $C$ and maps the half-planes $H_1$, $H_2$ to $H_4$, $H_3$ respectively. The structure just defined lies in the hyperelliptic component of the stratum $\mathcal{H}_g(2g+2;-2,-2)$ and by breaking up the zero into two zeros each of order $m=g+1$ we get a structure in $\mathcal{H}_g(g+1,g+1;-2,-2)$. We can eventually bubble $p-1$ copies of $(\C,\,dz)$ along a ray leaving from $P_0\in H_1$ and bubble $p-1$ copies along a ray leaving from $P_{2g}\in H_4$ to get a structure in the stratum $\mathcal{H}_g(2m;-p,-p)$. Again, we apply Lemma \ref{lem:breakhyp} above in order to obtain a hyperelliptic translation structure in $\mathcal{H}_g(m,m;-p,-p)$ with period character $\chi$ as desired.

\subsubsection{Exceptional case: $v$ and $w$ are parallel} 
In the above construction, the cylinder $C$ has non-empty interior because $v$ is not parallel to $w$ and thence the lines $r_1^+\,\cup\,r_0^+\,\cup\,r_4^+$ and $r_3^-\,\cup s_0^-\,\cup r_6^-$ are parallel but disjoint. In the case $v$ is parallel to $w$, the cylinder $C$ degenerates to a straight line. However this case can be handled as follows. Let $\rho\colon\textnormal{H}_1(S_{g,\,2},\,\Z)\longrightarrow \C$ be an auxiliary representation, see Definition \ref{supprep}, defined as
\begin{equation}
    \rho(\alpha_i)=\chi(\alpha_i)\,\quad\,\rho(\beta_i)=\chi(\beta_i)\,\,\text{ and }\,\, \rho(\gamma_1)=\rho(\gamma_2)=0,
\end{equation} where $\gamma_1,\gamma_2$ are simple loops each around one of the punctures. We can realize $\rho$ as the period character of some translation surface, say $(X,\omega)$, with poles in the hyperelliptic component of $\mathcal{H}_g(2m;-p,-p)$. We claim that, on $(X,\omega)$ there always exists a bi-infinite ray, say $l_0$, joining the poles and orthogonal to the saddle connection $e$, arose from the identification of $e^+$ with $e^-$, on its mid-point. The slope of $e$, after developing, is equal to $\arg(v)=\arg(w)\in\left[-\frac\pi2,\,\frac\pi2\,\right[$. Let $l_1^+\subset \C$ be a straight line with slope $\arg(w)+\frac\pi2$ and let $l_1^-=l_1^+\,+\,w$. Define $C$ as the strip between the lines $l_1^+$ and $l_1^-$. Notice that the lines $l_0,l_1^+,l_1^-$ all have the same slope by design. Slit $(X,\omega)$ along $l_0$ and call $l_0^+$, $l_0^-$ the resulting edges. Then glue $l_0^-$ with $l_1^+$ and $l_0^+$ with $l_1^-$, see also Definition \ref{gluecyl}. The resulting structure, say $(Y,\,\xi)$, is a translation surface with poles in the stratum $\mathcal{H}_g(2m;-p,-p)$ having period character $\chi$. By construction, the straight line $l_0$ is invariant under the hyperelliptic involution of $(X,\omega)$. As a consequence, $(Y,\xi)$ also admits a hyperelliptic involution that keeps the strip $C$ (after gluing) invariant. 

\smallskip

\noindent Finally we apply again Lemma \ref{lem:breakhyp}.  By breaking the unique zero of $(Y,\xi)$ we can realize $\chi$ as the period character of some structure in the hyperelliptic component of $\mathcal{H}_g(m,m;-p,-p)$ as desired.

\subsection{Simple poles}\label{ssec:hypnotzeressimp} Throughout this section we always assume $p=1$ and we aim to realize a representation $\chi\colon\textnormal{H}_1(S_{g,\,2},\,\Z)\longrightarrow \C$ as the period character of some translation surface in the hyperelliptic component of the following strata
\begin{equation}
    \mathcal{H}_g(2g;-1,-1)\,\,\text{ and }\,\,  \mathcal{H}_g(g,g;-1,-1).
\end{equation}

\noindent We shall distinguish two cases depending on whether $\chi$ is discrete, see Definition \ref{defn:kindofreps}. For a given system of handle generators $\{\alpha_1,\beta_1,\dots,\alpha_g,\beta_g\}$, we still assume that $a_i,\,b_i\in\C^*$ are the images of $\alpha_i\,,\,\beta_i$ via $\chi$ respectively and that $\arg(a_i),\,\arg(b_i)\in\left[-\frac\pi2,\,\frac\pi2\,\right[$. We begin with the following case:

\subsubsection{The representation $\chi$ is not discrete}\label{sssec:hypnotzeressp} Recall that any system of handle generators yields a splitting and hence a well-defined representation $\chi_g\colon\shomolz\longrightarrow\C$. Since $\chi$ is a representation of non-trivial-ends type, Lemma \ref{cor:posvol} applies and we can assume $\textnormal{vol}(\chi_g)>0$. We distinguish two sub-cases depending on whether $\chi$ is or is not real-collinear, see Definition \ref{def:realcoll}.

\smallskip

\paragraph{$\chi$ \textit{is not real-collinear}}\label{par:hypnotcol} Assume here that $\chi$ is not a real-collinear representation. Let $\gamma$ be a simple loop around a puncture and let $w=\chi(\gamma)$. Define $v=\sum_i(a_i+b_i)$. In principle, $v$ and $w$ can be parallel. Since $\chi$ is not real-collinear, there exists a handle generator, say $\beta_g$, such that $b_g$ and $w$ are not parallel. We replace the handle generators $\{\alpha_g,\,\beta_g\}$ with $\{\alpha_g\,\beta_g,\,\beta_g\}$ and hence we replace $v$ with $v+b_g$. It is easy to check that this latter is no longer parallel to $w$. Hence we assume that $v$ and $w$ not parallel.

\smallskip 

\noindent Let $P_0\in\C$ be any point and let $P_{2g}=P_0+v$. Let $r$ be the unique straight line passing through these points. Up to rotating the whole construction with an appropriate rotation in $\sorr<\glplus$, let us assume for simplicity that $r$ is horizontal, \textit{i.e.} parallel to the real line; in this case $\arg(v)=0$ and $w\in\C\setminus \mathbb R$. Define a chain, say $\mathcal C_1$, as in \eqref{eq:chainpolzer} starting from $P_0$ and necessarily ending at $P_{2g}$. We now need the following 

\begin{lem}\label{lem:polygon}
With the same notation above, there exists a system of handle generators $\{\alpha_1,\beta_1,\dots,\alpha_g,\beta_g\}$ such that the chain $\mathcal{C}_1$ entirely lies on the right of the line $r$.
\end{lem}

\noindent Suppose the Lemma holds. Then we can assume that the chain $\mathcal{C}_1$ entirely lies on the right of $r$ with respect to the orientation induced by $v$. Let $r_1^-$ as the sub-ray of $r$ leaving from $P_0$ and not passing through $P_{2g}$. Similarly, define $r_2^-$ as the sub-ray of $r$ leaving from $P_{2g}$ and not passing through $P_{0}$. Let $P_0'=P_0+w$, set $P_{2g}'=P_0'+v=P_0+v+w$, and let $r'$ be the unique straight line passing through these points. Clearly, $r'$ is parallel to $r$. We define a chain $\mathcal{C}_2$ exactly as in \eqref{eq:chainpolzerr} starting from $P_0'$ and ending at $P_{2g}'$. By design, $\mathcal{C}_2$ lies on the left of $r'$ with respect to the orientation induced by $v$. Let $r_1^+$ be the sub-ray of $r'$ leaving from $P_0'$ and not passing through $P_{2g}'$ and, similarly, define $r_2^+$ as the sub-ray of $r$ leaving from $P_{2g}'$ and not passing through $P_{0}'$. Define $C\subset \C$ as the region bounded by $r_1^-\,\cup\mathcal{C}_1\,\cup r_2^-$ and $r_1^+\,\cup\mathcal{C}_2\,\cup r_2^+$, see Figure \ref{fig:hypsimplepolesnodis}. 

\begin{figure}[!ht] 
\centering
\begin{tikzpicture}[scale=1.05, every node/.style={scale=0.85}]
\definecolor{pallido}{RGB}{221,227,227}
    \pattern [pattern=north east lines, pattern color=pallido]
    (-7,3)--(-4,3)--(-2,2)--(0,2.5)--(3,2)--(4,3)--(7,3)--(7,5)--(4,5)--(2,6)--(0,5.5)--(-3,6)--(-4,5)--(-7,5)--(-7,3);
    \draw[thin, black] (-7,3)--(-4,3);
    \draw[thick, orange] (-4,3)--(-2,2);
    \draw[thick, red] (-2,2)--(0,2.5);
    \draw[thick, orange] (0,2.5)--(3,2);
    \draw[thin, gray, dashed] (-4,3)--(4,3);
    \draw[thick, red] (3,2)--(4,3);
    \draw[thick, black] (4,3)--(7,3);
    
    \draw[thin, black, ->] (6,3.2)--(6,4.8);
    \draw[thin, gray, ->] (-0.5,3.25)--(0.5, 3.25);
    
    \fill (-4,3) circle (1.5pt);
    \fill (-2,2) circle (1.5pt);
    \fill (0,2.5) circle (1.5pt);
    \fill (3,2) circle (1.5pt);
    \fill (4,3) circle (1.5pt);
    
    \node at (0, 3.5) {$v$};
    \node at (6.25, 4) {$w$};
    
    \draw[thin, black] (-7,5)--(-4,5);
    \draw[thick, orange] (2,6)--(4,5);
    \draw[thick, red] (0,5.5)--(2,6);
    \draw[thick, orange] (-3,6)--(0,5.5);
    \draw[thin, gray, dashed] (-4,5)--(4,5);
    \draw[thick, red] (-4,5)--(-3,6);
    \draw[thin, black] (4,5)--(7,5);
  
    \fill (4,5) circle (1.5pt);
    \fill (2,6) circle (1.5pt);
    \fill (0,5.5) circle (1.5pt);
    \fill (-3,6) circle (1.5pt);
    \fill (-4,5) circle (1.5pt);
    
    \node at (-5.5,2.75) {$r_1^-$};
    \node at (-3,2.25) {$a_1^-$};
    \node at (-1,2) {$b_1^-$};
    \node at (1.5,2) {$a_2^-$};
    \node at (3.75,2.25) {$b_2^-$};
    \node at (5.5,2.75) {$r_2^-$};
    
    \node at (-5.5,5.25) {$r_1^+$};
    \node at (3,5.75) {$a_1^+$};
    \node at (1,6) {$b_1^+$};
    \node at (-1.5,6) {$a_2^+$};
    \node at (-3.75,5.75) {$b_2^+$};
    \node at (5.5,5.25) {$r_2^+$};

\end{tikzpicture}
\caption{Realization of a hyperelliptic translation surface of genus two with a single zero and two simple poles having non-discrete and non real-collinear period character $\chi$.}
\label{fig:hypsimplepolesnodis}
\end{figure}
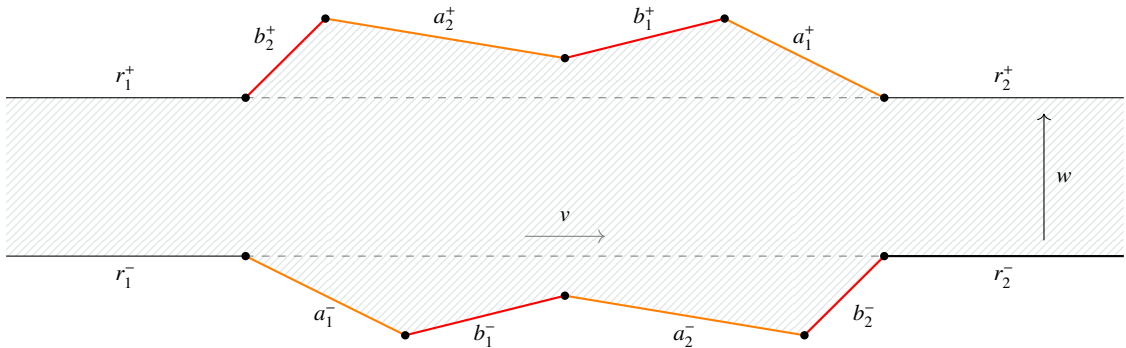

\noindent Finally, glue the half-rays $r_i^-$ and $r_i^+$ together and the edges $a_i^-$, $b_i^-$ with $a_i^+$, $b_i^+$ respectively, for $i=1,\dots,g$. The resulting object is a topological surface homeomorphic to $S_{g,2}$ equipped with a translation structure with poles having period character $\chi$. Notice that $C$ is invariant with respect to a rotation of order $2$ about a point $O\in\C$. Such a point can be explicitly determined; in fact it turns out to be the intersection point of the diagonals of the parallelogram with vertices $P_0,\, P_{2g},\, P_{2g}',\, P_0'$. As a consequence, the resulting structure above admits a hyperelliptic involution; in particular $\chi$ can be realized in the hyperelliptic component of $\mathcal{H}_g(2g;-1,-1)$. By breaking the zero into two zeros each of order $g$ we obtain a hyperelliptic structure in the stratum $\mathcal{H}_g(g,g;-1,-1)$ as a consequence of Lemma \ref{lem:breakhyp}. We are left to prove Lemma \ref{lem:polygon}.

\begin{proof}[Proof of Lemma \ref{lem:polygon}]
If $\mathcal{C}_1$ already lies on the right of $r$ there is nothing to prove and we are done. The geometric interpretation here is that the extended chain of edges $\cext\defeq\mathcal{C}_1\,\cup\,P_{2g}\mapsto P_{2g}-v=P_0$ is a closed piecewise geodesic curve in $\C$ that bounds a compact embedded polygon on $\C$ on its \textit{left}, according to the orientation. If $\mathcal{C}_1$ does not entirely lie on the right of $r$, there is no such a polygon bounded on the left of $\cext$  -- there might be, however, a polygon bounded on the right of $\cext$ meaning that the chain $\mathcal{C}_1$ entirely lies on the left of $r$. We will show that, by changing the last pair of handle generators, the chain $\cext$ bounds a compact embedded polygon on its left. First of all, we reorder the handles so that $\arg(b_i)\le\arg(b_g)$ for any $i=1,\dots,g-1$. Consider the handle generators $\{\alpha_g,\,\beta_g\}$. For any $n\in\Z^+$, the curves $\{\alpha_g\,\beta_g^n,\,\beta_g\}$ form another pair of handle generators and it is easy to check that the following conditions hold: 
\begin{equation}
    \chi(\alpha_g\,\beta_g^n)=a_g\,+\,nb_g \,\,\,\text{ and }\,\,\,\arg(a_g\,+\,nb_g)\in\left[-\frac\pi2,\,\frac\pi2\right[.
\end{equation}
From the second condition we can deduce that the new chain $\mathcal{C}_1(n)\subset\C$ is embedded, \textit{i.e.} no self-intersections, because all vectors $a_i,\,b_i$ point rightwards. Define $v_n=\sum_i(a_i\,+\,b_i)\,+\,n\,b_g$. By design the following equalities are straightforward: $v_0=v$ and $v_n=v+nb_g$. Observe that for $n\rightarrow+\infty$ then $\arg(v_n)\longrightarrow\arg(b_g)$. As a consequence, by taking $n$ big enough the inequality $\arg(b_i)\le\arg(v_n)$ holds for any $i=1,\dots,g-1$ and the chain 
\begin{equation}
    P_0\mapsto\cdots\mapsto P_0+\sum_{i=1}^g (a_i+b_i)+nb_g=P_{2g} \mapsto P_{2g}-v_n\mapsto P_0
\end{equation}
bounds a polygon on its left as desired.
\end{proof}

\paragraph{$\chi$ \textit{is real-collinear}}\label{par:hypcol} We now assume $\chi$ to be real-collinear. Up to replacing $\chi$ with $A\,\chi$, for some appropriate $A\in\glplus$, we can assume $\text{Im}(\chi)\subset \mathbb R$. Recall that we are under the initial assumption $\arg(a_i),\,\arg(b_i)\in\left[-\frac\pi2,\,\frac\pi2\,\right[$, hence $a_i,\,b_i>0$ in this special case. Let $\gamma$ be a simple loop around a puncture and let $w=\chi(\gamma)$. Let us assume $w>0$ for simplicity. Then Lemma \ref{lem:smallperiods} applies, and hence there exists a system of handle generators such that $v\defeq \,\sum_i(a_i\,+\,b_i)< w$ and we set $2\delta=w-v$. Let $l_1\subset\C$ be a straight line with slope $\frac\pi2$, let $l_2=l_1+w$, and define $\mathcal{S}$ as the vertical strip bounded by $l_1^+$ and $l_2^-$. Let $P_0\in\mathcal S\subset \C$ be any point at distance $\delta$ from $l_1$ and let $P_{2g}=P_0+v$. Notice that, $P_{2g}\in\mathcal{S}$ at distance $\delta$ from $l_2$. 

\begin{figure}[!ht] 
\centering
\begin{tikzpicture}[scale=1, every node/.style={scale=0.85}]
\definecolor{pallido}{RGB}{221,227,227}

    \pattern [pattern=north east lines, pattern color=pallido]
    (-7,3)--(-7,-2)--(-1,-2)--(-1,3);
    
    \draw[thin, black, ->] (-7,-2)--(-7,0.5);
    \draw[thin, black] (-7,0.5)--(-7,3);
    \draw[thin, black] (-7,-2)--(-6,-2);
    \draw[thick, orange] (-6,-2)--(-5,-2);
    \draw[thick, red] (-5,-2)--(-4,-2);
    \draw[thick, orange] (-4,-2)--(-3,-2);
    \draw[thick, red] (-3,-2)--(-2,-2);
    \draw[thin, black] (-2,-2)--(-1,-2);
    \draw[thin, black,->] (-1,-2)--(-1,0.5);
    \draw[thin, black] (-1,3)--(-1,0.5);
    
    \fill (-7,-2) circle (1.5pt);
    \fill (-6,-2) circle (1.5pt);
    \fill (-5,-2) circle (1.5pt);
    \fill (-4,-2) circle (1.5pt);
    \fill (-3,-2) circle (1.5pt);
    \fill (-2,-2) circle (1.5pt);
    \fill (-1,-2) circle (1.5pt);
    
    \pattern [pattern=north east lines, pattern color=pallido]
    (7,3)--(7,-2)--(1,-2)--(1,3);
    
    \draw[thin, black, ->] (7,-2)--(7,0.5);
    \draw[thin, black] (7,0.5)--(7,3);
    \draw[thin, black] (7,3)--(6,3);
    \draw[thick, orange] (6,3)--(5,3);
    \draw[thick, red] (5,3)--(4,3);
    \draw[thick, orange] (4,3)--(3,3);
    \draw[thick, red] (3,3)--(2,3);
    \draw[thin, black] (2,3)--(1,3);
    \draw[thin, black] (1,0.5)--(1,3);
    \draw[thin, black, ->] (1,-2)--(1,0.5);
    
    \fill (7,3) circle (1.5pt);
    \fill (6,3) circle (1.5pt);
    \fill (5,3) circle (1.5pt);
    \fill (4,3) circle (1.5pt);
    \fill (3,3) circle (1.5pt);
    \fill (2,3) circle (1.5pt);
    \fill (1,3) circle (1.5pt);
    
    \draw[thin, black, ->] (-6.75, 2)--(-1.25, 2); 
    \draw[thin, black, ->] (1.25, -1)--(6.75, -1); 
    
    \node at (-4, 2.25) {$w$};
    \node at (4, -1.25) {$w$};
    
    \node at (-4,0.5) {$\mathcal S_1$};
    \node at ( 4,0.5) {$\mathcal S_2$};
    
    \node at (-7.25, 0.5) {$l_1^+$};
    \node at (-0.75, 0.5) {$l_2^-$};
    \node at (0.75, 0.5) {$l_1^+$};
    \node at (7.25, 0.5) {$l_2^-$};
    
    \node at (-6.5,-2.25) {$r_1^-$};
    \node at (-5.5,-2.25) {$a_1^-$};
    \node at (-4.5,-2.25) {$b_1^-$};
    \node at (-3.5,-2.25) {$a_2^-$};
    \node at (-2.5,-2.25) {$b_2^-$};
    \node at (-1.5,-2.25) {$r_2^-$};
    
    \node at (1.5,3.25) {$r_1^+$};
    \node at (2.5,3.25) {$b_2^+$};
    \node at (3.5,3.25) {$a_2^+$};
    \node at (4.5,3.25) {$b_1^+$};
    \node at (5.5,3.25) {$a_1^+$};
    \node at (6.5,3.25) {$r_2^+$};
    
\end{tikzpicture}
\caption{The half-strips $\mathcal S_1$ and $\mathcal S_2$ represent the top and bottom sides of the strip $\mathcal{S}$ cut along a horizontal segment of length $w$ containing $e$. The edges $e_i^{-}$'s are in the ascending order from the left to the right and the $e_i^+$'s are in the descending order from the left to the right.}
\label{fig:hypsimplepolesnodisrealcoll}
\end{figure}
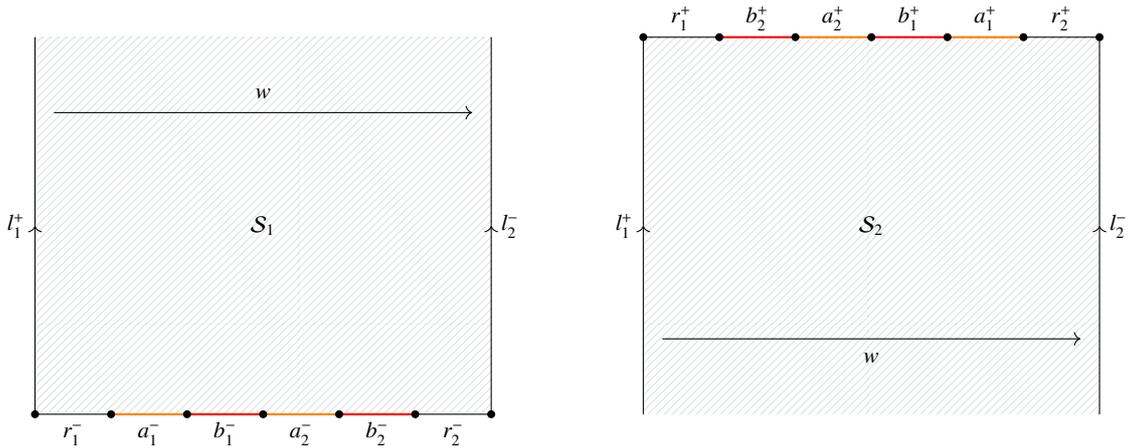

\noindent Let $e$ be the edge joining $P_0$ and $P_{2g}$. Slit $\mathcal S$ along $e$ and define $e^{\pm}$ the resulting edges according to our convention. We partition $e^-$ into $2g$ sub-edges $\{e_i^-\}_{1\le i\le 2g}$ of lengths $a_1,b_1,\dots,a_g,b_g$ respectively. In the same fashion, we partition the edge $e^+$ into $2g$ sub-edges $\{e_i^+\}_{1\le i\le 2g}$ of lengths $b_g,a_g,\dots,b_1,a_1$ respectively, see Figure \ref{fig:hypsimplepolesnodisrealcoll}. By gluing the lines $l_1^-$ and $l_2^+$ together, we get an infinite cylinder with period $w$ slit along the edge $e$. We next glue the edges $e_i^-$ with $e_{2g-i+1}^+$ together for any $i=1,\dots,2g$. The resulting object is homeomorphic to $S_{g,2}$ and it is equipped with a translation structure $(X,\omega)$ with period $\chi$. We can observe that $\mathcal{S}\in\C$ is invariant under a rotation of order $2$ about the mid-point of the segment $\overline{P_0\,P_{2g}}$. As a consequence, $(X,\omega)$ admits a hyperelliptic involution meaning that $\chi$ can be realized in the hyperelliptic component of $\mathcal{H}_g(2g;-1,-1)$. By breaking the zero into two zeros each of order $g$ we obtain a translation surface with period $\chi$ in the hyperelliptic component of the stratum $\mathcal{H}_g(g,g;-1,-1)$ as a consequence of  Lemma \ref{lem:breakhyp}.

\subsubsection{The representation $\chi$ is discrete of rank two}\label{sssec:disranktwo}
 
Now assume that $\chi$ is a discrete representation of rank two, see Definition \ref{defn:kindofreps}. The idea of this case is mostly subsumed in paragraph \S\ref{par:hypnotcol}. Up to $\glplus$, we assume without loss of generality that $\text{Im}(\chi)=\mathbb Z\oplus i\,\mathbb Z$. Lemma \ref{lem:disnormalform} above applies in this case and it guarantees the existence of a system of handle generators such that
\begin{itemize}
    \item $\chi(\alpha_g)\in\Z$ and $\chi(\beta_g)=i$,
    \smallskip
    \item $0<\chi(\alpha_j)<\chi(\alpha_g)$ and $\chi(\alpha_g)=\chi(\beta_g)$ for all $j=1,\dots,g-1$.
\end{itemize}

\noindent It is an easy matter to check that Lemma \ref{lem:polygon} holds in this case and hence we can proceed as done in \S\ref{par:hypnotcol} in order to get a translation surface with poles in the hyperelliptic component of the stratum $\mathcal{H}_g(2g;-1,-1)$ with period character $\chi$, see Figure \ref{fig:hypsimplepolesdisranktwo}. As above, by breaking the single zero into two zeros of order $g$, we can also realize $\chi$ as the period character of some hyperelliptic translation surface in $\mathcal{H}_g(g,g;-1,-1)$.

\begin{rmk}
By defining $v=\sum_i(a_i\,+\,b_i)\in\Z$ and $w\in\Z$ as above, in principle it can happen that $v$ and $w$ have the same argument. As done at the beginning of paragraph \S\ref{par:hypnotcol}, in this case it will be sufficient to replace the pair $\{\alpha_g,\,\beta_g\}$ with $\{\alpha_g\,\beta_g,\,\beta_g\}$ and proceed as above.
\end{rmk}

\begin{figure}[!ht] 
\centering
\begin{tikzpicture}[scale=1, every node/.style={scale=0.85}]
\definecolor{pallido}{RGB}{221,227,227}

    \pattern [pattern=north east lines, pattern color=pallido]
    (-7,2)--(-4,2)--(4,2)--(4,3)--(7,3)--(7,5)--(-4,5)--(-4,4)--(-7,4)--(-7,2);
    \draw[thin, black] (-7,2)--(-4,2);
    \draw[thick, orange] (-4,2)--(-2,2);
    \draw[thick, red] (-2,2)--(0,2);
    \draw[thick, orange] (0,2)--(4,2);
    \draw[thick, red] (4,2)--(4,3);
    \draw[thick, black] (4,3)--(7,3);
         
    \draw[thin, black] (7,5)--(4,5);
    \draw[thick, red] (-4,4)--(-4,5);
    \draw[thick, orange] (-4,5)--(0,5);
    \draw[thick, red] (0,5)--(2,5);
    \draw[thick, orange] (2,5)--(4,5);
    \draw[thin, black] (-4,4)--(-7,4);
    
    \draw[thin, black, ->] (6,3.2)--(6,4.8);
    \draw[thin, black, ->] (-6,2.2)--(-6,3.8);
    
    \fill (-4,2) circle (1.5pt);
    \fill (-2,2) circle (1.5pt);
    \fill (0,2) circle (1.5pt);
    \fill (4,2) circle (1.5pt);
    \fill (4,3) circle (1.5pt);
    
    \node at (6.25, 4) {$w$};
    \node at (-6.25, 3) {$w$};
    
    \fill (4,5) circle (1.5pt);
    \fill (2,5) circle (1.5pt);
    \fill (0,5) circle (1.5pt);
    \fill (-4,5) circle (1.5pt);
    \fill (-4,4) circle (1.5pt);
    
    \node at (-3,1.75) {$a_1^-$};
    \node at (-1,1.75) {$b_1^-$};
    \node at (2,1.75) {$a_2^-$};
    \node at (4.25,2.5) {$b_2^-$};
    
    \node at (3,5.25) {$a_1^+$};
    \node at (1,5.25) {$b_1^+$};
    \node at (-2,5.25) {$a_2^+$};
    \node at (-4.25,4.5) {$b_2^+$};
    
\end{tikzpicture}
\caption{Realization of a translation surface of genus two with poles admitting a hyperelliptic involution and having discrete period character of rank two.}
\label{fig:hypsimplepolesdisranktwo}
\end{figure}
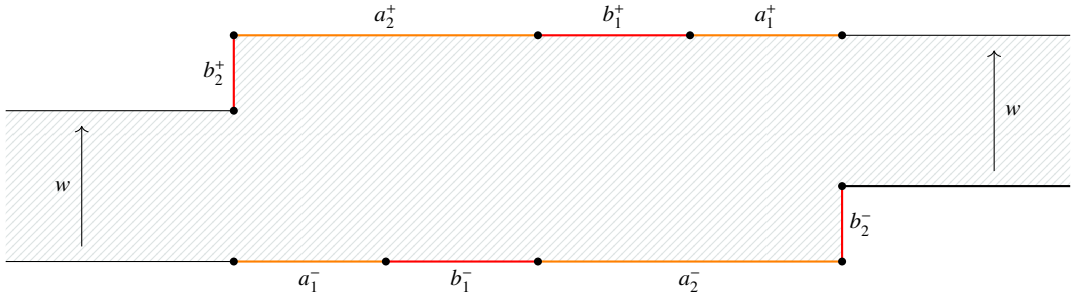

\subsubsection{The representation $\chi$ is discrete of rank one}\label{sssec:disrankone} We finally consider the case of discrete representation of rank one, see Definition \ref{defn:kindofreps}. For any such a representation  $\chi\colon\textnormal{H}_1(S_{g,\,2},\,\Z)\longrightarrow \C$ of rank one, we can replace $\chi$ with $A\,\chi$, for some appropriate $A\in\glplus$, and assume that  $\chi\colon\textnormal{H}_1(S_{g,\,2},\,\Z)\longrightarrow \Z$ is surjective. Let $\gamma$ be a simple loop around a puncture and let $w=\chi(\gamma)\in\Z$ be its period, \textit{i.e.} the residue. \cite[Theorem D]{CFG} provides necessary and sufficient conditions for such representations to appear as the holonomy of some translation surfaces with simple poles at the punctures. For the reader's convenience, we review here a simplified version of that result.

\begin{prop}[{\cite[Theorem D]{CFG}} for $n=2$]\label{prop:rankonechar}
Let $\chi\colon\textnormal{H}_1(S_{g,\,2},\,\Z) \longrightarrow \mathbb{Z}$ be a discrete representation of non-trivial-ends type. Then $\chi$ appears as the period character of some translation surface with simple poles at the punctures and zeros of prescribed orders $(m_1,m_2,\ldots, m_k)$ that satisfy the degree condition $m_1+\cdots+m_k = 2g$ if and only if 
    \begin{equation}\label{gencomp}
    |w|> \max\{m_1,m_2,\ldots, m_k\}
\end{equation}
where $w,\,-w\in\Z$ are the residues of the simple poles.
\end{prop}

\smallskip

\noindent According to Proposition \ref{prop:rankonechar} above, we need to distinguish two sub-cases according to the value of $w\in\Z$. In the case $w>2g$, then $\chi$ appears as the  period character of some translation surface with poles in both strata $\mathcal{H}_g(2g;-1,-1)$ and $\mathcal{H}_g(g,g;-1,-1)$. In paragraph \S\ref{par:easy} below, we shall prove that $\chi$ can be realized in the hyperelliptic component of these strata. In paragraph \S\ref{par:hard}, we shall consider the case $g<w\le 2g$, then $\chi$ can be realized only in the stratum $\mathcal{H}_g(g,g;-1,-1)$. Notice that this case requires an ad-hoc construction because we can no longer realize a structure with a single zero and then break it into two zeros of order $g$. Finally, if $0<w\le g$ then $\chi$ cannot be realized in both strata and hence it cannot appear as the period character of some hyperelliptic translation surface with simple poles.

\medskip

\paragraph{\textit{Case: $w>2g$}}\label{par:easy} This case is completely subsumed in paragraph \S\ref{par:hypcol}. In fact, Lemma \ref{lem:disnormalform} applies and hence there is a system of handle generators $\{\alpha_1,\beta_1,\dots,\alpha_g,\beta_g\}$ such that $\chi(\alpha_i)=\chi(\beta_i)=1$ for all $i=1,\dots,g$. Let $v=\sum_i(a_i\,+\,b_i)$ as usual. Since $w>2g$, then $w>v=2g$, and hence we can proceed exactly as in \S\ref{par:hypcol}. Therefore $\chi$ appears as the period character of some hyperelliptic translation surface with simple poles in $\mathcal{H}_g(2g;-1,-1)$. By breaking the zero, we get a translation surface with simple poles in the hyperelliptic component of $\mathcal{H}_g(g,g;-1,-1)$ as desired.

\medskip

\paragraph{\textit{Case: $g<w\le 2g$}}\label{par:hard} We finally deal with this last case. Let $\chi\colon\textnormal{H}_1(S_{g,\,2},\,\Z) \longrightarrow \mathbb{Z}$ be a representation and assume that, for a simple loop $\gamma$ around a puncture, we have that $g\,<\,w=\chi(\gamma)\le 2g$. Since $w\in\Z$, we can write it as $w=g+1+l$, where $l$ is a non-negative integer. We now define two half-infinite strips, say $\mathcal S_1,\,\mathcal S_2\subset \C$ of width $w$ and pointing in opposite directions. More precisely, let $P_0\in\C$ be any point and let $P_{2g+2}=P_0+w$. Let $s$ be the segment joining them and define $r_1,\,r_2=r_1+w$ as the straight lines orthogonal to $P_0$ and $P_{2g+2}$ respectively. Define $\mathcal S_1$ as the strip bounded by $r_1^+\,\cup\,s^-\,\cup\,r_2^-$ and, in a similar fashion, define $\mathcal S_2$ as the strip bounded by $r_1^+\,\cup\,s^+\,\cup\,r_2^-$. Consider these half-strips separately. We partition $s^-$ as the union of segments $s_1^-,\,\dots,\,s_{2g+2}^-$, from left to right, where each $s_i$ is of length $\frac12$, except that the last one is of width $\frac12 + l$, so the total length amounts to $\frac{2g+2}{2} + l = w$. Similarly, we partition $s^+$ as the union of segments $s_{2g+2}^+,\dots,s_{1}^+$, from the left to the right, where each $s_i$ is of length $\frac12$, except that the first one is of width $\frac12 + l$. Again, the total length amounts to $\frac{2g+2}{2} + l = w$. Glue the half-strips $\mathcal{S}_1,\,\mathcal{S}_2$ by identifying the edges with the same label and the rays $r_1^+$ with $r_2^-$ together. The resulting structure is a translation surface with two simple poles and two zeros each of order $g$. These two zeros appear in an alternating way as the endpoints of the $s_i$. Notice that any absolute period is an integer as it consists of an even number of the $s_i$. Finally, the structure is hyperelliptic because the half-strips are symmetric with respect to a rotation of order two by design.

\subsection{Non-hyperelliptic translation surfaces in genus two}\label{ssec:nothypgenus2} For certain strata of genus two meromorphic differentials we are now ready to complete the proof of Theorem \ref{mainthm}. In fact, by performing appropriate modifications to our construction developed in Sections \S\ref{ssec:hypzeres}

and \S\ref{ssec:hypnotzeressimp} we can prove the following propositions.

\begin{prop}\label{prop:genustwohypcaseone}
Let $\chi\colon\textnormal{H}_1(S_{2,\,1},\,\Z) \longrightarrow \C$ be a non-trivial representation.
Assume $\chi$ appears as the period character of some translation surface with poles in a  stratum $\mathcal{H}_2(4;-2)$ or $\mathcal{H}_2(2,2;-2)$, possibly both. Then $\chi$ can be realized in the non-hyperelliptic component of the same stratum. 
\end{prop}

\begin{prop}\label{prop:genustwohypcasetwo}
Let $\chi\colon\textnormal{H}_1(S_{2,\,2},\,\Z) \longrightarrow \C$ be a non-trivial representation.
Assume $\chi$ appears as the period character of some translation surface with poles in a  stratum $\mathcal{H}_2(4;-1,-1)$ or $\mathcal{H}_2(2,2;-1,-1)$, possibly both. Then $\chi$ can be realized in the non-hyperelliptic component of the same stratum. 
\end{prop}

\begin{proof}[Sketch of the proofs of Propositions \ref{prop:genustwohypcaseone} and \ref{prop:genustwohypcasetwo}.]

Let $\chi\colon\textnormal{H}_1(S_{2,\,1},\,\Z) \longrightarrow \C$ be a representation that appears as the period character of some translation surface in a stratum $\mathcal{H}_2(4;-2)$ or $\mathcal{H}_2(2,2;-2)$. In Section \S\ref{ssec:hypzeres} we have realized $\chi$ as the period character of some hyperelliptic translation surface by gluing broken half-planes designed in such a way that the resulting structure turned out to be hyperelliptic. The gist of the idea was to realize these broken half-planes so that they were invariant under a rotation of order $2$ about a certain point in $\C$. However, it is also possible to design broken half-planes so that they are no longer symmetric with respect to a rotation of order two. For instance, the broken half-plane $H_1$ in Section \S\ref{ssec:hypzeres} can be defined in the same way and then define $H_2$ as the broken half-plane bounded on its right by the chain \begin{equation}\label{eq:chainpolzernothyp}
    P_0\mapsto P_0+b_1=P_1\mapsto P_0+b_1+a_1=P_2\mapsto P_0+b_1+a_1+b_2=P_3\mapsto\cdots\mapsto P_0+\sum_{i=1}^g (a_i+b_i)=P_{2g}
\end{equation} along with the half-rays $r_1,\,r_2$ (we adopt the same notation). Notice that the chain above differs from \eqref{eq:chainpolzerr} because the edges $a_i,\,b_i$ are now sorted in a different way. Glue the half-planes as usual. The resulting translation surface still lies in $\mathcal{H}_2(4;-2)$ or $\mathcal{H}_2(2,2;-2)$ but it is no longer hyperelliptic. Recall that for these genus two strata under considerations, the connected components are distinguished by the hyperellipticity as well as by the spin parity, see Section \S\ref{mscc}. A direct computation shows that all hyperelliptic structures realized in \S\ref{ssec:hypzeres}
and \S\ref{ssec:hypnotzeressimp} have even spin parity whereas all structures realized by defining the broken half-plane $H_2$ as above have odd spin parity. In Section \S\ref{ssec:hypnotzeressimp} we have adopted the same strategy for representations $\chi\colon\textnormal{H}_1(S_{2,\,2},\,\Z) \longrightarrow \C$ and hence the same discussion holds for them, see Figure \ref{fig:simplepolesdisranktwo}.
\end{proof}

\begin{figure}[!ht] 
\centering
\begin{tikzpicture}[scale=1.05, every node/.style={scale=0.85}]
\definecolor{pallido}{RGB}{221,227,227}

    \pattern [pattern=north east lines, pattern color=pallido]
    (-7,2)--(-4,2)--(0,2)--(0,3)--(4,3)--(7,3)--(7,5)--(-4,5)--(-4,4)--(-7,4)--(-7,2);
    
    \draw[thin, black] (-7,2)--(-4,2);
    \draw[thick, orange] (-4,2)--(0,2);
    \draw[thick, red] (2,3)--(4,3);
    \draw[thick, orange] (0,3)--(2,3);
    \draw[thick, red] (0,2)--(0,3);
    \draw[thick, black] (4,3)--(7,3);
         
    \draw[thin, black] (7,5)--(4,5);
    \draw[thick, red] (-4,4)--(-4,5);
    \draw[thick, orange] (-4,5)--(0,5);
    \draw[thick, red] (0,5)--(2,5);
    \draw[thick, orange] (2,5)--(4,5);
    \draw[thin, black] (-4,4)--(-7,4);
    
    \draw[thin, black, ->] (6,3.2)--(6,4.8);
    \draw[thin, black, ->] (-6,2.2)--(-6,3.8);
    
    \fill (-4,2) circle (1.5pt);
    \fill (0,2) circle (1.5pt);
    \fill (0,3) circle (1.5pt);
    \fill (2,3) circle (1.5pt);
    \fill (4,3) circle (1.5pt);
    
    \node at (6.25, 4) {$w$};
    \node at (-6.25, 3) {$w$};
    
    \fill (4,5) circle (1.5pt);
    \fill (2,5) circle (1.5pt);
    \fill (0,5) circle (1.5pt);
    \fill (-4,5) circle (1.5pt);
    \fill (-4,4) circle (1.5pt);
    
    \node at (-2,1.75) {$a_2^-$};
    \node at (3,2.75) {$b_1^-$};
    \node at (1,2.75) {$a_1^-$};
    \node at (0.25,2.5) {$b_2^-$};
    
    \node at (3,5.25) {$a_1^+$};
    \node at (1,5.25) {$b_1^+$};
    \node at (-2,5.25) {$a_2^+$};
    \node at (-4.25,4.5) {$b_2^+$};
    
\end{tikzpicture}
\caption{Realization of a genus two translation surface with poles and discrete period character of rank two. In this case the bottom "zig-zag" line is obtained by sorting the edges in a different way with respect the order used in Figure \ref{fig:hypsimplepolesdisranktwo}. Notice that in this case the shadow area is no longer invariant under a rotation of order $2$ of $\C$. As a consequence, the structure obtained by gluing the edges according to the labels is no longer hyperelliptic. 
}
\label{fig:simplepolesdisranktwo}
\end{figure}
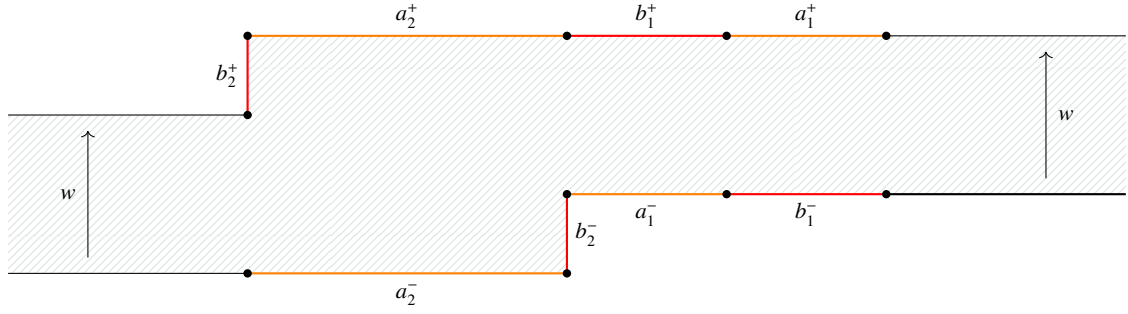

\begin{cor}\label{cor:genustwohyp}
For a partition $\kappa$ of $4$ and a partition $\nu$ of $2$, define $\mathcal{H}_2(\kappa;-\nu)$ as the stratum of meromorphic genus two differentials with zeros and poles of orders prescribed according to $\kappa$ and $-\nu$ respectively. Then Theorem \ref{mainthm} holds for any stratum $\mathcal{H}_2(\kappa;-\nu)$ thus defined.
\end{cor}

\begin{proof}
We begin with the following observation. Let $\kappa$ be a partition of $4$ and let $\nu$ be a partition of $2$. According to Boissy, see \cite[Theorem 1.2]{BC}, a stratum $\mathcal{H}_2(\kappa;-\nu)$ of genus two meromorphic differentials admits at most two connected components. More precisely, the stratum is connected whenever $\kappa\neq\{4\},\,\{2,\,2\}$. For $\kappa=\{4\}$ or $\kappa=\{2,\,2\}$ it admits two connected components one of which is hyperelliptic and the other is not. Suppose a representation $\chi$ can be realized in a certain stratum $\mathcal{H}_2(\kappa; -\nu)$. If the stratum is connected then the claim follows from \cite[Theorems C and D]{CFG}. We next assume that the stratum is not connected. This is one of the strata $\mathcal{H}_2(4;-2)$, $\mathcal{H}_2(2,2;-2)$, $\mathcal{H}_2(4;-1,-1)$, and $\mathcal{H}_2(2,2;-1,-1)$.
Now Proposition \ref{prop:hgdiffhyp} says that $\chi$ can be realized in the hyperelliptic component of each stratum. On the other hand, Propositions \ref{prop:genustwohypcaseone} and \ref{prop:genustwohypcasetwo} say that $\chi$ can be realized in the non-hyperelliptic component of each stratum. The claim thus  follows.
\end{proof}

\begin{rmk}\label{rmk:specialgenus2strata}
According to Boissy, \cite{BC}, the connected components of strata in Corollary \ref{cor:genustwohyp} are also distinguished by the spin parity, see \S\ref{sssec:spinpar}. In fact, by a direct computation, one can see that hyperelliptic translation surfaces have even parity and non-hyperelliptic ones have odd parity. See Figures \ref{fig:simplepolesdisranktwowithindex} and \ref{fig:simplepolesdisranktwowithindexeven} for an example related to Figures \ref{fig:hypsimplepolesdisranktwo} and  \ref{fig:simplepolesdisranktwo}.
\end{rmk}

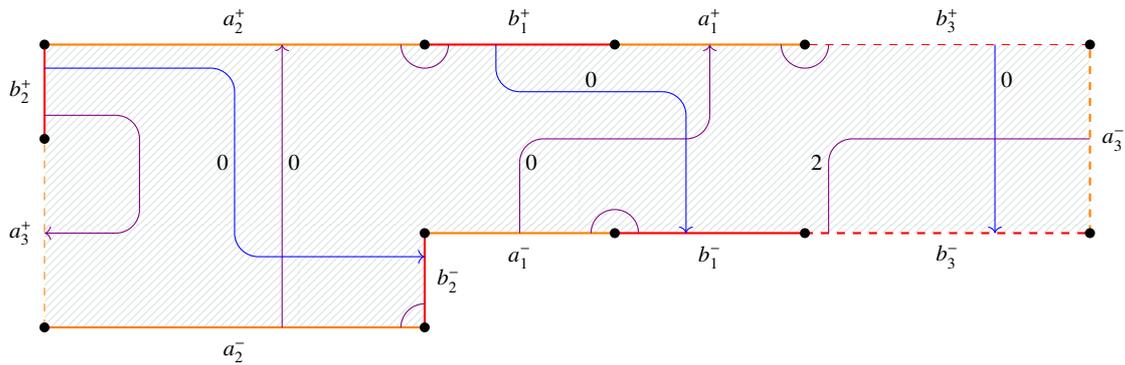
\begin{figure}[!ht] 
\centering
\begin{tikzpicture}[scale=1.25, every node/.style={scale=0.85}]
\definecolor{pallido}{RGB}{221,227,227}

    \pattern [pattern=north east lines, pattern color=pallido]
    (-4,2)--(0,2)--(0,3)--(4,3)--(7,3)--(7,5)--(-4,5)--(-4,4);
    
    \draw[thin, orange, dashed] (-4,2)--(-4,4);
    \draw[thick, orange] (-4,2)--(0,2);
    \draw[thick, red] (2,3)--(4,3);
    \draw[thick, orange] (0,3)--(2,3);
    \draw[thick, red] (0,2)--(0,3);
    \draw[thick, dashed, red] (4,3)--(7,3);
         
    \draw[thin, dashed, red] (7,5)--(4,5);
    \draw[thick, red] (-4,4)--(-4,5);
    \draw[thick, orange] (-4,5)--(0,5);
    \draw[thick, red] (0,5)--(2,5);
    \draw[thick, orange] (2,5)--(4,5);
    \draw[thick, orange, dashed] (7,3)--(7,5);
    
    \draw[thin, violet, ->] (-1.5,2)--(-1.5,5);
    \draw[thin, violet] (1,3)--(1,3.75);
    \draw [thin, violet] (1,3.75) arc [start angle =180, end angle=90 , radius = 0.25];
    \draw[thin, violet] (1.25,4)--(2.75,4);
    \draw [thin, violet] (2.75,4) arc [start angle =270, end angle=360 , radius = 0.25];
    \draw[thin, violet, ->] (3,4.25)--(3,5);
    \draw[thin, blue, ->] (6,5)--(6,3);
    
    \draw[thin, blue, <-] (2.75,3)--(2.75,4.25);
    \draw [thin, blue] (2.75,4.25) arc [start angle =0, end angle=90 , radius = 0.25];
    \draw[thin, blue] (2.5,4.5)--(1,4.5);
    \draw [thin, blue] (1,4.5) arc [start angle =270, end angle=180 , radius = 0.25];
    \draw[thin, blue, -] (0.75,4.75)--(0.75,5);
    
    \draw[thin, blue] (-4,4.75)--(-2.25,4.75);
    \draw [thin, blue] (-2.25,4.75) arc [start angle =90, end angle=0 , radius = 0.25];
    \draw[thin, blue] (-2,4.5)--(-2,3);
    \draw [thin, blue] (-2,3) arc [start angle =180, end angle=270 , radius = 0.25];
    \draw[thin, blue, ->] (-1.75,2.75)--(0,2.75);
    
    \draw[thin, violet] (7,4)--(4.5,4);
    \draw[thin, violet] (4.5,4) arc [start angle =90, end angle=180 , radius = 0.25];
    \draw[thin, violet] (4.25,3.75)--(4.25,3);
    \draw[thin, violet] (4.25,5) arc [start angle =0, end angle=-180 , radius = 0.25];
    \draw[thin, violet] (2.25,3) arc [start angle =0, end angle=180 , radius = 0.25];
    \draw[thin, violet] (0.25,5) arc [start angle =0, end angle=-180 , radius = 0.25];
    \draw[thin, violet] (0,2.25) arc [start angle =90, end angle=180 , radius = 0.25];
    \draw[thin, violet] (-4,4.25)--(-3.25,4.25);
    \draw[thin, violet] (-3.25,4.25) arc [start angle =90, end angle=0 , radius = 0.25];
    \draw[thin, violet] (-3,4)--(-3,3.25);
    \draw[thin, violet] (-3,3.25) arc [start angle =0, end angle=-90 , radius = 0.25];
    \draw[thin, violet, <-] (-4,3)--(-3.25,3);
    
    \fill (-4,2) circle (1.5pt);
    \fill (0,2) circle (1.5pt);
    \fill (0,3) circle (1.5pt);
    \fill (2,3) circle (1.5pt);
    \fill (4,3) circle (1.5pt);
    
    \fill (7,3) circle (1.5pt);
    \fill (7,5) circle (1.5pt);
    
    \node at (6.125, 4.625) {$0$};
    \node at (1.75, 4.625) {$0$};
    \node at (4.125, 3.75) {$2$};
    \node at (1.125, 3.75) {$0$};
    \node at (-1.375, 3.75) {$0$};
    \node at (-2.125, 3.75) {$0$};
    
    \fill (4,5) circle (1.5pt);
    \fill (2,5) circle (1.5pt);
    \fill (0,5) circle (1.5pt);
    \fill (-4,5) circle (1.5pt);
    \fill (-4,4) circle (1.5pt);
    
    \node at (-2,1.75) {$a_2^-$};
    \node at (3,2.75) {$b_1^-$};
    \node at (1,2.75) {$a_1^-$};
    \node at (0.25,2.5) {$b_2^-$};
    \node at (7.25,4) {$a_3^-$};
    \node at (5.5,2.75) {$b_3^-$};

    \node at (3,5.25) {$a_1^+$};
    \node at (1,5.25) {$b_1^+$};
    \node at (-2,5.25) {$a_2^+$};
    \node at (-4.25,4.5) {$b_2^+$};
    \node at (-4.25,3) {$a_3^+$};
    \node at (5.5,5.25) {$b_3^+$};
    
\end{tikzpicture}
\caption{Computation of the spin parity for the translation surface $(X,\omega)$ in Figure \ref{fig:simplepolesdisranktwo}. According to Remark \ref{twosimplepoles}, the structure depicted here can be obtained from $(X,\omega)$ in Figure \ref{fig:simplepolesdisranktwo} by truncating the cylindrical ends along waist geodesic curves. The dashed edges correspond to those obtained after truncation. The colored lines represent a symplectic base and the separate labels denote the indices of the respective curves. According to formula \eqref{eq:spinparity}, it is easy to check that $\varphi(\omega)=1\,\,(\textnormal{mod}\,2)$, hence the structure is not hyperelliptic.
}
\label{fig:simplepolesdisranktwowithindex}
\end{figure}

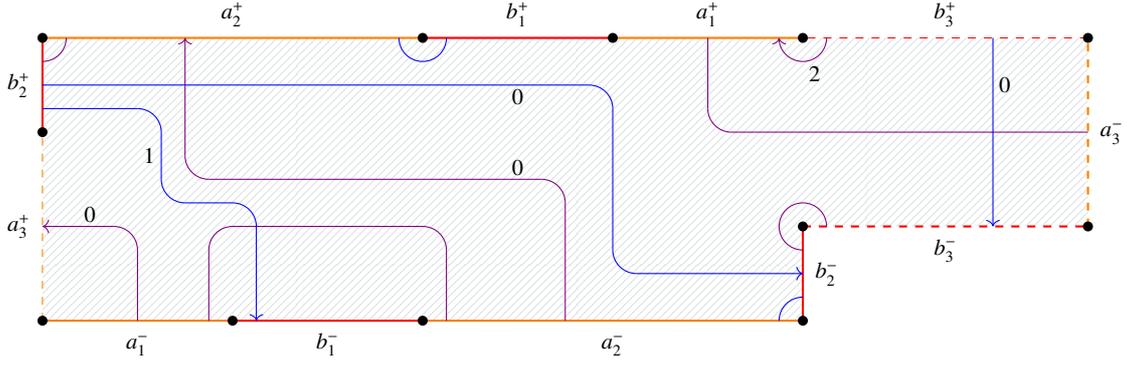
\begin{figure}[!ht] 
\centering
\begin{tikzpicture}[scale=1.25, every node/.style={scale=0.85}]
\definecolor{pallido}{RGB}{221,227,227}

    \pattern [pattern=north east lines, pattern color=pallido]
    (-4,2)--(0,2)--(4,2)--(4,3)--(7,3)--(7,5)--(-4,5)--(-4,4);
    
    \draw[thin, orange, dashed] (-4,2)--(-4,4);
    \draw[thick, orange] (0,2)--(4,2);
    \draw[thick, orange] (-4,2)--(-2,2);
    \draw[thick, red] (-2,2)--(0,2);
    \draw[thick, red] (4,2)--(4,3);
    \draw[thick, dashed, red] (4,3)--(7,3);
         
    \draw[thin, dashed, red] (7,5)--(4,5);
    \draw[thick, red] (-4,4)--(-4,5);
    \draw[thick, orange] (-4,5)--(0,5);
    \draw[thick, red] (0,5)--(2,5);
    \draw[thick, orange] (2,5)--(4,5);
    \draw[thick, orange, dashed] (7,3)--(7,5);
    
    \draw[thin, violet, ->] (-2.5,3.75)--(-2.5,5);
    \draw[thin, violet] (-2.5,3.75) arc [start angle =180, end angle=270 , radius = 0.25];
    \draw[thin, violet] (-2.25,3.5)--(1.25,3.5);
    \draw[thin, violet] (1.25,3.5) arc [start angle =90, end angle=0 , radius = 0.25];
    \draw[thin, violet] (1.5,3.25)--(1.5,2);
    
    \draw[thin, violet] (7,4)--(3.25,4); 
    \draw[thin, violet] (3.25,4) arc [start angle =270, end angle=180 , radius = 0.25];
    \draw[thin, violet] (3,4.25)--(3,5);
    \draw[thin, violet] (-3,2)--(-3,2.75);
    \draw[thin, violet] (-3,2.75) arc [start angle =0, end angle=90 , radius = 0.25];
    \draw[thin, violet, <-] (-4,3)--(-3.25,3); 
    
    \draw[thin, violet] (-2.25,2)--(-2.25,2.75);
    \draw[thin, violet] (-2.25,2.75) arc [start angle =180, end angle=90 , radius = 0.25];
    \draw[thin, violet] (-2,3)--(0,3);
    \draw[thin, violet] (0.25,2.75) arc [start angle =0, end angle=90 , radius = 0.25];
    \draw[thin, violet] (0.25,2.75)--(0.25,2);
    \draw[thin, violet] (-3.75,5) arc [start angle =0, end angle=-90 , radius = 0.25];
    \draw[thin, violet] (4,2.75) arc [start angle =270, end angle=0 , radius = 0.25];
    \draw[thin, violet, ->] (4.25,5) arc [start angle =0, end angle=-180 , radius = 0.25];
    
    \draw[thin, blue] (0.25, 5) arc [start angle =0, end angle=-180 , radius = 0.25];
    \draw[thin, blue] (4, 2.25) arc [start angle =90, end angle=180 , radius = 0.25];
    \draw[thin, blue] (-4,4.25)--(-3,4.25);
    \draw[thin, blue] (-3, 4.25) arc [start angle =90, end angle=0 , radius = 0.25];
    \draw[thin, blue] (-2.75, 4)--(-2.75,3.5);
    \draw[thin, blue] (-2.75,3.5) arc [start angle =180, end angle=270 , radius = 0.25];
    \draw[thin, blue] (-2.5, 3.25)--(-2, 3.25);
    \draw[thin, blue] (-2, 3.25) arc [start angle =90, end angle=0 , radius = 0.25];
    \draw[thin, blue, ->] (-1.75, 3)--(-1.75,2);
    
    \draw[thin, blue, ->] (6,5)--(6,3);
    
    \draw[thin, blue] (-4,4.5)--(1.75,4.5);
    \draw [thin, blue] (1.75,4.5) arc [start angle =90, end angle=0 , radius = 0.25];
    \draw[thin, blue] (2,4.25)--(2,2.75);
    \draw [thin, blue] (2,2.75) arc [start angle =180, end angle=270 , radius = 0.25];
    \draw[thin, blue, ->] (2.25,2.5)--(4,2.5);

    \fill (-4,2) circle (1.5pt);
    \fill (0,2) circle (1.5pt);
    \fill (-2,2) circle (1.5pt);
    \fill (4,2) circle (1.5pt);
    \fill (4,3) circle (1.5pt);
    
    \fill (7,3) circle (1.5pt);
    \fill (7,5) circle (1.5pt);
    
    \node at (6.125, 4.5) {$0$};
    \node at (1, 4.375) {$0$};
    \node at (4.125, 4.625) {$2$};
    \node at (1, 3.625) {$0$};
    \node at (-2.875, 3.75) {$1$};
    \node at (-3.5, 3.125) {$0$};
    
    \fill (4,5) circle (1.5pt);
    \fill (2,5) circle (1.5pt);
    \fill (0,5) circle (1.5pt);
    \fill (-4,5) circle (1.5pt);
    \fill (-4,4) circle (1.5pt);
    
    \node at (2,1.75) {$a_2^-$};
    \node at (-1,1.75) {$b_1^-$};
    \node at (-3,1.75) {$a_1^-$};
    \node at (4.25,2.5) {$b_2^-$};
    \node at (7.25,4) {$a_3^-$};
    \node at (5.5,2.75) {$b_3^-$};

    \node at (3,5.25) {$a_1^+$};
    \node at (1,5.25) {$b_1^+$};
    \node at (-2,5.25) {$a_2^+$};
    \node at (-4.25,4.5) {$b_2^+$};
    \node at (-4.25,3) {$a_3^+$};
    \node at (5.5,5.25) {$b_3^+$};
    
\end{tikzpicture}
\caption{Computation of the spin parity for the translation surface $(X,\omega)$ in Figure \ref{fig:hypsimplepolesdisranktwo}. According to Remark \ref{twosimplepoles}, the structure depicted here can be obtained from $(X,\omega)$ in Figure \ref{fig:hypsimplepolesdisranktwo} by truncating the cylindrical ends along waist geodesic curves. By using formula \eqref{eq:spinparity}, it is easy to check that $\varphi(\omega)=0\,\,(\textnormal{mod}\,2)$, hence the structure is hyperelliptic.
}
\label{fig:simplepolesdisranktwowithindexeven}
\end{figure}

\smallskip

\section{Higher genus meromorphic differentials with prescribed parity}\label{sec:hgcpar}

\noindent We aim to determine whether a representation $\chi\colon\shomolzn\longrightarrow \C$ can be realized as the period character of some translation surface with prescribed spin parity, see \S\ref{sssec:spinpar}. In this section we assume the representation $\chi$ to be non-trivial and we shall handle the trivial representation in Section \S\ref{sec:trirep}. Recall that a genus $g$ meromorphic differential $\omega$ on a Riemann surface $X$ determines a well-defined spin structure if and only if the set of zeros and poles is of even type, see Definition \ref{def:kindofstrata}. Our aim is to prove the following

\begin{prop}\label{prop:hgspin}
Let $\chi$ be a non-trivial representation and suppose it arises as the period character of some meromorphic genus $g$ differential in a stratum admitting two connected components distinguished by the spin parity. Then $\chi$ can be realized in both components of the same stratum as the period character of some translation surfaces with poles.
\end{prop}

\smallskip

\subsection{Inductive process}\label{ssec:indproc} The strategy we shall adopt in the present section is based on an inductive foundation on the genus $g$ of surfaces. We begin with an explanation of our strategy which we shall develop in Section \S\ref{ssec:intoproc}. In the explanation below we shall mainly consider strata of differentials with a single zero of maximal order because the general case follows by breaking a zero.

\subsubsection{Higher order poles}\label{sssec:hopbub} Let $\chi\colon\shomolzn\longrightarrow \C$ be a non-trivial representation, possibly of non-trivial-ends type, and let $\mathcal{H}_g(2m+2g-2;-2p_1,\dots,-2p_n)$ be a stratum of genus $g$ differentials of even type, where $n\ge2$ and $m=p_1+\cdots+p_n$. According to \cite[Theorem C]{CFG}, $\chi$ can be realized in such a stratum. Let $\{\alpha_1,\beta_1,\dots,\alpha_g,\beta_g\}$ be a system of handle generators, see Definition \ref{def:syshandle}, and define an auxiliary representation $\rho\colon\textnormal{H}_1(S_{1,n},\,\Z)\longrightarrow \C$ as follow
\begin{equation}\label{eq:auxsecondtype}
    \rho(\alpha)=\chi(\alpha_1),\quad\rho(\beta)=\chi(\beta_1),\quad \rho(\delta_i)=\chi(\delta_i) 
\end{equation} where $\{\alpha,\beta\}$ is a pair of handle generators for $\shomolzon$ and $\delta_i$ is a peripheral loop around the $i$-th puncture. Notice that this auxiliary representation differs from that defined in Definition \ref{supprep}. Since $\chi$ is a non-trivial representation, Lemmas \ref{lem:allhandholnonzero} and \ref{lem:handholnonzero} apply and we can assume $\chi(\alpha_1),\,\chi(\beta_1)\in\C^*$. As a consequence, the auxiliary representation $\rho$ is also non-trivial. Next, we realize $\rho$ as the period character of some translation surface with poles in the stratum $\mathcal{H}_1(2n;-2,\dots,-2)$ as in Section \S\ref{genusoneordertwo}. This structure serves as the base case for an inductive foundation. The inductive process consists in showing that at each step we always obtain a translation surface with poles with enough room to bubble a handle so that the resulting structure has the desired parity. Notice that, once the polar part $\nu=(2,\dots,2)$ is fixed, the genus determines the order of the zero uniquely; this is Gauss-Bonnet condition, see Remark \ref{gbcond}. Therefore bubbling a handle yields a sequence of mapping between strata as follows 
\begin{equation}\label{eq:chainofstrata}
    \mathcal{H}_1(2n;-\nu)\longmapsto  \mathcal{H}_2(2n+2;-\nu) \longmapsto 
    \cdots\longmapsto \mathcal{H}_g(2n+2g-2;-\nu)\longmapsto \mathcal{H}_{g+1}(2n+2g;-\nu)\longmapsto\cdots.
\end{equation}

\noindent According to Lemmas \ref{lem:spininv} and \ref{lem:spininv2}, at each step  bubbling a handle does not alter the spin parity. Therefore, as we shall see in Section \S\ref{ssec:intoproc}, the spin parity is completely determined by the rotation number of the initial genus one differential. The stratum $\mathcal{H}_1(2n;-2,\dots,-2)$ has exactly two connected components: One of these comprises genus one differentials with rotation number $k=1$ and the other one  comprises differentials with rotation number $k=2$. In the former case, each bubbling gets the access to the connected component of translation surfaces with even spin parity. In the latter case, each bubbling gets the access to the connected component of translation surfaces with odd spin parity.

\medskip

\noindent Once a structure in a stratum $\mathcal{H}_g(2n+2g-2;-2,\dots,-2)$ is realized with prescribed parity and period character $\chi$, then we can get the access to all other strata of genus $g$ differentials by bubbling copies of the differential $(\C,\,z\,dz)$ along suitable rays joining the single zero with the punctures. Lemma \ref{lem:invaspin} ensures that the parity remains unaltered. Finally, by breaking a zero, see \S\ref{sec:zerobreak}, we get the desired result for all possible strata $\mathcal{H}_{g}(\kappa;-\nu)$, where $\kappa=(2m_1,\dots,2m_k)\in2\,\Z_+^k$ and $\kappa, \nu$ satisfy the Gauss-Bonnet condition \eqref{gbeq}.

\subsubsection{Two exceptional cases}\label{sssec:excasespin} The strategy above does not apply for strata of differentials with a single pole of order $2$. The problem is due to the connectedness of the stratum $\mathcal{H}_1(2;-2)$, see \cite[Theorem 1.1]{BC} or Section \S\ref{mscc}. We bypass this issue by using genus two differentials as the base case for the induction. More precisely, given a representation $\chi\colon\textnormal{H}_1(S_{g,1},\,\Z)\longrightarrow \C$ and a system of handle generators $\{\alpha_1,\beta_1,\dots,\alpha_g,\beta_g\}$ we define an auxiliary representation $\rho\colon\textnormal{H}_1(S_{2,1},\,\Z)\longrightarrow \C$ as 
\begin{equation}\label{eq:auxsecondtype2}
    \rho(\alpha_1)=\chi(\alpha_1),\quad\rho(\beta_1)=\chi(\beta_1),\quad \rho(\alpha_2)=\chi(\alpha_2),\quad\rho(\beta_2)=\chi(\beta_2).
\end{equation}

\noindent Next, we realize $\rho$ as the period character of some translation surface in $\mathcal{H}_2(4;-2)$. We shall use this structure as the base case for the induction and then we can rely on the same kind of process described in Section \S\ref{sssec:hopbub}. Bubbling a handle yields a sequence of mapping similar to \eqref{eq:chainofstrata}, \textit{i.e.}

\begin{equation}\label{eq:chainofstrata2}
    \mathcal{H}_2(4;-2) \longmapsto \mathcal{H}_3(6;-2)\longmapsto
    \cdots\longmapsto \mathcal{H}_g(2g;-2)\longmapsto \mathcal{H}_{g+1}(2g+2;-2)\longmapsto\cdots.
\end{equation}

\noindent Recall that bubbling does not alter the spin parity of the structure (whenever it is defined). Since $\mathcal{H}_2(4;-2)$ has two connected components distinguished by the spin parity, by bubbling a genus two differential with even parity we get a translation surface with even parity in each stratum $\mathcal{H}_g(2g;-2)$. Similarly, by bubbling a genus two differential with odd parity we get a structure with odd parity in each stratum $\mathcal{H}_g(2g;-2)$. We next induct on the order of the pole. More precisely, by bubbling $p-1$ copies of $(\C,\,z\,dz)$ along an infinite ray joining the single zero and the pole we can access to the stratum $\mathcal{H}_g(2g+2p-2;-2p)$. Notice that, again, Lemma \ref{lem:invaspin} ensures that the spin parity only depends on the parity of the initial structure. By breaking a zero, we get the access to all strata of the form $\mathcal{H}_g(\kappa;-2p)$, where $\kappa=(2m_1,\dots,2m_k)\in2\,\Z_+^k$.

\medskip

\noindent In the light of Remark \ref{twosimplepoles}, for two-punctured surfaces there is an additional case to take into account. This is the second exceptional case. Let $\chi\colon\textnormal{H}_1(S_{g,\,2},\,\Z)\longrightarrow \C$ be a non-trivial representation of non-trivial-ends type and let $\mathcal{H}_g(2g;-1,-1)$ be a stratum of genus $g$ meromorphic differentials with two simple poles. A non-rational representation, see Definition \ref{ratchar}, can be realized in that stratum, see \cite[Theorem C]{CFG}. However, if $\chi$ is rational, then necessary and sufficient conditions for the realization are given by \cite[Theorem D]{CFG}, see also Proposition \ref{prop:rankonechar} above. Once again, we rely on an inductive foundation and, since the stratum $\mathcal{H}_1(2;-1,-1)$ is connected, genus two differentials will serve as the base case for the induction. In fact, the stratum $\mathcal{H}_2(4;-1,-1)$ has two connected components distinguished by the spin parity. We shall realize an auxiliary representation in this stratum with prescribed parity and then bubbling will provide the access to connected components of all the other strata $\mathcal{H}_g(2g,-1,-1)$ according to the sequence of mapping 

\begin{equation}\label{eq:chainofstrata3}
    \mathcal{H}_2(4;-1,-1) \longmapsto \mathcal{H}_3(6;-1,-1)\longmapsto
    \cdots\longmapsto \mathcal{H}_g(2g;-1,-1)\longmapsto \mathcal{H}_{g+1}(2g+2;-1,-1)\longmapsto\cdots.
\end{equation}

\noindent Once again, by breaking a zero we get the access to all strata of the form $\mathcal{H}_g(\kappa;-1,-1)$, with $\kappa\in2\,\Z_+^k$.

\smallskip

\noindent In both exceptional cases, the basic cases for the inductive process have already been realized in Section \S\ref{sec:hgchyp}, see Remark \ref{rmk:specialgenus2strata}. In the next section we move to develop our inductive process. We have already mentioned above that bubbling copies of $(\C,\,dz)$ along rays and breaking a zero into zeros of even order do not alter the spin parity. Since these operations provide the access to all other strata with poles of order greater than $2$ and multiple zeros, in what follow we can reduce to consider strata with poles of order $2$ and then prove Proposition \ref{prop:hgspin} for these strata. The generic case immediately follows.

\smallskip

\subsection{Into the process}\label{ssec:intoproc} We discuss the inductive foundation by distinguishing two cases according to Sections \S\ref{sssec:hopbub} and \S\ref{sssec:excasespin} above.

\subsubsection{Generic case}\label{sssec:gencase} Let $n\ge2$ and let $\chi\colon\shomolzn\longrightarrow \C$ be a non-trivial representation, possibly of non-trivial-ends type. Let $\mathcal{G}=\{\alpha_1,\beta_1,\dots,\alpha_g,\beta_g\}$ be a system of handle generators and let $\rho\colon\textnormal{H}_1(S_{1,n},\,\Z)\longrightarrow \C$ be an auxiliary representation defined as in \eqref{eq:auxsecondtype}. Assume without loss of generality that $\rho$ is a non-trivial representation. According to our constructions in Section \S\ref{genusonemero}, $\rho$ can be realized as the period character of some translation surface with poles, say $(X,\omega)$, in $\mathcal{H}_1(2n;-2,\dots,-2)$ with rotation number $k=1$ or $k=2$. 

\smallskip

\noindent More precisely, if $\text{vol}(\rho)>0$, then we realize $(X,\omega)$ as in \S\ref{pk1} in order to get a structure with rotation number $k=1$; otherwise we realize $(X,\omega)$ as in \S\ref{pvep} or \S\ref{pvop} to get a structure with rotation number $k=2$. In the case $\text{vol}(\rho)\le0$, then we realize $(X,\omega)$ as in \S\ref{nk1} in order to obtain a structure with rotation number $k=1$; otherwise we realize $(X,\omega)$ as in \S\ref{npvonp} or \S\ref{nvep} to get a structure with rotation number $k=2$, see also Table \ref{fig:flowdiagram}.

\smallskip

\noindent In order to consistent, let us adopt the notation of Section \S\ref{genusonemero}. We begin by noticing that in all cases mentioned above, the translation surface $(X,\omega)$ always contains an entire copy of $(\C,\,dz)$ or, at worst, a copy of $(\C,\,dz)$ with some compact set removed. 

\begin{rmk}\label{rmk:excasetwopunctures}
This latter case appears only if $n=2$ and we aim to realize $\rho\colon\textnormal{H}_1(S_{1,2},\,\Z)\longrightarrow \C$ as the period character of some structure $(X,\omega)$ with rotation number $2$. In fact, in this special case $(X,\omega)$ is obtained by gluing the closures of the exteriors of two isometric triangles each in a different copy of $(\C,\,dz)$, see \S\ref{nvep} for details.
\end{rmk}

\noindent Nevertheless, we always have enough room for bubbling handles with positive or non-positive volumes. We proceed in a recursive way as follows. Choose an initial point $O_1$ according to the following rule:

\begin{itemize}
    \item If $(X,\omega)$ contains an entire copy of $(\C,\,dz)$, then pick $O_1$ as any of the points $Q_i$ for $i=1,\dots,n-1$; otherwise
    \smallskip
    \item we are in the special case mentioned in Remark \ref{rmk:excasetwopunctures} above. The starting point $O_1$ can be taken as $P+\chi(\alpha)$, see Figure \ref{rotnek2}.
\end{itemize}

\noindent In both cases we can find a straight line $l_1$ passing through $O_1$ so that one of the two sides is an embedded half-plane, say $H_1$. Without loss of generality, we can orient $l_1$ so that $H_1$ lies on the left of $l_1$. We will show that all bubbling can be done within $H_1$.

\smallskip

\noindent We are now ready to implement the recursion we just alluded above, namely we will show that all bubbling can be done within $H_1$. For $j=2,\dots,g$, let $l_{j-1}$ be a straight line parallel and with the same orientation with respect to $l_1$ and passing through a point $O_{j-1}$ which will be recursively determined step by step. Finally, define $H_{j-1}$ as the half-plane on the left side of $l_{j-1}$. Next, consider the $j$-th pair of handle generators $\{\alpha_j,\,\beta_j\}\subset \mathcal{G}$. Then, depending on the value $\Im\left(\,\overline{\chi(\alpha_j)}\,\chi(\beta_j)\,\right)$, we proceed as follow:

\smallskip

\begin{itemize}
    \item If $\Im\left(\,\overline{\chi(\alpha_j)}\,\chi(\beta_j)\,\right)>0$, then we bubble a handle with positive volume. Up to replacing $\{\alpha_j,\,\beta_j\}$ with their inverses and renaming the curves if needed, we can assume that the edge, say $e_j$, joining $O_{j-1}$ with $O_{j-1}+\chi(\alpha_j)$ entirely lies in $H_{j-1}$. Let $\mathcal{P}_j\subset\C$ be the parallelogram bounded by the chain
    \begin{equation*}
        O_{j-1}\mapsto O_{j-1}+\chi(\alpha_j)\mapsto O_{j-1}+\chi(\alpha_j)+\chi(\beta_j)\mapsto O_{j-1}+\chi(\beta_j)\mapsto O_{j-1}.
    \end{equation*}
    According to our convention, let us denote by $a_j^+$ (respectively $a_j^-$) the edge of $\mathcal{P}_j$ parallel to $\chi(\alpha_j)$ that bounds the parallelogram on its right (respectively left). Similarly, we denote by $b_j^+$ (respectively $b_j^-$) the edge of $\mathcal{P}_j$ parallel to $\chi(\beta_j)$ that bounds the parallelogram on its right (resp. left). Next, we slit $H_{j-1}$ along $e_j$ and denote $e_j^{\pm}$ the resulting edges. Then identify the edge $b_j^+$ with $b_j^-$, the edge $a_j^+$ with $e_j^-$, and the edge $a_j^-$ with $e_j^+$. The resulting structure is a genus $j$ surface and the newborn handle has periods $\chi(\alpha_j)$ and $\chi(\beta_j)$. Finally define $O_j\defeq O_{j-1}+\chi(\alpha_j)$.\\
    
    \smallskip
    
    \item If $\Im\left(\,\overline{\chi(\alpha_j)}\,\chi(\beta_j)\,\right)=0$, then we bubble a handle with null volume. Up to replacing $\{\alpha_j,\,\beta_j\}$ with their inverses, we can assume that the edge, say $e_j$, joining $O_{j-1}$ with $O_{j-1}+\chi(\alpha_j)+\chi(\beta_j)$ entirely lies in $\overline H_{j-1}$. Notice that here we need to consider the closure of $H_{j-1}$ because $\chi(\alpha_j)$ and $\chi(\beta_j)$ can be parallel to $l_1$. Slit $e_j$ and denote $e_j^{\pm}$ the resulting sides. On $e_j^+$, define $a_j^+$ the sub-edge joining $O_{j-1}$ with $O_{j-1}+\chi(\alpha_j)$ and define $b_j^+$ the sub-edge joining $O_{j-1}+\chi(\alpha_j)$ with $O_{j-1}+\chi(\alpha_j)+\chi(\beta_j)$.
    On $e_j^-$, define $b_j^-$ the sub-edge joining $O_{j-1}$ with $O_{j-1}+\chi(\beta_j)$ and define $a_j^+$ the sub-edge joining $O_{j-1}+\chi(\beta_j)$ with $O_{j-1}+\chi(\alpha_j)+\chi(\beta_j)$. Identify the edge $a_j^+$ with $a_j^-$ and the edge $b_j^+$ with $b_j^-$. The resulting structure is a genus $j$ surface and the newborn handle has periods $\chi(\alpha_j)$ and $\chi(\beta_j)$. Finally define $O_j\defeq O_{j-1}+\chi(\alpha_j)$. In the case $O_j$ lies in $l_{j-1}$, then $l_{j-1}=l_j$ and $H_{j-1}=H_j$.\\
    \smallskip
    \item If $\Im\left(\,\overline{\chi(\alpha_j)}\,\chi(\beta_j)\,\right)<0$, then we bubble a handle of negative volume. Up to replacing $\{\alpha_j,\,\beta_j\}$ with their inverses, we can assume that the edge, say $e_j$, joining $O_{j-1}$ with $O_{j-1}+\chi(\alpha_j)+\chi(\beta_j)$ entirely lies in $H_{j-1}$. Let $\mathcal{Q}_j\subset H_{j-1}$ be the quadrilateral bounded by the chain
    \begin{equation*}
        O_{j-1}\mapsto O_{j-1}+\chi(\alpha_j)\mapsto O_{j-1}+\chi(\alpha_j)+\chi(\beta_j)\mapsto O_{j-1}+\chi(\beta_j)\mapsto O_{j-1}.
    \end{equation*}
    We can assume without loss of generality that $\mathcal{Q}_j\,\cap\,l_{j-1}=\{\,O_{j-1}\,\}$. In fact, if one of the edges of $\mathcal{Q}_j$ would lie on $l_{j-1}$, then we can replace $\{\alpha_j,\,\beta_j\}$ with a new set of handle generators obtained by applying suitable Dehn twists so that the resulting quadrilateral enjoys the desired property. Remove the interior of $\mathcal{Q}_j$ and denote by $a_j^+$, respectively $a_j^-$, the edge of $\mathcal{Q}_j$ parallel to $\chi(\alpha_j)$. Denote by $b_j^+$, respectively $b_j^-$, similarly as before. Identify the edge $a_j^+$ with $a_j^-$ and the edge $b_j^+$ with $b_j^-$. The resulting structure is a genus $j$ surface and the newborn handle has periods $\chi(\alpha_j)$ and $\chi(\beta_j)$ as desired. Finally define $O_j\defeq O_{j-1}+\chi(\alpha_j)$.\\
\end{itemize}

\noindent It remains to show that the translation surface, say $(Y,\,\xi)$, obtained after $g$ steps has the desired parity. This can be seen with a direct computation as follow:
\begin{align}
    \varphi(\xi) & = \sum_{j=1}^g \left(\textnormal{Ind}(\alpha_j)+1 \right)\left(\textnormal{Ind}(\beta_j)+1 \right) \,\,\,(\text{mod}\,2)\\
                 & = \left(\textnormal{Ind}(\alpha_1)+1 \right)\left(\textnormal{Ind}(\beta_1)+1 \right) \,+ \, \sum_{j=2}^g \left(\textnormal{Ind}(\alpha_j)+1 \right)\left(\textnormal{Ind}(\beta_j)+1 \right) \,\,\,(\text{mod}\,2)\\
                 & = \left(\textnormal{Ind}(\alpha_1)+1 \right)\left(\textnormal{Ind}(\beta_1)+1 \right)  \,\,\,(\text{mod}\,2)\,\,=
                 \begin{cases}
                 0,\,\,\,\text{ if }k=1\\
                 1,\,\,\,\text{ if }k=2.\\
                 \end{cases}
\end{align}

\noindent The third equality holds because $\textnormal{Ind}(\alpha_1),\,\textnormal{Ind}(\beta_1)$ are described according to Table \ref{tab:indexes} below and the quantity $\left(\textnormal{Ind}(\alpha_j)+1 \right)\left(\textnormal{Ind}(\beta_j)+1 \right)$ is even for any $j=2,\dots,g$ as  bubbling a handle does not alter the spin parity, see Lemmas \ref{lem:spininv} and \ref{lem:spininv2}. Therefore, the spin parity of the final structure is completely determined by the rotation number of the starting genus one differential.

\begin{table}[!ht]
    \centering
    \begin{tabular}{lcccc}
    \midrule
     & $\textnormal{Ind}(\alpha)$ & $\textnormal{Ind}(\beta)$  & $\left(\textnormal{Ind}(\alpha)+1 \right)\left(\textnormal{Ind}(\beta)+1 \right)$ & $\,\,(\text{mod }2)$\\ \midrule
    $k=1$ &  &  &  \\ \midrule
    \textbf{Positive volume} & $n-1$ & $1$ & $2n$ & 0 \\ \midrule
    \textbf{Non-positive volume} & $n$ & $1$ & $2(n+1)$ & 0 \\ \midrule
    $k=2$ &  &  &  \\ \midrule
    \textbf{Positive volume}\\ $n$ even & $n$ & $0$ & $n+1$ & 1 \\ \midrule
    \textbf{Positive volume}\\ $n$ odd & $2$ & $n-1$ & $3n$ & 1 \\ \midrule
    \textbf{Non-positive volume}\\ $n$ odd & $n-1$ & $2$ & $3n$ & 1 \\ \midrule
    \textbf{Non-positive volume}\\ $n$ even & $n$ & $2$ & $3(n+1)$ & 1 \\ \midrule
    \end{tabular}
    \caption{Indices of pairs of handle generators $\{\alpha,\,\beta\}$ for genus one differentials constructed as in Section \S\ref{genusoneordertwo}.}
    \label{tab:indexes}
\end{table}


\subsubsection{Exceptional case: Single pole of order $2$} Let $n=1$ and let $\chi\colon\shomolzo\longrightarrow \C$ be a non-trivial representation. Fix a set of handle generators $\mathcal{G}=\{\alpha_1,\beta_1,\dots,\alpha_g,\beta_g\}$, define an auxiliary representation as in \eqref{eq:auxsecondtype2}, and finally realize $\rho$ as the period character of some translation surface $(X,\omega)$ as in Section \S\ref{sssec:hyponehop}. In order to be consistent and facilitate reading, we adopt the same notation. 

\begin{rmk}
Since $X$ has genus two, the extremal points of the broken chain \eqref{eq:chainpolzer} are $P_0$ and $P_{4}$.
\end{rmk}

\noindent Let $O_1=P_{4}$ be an initial point and let $l_1$ be a straight line passing through $O_1$, orthogonal to $r_2$ and oriented so that the broken chains \eqref{eq:chainpolzer} and \eqref{eq:chainpolzerr} lie on the left. In this case we set $E_1$ as the half-plane on the right of $l_1$. We can now implement the same recursion as in \S\ref{sssec:gencase} and hence, after $g-2$ steps, we get a genus $g$ differential $\xi$ on a Riemann surface $Y$. A straightforward computation shows that the parity of $\xi$ is determined by the parity of $\omega$:
\begin{align}
    \varphi(\xi) & = \sum_{j=1}^g \left(\textnormal{Ind}(\alpha_j)+1 \right)\left(\textnormal{Ind}(\beta_j)+1 \right) \,\,\,(\text{mod}\,2)\\
                 & = \sum_{j=1}^2\left(\textnormal{Ind}(\alpha_j)+1 \right)\left(\textnormal{Ind}(\beta_j)+1 \right) \,+ \, \sum_{j=3}^g \left(\textnormal{Ind}(\alpha_j)+1 \right)\left(\textnormal{Ind}(\beta_j)+1 \right) \,\,\,(\text{mod}\,2)\\
                 & = \sum_{j=1}^2\left(\textnormal{Ind}(\alpha_j)+1 \right)\left(\textnormal{Ind}(\beta_j)+1 \right) \,\,\,(\text{mod}\,2)\,\,= \varphi(\omega),
\end{align}
\smallskip

\noindent where the third equality follows as a consequence of Lemmas \ref{lem:spininv} and \ref{lem:spininv2}. Since $(X,\omega)$ can be realized with both even or odd parity, Proposition \ref{prop:hgspin} follows in this case.

\subsubsection{Exceptional case: Two simple poles} Let $n=2$ and let $\chi\colon\textnormal{H}_1(S_{g,\,2},\,\Z)\longrightarrow \C$ be a representation of non-trivial-ends type. Let $w\in\C^*$ such that $\textnormal{Im}(\chi_2)=\langle\,w\,\rangle$, where $\chi_2$ is the representation encoding the polar part of $\chi$, see Section \S\ref{agv}. Let $\mathcal{G}=\{\alpha_1,\beta_1,\dots,\alpha_g,\beta_g\}$ be a system of handle generators and define an auxiliary representation $\rho\colon\textnormal{H}_1(S_{2,\,2},\,\Z)\longrightarrow \C$ as 
\begin{equation}\label{eq:auxsecondtype3}
    \rho(\alpha_1)=\chi(\alpha_1),\quad\rho(\beta_1)=\chi(\beta_1),\quad \rho(\alpha_2)=\chi(\alpha_2),\quad\rho(\beta_2)=\chi(\beta_2),\quad \rho(\delta_1)=\rho(\delta_2^{-1})=w\in\C^*.
\end{equation}

\noindent We already know from \S\ref{ssec:hypnotzeressimp} how to realize $\rho$ as the period character of some translation surfaces with poles, say $(X,\omega)\in\mathcal{H}_2(4;-1,-1)$, and prescribed parity. For simplicity, we shall adopt the same notation as therein and we proceed with a discussion case by case as follows.

\medskip

\paragraph{\textit{$\chi$ is not real-collinear.}} Consider first non real-collinear representations. By Corollary \ref{cor:posvol} we can assume that any pair $\{\alpha_j,\,\beta_j\}\subset\mathcal{G}$ has positive volume, \textit{i.e.} $\Im\left(\,\overline{\chi(\alpha_j)}\,\chi(\beta_j)\,\right)>0$. Without loss of generality, we can also assume that the auxiliary representation $\rho$ is also not real-collinear. Notice that $\rho$ can be discrete of rank two even if the overall representation $\chi$ is not (it cannot be discrete of rank one because all pairs of handle generators have positive volume). If $\rho$ is discrete of rank two, then Lemma \ref{lem:disnormalform} applies and we renormalize the handle generators $\{\alpha_1,\beta_1,\alpha_2,\beta_2\}$ so that 
\begin{equation*}
    \rho(\alpha_1)=\rho(\beta_1)\in\Z_+ \quad \text{ and } \quad \rho(\alpha_2)\in\Z_+,\,\,\rho(\beta_2)=i.
\end{equation*}

\noindent Under these conditions, we realize $(X,\omega)$ as in \S\ref{par:hypnotcol} or \S\ref{sssec:disranktwo} depending on whether $\rho$ is discrete or not. Let $O_1=P_{4}$ be the initial point. For any $j=3,\dots,g$ consider the pair of handle generators $\{\alpha_j,\,\beta_j\}$. Notice that least one between $\chi(\alpha_j)$ and $\chi(\beta_j)$ is not parallel to $w$, otherwise $\Im\left(\,\overline{\chi(\alpha_j)}\,\chi(\beta_j)\,\right)=0$. 

\begin{rmk}
By replacing $\{\alpha_j,\,\beta_j\}$ with either $\{\alpha_j^{-1},\,\beta_j^{-1}\}$,
$\{\,\beta_j,\,\alpha^{-1}_j\}$ or $\{\,\beta_j^{-1},\,\alpha_j\}$, and then renaming the new pair of handle generators as $\{\alpha_j,\,\beta_j\}$ (with a little abuse of notation), we can assume that $\chi(\alpha_j)$ is not parallel to $w$ and points rightwards, \textit{i.e.} $\arg\left(\, \chi(\alpha_j)\,\right)\in \left[ -\frac\pi2,\,\frac\pi2\right[$. Notice that this changing does not alter the volume of the handle.
\end{rmk}

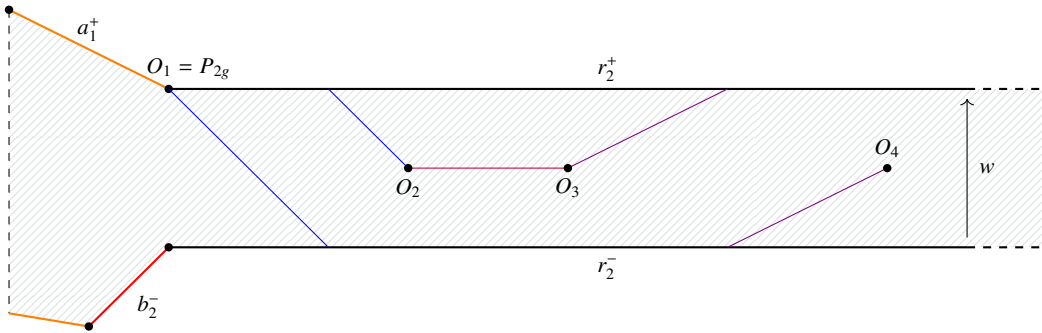
\begin{figure}[!ht] 
\centering
\begin{tikzpicture}[scale=1.05, every node/.style={scale=0.85}]
\definecolor{pallido}{RGB}{221,227,227}
    \pattern [pattern=north east lines, pattern color=pallido]
    (3,2)--(4,3)--(15,3)--(15,5)--(4,5)--(2,6)--(2,2.166666666);
    \draw[thin, black, dashed] (2,6)--(2,2.166666666);
    \draw[thick, orange] (2,2.1666666666)--(3,2);
    \draw[thick, red] (3,2)--(4,3);
    \draw[thick, black] (4,3)--(14,3);
    \draw[thick, black, dashed] (14,3)--(15,3);
    \draw[thick, black] (4,5)--(14,5);
    \draw[thick, black, dashed] (14,5)--(15,5);
    \draw[thick, orange] (2,6)--(4,5);
        
    \draw[thin, black, ->] (14, 3.125)--(14, 4.875);    
        
    \fill (3,2) circle (1.5pt);
    \fill (4,3) circle (1.5pt);

    \fill (4,5) circle (1.5pt);
    \fill (2,6) circle (1.5pt);

    \draw[thin, blue] (4,5)--(6,3);
    \draw[thin, blue] (6,5)--(7,4);
    \draw[thin, purple] (7,4)--(9,4);
    \draw[thin, violet] (9,4)--(11,5);
    \draw[thin, violet] (11,3)--(13,4);
    
    \fill (7,4) circle (1.5pt);
    \fill (9,4) circle (1.5pt);
    \fill (13,4) circle (1.5pt);
    
    \node at (14.25, 4) {$w$};
    
    \node at (4.25, 5.25) {$O_1=P_{2g}$};
    \node at (7, 3.75) {$O_2$};
    \node at (9, 3.75) {$O_3$};
    \node at (13, 4.25) {$O_4$};
    
    \node at (3.75,2.25) {$b_2^-$};
    \node at (9.5,2.75) {$r_2^-$};

    \node at (3,5.75) {$a_1^+$};
    \node at (9.5,5.25) {$r_2^+$};

\end{tikzpicture}
\caption{Bubbling handles of positive volume on a genus two differential constructed as in  \S\ref{par:hypnotcol}. All bubbling are performed inside a cylinder. Each coloured edge corresponds to a slit along which we bubble a handle with positive volume.}
\label{fig:bubbhandinacyl}
\end{figure}

\noindent For $j=3,\dots,g$, let $O_{j-1}\defeq O_{j-2}+\chi(\alpha_j)$ and let $e_{j-2}$ be the edge joining $O_{j-2}$ with $O_{j-1}$, see Figure \ref{fig:bubbhandinacyl}. Notice that $e_j$ is a geodesic segment that "wraps" around the cylindrical ends on $(X,\omega)$ without overlapping itself because $\chi(\alpha_j)$ is not parallel to $w$ by design. Let $\mathcal{P}_j\subset\C$ be the parallelogram bounded by the chain
\begin{equation*}
    O_{j-2}\mapsto O_{j-2}+\chi(\alpha_j)\mapsto O_{j-2}+\chi(\alpha_j)+\chi(\beta_j)\mapsto O_{j-2}+\chi(\beta_j)\mapsto O_{j-2}.
\end{equation*}
\noindent As usual, according to our convention, we denote by $a_j^+$, respectively $a_j^-$, the edge of $\mathcal{P}_j$ parallel to $\chi(\alpha_j)$ that bounds the parallelogram on its right, respectively left. Similarly, we denote by $b_j^+$, respectively $b_j^-$, the edge of $\mathcal{P}_j$ parallel to $\chi(\beta_j)$ that bounds the parallelogram on its right, respectively left. Next, we slit $(X,\omega)$ along $e_{j-2}$ and denote by $e_{j-2}^{\pm}$ the resulting sides. Identify the edge $e_{j-2}^-$ with $a_j^+$ and $e_{j-2}^+$ with $a_j^-$. Finally, identify the edge $b_j^+$ with $b_j^-$. After repeating this process $g-2$ times, we obtain a translation surface $(Y,\xi)$ with period character $\chi$ and spin parity $\varphi(\xi)$ equal to $\varphi(\omega)$ because bubbling handles with positive volume, as in Section \S\ref{ssec:bubhandle}, does not alter the spin parity. Since $(X,\omega)$ can be realized with prescribed parity, both even and odd parity are achievable, and the desired conclusion follows in this case.

\medskip

\paragraph{\textit{$\chi$ is real-collinear but not discrete.}}\label{par:realcolnotdis} We now assume $\chi$ to be real-collinear but not discrete. Up to replacing $\chi$ with $A\,\chi$, for some $A\in\glplus$, we can assume that $\textnormal{Im}(\,\chi\,)\subset \R$. Lemma \ref{lem:smallperiods} applies and hence we can find a system of handle generators $\mathcal{G}=\{\alpha_1,\beta_1,\dots,\alpha_g,\beta_g\}$ such that $a_j=\chi(\alpha_j)$ and $b_j=\chi(\beta_j)$ satisfy the inequality
\begin{equation}\label{eq:length}
    v=\sum_{j=1}^g a_j\,+\,b_j < w.
\end{equation}

\noindent By introducing an auxiliary representation $\rho$ as in \eqref{eq:auxsecondtype3} we first realize a translation surface $(X,\omega)$ in the stratum $\mathcal{H}_2(4;-1,-1)$ with period character $\rho$ and prescribed parity as in  \S\ref{par:hypcol}. By construction, there is a closed saddle connection, say $s$, of length 
\begin{equation*}
    w-\sum_{j=1}^2 a_j\,+\,b_j.
\end{equation*}

\noindent Since $v<w$ we have enough room on $s$ for bubbling $g-2$ handles with zero volume.  After $g-2$ steps we obtain a translation surface, say $(Y,\xi)$, with period character $\chi$ and spin parity $\varphi(\xi)$ equal to $\varphi(\omega)$. Since the initial structure $(X,\omega)$ can be realized with either even or odd parity, the desired conclusion follows in this case.

\medskip

\paragraph{\textit{$\chi$ is discrete of rank one.}} We are left to consider discrete representations of rank one that require a  deeper discussion. We can assume $\chi\colon\textnormal{H}_1(S_{g,2},\,\Z)\longrightarrow \Z$ without loss of generality. Recall that realizing such a representation in a certain stratum $\mathcal{H}_g(2m_1,\dots,2m_k;-1,-1)$ depends only on whether the condition
\begin{equation}\label{eq:ratcond}
 \max\{ 2m_1,\dots, 2m_k\} <\,w= \chi(\gamma)\in\Z
\end{equation}

\noindent holds, where $\gamma$ is a simple loop around a puncture. 

\medskip

\noindent In the case $w>2g$, the auxiliary representation $\rho$ defined as in \eqref{eq:auxsecondtype3} can be realized as in \S\ref{par:easy} as the holonomy of some translation surface with poles and with prescribed parity. On the other hand, this latter construction can be done as in \S\ref{par:hypcol}. Therefore, we can proceed exactly as in \S\ref{par:realcolnotdis}. 

\medskip

\noindent Let us assume $w\le 2g$ and let $\mathcal{H}_g(2m_1,\dots,2m_k;-1,-1)$ be a stratum such that the condition \eqref{eq:ratcond} holds. It is easy to observe that $m_1+\cdots+m_k=g$ for this stratum; therefore $m_i\le g$ and the equality holds if and only if $k=1$. Without loss of generality, assume that
\begin{equation}\label{eq:zerosordering}
    2m_1\ge 2m_2\ge \cdots \ge 2m_k
\end{equation}

\noindent holds. Consider the auxiliary representation $\rho\colon\textnormal{H}_1(S_{m_1,\,2},\,\Z)\longrightarrow \Z$ defined as follows
\begin{equation}
    \rho(\alpha_i)=\rho(\beta_i)=1, \,\,\,\text{ for } i=1,\dots,m_1, \,\,\,\text{ and }\,\,\rho(\delta_1)=\rho(\delta_2^{-1})=w.
\end{equation}

\noindent Since $w>2m_1$, then $\rho$ can be realized as the period character of some translation surface, say $(X_1,\omega_1)$, in the stratum $\mathcal{H}_{m_1}(2m_1;-1,-1)$ with prescribed parity, see paragraph \S\ref{par:hard}. If $k=1$ we are done; otherwise we proceed as follows.  

\smallskip

\noindent From paragraph \S\ref{par:hard}, we recall that $(X_1,\omega_1)$ is obtained by gluing two infinite half-strips $\mathcal S_1, \mathcal{S}_2\subset \C$. Let us focus on $\mathcal S_1$. This region is bounded by a segment $s^-$ of length $w$ and two half-rays $r_1^+$ and $r_2^-$. The sign as usual denote on which side the region is bounded according to their orientation. Since $w>2m_2$, the interior of $\mathcal S_1$ contains a segment, say $s_2$, of length $2m_2$ and parallel to $s^-$. On the other hand, the interior of $\mathcal S_1$ is embedded in $(X_1,\omega_1)$ and hence there is an isometric copy of $s_2$ inside $(X_1,\omega_1)$. With a little abuse of notation, this latter is also denoted by $s_2$. Next, we divide $s_2\subset (X_1,\omega_1)$ into $2m_2$ sub-segments, say $s_{2,\,1},s_{2,\,2},\dots,s_{2,\,2m_2-1},s_{2,\,2m_2}$, each of length $1$. Slit all of them and re-glue as follows: $s_{2,\,2i-1}^-$ is identified with $s_{2,\,2i}^+$ and $s_{2,\,2i-1}^+$ is identified with $s_{2,\,2i}^-$. The resulting space, say $(X_2,\omega_2)$, is a surface of genus $m_1+m_2$. As a consequence of Lemma \ref{lem:spininv2}, $(X_1,\omega_1)$ and $(X_2,\omega_2)$ have the same parity. Again, if $k=2$ we are done, otherwise there is always a segment, say $s_3$ of length $2m_3$, that lies in the interior of $\mathcal S_1$.  If this is the case, we proceed as we have just done. After $k$ steps, we obtain a surface $(X_k,\omega_k)=(Y,\xi)$ of genus $m_1+\cdots+m_k=g$ with period character $\chi$. Moreover, since bubbling a handle with zero volume does not alter the spin parity, see Lemma \ref{lem:spininv2}, the resulting structure has the same parity as $(X_1,\omega_1)$. This latter, according to Section \S\ref{sec:hgchyp} can be realized with even or odd parity, see also Remark \ref{rmk:specialgenus2strata}, hence Proposition \ref{prop:hgspin} also holds for discrete representations of rank one. Notice that the same argument would have been valid if we have chosen $\mathcal S_2$ in place of $\mathcal S_1$. Since there are no other cases to consider, this concludes the proof of Proposition \ref{prop:hgspin} and indeed the proof of Theorem \ref{mainthm} which is specific for non-trivial representations.

\medskip

\section{Meromorphic exact differentials}\label{sec:trirep}

\noindent We finally consider the trivial representation and we aim to prove Theorem \ref{thm:mainthm2}. On a compact Riemann surface $\overline{X}$, any non-constant rational function $f\colon\overline{X}\longrightarrow \cp$ yields a finite degree branched covering and the meromorphic differential $\omega=d\!f$ has trivial absolute periods, \textit{i.e.} $\omega$ determines a trivial period character. Let us denote by $X=\overline{X}\setminus\{\textnormal{ poles of } \omega\,\}$. Then the couple $(X,\omega)$ is a translation surface with poles in the sense of Definition \ref{tswp}. Conversely, if $(X,\omega)$ is a translation surface with poles on $S_{g,n}$ and with zero absolute periods, then the developing map, see Section \S\ref{sec:tswp}, boils down to a holomorphic mapping $X\longrightarrow \C$ that extends to a rational function $f\colon\overline{X}\longrightarrow \cp$ and $\omega=d\!f$.

\begin{defn}\label{def:exactdiff}
A meromorphic differential $\omega$ on a compact Riemann surface is called an  \textit{exact differential} if all absolute periods of $\omega$ are equal to zero. On a Riemann surface $X$ of finite type $(g,n)$ we say that a holomorphic differential $\omega$ with finite poles at the punctures is \textit{exact} if all absolute periods are zero and all poles have zero residue.
\end{defn}

\subsection{Bubbling handles with trivial periods}\label{ssec:gluingtrihand} We describe here how to glue handles with trivial periods. More precisely, we provide a surgery to add a handle with trivial periods on a genus zero differential in order to obtain a genus one differential with prescribed rotation number, see \S\ref{sssec:tripresrot}. We next provide an alternative construction, see \S\ref{sssec:trialternhand}, which will be useful later on. 
\smallskip

\noindent Let $(X,\omega)$ be any translation structure on a surface $S_{g,n}$ and let $\textnormal{dev}\colon\widetilde{S}_{g,\,n}\longrightarrow \C$ be its developing map. We introduce the following terminology.

\begin{defn}[$m$-pods and twins]\label{def:twins}
On a translation surface $(X,\omega)$, let $P$ be any branch point of order $m$. An embedded $m$-pod at $P$ is a collection of $m+1$ embedded paths $c_i\colon [\,0,\,1\,]\longrightarrow (X,\omega)$, for $i=0,\dots,m$, which meet exactly at $P$, each of which is injectively developed, and all of which overlap once developed, \textit{i.e.} there is a determination of the developing map around $c_0\,\cup\,\cdots\,\cup\,c_m$ which injectively develop $c_0,\,c_1,\dots,c_m$ to the same arc $\widehat{c}\subset \C$. For any $2\le k\le m$, the paths $c_{i_1},\dots,c_{i_k}$ are call \textit{twin paths}. For any pair $c_i,\,c_j$ we may notice that the angle at $P$ between them is a multiple of $2\pi$.
\end{defn}

\noindent The following technical lemma is straightforward and the proof is left to the reader.

\begin{lem}\label{lem:techlemtwins}
On a translation surface $(X,\omega)$, let $P$ be any branch point of order $m$ and let $B_{4\varepsilon}(P)$ be an open metric ball centered at $P$. Break $P$ into two zeros, say $P_1$ and $P_2$ of orders $m_1$ and $m_2$ respectively, such that $P_1$ and $P_2$ are joined by a saddle connection, say $s$, of length $\epsilon$. Then there are $m_1$ paths, say $c_1,\dots,c_{m_1}$, all leaving from $P_1$, such that $s$ and $c_i$ are twins for every $i=1,\dots,m_1$. Similarly, there are $m_2$ paths, say $c_{m_1+1},\dots,c_m$, all leaving from $P_2$, such that $c_{m_1+j}$ and $s$ are twins for every $j=1,\dots,m_2$.
\end{lem}

\smallskip

\subsubsection{\textit{Handles with trivial periods}}\label{sssec:tripresrot} Let $(X,\omega)$ be a translation surface with poles, let $P_1,\,P_2\in(X,\omega)$ be two branch points of orders $m_1,\,m_2$, and let $s$ be the resulting saddle connection of length $\varepsilon$. Fix an orientation on $s$, say from $P_1$ to $P_2$. The saddle connection $s$ has $m_1$ twins leaving from $P_1$ and it determines on each of them an obvious outbound orientation. In the same fashion, $s$ has $m_2$ twin paths leaving from $P_2$ and it determines on each of them an obvious inbound orientation to $P_2$. Let $c_2$ be any twin of $s$ leaving from $P_2$ and let $\alpha$ be a simple closed curve around $c_2$. Since $\alpha$ winds around $P_2$ but not around the other branch point, it has index equal to $m_2+1$ because $P_2$ has order $m_2$.

\smallskip

\noindent Let $c_1$ be the twin of $s$ leaving from $P_1$ that forms an angle of $2\pi$ on its right. Define $c_2$ as the twin of $s$ leaving from $P_2$ that forms an angle of $2\pi(r+1)$ on its left. Next, slit both $c_1$ and $c_2$ and we denote the resulting sides $c_1^{\pm}$ and $c_2^{\pm}$ where the signs are taken according to our convention. Let $\beta$ be a smooth path starting from the mid-point of $c_1^-$ to the mid-point of $c_2^+$ that crosses $s$ and it does not contain any branch point in its interior (hence it misses both $P_1$ and $P_2$). Then, identify $c_1^+$ with $c_2^-$ and similarly identify $c_1^-$ with $c_2^+$.

\smallskip 

\noindent If the initial translation surface $(X,\omega)$ has genus $g$, the resulting surface $(Y,\xi)$ has genus $g+1$. The smooth path $\beta$ closes up to a simple closed  curve such that along with $\alpha$ they provide a basis of handle generators, see Definition \ref{def:syshandle} for the newborn handle. It can be checked that $\beta$ has index equal to $r$ because it turns of an angle $2\pi$ around $P_1$ counterclockwise and then it turns of an angle $2\pi(r+1)$ around $P_2$ clockwise. See Figure \ref{fig:addtrihandlepresrot}. 

\begin{figure}[!ht] 
\centering
\begin{tikzpicture}[scale=1, every node/.style={scale=0.85}]
\definecolor{pallido}{RGB}{221,227,227}

    \pattern [pattern=north east lines, pattern color=pallido] (-6,3)--(7,3)--(7,-3)--(-6,-3)--(-6,3);
    
    \draw[thin, black]  (2, 0) ++(0:3mm) arc (00:180:3mm);
    \draw[thin, black]  (-2, 0) ++(0:3mm) arc (0:-180:3mm);
    
    \draw[->, thick, red] (-2,0)--(-0.5,0);
    \draw[thick, red] (-0.5,0)--(2,0);
    
    \draw[thin, violet] (-3.5,0)--(-3.5, -1);
    \draw[thin, violet] (-3, -1.5) arc [start angle =270, end angle=180 , radius = 0.5];
    \draw[thin, violet] (-3, -1.5)--(-0.5, -1.5);
    \draw[thin, violet]  (-0.5, -1) ++(0:5mm) arc (0:-90:5mm);
    \draw[thin, violet] (0,-1)--(0,1);
    \draw[thin, violet]  (0.5, 1) ++(180:5mm) arc (180:90:5mm);
    \draw[thin, violet] (0.5,1.5)--(3,1.5);
    \draw[thin, violet] (3, 1) ++(90:5mm) arc (90:0:5mm);
    \draw[thin, violet] (3.5, 1)--(3.5, 0);
    
    \draw[thin, blue] (2,1)--(5,1);
    \draw[thin, blue]  (5, 0) ++(-90:10mm) arc (-90:90:10mm);
    \draw[thin, blue] (2,-1)--(5,-1);
    \draw[thin, blue] (2, 0) ++(90:10mm) arc (90:270:10mm);
    
    \fill [white] (-5,0) to [out=20, in=160] (-2,0) to [out=200, in=340] (-5,0);
    \fill [white] (2,0)  to [out=340, in=200] (5,0) to [out=160, in=20] (2,0);
    
    \draw[ultra thin, black] (-5,0) to [out=20, in=160] (-2,0);
    \draw[ultra thin, black] (-5,0) to [out=340, in=200] (-2,0);
    \draw[ultra thin, black] (2,0)  to [out=340, in=200] (5,0);
    \draw[ultra thin, black] (2,0)  to [out=20, in=160] (5,0);
    
    \draw (5,0) circle (2.5pt);
    \fill [white] (5,0) circle (2.5pt);
    \fill [black] (-5,0) circle (2.5pt);
    
    \fill [black] (2,0) circle (2.5pt);
    \fill [white] (-2,0) circle (2.5pt);
    \draw (-2,0) circle (2.5pt);
    
    \node at (-0.5, 0.25) {$s$};
    \node at (-4.25, 0.5) {$c_1^+$};
    \node at (-4.25, -0.5) {$c_1^-$};
    \node at (4.25, 0.5) {$c_2^+$};
    \node at (4.25, -0.5) {$c_2^-$};
    \node at (3.5, -0.75) {$\alpha$};
    \node at (3.5, -1.25) {$m_2+1$};
    \node at (-2, -1.25) {$\beta$};
    \node at (-2, -1.75) {$r$};
    \node at (2, 0.5) {$2\pi(r+1)$};
    \node at (-2, 0.5) {$P_1$};
    \node at (2, -0.5) {$P_2$};
    \node at (-2, -0.5) {$2\pi$};

\end{tikzpicture}
\caption{Adding a handle with trivial periods and handle generators with prescribed indices. The orange segment is a saddle connection joining two zeros of odd orders. The blue curve $\alpha$ has index $m$ whereas the violet curve $\beta$ has index $r$.}
\label{fig:addtrihandlepresrot}
\end{figure}
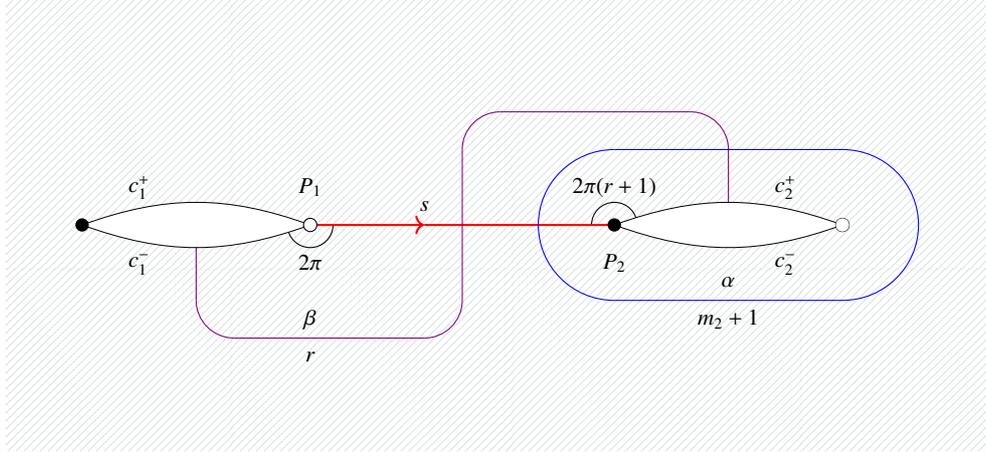

\noindent By denoting $P_1'$ the extremal point of $c_1$ other than $P_1$ and by $P_2'$ the extremal point of $c_2$ other than $P_2$, we can notice that they are both regular. Next identify $P_1$ with $P_2'$ and $P_2$ with $P_1'$. Once the identification is done, we get two branch points of angles $2\pi(m_1+2)$ and $2\pi(m_2+2)$. The following holds.

\begin{lem}\label{lem:addhandtrieven}
Let $(X,\,\omega)\in\mathcal{H}_0(m_1,\dots,m_{k-1}-1,\,m_k-1;\,-p_1,\dots,-p_n)$ be a genus zero differential with trivial periods. Let $P_{k-1}$ and $P_k$ be the zeros of orders $m_{k-1}-1$ and $m_k-1$ respectively and assume there is a saddle connection joining them. Let $(Y,\,\xi)$ be the translation surface obtained by bubbling a handle with trivial periods as described above. Then $(Y,\xi)$ is a genus one differential with rotation number equal to $\gcd(r,m_1,\dots,m_k,p_1,\dots,p_n)$.
\end{lem}

\begin{proof}
The fact that $(Y,\xi)$ is a genus one differential directly follows from the construction. By adopting the notation above,  $\alpha$ has index equal to $m_k$ and $\beta$ has index $r$. Then the desired conclusion follows.
\end{proof}

\smallskip

\subsubsection{\textit{Alternative construction}}\label{sssec:trialternhand} We now introduce an alternative way for adding a handle with trivial periods. We shall use this version later on for the inductive foundation in Section \S\ref{ssec:meroexdiffspin}. Let $(X,\omega)\in\Omega\mathcal{M}_{g,n}$ be a translation surface with poles and let $P\in(X,\omega)$ be a branch point of even order $2m$, \textit{i.e.} the angle at $P$ is $(4m+2)\pi$. Let $B_{4\varepsilon}(P)$ be an open ball of radius $4\varepsilon$ at $P$. Consider a $3$-pod at $P$ made of three geodesic twin paths $c_1,\,c_2,\,c_3$ all leaving from $P$ with length $\varepsilon$ and such that the angle between $c_1$ and $c_2$ and the angle between $c_2$ and $c_3$ are both $2\pi$. Then the angle between $c_3$ and $c_1$ is $(4m-2)\pi$. Let $Q_1,\,Q_2,\,Q_3$ denote the extremal regular points of $c_1,\,c_2,\,c_3$ respectively other than $P$. For each $i=1,2,3$, assume that $c_i$ is oriented from $P$ to $Q_i$. By slitting all of these paths, the branch point $P$ splits into three points $P_1,\,P_2,\,P_3$ and we get a surface of genus $g$ with piecewise geodesic boundary. We can assume that  $P_1,\,Q_1,\,P_2,\,Q_2,\,P_3,\,Q_3$ are cyclically ordered as shown in Figure \ref{fig:addtrihandleoddrot}. Then the corner angles at all vertices but $P_3$ is $2\pi$ and the angle at $P_3$ is $(4m-2)\pi$.

\smallskip

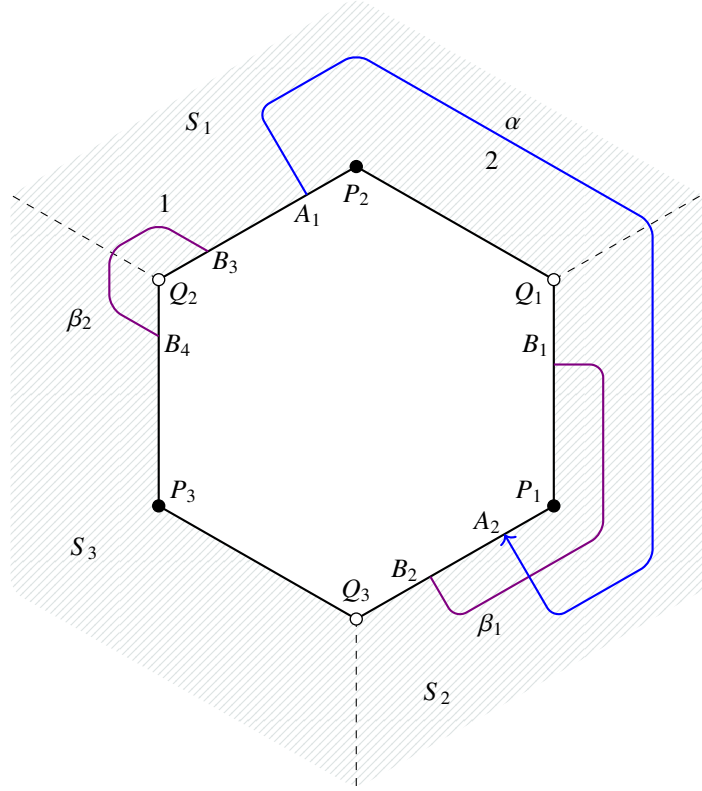
\begin{figure}[!ht]
    \centering
    \begin{tikzpicture}[scale=0.75, every node/.style={scale=1}]
    \definecolor{pallido}{RGB}{221,227,227}
    \pattern [pattern=north east lines, pattern color=pallido, rounded corners=5pt]  (30:7cm) -- (90:7cm) -- (150:7cm) -- (210:7cm) -- (270:7cm)--(330:7cm);
    
    \fill [white] (30:4cm) -- (90:4cm) -- (150:4cm) -- (210:4cm) -- (270:4cm)--(330:4cm);

    \newdimen\R
    \R=4cm
    \draw[thick] (330:\R) foreach \x in {30,90,...,330} { -- (\x:\R) };
    \foreach \x in {30,150,270} { 
        \draw[dashed] (\x:\R) -- (\x:{\R+3cm});
        \draw[line width=0.5pt,black,fill=white] (\x:\R) circle (3pt);
    }
    \foreach \x in {90,210,330} { 
        \draw[line width=0.5pt,black,fill=black] (\x:\R) circle (3pt);
    }

        \begin{scope}[shift=(150:\R)]
        \draw[thick, violet, rounded corners=5pt] 
            (30:1cm) -- (90:1cm) -- (150:1cm) -- (210:1cm) -- (270:1cm);
    \end{scope}
    
    \begin{scope}[shift=(330:\R)]
        \draw[thick, violet, rounded corners=5pt] 
            (210:2.5cm) -- ([shift=(330:1cm)]210:3cm) -- (330:1cm) -- ([shift=(330:1cm)]90:3cm) -- (90:2.5cm);
    \end{scope}

    \begin{scope}[shift=(30:\R)]
        \draw[->, thick, blue, rounded corners=5pt] 
            ([shift=(210:1cm)]150:4cm) -- ([shift=(150:2cm)]150:4cm) -- ([shift=(90:2cm)]150:4cm) -- (30:2cm) -- ([shift=(330:2cm)]270:4cm) -- ([shift=(270:2cm)]270:4cm) -- ([shift=(210:1cm)]270:4cm);
    \end{scope}
    
    \node at (60:4.75cm) {$2$};
    \node at (135:4.75cm) {$1$};
    
    \node at (120:5.5cm) {$S_1$};
    \node at (285:5.5cm) {$S_2$};
    \node at (210:5.5cm) {$S_3$};
    
    \node at (60:5.5cm) {$\alpha$};
    \node at (300:4.75cm) {$\beta_1$};
    \node at (165:5cm) {$\beta_2$};
    
    \node at (105:3.25cm) {$A_1$};
    \node at (315:3.25cm) {$A_2$};
    \node at (015:3.25cm) {$B_1$};
    \node at (285:3.25cm) {$B_2$};
    \node at (135:3.25cm) {$B_3$};
    \node at (165:3.25cm) {$B_4$};
    
    \node at (030:3.5cm) {$Q_1$};
    \node at (090:3.5cm) {$P_2$};
    \node at (150:3.5cm) {$Q_2$};
    \node at (210:3.5cm) {$P_3$};
    \node at (270:3.5cm) {$Q_3$};
    \node at (330:3.5cm) {$P_1$};
    
    \end{tikzpicture}
    \caption{Adding a handle with trivial periods. The blue curve $\alpha$ has index $2$ and violet curve $\beta$ has index $0$.}
    \label{fig:addtrihandleoddrot}
\end{figure}

\smallskip

\noindent Notice that, by prolonging each $c_i$ to a geodesic ray $r_i$ of length $4\epsilon$, the open ball $B_{4\varepsilon}(P)$ is divided into  three sectors, say $S_1,\,S_2,\,S_3$. We assume $S_i$ is the sector containing $P_i$ after slitting $c_1,\,c_2,\,c_3$. We next consider the following smooth arcs on the surface with boundary obtained after slitting. See Figure \ref{fig:addtrihandleoddrot}.

\begin{itemize}
    \item Let $\alpha$ be a smooth arc joining a point $A_1\in{P_2\,Q_2}$ at distance $\varepsilon/3$ from $P_2$ and a point $A_2\in{P_1\,Q_3}$ at distance $\varepsilon/3$ from $P_1$. We can take this arc such that it lies in the sectors $S_1$ and $S_2$ without crossing the sector $S_3$.
    \smallskip
    \item Let $\beta_1$ be a smooth arc joining a point $B_1\in P_1\,Q_1$ at distance $2\varepsilon/3$ from $P_1$ and a point $B_2\in P_1\,Q_3$ at distance $2\varepsilon/3$ from $P_1$. We can assume that $\beta_1$ lies entirely in the sector $S_1$. Notice that, by construction, $\alpha$ and $\beta_1$ cross. 
    \smallskip
    \item Finally, let $\beta_2$ be a smooth arc joining a point $B_3\in P_3\,Q_3$ at distance $2\varepsilon/3$ from $P_2$ and a point $B_4\in P_3\,Q_2$ at distance $2\varepsilon/3$ from $P_3$. We can take $\beta_2$ such that it does not cross the sector $S_1$ and is disjoint from $\alpha$.
\end{itemize}

\smallskip

\noindent We next glue the edges of the hexagonal boundary as follows. First, identify $P_1\,Q_1$ with $P_3\,Q_2$, then $P_2\,Q_1$ with $P_3\,Q_3$ and, finally, $P_2\,Q_2$ with $P_1\,Q_3$. The resulting surface, say $(Y,\xi)$ has genus $g+1$. By construction, the arc $\alpha$ closes up to a simple closed curve of index $2$. Moreover, $B_1$ identifies with $B_4$ and $B_2$ identifies with $B_3$. Therefore, $\beta_1\cup\beta_2$ yields a simple closed curve $\beta$ of index $2$. The pair of curves $\{\alpha,\,\beta\}$ determines a set of handle generators for the newborn handle. The points $P_1,\,P_2,\,P_3$ are identified according to our construction and they determine a branch point of angle $(4m+2)\pi$. Similarly, $Q_1,\,Q_2,\,Q_3$ are also identified and they determine a branch point of angle $6\pi$ because the extremal points are all assumed to be regular. The following holds.

\begin{lem}\label{lem:addhandtriodd}
Let $(X,\,\omega)\in\mathcal{H}_g(2m_1,\dots,2m_k;\,-2p_1,\dots,-2p_n)$ be a translation surface with poles and trivial periods. Let $(Y,\,\xi)$ be the translation surface obtained by adding a handle with trivial periods as described above. Then $(Y,\xi)$ also belongs to a stratum of even type, it has trivial periods, and its parity is given by $\varphi(\xi)=\varphi(\omega)+1$.
\end{lem}

\begin{proof}
The fact that $(Y,\xi)$ belongs to a stratum of even type directly follows from the construction. By adopting the notation above, since $\alpha$ has index $2$ and $\beta$ has index $0$, then 
\[
\left(\textnormal{Ind}(\alpha)+1\right)\left(\textnormal{Ind}(\beta)+1\right)=3\cdot 1\equiv 1\,\,(\text{mod}\,2).
\]
Therefore $\varphi(\xi)=\varphi(\omega)+1$.
\end{proof}

\smallskip

\noindent We shall use the surgeries above to realize translation surfaces with poles and trivial periods in a given stratum with prescribed parity.

\subsection{Genus zero exact differentials}\label{ssec:genzero} 
We briefly recall the strategy adopted in \cite[Section \S8]{CFG} to realize the trivial representation as the period character of some exact differentials on $\cp$ with prescribed zeros and poles. The aim is to make the autonomous explanation of the next subsections as self-contained as possible.

\smallskip

\noindent Suppose we want to realize the trivial representation in a stratum $\mathcal{H}_0(m_1,\dots,m_k;-p_1,\dots,-p_n)$. We consider $n$ translation surfaces $(\C,\omega_i)$ for $1\leq i\leq n$, where the differential $\omega_i$ has a pole of order $p_i\ge2$ at the infinity and a zero of order $p_i-2$ at $0\in\C$. We consider some, possibly all, of these translation surfaces and we glue them, by slitting along a segment of finite length, to define a sequence of translation surfaces $(Y_j,\,\xi_j)$. Here, each $(Y_j,\,\xi_j)$ is a sphere with an exact meromorphic differential with poles of orders $p_1, \ldots, p_t$ for some $1\leq t \leq n$ and zeros of orders $m_1, \ldots, m_{s-1}$ and  $\widetilde{m}_s$ for some $1\leq s\leq k$ where $\widetilde{m}_s \geq m_s$ and $\widetilde{m}_s \leq m_s+ m_{s+1} + \cdots+m_k$. Since the degree of any meromorphic differential on the sphere is $-2$, we have
\begin{equation}\label{eq:87} 
    \sum_{l=1}^{s} m_l + (\,\widetilde{m}_s - m_s\,) = \sum_{i=1}^{t} p_i -2.
\end{equation}
\noindent The sequence $(Y_j,\,\xi_j)$ is constructed such that, as $j$ increases, $s$ and $t$ increase until $t=n$. The process runs until all the $(\C,\omega_i)$'s are exhausted. We finally break the last zero to get the desired differential. 

\begin{rmk}\label{rmk:wealwaysbreak}
The key observation here is for the translation surface $(Y_1,\,\xi_1)$ obtained after the first step with poles $p_1,\dots,p_t$ and two zeros $m_1, \widetilde{m}_2$, where $\widetilde{m}_2\ge m_2$. In the case $t=n$, then only the following cases arise:
\begin{itemize}
    \item $\widetilde{m}_2=m_2$ and hence the desired genus zero differential belongs to $\mathcal{H}_0(m_1,m_2;-p_1,\dots,-p_n)$; otherwise
    \smallskip
    \item $\widetilde{m}_2>m_2$ and hence the desired genus zero differential belongs to a certain stratum $\mathcal{H}_0(\kappa;-\nu)$ where $\kappa=(m_1,m_2,\dots,m_k)$ and $\nu=(p_1,\dots,p_n)$. In this case, the zero of order $\widetilde{m}_2$ is broken into $k-1$ zeros of orders $m_2,\dots,m_k$. 
\end{itemize}
\end{rmk}

\noindent This constriction can be performed in such a way that if $P_i$ and $P_j$ are two zeros such that $|\,i-j\,|=1$, then there exists at least one saddle connection joining $P_i$ and $P_j$, that is a geodesic segment with no zeros in its interior. Finally, we can arrange the zeros so that two different zeros have different developed images. As a consequence, if $s_i$ is the saddle connection joining $P_i$ with $P_{i+1}$, all twins of $s_i$ leaving from $P_i$ do not contain any zero other than $P_i$ and $P_{i+1}$ themselves. See \cite[Section \S8.3.4]{CFG} for more details. 

\subsection{Trivial periods and prescribed rotation number}\label{ssec:meroexdiffgenone}  The aim of this section is to prove the following proposition concerning genus one differentials.

\begin{prop}\label{prop:genusonetri}
Suppose the trivial representation $\chi\colon\shomolzon\longrightarrow \C$ can be realized in a stratum $\mathcal{H}_1(m_1,\dots,m_k;\,-p_1,\dots,-p_n)$. Then it can be realized in each connected component of the same stratum with the only exceptions being the strata $\mathcal{H}_1(2,2;\,-4)$, $\mathcal{H}_1(2,2,2;\,-2,-2,-2)$  and $\mathcal{H}_1(3,3;-3,-3)$.
\end{prop}

\noindent The proof of this proposition is based on the construction done before in subsection \S\ref{sssec:tripresrot}. We recall for the reader's convenience that the connected components of $\mathcal{H}_1(m_1,\dots,m_k;\,-p_1,\dots,-p_n)$ are distinguished by the divisors of $\gcd(m_1,\dots,m_k,p_1,\dots,p_n)$. Clearly, if the greatest common divisor is one there is nothing to prove and the realization is subject to \cite[Theorem B]{CFG}. Throughout the present subsection, for a generic stratum let $p=p_1+\cdots+p_n\,=\,m_1+\cdots+m_k$ and let $r$ be a divisor of $\gcd(m_1,\dots,m_k;p_1,\dots,p_n)$. We shall need to distinguish several cases according to certain mutual relationships among $p,\,r,\,n$ and $k$.

\subsubsection{Strata with two zeros}\label{ssec:tricasetwozerostrata} It is not hard to show that the trivial representation $\chi:\shomolzon\longrightarrow \C$ cannot be realized in every stratum with a single zero of maximal order because the Hurwitz type inequality \eqref{eq:triconstr} never holds for these strata. Therefore strata with two zeros appear as the simplest minimal strata in which the trivial representation can be realized. Recall that our aim is to realize $r$ as the prescribed rotation number for a trivial representation. The first case we consider is the following

\medskip 

\paragraph{\textit{All poles have order $r$}}\label{par:allsameordertritwozer} Let us realize the trivial representation in the stratum $\mathcal{H}_1(m_1,m_2;-r^n)$. In this case $m_1+m_2=nr$ and $m_1,\,m_2$ are both divisible by $r$ -- otherwise the stratum would be connected and there is nothing to prove. Moreover, $m_1$ and $m_2$ are also subject to the Hurwitz type inequality \eqref{eq:triconstr}; hence
\begin{equation}\label{eq:lowerboundr}
n r=m_1+m_2\le 2n r-2n-2 \qquad \Longleftrightarrow \qquad r\ge3.
\end{equation}

\begin{rmk}
As a cross-check, the reader may observe with a direct computation that the trivial representation cannot be realized in every stratum $\mathcal{H}_1(m_1,m_2; -2^n)$. Hence the inequality $r\ge3$ above is sharp.
\end{rmk}

\smallskip

\noindent In this case we first realize the trivial representation in $\mathcal H_0(m_1-1,m_2-1;-r^n)$ by using the saddle connection configuration description in \cite[Section \S9.3]{EMZ}. More precisely, we choose $n$ pairs of positive integers, say $(a_i, b_i)$, such that $a_i + b_i = r$ for $i = 1,\dots,n$, and moreover, $a_1 +\cdots + a_n = m_1$ and $b_1 + \cdots+ b_n = m_2$. 

\begin{rmk}\label{rmk:excrit}
The existence of these pairs $(a_i,b_i)$ can be explained as follows. We first notice that, the Hurwitz type inequality \eqref{eq:triconstr} for $\mathcal H_1(m_1, m_2; -r^n)$ implies $n < m_1, m_2 < (r-1)n$. We want pairs of positive integers $a_i, b_i$ to satisfy the conditions above.  In particular, $0 < a_i < r$ ensures $b_i$ to be positive. The extreme case happens when $m_1 = n+1$, where we can take $a_1 = 2$ and $a_2 = \cdots = a_n = 1$ -- recall that $r$ is at least $3$, see \eqref{eq:lowerboundr} above.  If $m_1$ increases, we increase $a_2$, until $a_2$ reaches $r-1$, then we increase $a_3$, and so on.
\end{rmk}

\noindent For every $i=1,\dots,n$, consider a copy of $(\C,\, z^{r-2}dz)$ and fix any vector $c\in\mathbb C^*$. Let $(\cp,\omega_i)$ be the translation surface with poles obtained by breaking the zero of $(\C,\, z^{r-2}dz)$ into two zeros, say $P_i$ and $Q_i$, of orders $a_i-1$ and $b_i-1$ respectively so that the resulting saddle connection joining them, say $c_i$, is parallel to $c$ with the same length, \textit{i.e.} $|c|=|c_i|$. See Figure \ref{genuszerotri} below.

\begin{figure}[ht!]
    \centering
    \begin{tikzpicture}[scale=1, every node/.style={scale=0.8}]
    \definecolor{pallido}{RGB}{221,227,227}

    \foreach \x [evaluate=\x as \coord using  4.5*\x] in {0,1,2.25} 
    {
    \pattern [pattern=north west lines, pattern color=pallido]
    (\coord, -1.25)--(\coord+4, -1.25)--(\coord+4, 1.25)--(\coord, 1.25)--(\coord,-1.5);
    
    \draw[thin, red] (\coord+1,0)--(\coord+3,0);
    
    \fill [black] (\coord+1,0) circle (1.5pt);
    \fill [black] (\coord+3,0) circle (1.5pt);
    }
    
    \foreach \x [evaluate=\x as \coord using  4.5*\x] in {0} 
    {
    \draw[dashed, black, thin] (\coord+1,0)--(\coord,0);
    \draw[dashed, black, thin]
    (\coord+1,0)--(\coord, 0.5);
    \draw[dashed, black, thin]
    (\coord+1,0)--(\coord, -0.5);
    
    \node at (\coord, 0.75) {$r-2$};
    \node at (\coord+1, -0.375) {$a_1-1$};
    \node at (\coord+3, -0.375) {$b_1-1$};
    \node at (\coord+2, 0.375) {$c_1^+$};
    }
    
    \foreach \x [evaluate=\x as \coord using  4.5*\x] in {1} 
    {
    \draw[dashed, black, thin] (\coord+3,0)--(\coord+4,0);
    \draw[dashed, black, thin] (\coord+3,0)--(\coord+4,0.5);
    \draw[dashed, black, thin] (\coord+3,0)--(\coord+4,-0.5);
    
    \node at (\coord+4, 0.75) {$r-2$};
    \node at (\coord+1, -0.375) {$a_2-1$};
    \node at (\coord+3, -0.375) {$b_2-1$};
    
    }
    
    \foreach \x [evaluate=\x as \coord using  4.5*\x] in {2.25} 
    {
    \draw[dashed, black, thin] (\coord+1,0)--(\coord,0.5);
    \draw[dashed, black, thin] (\coord+1,0)--(\coord,-0.5);
    \draw[dashed, black, thin] (\coord+3,0)--(\coord+4,0.5);
    \draw[dashed, black, thin] (\coord+3,0)--(\coord+4,-0.5);
    
    \node at (\coord, 0.75) {$k$};
    \node at (\coord+4, 0.75) {$r-k$};
    \node at (\coord+1, -0.375) {$a_n-1$};
    \node at (\coord+3, -0.375) {$b_n-1$};
    \node at (\coord+2, -0.375) {$c_n^-$};
    }
    
    \node at (9.35, 0) {$\dots$};

    \end{tikzpicture}
    \caption{Realization of a genus zero differential in $\mathcal H_0(m_1-1,m_2-1;\,-r^n)$.}
    \label{genuszerotri}
\end{figure}
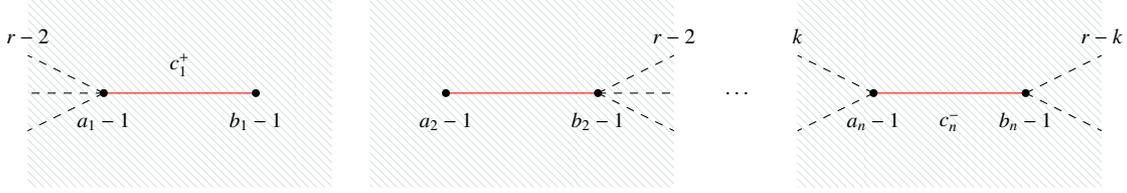

\noindent Slit every $(\cp,\omega_i)$ along $c_i$ and denote by $c_i^{\pm}$ the resulting edges according to our convention. Identify $c_i^-$ with $c_{i+1}^+$ for every $i=1,\dots,n-1$ and finally glue $c_n^-$ with $c_1^+$. The resulting space is now a pair $(\cp, \omega)$ where $\omega$ is a meromorphic differential with two zeros, say $P_o$ and $Q_o$, of orders $m_1-1$ and $m_2-1$ and $n$ poles of order $r$. Let $s_1$ be the saddle connection resulting from the identification of $c_1^+$ with $c_n^-$. We extend the notation by defining $s_{i+1}$ to be the saddle connection resulting from the identification of $c_i^-$ with $c_{i+1}^+$ for every $i=1,\dots,n-1$. In this space, we define the \textit{beginning saddle connection} as the saddle connection resulting from the identification of $c_1^+$ with $c_n^-$. This is purely a convention as in principle any other saddle connection arising from the identification of $c_i^+$ with $c_{i+1}^-$ could be equally entitled as the beginning saddle connection. We shall use this terminology later on. 

\smallskip

\noindent In order to bubble the desired handle, we shall need to slit two paths, say $s_o$ and $s_{n+1}$, at $P_1$ and $Q_n$ respectively with the same holonomy as the saddle connections, \textit{i.e.} we want to find $s_o$ and $s_{n+1}$ such that any pair $s_i,\,s_j$ in the collection $s_o,\,s_1,\dots,s_n,\,s_{n+1}$ are twins in $(\cp,\,\omega)$. We need to make sure that $s_o$ and $s_{n+1}$ exists and they do not coincide with any pre-existing saddle connection in $(\cp, \omega)$. For this purpose, we need to choose the pairs $(a_i,b_i)$ in a proper way. Let $A_i = a_1 + \cdots + a_i$ and $B_i = b_1 + \cdots + b_i$, for $i = 1,\dots, n$. The gist of idea is to find suitable pairs of positive integers $(a_i, b_i)$ as above such that there exist two positive integers, say $A$ and $B$, satisfying that $A \equiv B\, (\text{mod}\, r)$ and $A \neq A_i$ and $B \neq B_i$ for all $i=1,\dots,n$. Under these conditions, we pick $s_o$ that forms an angle $2\pi\, A$ away from the beginning saddle connection at $P_o$. This is equivalent to pick $s_o$ that forms an angle $2\pi\, A$ away from $c_1^+$ in $(\cp,\,\omega_1)$. Similarly, we pick $s_{n+1}$ that forms an angle $2\pi\, B$ away from the beginning saddle connection at $Q_o$, \textit{i.e.} that forms an angle $2\pi\, B$ away from $c_n^-$ in $(\cp,\,\omega_n)$. Then the two slits $s_o$ and $s_{n+1}$ do not coincide with any existing saddle connections on $(\cp,\,\omega)$ by the assumption on $A$ and $B$. More precisely, we consider first the extreme case $m_1 = n+1$, see Remark \ref{rmk:excrit}. Then we can choose the pairs $(a_i, b_i)$ so that $a_1 = 2$ and $a_2 = \cdots = a_n = 1$, $b_1 = r-2$ and $b_2 = \cdots = b_n = r-1$. Since $r - 1 \geq 2$, then $A = B = 1$ works. If $m_1$ increases, we increase $a_1$ until it reaches $r-1$, then we increase $a_2$, and so on.  For all of them we can use $A = B = 1$. Once we reach $a_1 = \cdots = a_{n-1} = r-1$, that means $b_1 = \cdots = b_{n-1} = 1$, $m_1$ might still increase, and then we need to increase $a_n$, \textit{i.e.}, decrease $b_n$, until it drops to $2$ -- in this case $m_2=n+1$ -- and $A = B = 1$ still works.  

\smallskip

\begin{figure}
    \centering
    \begin{tikzpicture}[scale=0.85, every node/.style={scale=0.8}]
    
    \definecolor{pallido}{RGB}{221,227,227}

    \foreach \x [evaluate=\x as \coord using  4*\x] in {0,2} 
    {
    \pattern [pattern=north west lines, pattern color=pallido]
    (\coord-3.75, -2)--(\coord-3.75, 2)--(\coord+3.5, 2)--(\coord+3.5, -2)--(\coord-3.75,-2);
    
    \draw[thin, red] (\coord-1,0)--(\coord+1,0);
    
    \draw[thin, orange] (-1,0)--(-3,0);
    \fill [black] (-3,0) circle (1.5pt);
    
    \fill [black] (\coord-1,0) circle (1.5pt);
    \fill [black] (\coord+1,0) circle (1.5pt);
    
    \draw[thin, blue, ->] (\coord, -1)--(\coord,+1);
    \draw[thin, blue] (\coord, 1) arc [start angle=0, end angle=90 , radius = 0.5];
    \draw[thin, blue] (\coord-0.5,1.5)--(\coord-2.75,1.5);
    \draw[thin, blue] (\coord-2.75, 1.5) arc [start angle=90, end angle=180 , radius = 0.5];
    \draw[thin, blue] (\coord-3.25, 1)--(\coord-3.25,-1);
    \draw[thin, blue] (\coord-3.25, -1) arc [start angle=180, end angle=270 , radius = 0.5];
    \draw[thin, blue] (\coord-2.75, -1.5)--(\coord-0.5, -1.5);
    \draw[thin, blue] (\coord-0.5, -1.5) arc [start angle=-90, end angle=0 , radius = 0.5];
    }

    \draw[thin, dashed, black] (-1,0)--(-3.75,1.5);
    \draw[thin, dashed, black] (-1,0)--(-3.75,-1.5);
    \draw[thin, dashed, black] (1,0)--(3.5,1.5);
    \draw[thin, dashed, black] (1,0)--(3.5,0.5);
    \draw[thin, dashed, black] (1,0)--(3.5,-0.5);
    \draw[thin, dashed, black] (1,0)--(3.5,-1.5);
    
    \draw[thin, dashed, black] (9,0)--(11.5,2);
    \draw[thin, dashed, black] (9,0)--(11.5,1);
    \draw[thin, dashed, black] (9,0)--(11.5,0);
    \draw[thin, dashed, black] (9,0)--(11.5,-1);
    \draw[thin, dashed, black] (9,0)--(11.5,-2);

    \foreach \x [evaluate=\x as \coord using  4*\x] in {0,2} 
    {
    \pattern [pattern=north west lines, pattern color=pallido]
    (\coord-3.75, -7)--(\coord-3.75, -3)--(\coord+3.5, -3)--(\coord+3.5, -7)--(\coord-3.75,-7);
    \draw[thin, red] (\coord-1,-5)--(\coord+1,-5);
    \fill [black] (\coord-1,-5) circle (1.5pt);
    \fill [black] (\coord+1,-5) circle (1.5pt);
    
    \draw[thin, blue, ->] (\coord, -6)--(\coord,-4);
    \draw[thin, blue] (\coord, -4) arc [start angle=0, end angle=90 , radius = 0.5];
    \draw[thin, blue] (\coord-0.5, -3.5)--(\coord-2.75,-3.5);
    \draw[thin, blue] (\coord-2.75, -3.5) arc [start angle=90, end angle=180 , radius = 0.5];
    \draw[thin, blue] (\coord-3.25, -4)--(\coord-3.25,-6);
    \draw[thin, blue] (\coord-3.25, -6) arc [start angle=180, end angle=270 , radius = 0.5];
    \draw[thin, blue] (\coord-2.75, -6.5)--(\coord-0.5, -6.5);
    \draw[thin, blue] (\coord-0.5, -6.5) arc [start angle=-90, end angle=0 , radius = 0.5];
    
    }
    
    \draw[thin, dashed, black] (1,-5)--(3.5,-3);
    \draw[thin, dashed, black] (1,-5)--(3.5,-4);
    \draw[thin, dashed, black] (1,-5)--(3.5,-5);
    \draw[thin, dashed, black] (1,-5)--(3.5,-6);
    \draw[thin, dashed, black] (1,-5)--(3.5,-7);

    \draw[thin, dashed, black] (9,-5)--(11.5,-3);
    \draw[thin, dashed, black] (9,-5)--(11.5,-4);
    \draw[thin, dashed, black] (9,-5)--(11.5,-5);
    \draw[thin, dashed, black] (9,-5)--(11.5,-6);
    \draw[thin, dashed, black] (9,-5)--(11.5,-7);
    
    \node at (0,-8) {$\vdots$};
    \node at (8,-8) {$\vdots$};
    
    \foreach \x [evaluate=\x as \coord using  4*\x] in {0,2} 
    {
    \pattern [pattern=north west lines, pattern color=pallido]
    (\coord-3.75, -13)--(\coord-3.75, -9)--(\coord+3.5, -9)--(\coord+3.5, -13)--(\coord-3.75,-13);
    \draw[thin, red] (\coord-1,-11)--(\coord+1,-11);
    
    \draw[thin, orange] (9,-11)--(10.625,-11.975);
    \fill [black] (10.625,-11.975) circle (1.5pt);
    
    \fill [black] (\coord-1,-11) circle (1.5pt);
    \fill [black] (\coord+1,-11) circle (1.5pt);
    
    \draw[thin, blue, ->] (\coord, -12)--(\coord,-10);
    \draw[thin, blue] (\coord, -10) arc [start angle=0, end angle=90 , radius = 0.5];
    \draw[thin, blue] (\coord-0.5,-9.5)--(\coord-2.75,-9.5);
    \draw[thin, blue] (\coord-2.75, -9.5) arc [start angle=90, end angle=180 , radius = 0.5];
    \draw[thin, blue] (\coord-3.25, -10)--(\coord-3.25,-12);
    \draw[thin, blue] (\coord-3.25, -12) arc [start angle=180, end angle=270 , radius = 0.5];
    \draw[thin, blue] (\coord-2.75, -12.5)--(\coord-0.5, -12.5);
    \draw[thin, blue] (\coord-0.5, -12.5) arc [start angle=-90, end angle=0 , radius = 0.5];
    
    }
    
    \draw[thin, dashed, black] (1,-11)--(3.5,-9);
    \draw[thin, dashed, black] (1,-11)--(3.5,-10);
    \draw[thin, dashed, black] (1,-11)--(3.5,-11);
    \draw[thin, dashed, black] (1,-11)--(3.5,-12);
    \draw[thin, dashed, black] (1,-11)--(3.5,-13);

    \draw[thin, dashed, black] (9,-11)--(11.5,-9);
    \draw[thin, dashed, black] (9,-11)--(11.5,-10);
    \draw[thin, dashed, black] (9,-11)--(11.5,-11);
    \draw[thin, dashed, black] (9,-11)--(11.5,-12);
    \draw[thin, dashed, black] (9,-11)--(11.5,-13);
    
    \draw[thick, violet, ->] (-1.5, 0) arc [start angle =180, end angle=0 , radius = 0.5];
    \draw[thick, violet, ->] (7.5, -11) arc [start angle =-180, end angle=-30 , radius = 1.5];
    
    \node at (-1.625,-1.625) {$\alpha$};
    \node at (-1,0.625) {$\beta$};
    
    \node at (3,0) {$r-2$};
    \node at (12,0) {$r-1$};
    \node at (3,-4.75) {$r-1$};
    \node at (12,-5) {$r-1$};
    \node at (3,-10.75) {$r-1$};
    \node at (12,-11) {$r-1$};
    
    \end{tikzpicture}
    \caption{Realization of the trivial representation in the stratum $\mathcal H_1(n+1, nr-n-1;\,-r^n)$. This figure depicts the extreme case in which $m_1=n+1$, that is $a_1=2$, $a_2=\cdots=a_n=1$ and $b_1=r-2$ and $b_2=\cdots=b_n=r-1$. The close curve $\alpha$ has index $n+1$ whereas the close $\beta$ has index $0$.}
    \label{fig:allpolesrtricase}
\end{figure}
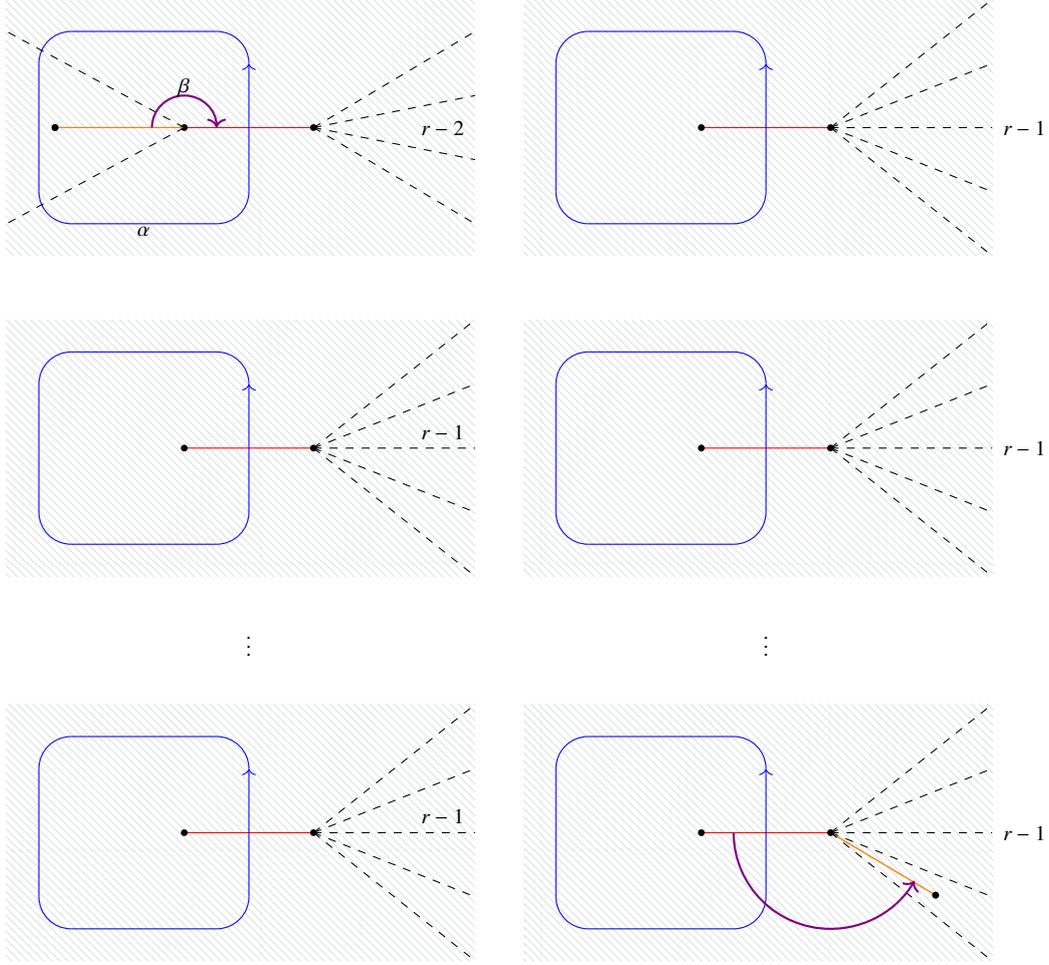

\smallskip

\noindent We then choose the slits $s_o$ and $s_{n+1}$ to be of angle $2\pi$ away from the beginning saddle connection because $A=B=1$ and hence they do not coincide with any saddle connection $s_1,\dots,s_n$. Let $P_{n+1}$ and $Q_{n+1}$ be the extremal points of $s_{n+1}$ and $s_o$ respectively. Slit both and denote the resulting sides as $s_o^{\pm}$ and $s_{n+1}^{\pm}$ according to our convention. If any saddle connection $s_i$ is oriented from $P_o$ to $Q_o$, then $s_o$ is oriented from $P_o$ to $Q_{n+1}$ and, similarly, $s_{n+1}$ is oriented from $P_{n+1}$ to $Q_o$. Identify $s_o^+$ with $s_{n+1}^-$ and identify $s_o^-$ with $s_{n+1}^+$. Notice that the points $P_o$ and $P_{n+1}$ as well as $Q_o$ and $Q_{n+1}$ are now identified. The resulting space is a genus one surface equipped with a meromorphic differential. By construction, such a structure lies in $\mathcal H_1(m_1,m_2;\,-r^n)$ and has trivial period character. See figure \ref{fig:allpolesrtricase}.

\medskip

\paragraph{\textit{Al least one pole has order bigger than $r$ -- generic case}}\label{par:genericgenusonetritwozeros} We now assume $p>nr$ which means at least one pole has order bigger than $r$, say $p_1$, without loss of generality. Assume for a moment that both zeros have order at least $p-n-r$. In this case we may observe that $2p - 2r - 2n \leq m_1 + m_2 = p$, that means 
\begin{equation}\label{eq:ineqexcases}
    (n+1)r \leq p \leq 2r + 2n\,\,\Longleftrightarrow\,\, (n-1)(r-2) \leq 2
\end{equation}

\noindent Therefore, in the case both zeros have order at least $p-n-r$, there are several isolated cases to consider which we define as special and we shall treat them separately, see Paragraphs \S\ref{par:exceptionalgenusonetritwozeros} and \S\ref{par:exceptionalgenusonetritwozeroslowerr}. 

\smallskip 

\noindent In the present paragraph, we assume $m_1<p-n-r$. Suppose we want to realize the trivial representation in the stratum $\mathcal H_1(m_1,m_2;\,-p_1,\dots,-p_n)$. Of course, $r$ divides $p_i$, hence we can write this latter as $p_i=r+h_i\,r$, for some $h_i\ge0$ for every $i=1,\dots,n$. Since $p>nr$, it follows that $h_i\ge1$ for some $i$. Suppose $h_1=1$ and $h_i=0$ for $i=2,\dots,n$. Then we can realize the trivial representation as the period character of some translation surface in the stratum $\mathcal H_1(m_1,m_2-r;\,-r^n)$ as done in the previous Paragraph \S\ref{par:allsameordertritwozer}. Notice that this is possible because Hurwitz type inequality \eqref{eq:triconstr} still holds due to the assumption that $m_1 <(p - r)-n$. 

\begin{rmk}
Notice that, since $m_1\le p-n-r-1$ and $m_1+m_2=nr+r > nr$, because we are assuming $h_1=1$,  it directly follows that 
\begin{equation}
    m_2> nr-m_1\ge nr-p+n+r+1=r+1.
\end{equation} Therefore $m_2-r\ge 1$.
\end{rmk}

\noindent Next, we bubble $r$ planes along a ray from the zero of order $m_2-r$ to any pole of order $r$ as described in Definition \ref{bubplane}. Notice from the previous construction that this is always possible. If a closed curve crosses the ray before bubbling, then its index alters by $r$ after, hence the rotation number remains unchanged. The resulting structure lies by design in the stratum $\mathcal H_1(m_1,m_2;\,-p_1,-r^{n-1})$ as desired. The most general case follows by applying a double induction. More precisely, we first induct on $h_1$ by keeping $h_2=\cdots=h_n=0$ and then we induct on the number of poles with order greater than $r$ that is on the number of poles whose orders $p_i$ satisfy $h_i\ge1$.

\begin{rmk}
Notice that, if $p_i=2$ for all $i=1,\dots,n$ then the trivial representation cannot be realized in the stratum $\mathcal H_1(m_1,m_2;\,-2^n)$ because $m_1$ or $m_2$ is at least $n > p - n - 1=n-1$, and we would have a contradiction with Hurwitz type inequality \eqref{eq:triconstr}. 
\end{rmk}

\smallskip

\paragraph{\textit{Al least one pole has order bigger than $r\ge3$  -- special cases}}\label{par:exceptionalgenusonetritwozeros} In the present paragraph we continue to assume that $p>nr$ and we suppose that both zeros have order at least $p-n-r$. In order to avoid a further proliferation of paragraphs and sub-paragraphs, we list all these special cases here and we handle them one by one them by using Lemmas. Let us consider again the inequality \eqref{eq:ineqexcases} above, that is $(n-1)(r-2) \leq 2$. 

\begin{rmk}
In the first place we notice that $r\ge5$ readily implies that $n=1$ and therefore $p\ge6$. In other words, under the conditions above, $r\ge5$ only for strata with a single pole. 
\end{rmk}

\noindent We consider first the case $n=1$, that is the case of strata $\mathcal H_1(m_1,m_2;\,-p)$ with $m_1+m_2=p$. The following holds for $r\ge3$.

\begin{lem}
Let $\mathcal H_1(m_1,m_2;\,-p)$ be a stratum of genus $1$-differentials. Suppose that $r\ge3$ is a divisor of $\gcd(m_1,m_2)$ such that $m_i\ge p-r-1$, for both $i=1,2$. Then the trivial representation can be realized in the connected component of the same stratum with rotation number $r$.
\end{lem}

\begin{proof}
Observe that $p>r$ because $r$ divides $\gcd(m_1,m_2)$. We first realize the trivial representation as the period character of some  genus-zero differential in the stratum $\mathcal{H}_0(m_1-1,m_2-1;\,-p)$. Let $(\C,\,dz)$ be the complex plane seen as a translation surface. Pick two distinct points, say $P$ and $Q$, and let $s$ denote the edge joining them. Let $r_P$ be a half-ray leaving from $P$ and let $r_Q$ be a half-ray leaving from $Q$ such that $r_P\,\cap\,r_Q=\phi$; \textit{e.g.} pick $r_P$ and $r_Q$ so that $r_P\,\cup\,s\,\cup\,r_Q$ are aligned. Bubble $m_1-1$ copies of $(\C,\,dz)$ along the ray $r_P$ and then bubble $m_2-1$ copies of $(\C,\,dz)$ along the ray $r_Q$. Since $m_1+m_2-2=p$ resulting space is now a genus zero meromorphic differential in the stratum $\mathcal{H}_0(m_1-1,m_2-1;\,-p)$ as desired. Orient $s$ from $P$ to $Q$ and denote by $s^-$ and $s^+$ the left and right side of $s$ respectively, according to our convention. In order to bubble a handle with trivial periods as described in \S\ref{sssec:tripresrot} above, choose a segment $e_P$ from $P$ so that its angle with respect to $s^-$ is $2\pi$ and choose a segment $e_Q$ from $Q$ so that its angle with respect to $s^+$ is $2\pi(r+1)$. Slit the edges $e_P$ and $e_Q$ and denote the resulting sides $e_P^{\pm}$ and $e_Q^{\pm}$, according to our convention. Then glue $e_P^+$ with $e_Q^-$ and similarly glue $e_P^-$ with $e_Q^+$. The resulting space is now a genus one differential with trivial periods in the stratum $\mathcal H_1(m_1,m_2;\,-p)$ with rotation number $r$ as desired. In fact, consider an arc joining the midpoint of $e_P$ with the midpoint of $s$ that winds clockwise around $P$. Next, prolong this arc with another one joining the midpoint of $s$ with the midpoint of $e_Q$ that winds counterclockwise around $Q$. After the cut and paste described above, this arc closes up to a simple close curve $\alpha$ with index $r$ by construction. Then pick $\beta$ as any simple loop around $e_P$ (or $e_Q$). This procedure is doable as long as $\max(m_1,m_2)>r+1$. Since $m_1,\,m_2$ are both divisible by $r$, the only exception is given by $\mathcal H_1(r, r;\, -2r)$, but then we can choose both slits of turning angle $2\pi$ so that their difference is zero and the same construction works.
\end{proof}

\noindent We next suppose $n=2,3$. In this case the possible values for $r$, so that inequality \eqref{eq:ineqexcases} holds, are $1\le r\le 4$ if $n=2$ and $1\le r\le 3$ if $n=3$. We shall consider $r=1,2$ in \S\ref{par:exceptionalgenusonetritwozeroslowerr} and hence we just focus on the other cases listed below. Since the total polarity $p$ is bounded between $(n+1)r\le p\le 2(n+r)$, see inequality \eqref{eq:ineqexcases}, there are a few isolated cases to consider as follows:

\begin{itemize}
    \item[1.] $n=2$ and $r=3$. In this case we consider strata of the form $\mathcal H_1(3k_1, 3k_2;\,-3h_1,-3h_2)$ for some positive integers $k_1,k_2,h_1,h_2$. A direct check shows that $9\le p\le10$. Since $3h_1+3h_2=10$ has no integral solutions, we only need to consider $p=9$. This forces $k_1=h_1=1$ and $k_2=h_2=2$ -- up to relabelling. Notice that $h_1$ and $h_2$ cannot be zero because we consider strata with two poles. Therefore we only need to consider the stratum $\mathcal H_1(3,6;\,-3,-6)$. Since
    \begin{equation}
        m_1=3<4=9-2-3=p-n-r,
    \end{equation} the argument developed in Paragraph \S\ref{par:genericgenusonetritwozeros} applies and hence we are done in this case.
    \smallskip
    \item[2.] $n=2$ and $r=4$. This case is quite similar. We consider strata of the form $\mathcal H_1(4k_1, 4k_2;\,-4h_1,-4h_2)$ for some positive integers $k_1,k_2,h_1,h_2$. A direct check shows that $p=12$ and the Diophantine equation $4h_1+4h_2=12$ has $(h_1,\,h_2)=(h_2,\,h_1)=(1,2)$ as the only possible solutions with positive integers. Assume, without loss of generality that $(h_1,\,h_2)=(1,2)$ -- the other pair leads to the same stratum, namely $\mathcal H_1(4,8;\,-4,-8)$. Since 
    \begin{equation}
        m_1=4<6=12-2-4=p-n-r,
    \end{equation} we fall once again in the case considered in Paragraph \S\ref{par:genericgenusonetritwozeros} and hence, even in this case, we are done.
    \smallskip
    \item[3.] $n=3$ and $r=3$. We consider strata of the form $\mathcal H_1(3k_1,3k_2;\,-3h_1,-3h_2,-3h_3)$ for some positive integers $k_1,k_2,h_1,h_2,h_3$. A direct check shows that $p=12$ and hence $(1,1,2)$ is the only triple of positive integers that satisfies $3(h_1+h_2+h_3)=12$, up to permutation. There are two strata to consider in this case according to the following partitions of $4$ as sum of two positive integers:
    \smallskip
    \begin{itemize}
        \item[i.] $k_1=k_2=2$ that leads to $\mathcal H_1(6,6;\,-3,-3,-6)$, and
        \smallskip
        \item[ii.] $k_1=1$ and $k_2=3$ that leads to $\mathcal H_1(3,9;\,-3,-3,-6)$.
        \smallskip
    \end{itemize}
    In the case \textit{ii.} above it is easy to check that $m_2=9>8=p-n-1$ and hence the trivial representation cannot be realized in this stratum because the Hurwitz type inequality \eqref{eq:triconstr} does not hold. The case \textit{i.} above, however, needs a special treatment. Notice that we cannot apply the induction, as in \S\ref{par:genericgenusonetritwozeros}, to the stratum $\mathcal H_1(3,6;\,-3,-3,-3)$ because the trivial representation cannot be realized in this latter; in fact, the Hurwitz type inequality \eqref{eq:triconstr} does not hold.
\end{itemize}

\begin{lem} The trivial representation can be realized in the connected component of $\mathcal H_1(6,6;\,-3,-3,-6)$ with rotation number $r=3$.
\end{lem}

\begin{figure}[!ht]
    \centering
    \begin{tikzpicture}[scale=1, every node/.style={scale=0.8}]
    
    \definecolor{pallido}{RGB}{221,227,227}

    \foreach \x [evaluate=\x as \coord using  4*\x] in {0,2} 
    {
    \pattern [pattern=north west lines, pattern color=pallido]
    (\coord-3.5, -2)--(\coord-3.5, 2)--(\coord+3.5, 2)--(\coord+3.5, -2)--(\coord-3.5,-2);
    
    \draw[thin, red, ->] (\coord-1,0)--(\coord,0);
    \draw[thin, red] (\coord,0)--(\coord+1,0);
    \fill [black] (\coord-1,0) circle (1.5pt);
    \fill [black] (\coord+1,0) circle (1.5pt);
    }
    
    \draw[thin, black, dashed] (-1,0)--(-1,2);
    \draw[thin, black, dashed] (-1,0)--(-1,-2);
    
    \draw[thin, black, dashed] (9,0)--(9,2);
    \draw[thin, black, dashed] (9,0)--(9,-2);
    
    \pattern [pattern=north west lines, pattern color=pallido]
    (0.5, -3)--(7.5, -3)--(7.5, -7)--(0.5,-7)--(0.5,-3);
    
    \draw[red, thin, ->] (3,-5)--(4,-5);
    \draw[red, thin] (4,-5)--(5,-5);
    \fill [black] (3,-5) circle (1.5pt);
    \fill [black] (5,-5) circle (1.5pt);
    
    \draw[thin, black, dashed] (3,-5)--(3,-3);
    \draw[thin, black, dashed] (3,-5)--(0.5,-5);
    \draw[thin, black, dashed] (3,-5)--(3,-7);
    
    \draw[thin, black, dashed] (5,-5)--(5,-3);
    \draw[thin, black, dashed] (5,-5)--(7.5,-5);
    \draw[thin, black, dashed] (5,-5)--(5,-7);
    
    \draw[thin, orange, ->] (-1,0)--(-2,0);
    \draw[thin, orange] (-2,0)--(-3,0);
    \draw[thin, orange] (5,-5)--(5.75,-5.75);
    \draw[thin, orange, <-] (5.75,-5.75)--(6.5,-6.5);
    
        \fill [black] (-1,0) circle (1.5pt);
        \fill [black] (-3,0) circle (1.5pt);
        
        \fill [black] (5,-5) circle (1.5pt);
        \fill [black] (6.5,-6.5) circle (1.5pt);
        
    \node at (0,0.25) {$s^+$};
    \node at (4,-5.25) {$s^-$};
    
    \node at (-2,0.35) {$e_1^+$};
    \node at (-2,-0.35) {$e_1^-$};
    
    \node at (6,-5.5) {$e_2^+$};
    \node at (5.5,-6) {$e_2^-$};
    
    \node at (2.5,-1.5) {$(\cp, \omega_1)$};
    \node at (5.5,-1.5) {$(\cp, \omega_2)$};
    \node at (1.5,-6.5) {$(\cp, \omega_3)$};
    
    \end{tikzpicture}
    \caption{Realization of the trivial representation in the connected component of the stratum $\mathcal H_1(6,6;\,-3,-3,-6)$ with rotation number $r=3$. The simple closed curves $\alpha,\,\beta$ can be drawn exactly as in Figure \ref{fig:allpolesrtricase} above. It is easy to check that the resulting genus one-differential has rotation number $3$.}
    \label{fig:specialcase}
\end{figure}
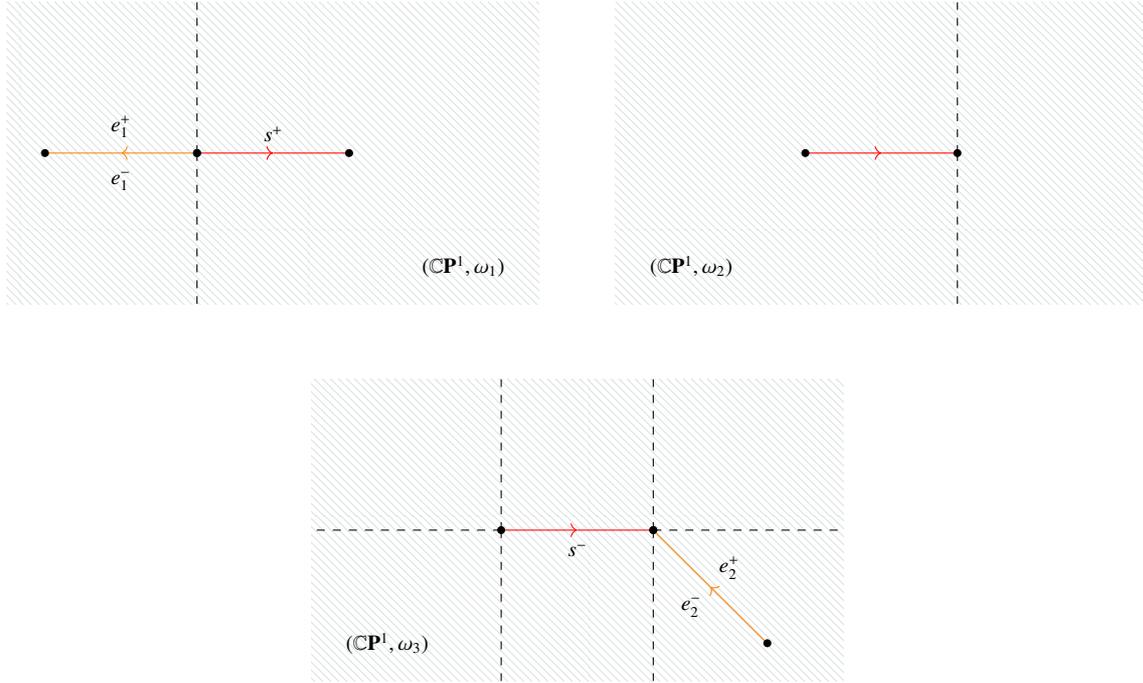

\begin{proof}[Sketch of the proof]
The proof is very similar to the argument developed in \S\ref{par:genericgenusonetritwozeros} and hence we provide it in brief. We first realize the trivial representation in $\mathcal H_0(5,5;\,-3,-3,-6)$ by using the saddle connection configuration description in \cite{EMZ} with $(a_1, b_1)=(2,1)$, $(a_2,b_2)=(1,2)$ and $(a_3,b_3)=(3,3)$. Notice that $a_1+a_2+a_3 =m_1+1=6$ and $b_1 + b_2 + b_3 = m_2+1=6$. 

\smallskip

\noindent For every $i=1,\dots,n$, consider a copy of $(\C,\, z^{r-2}dz)$ and fix any vector $c\in\mathbb C^*$. Let $(\cp,\omega_i)$ be the translation surface with poles obtained by breaking the zero of $(\C,\, z^{r-2}dz)$ into two zeros, say $P_i$ and $Q_i$, of orders $a_i-1$ and $b_i-1$ respectively so that the resulting saddle connection joining them, say $c_i$, is parallel to $c$ with the same length, \textit{i.e.} $|c|=|c_i|$.

\smallskip

\noindent Next, we slit every $(\cp,\omega_i)$ along $c_i$ and denote by $c_i^{\pm}$ the resulting edges according to our convention. Identify $c_i^-$ with $c_{i+1}^+$ for every $i=1,2$ and glue $c_3^-$ with $c_1^+$. The resulting space is now a pair $(\cp, \omega)$ where $\omega$ is a meromorphic differential with two zeros, say $P_o$ and $Q_o$, of orders $5$, two poles of order $3$ and one pole of order $6$. Let $s$ be the saddle connection resulting from the identification of $c_1^+$ with $c_n^-$. Orient $s$ from $P_o$ to $Q_o$.

\smallskip

\noindent In $(\cp,\,\omega)$ there is a segment, say $e_1$, leaving from $P_o$ such that $s$ and $e_1$ are twins, see Definition \ref{def:twins}, and the angle on the left is $2\pi$. Similarly, there is only segment, say $e_2$, leaving from $Q_o$ such that $s$ and $e_2$ are twins and the angle on the right is $2\pi$. Slit both $e_1$ and $e_2$ and denote the resulting sides $e_1^{\pm}$ and $e_2^{\pm}$ according to our convention. Glue $e_1^+$ with $e_2^-$ and glue $e_1^-$ with $e_2^+$. The resulting space is a genus one surface equipped with a translation structure with trivial periods. By construction it lies in the connected component of $\mathcal H_1(6,6;\,-3,-3,-6)$ with rotation number $r=3$. See Figure \ref{fig:specialcase}.
\end{proof}

\smallskip

\noindent We finally notice that for $n\ge4$ the inequality \eqref{eq:ineqexcases} holds if and only if $r\le2$. Therefore there are no other cases to consider for $r\ge3$.

\medskip

\paragraph{\textit{Al least one pole has order bigger than $r\le2$  -- special cases}}\label{par:exceptionalgenusonetritwozeroslowerr} We continue to assume that $p>nr$ and that both zeros have order at least $p-n-r$. It remains to consider the cases $r=1,2$. For both of them we no longer have any upper bound on the number of poles because \eqref{eq:ineqexcases} holds for every $n\ge1$.

\smallskip

\noindent If $r = 1$, the assumption implies that $m_1, m_2\geq(p-1)-n$, but the Hurwitz type inequality \eqref{eq:triconstr} then implies that $m_1 = m_2 = p- n -1$.  Hence $p = m_1 + m_2 = 2p - 2n - 2$, so $p = 2n+2$ and $m_1 = m_2 = n+1$.  Since each $p_i \geq 2$, the only cases are 
\smallskip
\begin{itemize}
    \item[1.] $\mathcal H_1(n+1, n+1; -2^{n-1}, -4)$. This stratum is connected whenever $n$ is even whereas it has exactly two connected components whenever $n$ is odd. Therefore, this latter case is worth of interest for us. See Lemmas \ref{lem:specstrat1} and \ref{lem:excaseone}.
    \smallskip
    \item[2.] $\mathcal H_1(n+1, n+1; -2^{n-2}, -3, -3)$. It is an easy matter to check that such a stratum is connected for every $n\ge3$ because $\gcd(2,3,n+1)=1$. On the other hand, for $n=2$ we get $\mathcal H_1(3,3;\,-3,-3)$ which has two connected components. See Lemma \ref{lem:excasetwo} for this special case which turns out to be an exceptional stratum.
\end{itemize}

\smallskip 

\noindent If $r=2$, the assumption implies that $m_1, m_2 \geq (p - 2) - n$.  Hence $p = m_1 + m_2 \geq 2p - 2n - 4$ and $p \leq 2n + 4$.  The Hurwitz inequality is that $m_1, m_2 \leq p - n - 1$.  Since each $p_i$ is even, the only cases are 
\smallskip
\begin{itemize}
    \item[3.] $\mathcal H_1(n+1, n+1; -2^{n-1},\, -4)$. If $n$ is even the stratum is connected and it has two connected if $n$ is odd. See Lemmas \ref{lem:specstrat1} and \ref{lem:excaseone}.
    \smallskip
    \item[4.] $\mathcal H_1(n+2, n+2; -2^{n-1}, -6)$. If $n$ is odd the stratum is connected and hence the interesting cases arise for $n$ even. See Lemma \ref{lem:specstrat2}. Finally,
    \smallskip
    \item[5.] the last family of strata is given by $\mathcal H_1(n+2, n+2; -2^{n-2}, -4, -4)$. For $n$ odd the stratum is connected and it has two connected components if $n$ is even. See Lemma \ref{lem:specstrat3}.
\end{itemize}

\medskip 

\noindent We have the following Lemmas.

\begin{lem}\label{lem:specstrat2}
Let $n=2k$ be an even positive integer. For every $k\ge1$, the trivial representation can be realized in both connected components of the stratum $\mathcal H_1(2k+2,2k+2;\,-2^{2k-1}, -6)$.
\end{lem}

\begin{proof}
In this case we only need to show that the trivial representation can be realized in the connected component of $\mathcal H_1(2k+2,2k+2;\,-2^{2k-1}, -6)$ with rotation number $2$. In fact, since 
\begin{equation}
    n+2 < 2(n-1)+6-n-1=n+3
\end{equation} for every $n\ge0$, it follows that $m_1=m_2< p-n-1$ we can proceed as in \S\ref{par:genericgenusonetritwozeros} and hence we are already done if $r=1$. In order to realize the trivial representation as the period character of some exact differential with rotation number $2$ we proceed as follows. We use the saddle connection configuration description as in \cite{EMZ}. We need to find pairs $(a_i,b_i)$ such that $a_1 + b_1 = 6$ and $a_i + b_i = 2$ for all $2\leq i\leq n$, that satisfy $a_1 + \cdots + a_n = b_1 + \cdots + b_n = n+1$. We choose these pairs to be $(3,3),\, (1,1),\, \dots,\, (1,1)$. More precisely, consider the differential $z^4\,dz$ on $\C$, break the sole zero into two zeros of the same order and denote by $c_1$ the resulting saddle connection. Denote the resulting translation surface as $(\C,\,\omega)$. Next, we consider $n-1$ copies of $(\C,\,dz)$ and we consider on each of them a marked segment $c_i$, for $2\leq i\leq n$ so that $c_1,\dots,c_n$ are all parallel with the same length. The desired handle with trivial periods can be realized in $(\C,\omega)$. Let $P$ and $Q$ denote the zeros of order two of $(\C,\,\omega)$ and let $c_1$ be the saddle connection joining them. According to our Lemma \ref{lem:techlemtwins}, there are two edges $e_1$ and $e_2$ leaving from $P$ such that $s,\,e_1,\,e_2$ are pairwise twins. Similarly, there are two edges $e_3,\,e_4$ leaving from $Q$ such that $c_1,\,e_3,\,e_4$ are pairwise twins. Assume without loss of generality that $c_1,e_1,e_2$ and $s,e_3,e_4$ are ordered counterclockwise around $P$ and $Q$ respectively. Slit the edges $e_1$ and $e_4$ and denote the resulting sides $e_1^{\pm}$ and $e_4^{\pm}$ according to our convention. Then glue the edges $e_1^+$ with $e_4^-$ and similarly glue $e_1^-$ with $e_4^+$. The resulting space is a genus one differential, say $(Y,\xi)\in\mathcal H_1(3,3;\,-6)$ with trivial periods and rotation number $1$, see Figure \ref{fig:interstruc336}.

\smallskip

\begin{figure}[!ht]
    \centering
    \begin{tikzpicture}[scale=1.2, every node/.style={scale=1}]
    
    \definecolor{pallido}{RGB}{221,227,227}
   
    \pattern [pattern=north west lines, pattern color=pallido]
    (-6,-3)--(-6,3)--(6,3)--(6,-3)--(-6,-3);
    
    \draw[thick, orange](-1.5,0)--(0,0);
    \draw[thick, orange](0,0)--(1.5,0);
    
    \draw[thin, dashed] (-1.5,0)--(-6,3);
    \draw[thin, dashed] (-1.5,0)--(-6,0);
    \draw[thin, dashed] (-1.5,0)--(-6,-3);
    
    \draw[thin, dashed] (1.5,0)--(6,3);
    \draw[thin, dashed] (1.5,0)--(6,0);
    \draw[thin, dashed] (1.5,0)--(6,-3);
    
    \draw[thin, red] (-1.5,0)--(-4.5,1);
    \draw[thin, red, dashed] (-1.5,0)--(-4.5,-1);
    
    \draw[thin, red] (1.5,0)--(4.5,1);
    \draw[thin, red, dashed] (1.5,0)--(4.5,-1);
    
    \fill [black] (-1.5,0) circle (1.5pt);
    \fill [black] (1.5,0) circle (1.5pt);
    
    \fill [black] (-4.5,1) circle (1.5pt);
    \fill [black] (4.5,-1) circle (1.5pt);
    
    \fill [black] (4.5,1) circle (1.5pt);
    \fill [black] (-4.5,-1) circle (1.5pt);
    
    \node at (-0.74,0.25) {$c_1$};
    
    \node at (-3.5,0.35) {$e_1$};
    \node at (-3.5,-0.35) {$e_2$};
    
    \node at (3.5,0.35) {$e_4$};
    \node at (3.5,-0.35) {$e_3$};
    
    \node at (-1.5,-0.35) {$P$};
    \node at (1.5, -0.35) {$Q$};
    
    \draw[thin, blue] (0, 0) arc [start angle=0, end angle=90 , radius = 1.5];
    \draw[thin, blue, <-] (-1.5, 1.5) arc [start angle=90, end angle=161 , radius = 1.5];
    \draw[thin, blue, ->] (0, 0) arc [start angle=180, end angle=270 , radius = 1.5];
    \draw[thin, blue] (1.5, -1.5) arc [start angle=270, end angle=377 , radius = 1.5];
    
    \node at (-1.5, 1.75) {$\alpha$};

    \end{tikzpicture}
    \caption{Realizing the intermediate structure $(Y,\xi)\in\mathcal H_1(3,3;\,-6)$. The orange edge is the saddle connection joining $P$ and $Q$. The red edges $e_1,\,e_2$ are the twins of $c_1$ leaving from $P$ and the edges $e_3$ and $e_4$ are the twins of $c_1$ leaving from $Q$. The dashed edges are those \textit{not} involved in the construction. The curve $\alpha$ has index $1$ because it winds clockwise around $P$ once and counterclockwise around $Q$ twice. Therefore the resulting structure has rotation number one. If we had chosen $e_3$ in place of $e_4$, then its index would have been $0$ because $\alpha$ would wind around $Q$ only once. In this way we would obtain a structure in $\mathcal H_1(3,3;\,-6)$ with rotation number $3$.}
    \label{fig:interstruc336}
\end{figure}
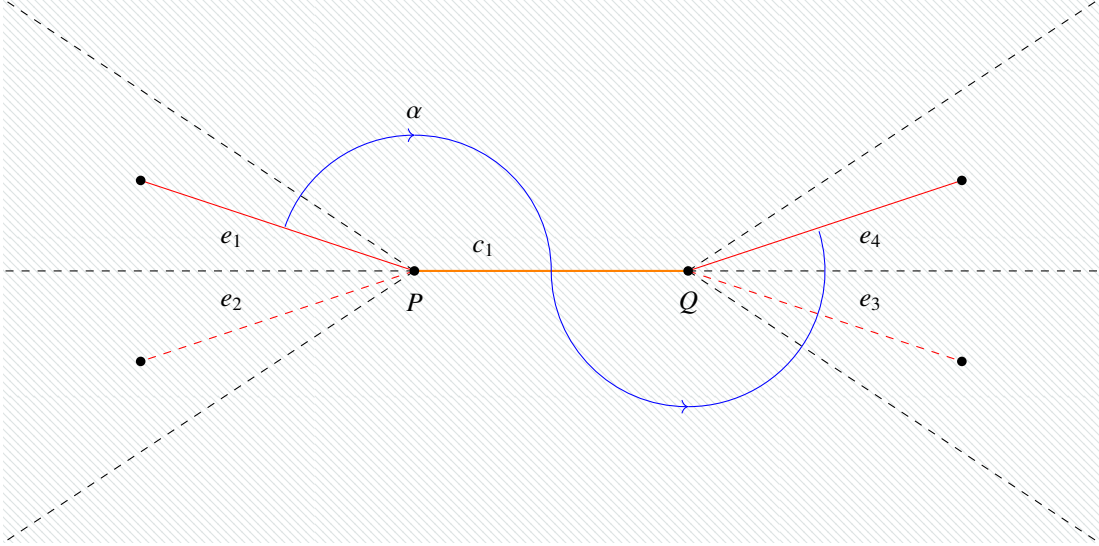

\noindent In fact, on $(\C,\,\omega)$ we consider an arc joining the midpoint of $e_1$ with the midpoint of $c_1$ that winds clockwise around $P$. Next, we prolong the arc above with another arc leaving from the midpoint of $c_1$ with the midpoint of $e_4$ that winds counterclockwise around $Q$. Notice that, since $c_1,e_3,e_4$ are ordered counterclockwise, this second arc winds twice around $Q$. In the resulting space $(Y,\xi)$ the union of these arc just described close up to a simple close curve, say $\alpha$ of index one, see Figure \ref{fig:interstruc336}. We define $\beta$ as any simple close loop around $e_1$. Notice that its index is equal to $3$.

\smallskip

\noindent In the resulting space $(Y,\xi)$, there is by construction a saddle connection, say $c_1$ with a small abuse of notation, joining the two zeros of order $3$ which is still parallel to $c_2,\dots,c_n$. We glue the remaining $n-1$ copies of $(\C,\,dz)$ by slitting the saddle connections $c_i$ as already explained in Paragraph \S\ref{par:allsameordertritwozer}. In this process, the indexes of $\alpha$ and $\beta$ are both altered by $\pm(n-1)$. Since $n$ is even it follows that $n-1$ is odd. In particular both $\textnormal{Ind}(\alpha)\pm(n-1)$ and $\textnormal{Ind}(\beta)\pm(n-1)$ are even. Therefore, the resulting translation surface lies in $\mathcal H_1(n+2,n+2;\,-2^{n-1},-6)$ and it has rotation number $2$ as desired.
\end{proof}

\noindent Notice that, if we had chosen $e_3$ in place of $e_4$ we would have got a translation surface in the same stratum with rotation number one. Finally,

\begin{lem}\label{lem:specstrat3}
Let $n=2k$ be an even positive integer. For every $k\ge1$, the trivial representation can be realized in both connected components of the stratum $\mathcal H_1(2k+2,2k+2;\,-2^{2k-2}, -4,-4)$.
\end{lem}

\begin{proof}
This proof is quite similar to that of Lemma \ref{lem:specstrat2}. In fact, even in this case we only need to show that the trivial representation can be realized in the connected component of $\mathcal H_1(2k+2,2k+2;\,-2^{2k-2}, -4,-4)$ with rotation number $2$. In fact, since 
\begin{equation}
    n+2 < 2(n-2)+8-n-1=n+3
\end{equation} for every $n\ge0$, it follows that $m_1=m_2< p-n-1$ we can proceed as in \S\ref{par:genericgenusonetritwozeros} and hence we are already done if $r=1$. In order to realize the trivial representation as the period character of some exact differential with rotation number $2$ we proceed as follows. We use the saddle connection configuration description as in \cite{EMZ}. We need to find pairs $(a_i,b_i)$ such that $a_1 + b_1 = a_2 + b_2 = 4$ and $a_i + b_i = 2$ for all $3\leq i\leq n$, and $a_1 + \cdots + a_n = b_1 + \cdots + b_n = n+2$. We choose these pairs to be $(2,2),\, (2,2),\, (1,1),\, \dots,\, (1,1)$. More precisely, in this case consider two copies of $(\C,\,z^2dz)$ and break on each of them the sole zero into two zeros each of order one. In both case the sole zero of $z^2dz$ can be broken so that the resulting saddle connections, say $c_1$ and $c_2$ are parallel with the same length. Next we consider $n-2$ copies of $(\C,\,dz)$, each of which with a marked segment $c_i$, for $3\leq i\leq n$ so that $c_1,c_2,c_3,\dots,c_n$ are all parallel with the same length. The
desired handle with trivial periods can be bubbled inside the copy $(\C,\omega_1)$. We then proceed as in the proof of Lemma \ref{lem:specstrat2} to realize the desire structure in the stratum $\mathcal H_1(2k+2,2k+2;\,-2^{2k-2}, -4,-4)$ with rotation number $2$ and trivial periods. More precisely, we first realize an intermediate structure $(Y,\xi)\in\mathcal H_1(2,2;-4)$ with rotation number $2$ and then we glue the remaining translation surfaces by slitting the saddle connection as done in Paragraph \S\ref{par:allsameordertritwozer}.
\end{proof}

\begin{lem}\label{lem:specstrat1}
Let $n=2k+1$ be an odd positive integer. For every $k\ge1$, the trivial representation can be realized in both connected components of the stratum $\mathcal H_1(2k+2,2k+2;\,-2^{2k}, -4)$.
\end{lem}

\begin{proof}
This is the most delicate case to consider. Once again, we wish to realize the trivial representation by using the saddle connection configuration description as in \cite{EMZ}. We need to find $n$ pairs $(a_i,b_i)$ such that $a_1 + b_1 = 4$ and $a_i + b_i = 2$ for all $2\leq i\leq n$, that satisfy $a_1 + \cdots + a_n = b_1 + \cdots + b_n = n+1$. The only possible case is given by the pairs $(2,2),\,(1,1),\dots,(1,1)$. We shall need to distinguish two cases depending in which connected component of $\mathcal H_1(2k+2,2k+2;\,-2^{2k}, -4)$ we aim to realize the trivial representation.

\smallskip

\noindent Suppose we aim to realize a translation surface with trivial periods and rotation number two in that stratum. Then we can proceed exactly as in Lemma \ref{lem:specstrat3} above. We may observe that the only difference is that there is only one pair $(2,2)$ but this is irrelevant to the construction because the handle with trivial periods is realized in $(\C,\omega_1)$ -- in the above notation. Therefore this case is essentially subsumed in the proof of Lemma \ref{lem:specstrat3}.

\smallskip

\begin{figure}
    \centering
    \begin{tikzpicture}[scale=1, every node/.style={scale=0.8}]
    
    \definecolor{pallido}{RGB}{221,227,227}
   
    \pattern [pattern=north west lines, pattern color=pallido]
    (-6,-3)--(-6,3)--(6,3)--(6,-3)--(-6,-3);
    
    \draw[thick, orange](-1.5,0)--(0,0);
    \draw[thick, orange](0,0)--(1.5,0);
    
    \draw[thin, dashed] (-1.5,0)--(-1.5, 3);
    \draw[thin, dashed] (-1.5,0)--(-1.5,-3);
    
    \draw[thin, dashed] (1.5,0)--(1.5, 3);
    \draw[thin, dashed] (1.5,0)--(1.5,-3);
    
    \draw[thin, red] (-1.5,0)--(-4.5,0);
    
    \draw[thin, red] (1.5,0)--(4.5,0);
    
    \fill [black] (-1.5,0) circle (1.5pt);
    \fill [black] ( 1.5,0) circle (1.5pt);
    
    \fill [black] (-4.5,0) circle (1.5pt);
    
    \fill [black] (4.5,0) circle (1.5pt);
    
    \node at (-0.74, 0.25) {$c_1^-$};
    \node at (-0.74,-0.25) {$c_1^+$};
    
    \node at (-3.5,0.35) {$e_P^+$};
    \node at (-3.5,-0.35) {$e_P^-$};
    
    \node at (3.5,0.35) {$e_Q^+$};
    \node at (3.5,-0.35) {$e_Q^-$};
    
    \node at (-1.75,-0.35) {$P$};
    \node at (1.75, -0.35) {$Q$};
    
    \draw[thin, blue] (0, 0) arc [start angle=0, end angle=90 , radius = 1.5];
    \draw[thin, blue, <-] (-1.5, 1.5) arc [start angle=90, end angle=180 , radius = 1.5];
    \draw[thin, blue, ->] (0, 0) arc [start angle=180, end angle=270 , radius = 1.5];
    \draw[thin, blue] (1.5, -1.5) arc [start angle=270, end angle=360 , radius = 1.5];
    
    \node at (-1.75, 1.75) {$\alpha$};

    \foreach \x [evaluate=\x as \coord using  3*\x] in {-1.5,0.5} 
    {
    \pattern [pattern=north west lines, pattern color=pallido]
    (\coord, -3.5)--(\coord+4.5,-3.5)--(\coord+4.5,-6.5)--(\coord,-6.5)--(\coord,-3.5);
    
    \draw[thin, blue] (-2.25, -5) arc [start angle=0, end angle=180 , radius = 1];
    \draw[thin, blue, ->] (-2.25, -5) arc [start angle=0, end angle=-180 , radius = 1];
    
    \draw[thick, orange] (-3.25, -5)--(-1.25, -5);
    
    \draw[thick, red] (2.75, -5)--(4.75, -5);
    
    \fill [black] (\coord+1.25, -5) circle (1.5pt);
    \fill [black] (\coord+3.25, -5) circle (1.5pt);
    }
    
    \node at (-5.25,-5) {$\cdots$};
    \node at ( 0.75,-5) {$\cdots$};

    \node at (-3.25, -6.25) {$\alpha$};
    
    \node at (5.5, -2.75) {$(\mathbb C,\, \omega)$};
    
    \node at (-0.5, -6.25) {$(\mathbb C,\, \omega_i)$};
    
    \node at (5.5, -6.25) {$(\mathbb C,\, \omega_n)$};
    
    \node at (-1.5, -4.75) {$c_i^-$};
    \node at (-1.5, -5.25) {$c_i^+$};
    
    \node at (3.75, -4.75) {$c_n^-$};
    \node at (3.75, -5.25) {$c_n^+$};

    \end{tikzpicture}
    \caption{Realizing a translation surface with trivial periods and rotation number one in the stratum $\mathcal H_1(2k+2,2k+2;\,-2^{2k}, -4)$. The blue curve $\alpha$ closes up to a simple close curve in the resulting space and has index $\textnormal{Ind}(\alpha)=\pm(n-2)$.}
    \label{fig:specstrat1}
\end{figure}
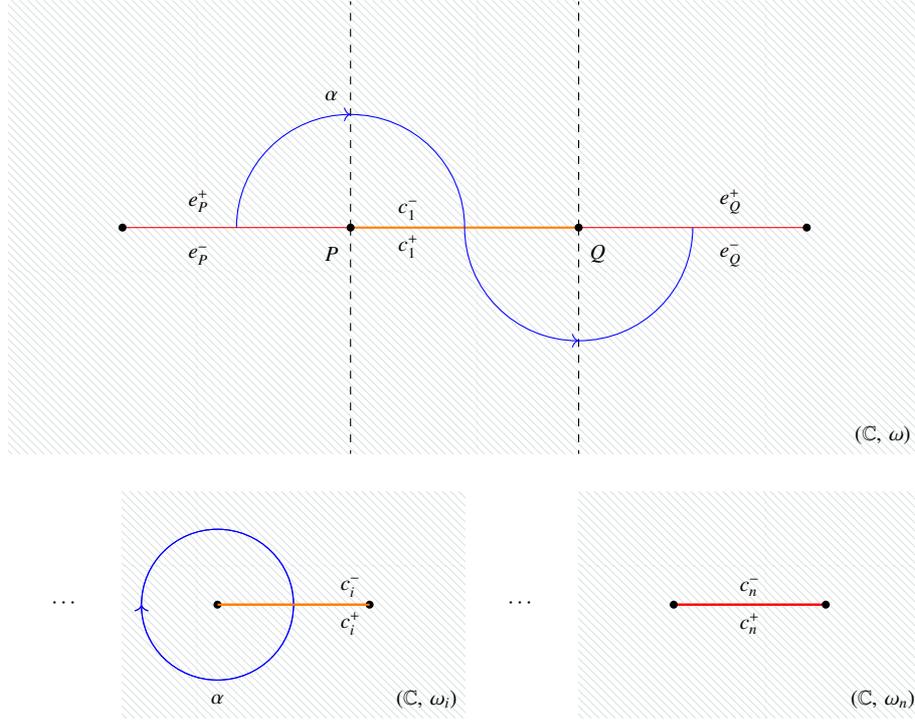

\noindent We just show how to realize a translation surface with trivial periods in $\mathcal H_1(2k+2,2k+2;\,-2^{2k}, -4)$ with rotation number one. Let $(\C,\omega)$ be a translation with trivial periods, two zeros, say $P$ and $Q$, each of order one and a single pole of order $4$. Denote by $c_1$ the saddle connection joining them. According to our technical Lemma \ref{lem:techlemtwins} there is a segment $e_P$ at $P$ such that $e_P$ and $c_1$ are twins. By the same lemma, there is a segment $e_Q$ at $Q$ such that $e_Q$ and $c_1$ are twins. Next pick $n-1$ copies of $(\C,\,dz)$ and denote them by $(\C,\omega_i)$ for $2\leq i\leq n$. On each copy of $(\C,\,dz)$ we consider an edge $c_i$ such that $c_i,\,c_j$ are parallel and with the same length. On $(\C,\,\omega)$ slit $c_1,\, e_P,$ and $e_Q$ and denote by $c_1^{\pm}$, $e_P^{\pm}$ and $e_Q^{\pm}$ the resulting edges according to our convention. Next, for every $2\leq i\leq n$, slit $(\C,\,\omega_i)$ along $c_i$ and denote the resulting sides as $c_i^{\pm}$. We glue the sides as follows: For $1\leq i \leq n-1$ identify $c_i^+$ with $c_{i+1}^-$ and identify $c_{n-1}^+$ with $c_1^-$. Then identify the edge $e_P^+$ with $e_Q^-$, next identify $e_P^-$ with $c_n^+$ and finally identify $e_Q^+$ with $c_n^-$. The resulting space is a translation surface in the stratum $\mathcal H_1(2k+2,2k+2;\,-2^{2k}, -4)$ as desired. It remains to show it has rotation number $1$. We can define $\alpha$ as the simple closed curve shown in Figure \ref{fig:specstrat1}; since $n$ is odd it follows that $\textnormal{Ind}(\alpha)=\pm(n-2)$ is also odd and hence the rotation number is necessarily one.
\end{proof}

\noindent We finally consider the exceptional cases. The following one corresponds to the case $k=0$ of Lemma \ref{lem:specstrat1} above. We first introduce the following definition.

\begin{defn}
Let $\mathcal H_1(m_1,\dots,m_k;\,-p_1,\dots,-p_n)$ be a stratum of genus-one differentials. We define the \textit{primitive component} of this stratum as the unique connected component with rotation number (equivalently torsion number) one. 
\end{defn}

\noindent Notice that, if a stratum is connected then it coincides with its primitive component. In Lemma \ref{lem:excaseone}, Lemma \ref{lem:excasetwo}, and Lemma \ref{lem:222-2-2-2} we shall consider strata with exactly two connected components; one of which is primitive and the other is \textit{non-primitive}.

\begin{lem}\label{lem:excaseone}
The trivial representation can only be realized in the connected component of $\mathcal H_1(2,2;\,-4)$ with rotation number $2$, that is the non-primitive component.
\end{lem}

\begin{proof}
Suppose that $(S,\omega)$ is a translation surface in $\mathcal H_1(2,2;\,-4)$ with trivial periods. Let $Z_1$, $Z_2$, and $P$ be the two zeros and the unique pole of $\omega$, respectively. The trivial period condition is equivalent to the existence of a triple cover of $S$ to $\mathbb{C}\mathbf{P}^1$ totally ramified at $Z_1$, $Z_2$, and $P$, where $\omega$ arises from differentiating the rational function associated to the cover. This implies the linear equivalence relation $3Z_1\sim 3Z_2\sim 3P$. On the other hand, by assumption $2Z_1 + 2Z_2\sim 4P$. It follows that 
$Z_1+Z_2 \sim 3Z_1+3Z_2-(2Z_1 + 2Z_2) \sim 2P$, which implies by definition that $(S,\omega)$ belongs to the non-primitive component of $\mathcal H_1(2,2;\,-4)$.  
\end{proof}

\begin{lem}\label{lem:excasetwo}
The trivial representation can only be realized in the connected component of $\mathcal H_1(3,3;\,-3,-3)$ with rotation number $3$, that is the non-primitive component.
\end{lem}

\begin{proof}
Suppose that $(S,\omega)$ is a translation surface in $\mathcal H_1(3,3;\,-3,-3)$ with trivial periods. Let $Z_1$, $Z_2$, $P_1$, and $P_2$ be the zeros and poles of $\omega$, respectively. The trivial period condition is equivalent to the existence of a quadruple cover of $S$ to $\mathbb{C}\mathbf{P}^1$ totally ramified at $Z_1$, $Z_2$, and having $2P_1+2P_2$ as the fiber over infinity. This implies the linear equivalence relation $4Z_1\sim 4Z_2\sim 2P_1+2P_2$. On the other hand, by assumption $3Z_1 + 3Z_2\sim 3P_1+3P_2$. It follows that 
$Z_1+Z_2 \sim 4Z_1+4Z_2-(3Z_1 + 3Z_2) \sim P_1 + P_2$, which implies by definition that $(S,\omega)$ belongs to the non-primitive component of $\mathcal H_1(3,3;\,-3,-3)$.  
\end{proof}

\smallskip

\subsubsection{Strata with more than two zeros} We now consider the general case of $\mathcal H_1(m_1, \dots, m_k; -p_1, \dots, -p_n)$ with $k\ge3$. We begin with the following observation reported here as

\begin{lem}\label{lem:redlemtricase}
Let $\kappa=(m_1,\dots,m_k)$ and $\nu=(p_1,\dots,p_n)$ be tuples such that $m_1+\cdots+m_k=p_1+\cdots+p_n$. Suppose there is a pair $m_i,\,m_j$ such that $m_i+m_j<p-n$. Let $\overline\kappa=(m_1,\dots,\widehat m_i,\widehat m_j,\dots,m_i+m_j,\dots,m_n)$. Finally, let $r$ be a divisor of $\gcd(m_1,\dots,m_k,p_1,\dots,p_n)$. If the trivial representation can be realized in the connected component of $\mathcal H_1(\,\overline\kappa;-\nu)$ with rotation number $r$, then it can also be realized in the connected component of $\mathcal H_1(\kappa;-\nu)$ with rotation number $r$.
\end{lem}

\noindent Explicitly, if there exist a pair $m_i,\, m_j$ such that $m_i + m_j < p - n$, then we can do induction by splitting a zero of order $m_i + m_j$ as the Hurwitz type inequality \eqref{eq:triconstr} still holds for the merged zero order. It remains to determine whether the trivial representation can be realized in the stratum $\mathcal H_1(\,\overline\kappa;-\nu)$ above. In the case there is another pair of zeros, say $m_h,\,m_l$ such that $m_h+m_l<p-n$, then Lemma \ref{lem:redlemtricase} applies once again. This process iterated finitely many times leads to one of the following situations:

\smallskip

\begin{itemize}
    \item[1.] We fall in a stratum with two zeros and hence we may apply the results of Subsection \S\ref{ssec:tricasetwozerostrata}, or
    \smallskip
    \item[2.] we end up with a stratum $\mathcal H_1(m_1, \dots, m_k; -\nu)$ in which $m_i + m_j \geq p - n$ for all pairs $i, j$.
\end{itemize}

\smallskip

\noindent In Paragraph \S\ref{par:exceptionalgenusonetritwozeroslowerr} we have shown the existence of two strata with two zeros that are exceptional in the sense the trivial representation cannot be realized in every connected component. These are $\mathcal H_1(2,2;-4)$ and $\mathcal H_1(3,3;-3,-3)$. We need to make sure that, by iterating the process just described, if we fall in a stratum with two zeros then this stratum cannot be exceptional. It is easy to check that, by merging zeros, one gets the access to $\mathcal H_1(2,2;-4)$ only from strata $\mathcal H_1(1,1,1,1;\,-4)$ and $\mathcal H_1(1,1,2;\,-4)$ which are connected. Similarly, by merging zeros one gets the access to $\mathcal H_1(3,3;\,-3,-3)$ from $\mathcal H_1(1,2,3;\,-3,-3)$ and $\mathcal H_1(1,1,2,2;\,-3,-3)$ which are also connected. Since we begin with strata $\mathcal H_1(m_1, \dots, m_k; -\nu)$ with more than one connected component we never fall in those exceptional cases.

\medskip

\noindent Suppose to be in the second case above. If $m_i + m_j \geq p - n$ for all pairs $i, j$, then summing over all pairs we obtain that 
\begin{equation}\label{eq:condone}
    (k-1)p \geq \frac{k(k-1)}{2}\,(p-n)
\end{equation} that is $(k-2)p \leq kn$.  Since 
\begin{equation}\label{eq:condtwo}
    (k-2)nr \leq (k-2)p
\end{equation} and $k \ge 3$, it necessarily follows that $r \leq 3$. 

\begin{rmk}
Let $l=\gcd(p_1,\dots,p_n)$ and notice that $n\,l\leq p$. It is easy to check that the equality holds if and only if $p_i=l$ for all $i=1,\dots,n$. The equation \eqref{eq:condtwo} above also holds if we replace $r$ with $l$, in fact 
\begin{equation}
    (k-2)nl\leq (k-2)p
\end{equation}
and, since $(k-2)p\leq kn$ and $k\ge3$, we get that $l\leq 3$. On the other hand,
\begin{equation}
    \gcd(m_1,\dots,m_k,p_1,\dots,p_n)\leq \gcd(p_1,\dots,p_n)
\end{equation} holds and, as a consequence, $\gcd(m_1,\dots,m_k,p_1,\dots,p_n)$ cannot exceed $l\leq3$. Therefore, for a given stratum $\mathcal H_1(\kappa;-\nu)$ that satisfies the condition $m_i + m_j \geq p - n$ for all pairs $i, j$, it necessarily follows $\gcd(\kappa, \nu)$ cannot be greater than $3$.
\end{rmk}

\smallskip

\noindent Assume in the first place $\gcd(p_1,\dots,p_n)=r=3$. In this case the inequalities \eqref{eq:condone} and \eqref{eq:condtwo} above readily implies $k\le 3$; on the other hand $k$ is supposed to be at least $3$ by assumption and hence the only possible case is $k=3$. For $n=3h$ and $h\ge1$, the collection of strata $\mathcal H_1(3h,3h,3h;\,-3^{3h})$ is the only one that satisfies all these conditions. In principle one might need to deal with strata of the form $\mathcal H_1(3h_1,3h_2,3h_3;-3^{h_1+h_2+h_3})$ with $h_1\le h_2\le h_3$. If we assume $3h_1+3h_2\ge p-n=2(h_1+h_2+h_3)$, we readily obtain $h_1+h_2\ge 2h_3$ and hence $h_1+h_2= 2h_3$. This implies $h_1=h_2=h_3$. We have the following

\begin{lem}\label{lem:triplen}
The trivial representation can be realized in each connected component of the stratum of differentials $\mathcal H_1(3h,3h,3h;\,-3^{3h})$.
\end{lem}

\begin{proof}
Let $n=3h\ge3$. We can first realize the trivial representation in the stratum $\mathcal H_0(2n-2, n; -3^n)$. This is possible because the Hurwitz type inequality \eqref{eq:triconstr} is satisfied. Then split the first zero into two nearby zeros each of order $n-1$. By construction there is a saddle connection, say $s$, joining the new-born zeros, say $P_1$ and $P_2$. Moreover, our Lemma \ref{lem:techlemtwins} applies and hence there are $n-1$ paths $c_1,\dots,c_{n-1}$ all leaving from $P_1$ such that $s$ and $c_i$ are twins. Similarly, there are $n-1$ paths $c_{n+1},\dots, c_{2n-1}$, all leaving from $P_2$ such that $s$ and $c_{n+i}$ are twins for $i=1,\dots,n-1$. We may suppose without loss of generality that both collections of paths $c_1,\dots,c_{n-1},s$, all based at $P_1$, and $s,c_{n+1},\dots,c_{2n-1}$, all based at $P_2$, are in positive cyclic order. For any $i=1,\dots,n-1$, slit $c_{i}$ and $c_{n+i}$ and denote $c_i^{\pm}$ and $c_{n+i}^{\pm}$ the resulting edges. Then glue $c_i^+$ with $c_{n+i}^-$ and glue $c_i^-$ with $c_{n+i}^+$. The resulting space is a genus one differential with rotation number $3$. Notice that the choice of the slit is crucial here because the angle difference divisible by $3$ modulo $2\pi$. A different choice of paths, so that the angle difference is no longer divisible by $3$ modulo $2\pi$ would have produced a genus one-differential with rotation number $1$.
\end{proof}

\noindent Since there are no other strata $\mathcal H_1(m_1,\dots,m_k;-p_1,\dots,-p_n)$ with $\gcd(m_1,\dots,m_k;p_1,\dots,p_n)=3$, we next consider the case of strata with $\gcd(m_1,\dots,m_k;p_1,\dots,p_n)=2$.

\medskip

\noindent We now assume $r=\gcd(m_1,\dots,m_k;p_1,\dots,p_n)=2$ and let $\mathcal H_1(m_1,\dots,m_k;\, -p_1,\dots,-p_n)$ be a stratum of genus one differentials. The key observation here is the following reduction process: If there is some pole of order $p_i > 2$ and if there is at most one zero, say $P_j$, of order $m_j \geq p-n-2$, then we can do induction from the connected component of the reduce stratum $\mathcal H_1( \dots, m_j-2, \dots ; \dots, -p_i+2, \dots )$ with the rotation number $1$ or $2$, by bubbling $2$ copies of $(\C,\,dz)$ along a ray from $P_j$ to the pole of order $p_i-2$. 

\noindent The reduced stratum now satisfies only of the following conditions:

\smallskip

\begin{itemize}
    \item[1.] In the reduced stratum there is at most one zero of order $m_h=m_j-2\ge p-n-4$. In this case we can reduce the stratum one more time.
    \smallskip
    \item[2.] There are now at least two zeros of orders $m_h,\,m_l\ge p-n-4$.
\end{itemize}

\smallskip

\noindent Suppose to be in the second case, \textit{i.e.} suppose there are two zeros of orders $m_1, m_2 \geq p-n-2$.  Then
\begin{equation}
    2p - 2n - 4\leq m_1+m_2 \leq p-2(k-2),
\end{equation}
\noindent where the second inequality holds because $m_i\ge2$ for all $i=1,\dots,k$. There is only one case we need to consider given by the stratum $\mathcal H_1(n, n, 2; -2, \dots, -2, -4)$ with $n$ even. The following holds

\begin{lem}\label{lem:h1222-2-4}
Let $n=2h$. The trivial representation can be realized in each connected component of the stratum $\mathcal H_1(2h,\,2h, 2; -2^{2h-1}, -4)$.
\end{lem}

\begin{proof}
We begin by realizing the trivial representation as the period character of genus $0-$differential in the stratum $\mathcal H_0(n,n;\,-2^{n-1},-4)$. Notice that this is possible because the Hurwitz type inequality \eqref{eq:triconstr} holds and the realization is guaranteed by \cite[Theorem B]{CFG}. Let $(Y,\xi)$ such a structure. Let $Q$ be a zero of order $n$ and let $\varepsilon>0$ such that the open ball $B_{4\varepsilon}(Q)$ does not contain the other zero. Break $Q$ into two zeros of order $n-1$ and $1$ as described in \S\ref{sec:zerobreak}. Since $n=2h>0$ both zeros have order at least one. Denote the newborn zeros as $P_1$ and $P_2$ respectively. Moreover, we can break $Q$ so that the resulting saddle connection, say $s$, joining $P_1$ and $P_2$ has length $\varepsilon$. According to our Lemma \ref{lem:techlemtwins}, there are $n-1$ edges, say $e_1,\dots,e_{n-1}$, leaving from $P_1$ so that $s,\,e_1,\dots,e_{n-1}$ are pairwise twins. By the same Lemma, there is an edge $e_n$ leaving from $P_2$ so that $s,\,e_n$ are twins. Orient $s$ from $P_1$ to $P_2$ and denote by $s^+$ and $s^-$ the right and left respectively. Let $e_i$ be any edge from $P_1$, slit the edges $e_i$ and $e_n$ and denote the resulting sides as $e_i^{\pm}$ and $e_n^{\pm}$ according to our convention. Then identify $e_i^+$ with $e_n^-$ and identify $e_i^-$ with $e_n^+$. The resulting space is a genus one differential in the stratum $\mathcal H_1(2h,\,2h, 2; -2^{2h-1}, -4)$. Suppose, without loss of generality that $e_1,\dots,e_{n-1}$ are ordered counterclockwise around $P_1$. If $i$ is odd, then it is easy to check that the resulting structure has rotation number $2$ whereas, if $i$ is even the resulting structure has rotation number one. This construction works as long as $n=2h\ge4$ and the case $n=2$, \textit{i.e.} $h=1$, needs a special treatment as follows. 

\smallskip

\begin{figure}[!ht]
    \centering
    \begin{tikzpicture}[scale=1, every node/.style={scale=0.9}]
    
    \definecolor{pallido}{RGB}{221,227,227}

    \foreach \x [evaluate=\x as \coord using  \x] in {0,8} 
    {
    \pattern [pattern=north west lines, pattern color=pallido]
    (\coord,6.5)--(\coord+7,6.5)--(\coord+7,0.5)--(\coord,0.5)--(\coord,6.5);
    
    \draw[thin, orange] (2,2.5)--(2,4.5);
    \draw[thin, red] (13,4.5)--(13,6.5);
    
    \fill [black] (2,2.5) circle (1.5pt); 
    \fill [black] (2,4.5) circle (1.5pt); 
    \fill [black] (13,4.5) circle (1.5pt); 
    
    \node at (\coord+6.5, 1) {$\mathbb C$};
    
    }
    
    \foreach \x [evaluate=\x as \coord using  \x] in {0,8} 
    {
    \pattern [pattern=north west lines, pattern color=pallido]
    (\coord,-6.5)--(\coord+7,-6.5)--(\coord+7,-0.5)--(\coord,-0.5)--(\coord,-6.5);
    
    \draw[thin, orange] (\coord+2,-2.5)--(\coord+2,-4.5);
    \draw[thin, red] (\coord+5, -2.5)--(\coord+5,-0.5);
    
    \fill [black] (\coord+2,-2.5) circle (1.5pt); 
    \fill [black] (\coord+2,-4.5) circle (1.5pt); 
    \fill [black] (\coord+5,-2.5) circle (1.5pt); 
    
    \node at (\coord+6.5, -6) {$\mathbb C$};
    
    }
    
    \draw[thin, blue] (1,-2.5)--(1,-4.5);
    \draw[thin, blue] (1, -2.5) arc [start angle=180, end angle=90 , radius = 1];
    \draw[thin, blue] (2, -1.5) arc [start angle=90, end angle=0 , radius = 1];
    \draw[thin, blue, ->] (3,-2.5)--(3,-4.5);
    \draw[thin, blue] (3, -4.5) arc [start angle=360, end angle=270 , radius = 1];
    \draw[thin, blue] (2, -5.5) arc [start angle=270, end angle=180 , radius = 1];
    
    \draw[thin, violet, ->] (2,-3.5)--(3.5, -3.5);
    \draw[thin, violet] (3.5, -3.5) arc [start angle=270, end angle=360 , radius = 0.5];
    \draw[thin , violet] (4,-3)--(4,-2);
    \draw[thin, violet] (4, -2) arc [start angle=180, end angle=90 , radius = 0.5];
    \draw[thin, violet] (4.5,-1.5)--(5, -1.5);
    \draw[thin, violet] (13,-1.5)--(13.5, -1.5);
    \draw[thin, violet] (13.5, -1.5) arc [start angle=90, end angle=0 , radius = 0.5];
    \draw[thin , violet] (14,-2)--(14,-5);
    \draw[thin, violet] (14, -5) arc [start angle=360, end angle=270 , radius = 0.5];
    \draw[thin, violet, ->] (13.5,-5.5)--(9.5, -5.5);
    \draw[thin, violet] (9.5, -5.5) arc [start angle=270, end angle=180 , radius = 0.5];
    \draw[thin , violet] (9,-5)--(9,-4);
    \draw[thin, violet] (9, -4) arc [start angle=180, end angle=90 , radius = 0.5];
    \draw[thin, violet] (9.5,-3.5)--(10,-3.5);
    \draw[thin, violet, <-] (3, 2.5) arc [start angle=0, end angle=360 , radius = 1];
    
    \node at (1.75, 4) {$e_1^-$};
    \node at (2.25, 4) {$e_1^+$};
    
    \node at (1.75, -4) {$e_3^-$};
    \node at (2.25, -4) {$e_3^+$};
    
    \node at (9.75, -4) {$e_4^-$};
    \node at (10.25, -4) {$e_4^+$};
    
    \node at (12.75, 5) {$r_2^-$};
    \node at (13.25, 5) {$r_2^+$};
    
    \node at (4.75, -2) {$r_3^-$};
    \node at (5.25, -2) {$r_3^+$};
    
    \node at (12.75, -2) {$r_4^-$};
    \node at (13.25, -2) {$r_4^+$};
    
    \node at (3.25, -4.5) {$\alpha$};
    \node at (8.75, -4.5) {$\beta$};

    \end{tikzpicture}
    \caption{Realizing a translation surface with trivial periods and rotation number one in the stratum $\mathcal H_1(2,2,2;\,-2,-4)$. The curve $\alpha$ has index one whereas the curve $\beta$ has index two.}
    \label{fig:h1222-2-41}
\end{figure}
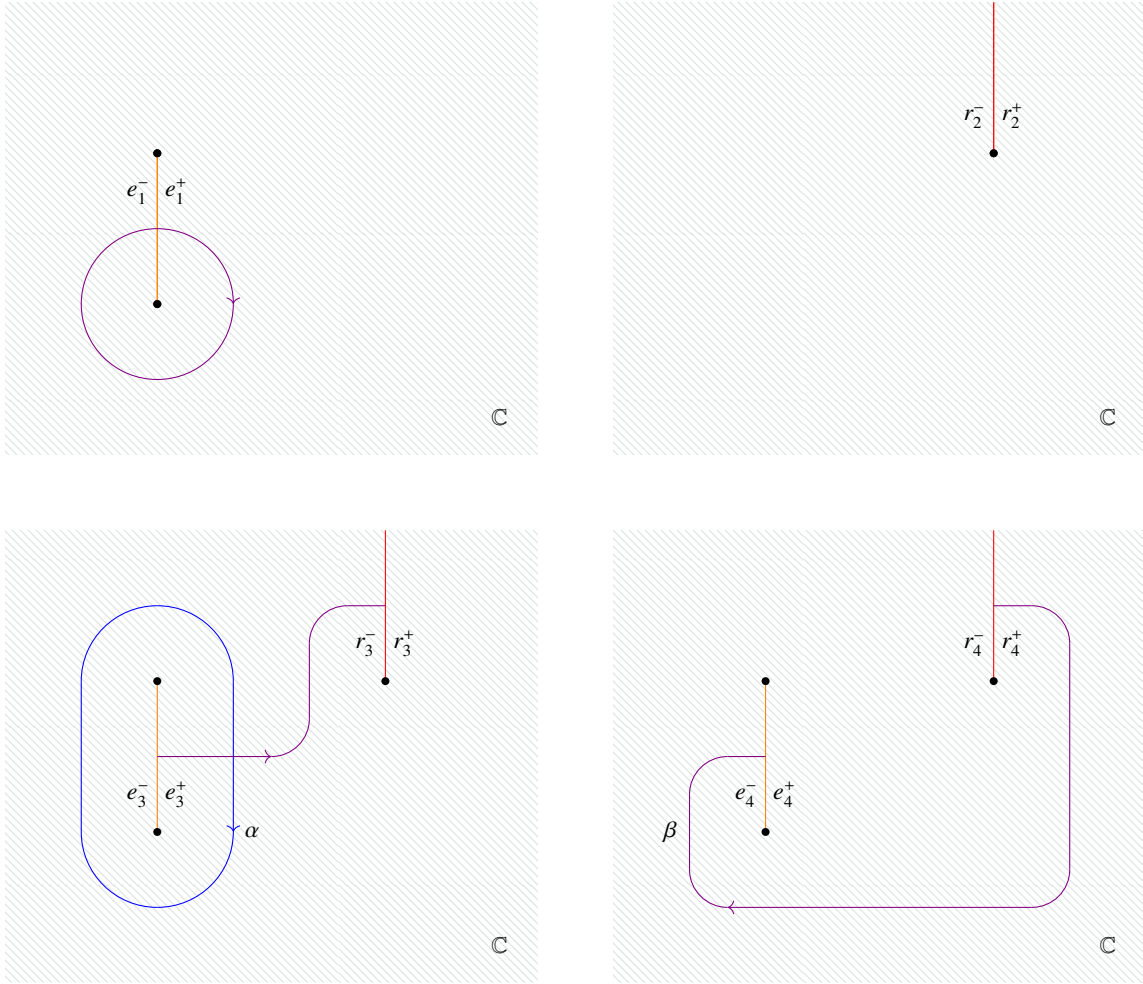

\noindent We provide explicit constructions for realizing the trivial representation in both connected components of the stratum $\mathcal H_1(2,2,2;\,-2,-4)$. Let us start by realizing a translation surface in this stratum with rotation number $1$. Let $P,\,Q,\,R$ be three not aligned points on the complex plane equipped with the standard differential $dz$. Let $e$ be the edge joining $P$ and $Q$ and let $r$ be a half-ray leaving from $R$ such that $e\,\cap\,r=\phi$. Consider four copies of $(\C,\,dz)$ and perform the following surgeries:
\begin{itemize}
    \item[1.] Slit the first copy along $e_1$ and denote the resulting sides as $e_1^{\pm}$,
    \smallskip
    \item[2.] slit the second copy along $r_2$ and denote the resulting sides as
    $r_2^{\pm}$, finally
    \smallskip
    \item[3.] slit the third and fourth copy along the edges $e_3$ and $e_4$ respectively and along the rays $r_3$ and $r_4$ respectively. Denote the resulting sides as $e_3^{\pm}$, $e_4^{\pm}$, $r_3^{\pm}$ and $r_4^{\pm}$;
\end{itemize}
where the signs are always taken according our convention. Identify the edge $e_1^+$ with $e_3^-$, then $e_3^+$ with $e_4^-$ and, finally, identify $e_4^+$ with $e_1^-$. Next identify the half-ray $r_2^+$ with $r_3^-$, then $r_3^+$ with $r_4^-$ and, finally, identify $r_4^+$ with $r_2^-$. The resulting surface is a genus one differential with trivial periods and rotation number $1$ as desired. In fact, a close loop around anyone of the $e_i$'s, say $\alpha$, has index one and this forces the rotation number to be one as well. See Figure \ref{fig:h1222-2-41}.

\smallskip

\noindent We finally realize the trivial representation in the connected component of $\mathcal H_1(2,2,2;\,-2,-4)$ with rotation number $2$. Once again let $P,\,Q,\,R$ be three not aligned points on the complex plane equipped with the standard differential $dz$. Let $a$ be the edge joining $P$ and $Q$, let $b$ be the edge joining $Q$ and $R$ and, finally, let $c$ the edge joining $P$ and $R$. As above, let $r$ be a half-ray leaving from $R$ such that it does not intersect none of the segments $a,b,c$. Consider four copies of $(\C,\,dz)$ and perform the following surgeries:
\begin{itemize}
    \item[1.] Slit the first copy along the edges $a_1,\,b_1$ and $c_1$ and the ray $r_1$. Denote the resulting sides as $a_1^{\pm}$, $b_1^{\pm}$, $c_1^{\pm}$ and $r_1^{\pm}$.
    \smallskip
    \item[2.] slit the second copy along $b_2$ and $r_2$ and denotes the resulting sides as $b_2^{\pm}$ and $r_2^{\pm}$, 
    \smallskip
    \item[3.] slit the third copy along the edges $c_3$ and $r_3$. Denote the resulting sides as $c_3^{\pm}$ and $r_3^{\pm}$, finally
    \smallskip
    \item[4.] slit the fourth copy along $a_4$ and denote the resulting sides as $a_4^{\pm}$, 
\end{itemize} 
where the signs are always taken according our convention. Identify the edge $a_1^+$ with $a_4^-$ and $a_1^-$ with $a_4^+$. Next, identify $b_1^+$ with $b_2^-$ and $b_1^-$ with $b_2^+$. Finally, identify $c_1^+$ with $c_3^-$ and $c_1^-$ with $c_3^+$. Then identify the half-ray $r_1^+$ with $r_2^-$, then $r_2^+$ with $r_3^-$ and finally $r_3^+$ with $r_1^-$. The resulting surface is a genus one differential with trivial periods and rotation number $2$ as desired. In fact, we can find two simple closed curves $\alpha,\,\beta$ such that $\textnormal{Ind}(\alpha)=\textnormal{Ind}(\beta)=2$. See Figure \ref{fig:h1222-2-42}.
\end{proof}

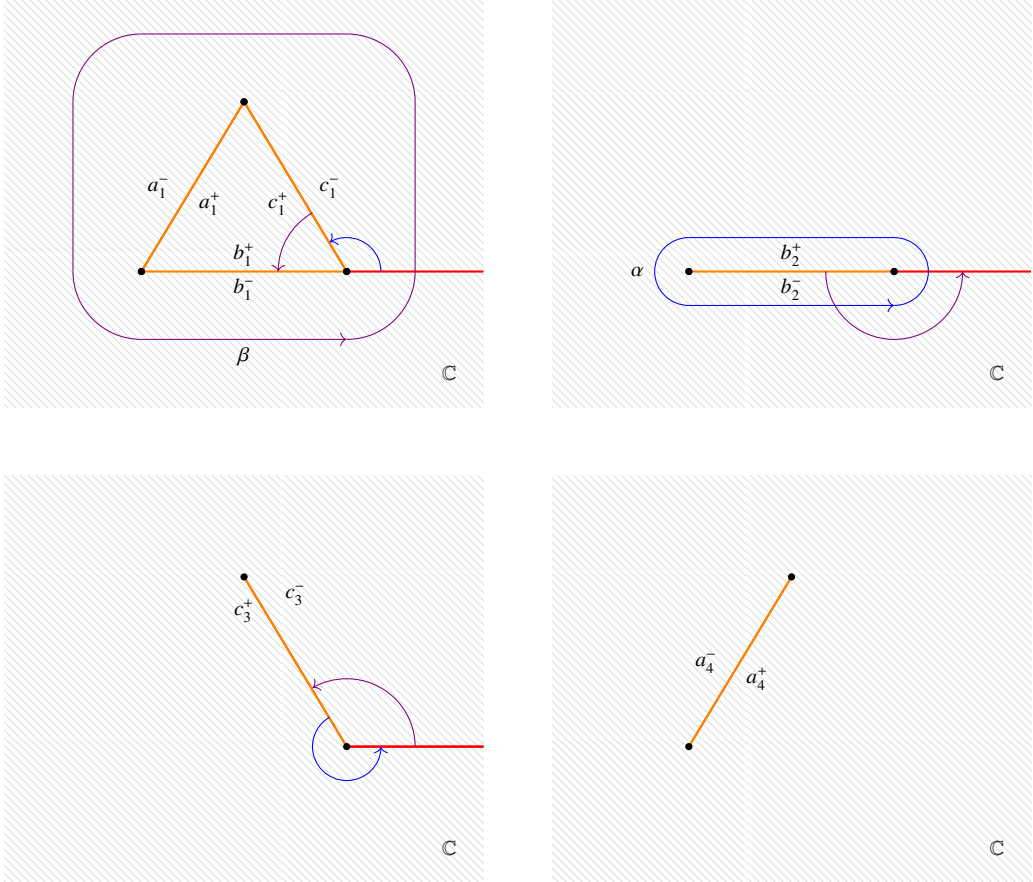
\begin{figure}[!ht]
    \centering
    \begin{tikzpicture}[scale=0.9, every node/.style={scale=0.8}]
    
    \definecolor{pallido}{RGB}{221,227,227}

    \foreach \x [evaluate=\x as \coord using  \x] in {0,8} 
    {
    \pattern [pattern=north west lines, pattern color=pallido]
    (\coord,6.5)--(\coord+7,6.5)--(\coord+7,0.5)--(\coord,0.5)--(\coord,6.5);
    
    \draw[thick, orange] (\coord+2, 2.5)--(\coord+5, 2.5);
    \draw[thick, orange] (2, 2.5)--(3.5, 5);
    \draw[thick, orange] (5, 2.5)--(3.5, 5);
    
    \draw[thick, red] (\coord+5, 2.5)--(\coord+7, 2.5);
    
    \fill [black] (\coord+2, 2.5) circle (1.5pt);
    \fill [black] (\coord+5, 2.5) circle (1.5pt);
    
    \fill [black] (3.5, 5) circle (1.5pt);

    \pattern [pattern=north west lines, pattern color=pallido]
    (\coord,-6.5)--(\coord+7,-6.5)--(\coord+7,-0.5)--(\coord,-0.5)--(\coord,-6.5);
    
    \draw[thick, orange] (5, -4.5)--(3.5, -2);
    \draw[thick, orange] (10, -4.5)--(11.5, -2);
    
    \draw[thick, red] (5, -4.5)--(7, -4.5);

    \node at (\coord+6.5,1) {$\mathbb C$};
    \node at (\coord+6.5,-6) {$\mathbb C$};
    
    }
    
    \draw[thin, blue] (10, 3)--(13, 3);
    \draw[thin, blue] (13, 3) arc [start angle=90, end angle=-90 , radius = 0.5];
    \draw[thin, blue, ->] (10, 2)--(13, 2);
    \draw[thin, blue] (10, 2) arc [start angle=270, end angle=90 , radius = 0.5];
    \draw[thin, blue, ->] (5.5, 2.5) arc [start angle=0, end angle=120 , radius = 0.5];
    \draw[thin, blue, <-] (5.5, -4.5) arc [start angle=360, end angle=120 , radius = 0.5];
    
    \draw[thin, violet] (6, 2.5)--(6, 5);
    \draw[thin, violet] (6,5) arc [start angle=0, end angle=90 , radius = 1];
    \draw[thin, violet] (5,6)--(2,6);
    \draw[thin, violet] (2,6) arc [start angle=90, end angle=180 , radius = 1];
    \draw[thin, violet] (1,5)--(1,2.5);
    \draw[thin, violet] (1,2.5) arc [start angle=180, end angle=270 , radius = 1];
    \draw[thin, violet,->] (2,1.5)--(5,1.5);
    \draw[thin, violet] (5,1.5) arc [start angle=270, end angle=360 , radius = 1];
    
    \draw[thin, violet, ->] (6, -4.5) arc [start angle=0, end angle=120 , radius = 1];
    
    \draw[thin, violet, <-] (4, 2.5) arc [start angle=180, end angle=120 , radius = 1];
    
    \draw[thin, violet, ->] (12, 2.5) arc [start angle=180, end angle=360 , radius = 1];
    
    \fill [black] ( 2, 2.5) circle (1.5pt);
    \fill [black] ( 5, 2.5) circle (1.5pt);
    \fill [black] (10, 2.5) circle (1.5pt);
    \fill [black] (13, 2.5) circle (1.5pt);
    
    \fill [black] (3.5, -2) circle (1.5pt);
    \fill [black] (11.5, -2) circle (1.5pt);
    \fill [black] (5, -4.5) circle (1.5pt);
    \fill [black] (10, -4.5) circle (1.5pt);
    
    \node at (3, 3.5) {$a_1^+$};
    \node at (2.25, 3.75) {$a_1^-$};
    
    \node at (11, -3.5) {$a_4^+$};
    \node at (10.25, -3.25) {$a_4^-$};
    
    \node at (3.5, 2.75) {$b_1^+$};
    \node at (3.5, 2.25) {$b_1^-$};
    
    \node at (11.5, 2.75) {$b_2^+$};
    \node at (11.5, 2.25) {$b_2^-$};
    
    \node at (4, 3.5) {$c_1^+$};
    \node at (4.75, 3.75) {$c_1^-$};
    
    \node at (3.5, -2.5) {$c_3^+$};
    \node at (4.25, -2.25) {$c_3^-$};
    
    \node at (9.25,2.5) {$\alpha$};
    \node at (3.5,1.25) {$\beta$};

    \end{tikzpicture}
    \caption{realizing a translation surface with trivial periods and rotation number one in the stratum $\mathcal H_1(2,2,2;\,-2,-4)$. It is not hard to check that both curves $\alpha$ and $\beta$ have index two.}
    \label{fig:h1222-2-42}
\end{figure}

\medskip

\noindent Finally, we still assume $\gcd(m_1,\dots,m_k,p_1,\dots,p_n)=2$ and we suppose that $p_i=2$ for all $i=1,\dots,n$, that is we consider the case of strata $\mathcal H_1(m_1,\dots,m_k;\,-2^n)$. In this particular case, since $m_i + m_j \geq n$ for all possible pairs $(i, j)$, we just have a few isolated cases to consider corresponding to $k=3,4$. In fact, equation \eqref{eq:condone} implies that $2n(k-2)\leq kn$; that is $k\leq 4$. Since $k\ge3$, it follows that $3\leq k \leq 4$.

\smallskip

\noindent If $k=4$, the only case to consider is the stratum $\mathcal H_1(m, m, m, m; -2^n)$ where $n = 2m$ and $m$ is even. In principle, one might need to deal with strata of the form $\mathcal H_1(2h_1,2h_2,2h_3,2h_4;\,-2^{h_1+h_2+h_3+h_4})$, where, without loss of generality, $h_1\leq h_2\leq h_3\leq h_4$. If we assume $2h_1+2h_2\geq n=h_1+h_2+h_3+h_4$, we readily obtain that the only possibility is that $h_1=h_2=h_3=h_4$.

\begin{lem}
Let $n=2m$. The trivial representation can be realized in each connected component of the stratum $\mathcal H_1(m,m,m,m; -2^n)$.
\end{lem}

\begin{proof}
In the first place we notice that if $m$ if odd then $\mathcal H_1(m,m,m,m;\,-2^n)$ is connected and hence the trivial representation is realisable in this stratum if and only if $m\leq n-1$, see \cite[Theorem B]{CFG}; that is satisfied for $m\geq1$ -- compare the Hurwitz type inequality \eqref{eq:triconstr} with the equation $4m=2n$. Therefore, in what follows we assume $m$ to be even. 

\smallskip

\noindent We first provide a generic construction that works as long as $m\geq4$ and subsequently we handle the case $m=2$ in a different way. Therefore suppose in the first place that $m\geq4$.  We realize the trivial representation as the period character of genus $0-$differential in the stratum $\mathcal H_0(2m-2,m,m;\,-2^n)$. Notice that this is possible because the Hurwitz type inequality \eqref{eq:triconstr} holds and the realization is guaranteed by \cite[Theorem B]{CFG}. Let $(Y,\xi)$ such a structure. Let $Q$ be a zero of order $2m-2$ and let $\varepsilon>0$ such that the open ball $B_{4\varepsilon}(Q)$ does not contain the other zero. Break $Q$ into two zeros of order $m-1$ and $m-1$ as described in \S\ref{sec:zerobreak}. Since $m>0$ is even, both zeros have order at least one. Denote the newborn zeros as $P_1$ and $P_2$ respectively. Moreover, we can break $Q$ so that the resulting saddle connection, say $s$, has length $\varepsilon$. Now we can proceed as in the first part of Lemma \ref{lem:h1222-2-4}. By Lemma \ref{lem:techlemtwins}, there are $m-1$ edges, say $e_1,\dots,e_{m-1}$, leaving from $P_1$ so that $s,\,e_1,\dots,e_{m-1}$ are pairwise twins and there are $m-1$ edges, say $e_m, \dots,e_{2m-2}$, leaving from $P_2$ so that $s,\,e_m,\dots,e_{2m-2}$ are twins. Orient $s$ from $P_1$ to $P_2$ and denote by $s^+$ and $s^-$ the right and left respectively. Let $e_i$ be any edge from $P_1$, slit the edges $e_i$ and $e_m$ and denote the resulting sides as $e_i^{\pm}$ and $e_m^{\pm}$ according to our convention. Then identify $e_i^+$ with $e_m^-$ and identify $e_i^-$ with $e_n^+$. The resulting space is a genus one differential in the stratum $\mathcal H_1(m,m,m,m;\,-2^n)$ with trivial periods as desired. Suppose, without loss of generality that $e_1,\dots,e_{m-1}$ are ordered counterclockwise around $P_1$. If $i$ is odd, then it is easy to check that the resulting structure has rotation number $2$ whereas, if $i$ is even the resulting structure has rotation number one.

\smallskip

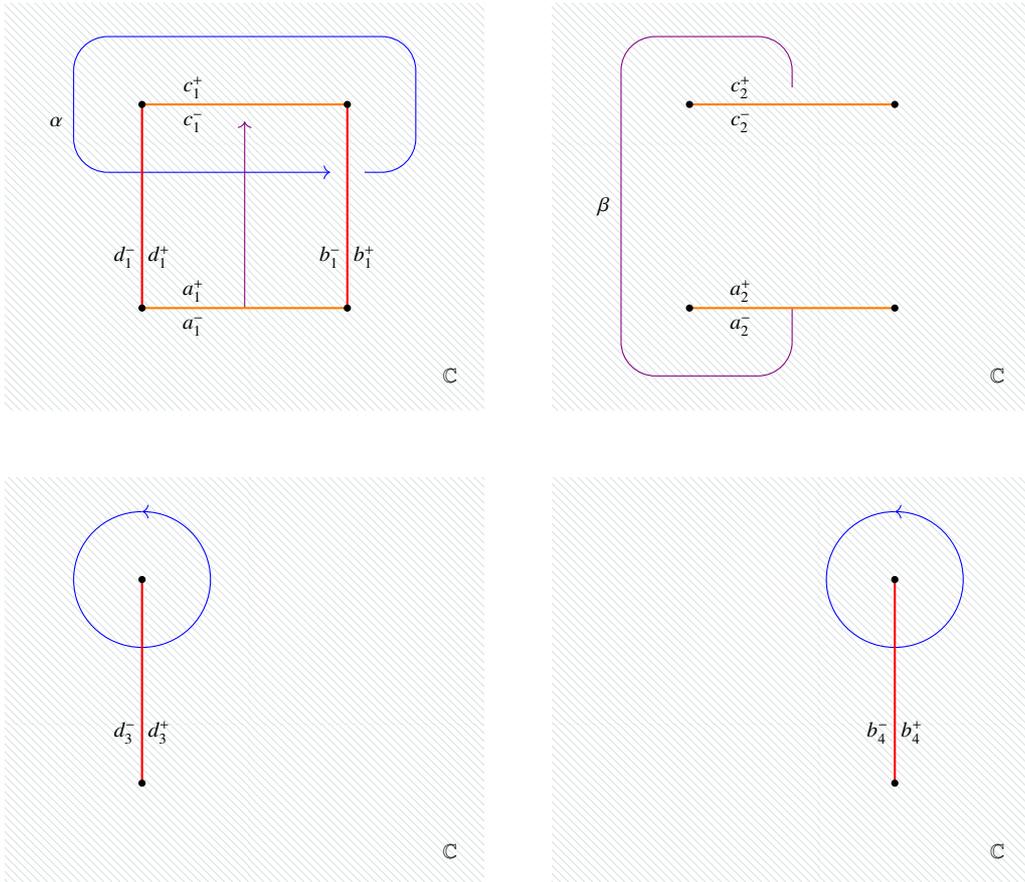
\begin{figure}[!ht]
    \centering
    \begin{tikzpicture}[scale=0.9, every node/.style={scale=0.8}]
    
    \definecolor{pallido}{RGB}{221,227,227}

    \foreach \x [evaluate=\x as \coord using  \x] in {0,8} 
    {
    \pattern [pattern=north west lines, pattern color=pallido]
    (\coord,6.5)--(\coord+7,6.5)--(\coord+7,0.5)--(\coord,0.5)--(\coord,6.5);

    \pattern [pattern=north west lines, pattern color=pallido]
    (\coord,-6.5)--(\coord+7,-6.5)--(\coord+7,-0.5)--(\coord,-0.5)--(\coord,-6.5);
    
    \node at (\coord+6.5, +1) {$\mathbb C$};
    \node at (\coord+6.5, -6) {$\mathbb C$};
    
    }
    
    \draw[blue, thin, ->] (1.5,4)--(4.75,4);
    \draw[blue, thin] (5.25,4)--(5.5,4);
    \draw[thin, blue] (5.5, 4) arc [start angle=270, end angle=360 , radius = 0.5];
    \draw[blue, thin] (6,4.5)--(6, 5.5);
    \draw[thin, blue] (6, 5.5) arc [start angle=0, end angle=90 , radius = 0.5];
    \draw[blue, thin] (1.5,6)--(5.5,6);
    \draw[thin, blue] (1.5, 6) arc [start angle=90, end angle=180 , radius = 0.5];
    \draw[blue, thin] (1,4.5)--(1, 5.5);
    \draw[thin, blue] (1, 4.5) arc [start angle=180, end angle=270 , radius = 0.5];
    
    \draw[thin, blue, ->] ( 2, -1) arc [start angle=90, end angle=450 , radius = 1];
    \draw[thin, blue, ->] (13, -1) arc [start angle=90, end angle=450 , radius = 1];
    
    \draw[thin, violet, ->] (3.5,2)--(3.5,4.75);
    \draw[thin, violet] (11.5, 5.25)--(11.5, 5.5);
    \draw[thin, violet] (11.5, 5.5) arc [start angle=0, end angle=90 , radius = 0.5];
    \draw[violet, thin] (11,6)--(9.5,6);
    \draw[thin, violet] (9.5, 6) arc [start angle=90, end angle=180 , radius = 0.5];
    \draw[violet, thin] (9,5.5)--(9, 1.5);
    \draw[thin, violet] (9, 1.5) arc [start angle=180, end angle=270 , radius = 0.5];
    \draw[thin, violet] (9.5, 1)--(11,1);
    \draw[thin, violet] (11, 1) arc [start angle=270, end angle=360 , radius = 0.5];
    \draw[violet, thin] (11.5, 1.5)--(11.5, 2);

    \draw[thick, orange] (2, 2)--(5,2);
    \draw[thick, orange] (2, 5)--(5,5);
    \draw[thick, orange] (10,2)--(13,2);
    \draw[thick, orange] (10,5)--(13,5);
    
    \draw[thick, red] (2,2)--(2,5);
    \draw[thick, red] (5,2)--(5,5);
    \draw[thick, red] (2,-2)--(2,-5);
    \draw[thick, red] (13,-2)--(13,-5);
    
    \fill [black] (2,2) circle (1.5pt);
    \fill [black] (5,2) circle (1.5pt);
    \fill [black] (2,5) circle (1.5pt);
    \fill [black] (5,5) circle (1.5pt);
    
    \fill [black] (10,2) circle (1.5pt);
    \fill [black] (13,2) circle (1.5pt);
    \fill [black] (10,5) circle (1.5pt);
    \fill [black] (13,5) circle (1.5pt);
    
    \fill [black] ( 2,-2) circle (1.5pt);
    \fill [black] ( 2,-5) circle (1.5pt);
    \fill [black] (13,-2) circle (1.5pt);
    \fill [black] (13,-5) circle (1.5pt);
    
    \node at (2.75, 1.75) {$a_1^-$};
    \node at (2.75, 2.25) {$a_1^+$};
    
    \node at (2.75, 4.75) {$c_1^-$};
    \node at (2.75, 5.25) {$c_1^+$};
    
    \node at (10.75, 1.75) {$a_2^-$};
    \node at (10.75, 2.25) {$a_2^+$};
    
    \node at (10.75, 4.75) {$c_2^-$};
    \node at (10.75, 5.25) {$c_2^+$};
    
    \node at (1.75, 2.75) {$d_1^-$};
    \node at (2.25, 2.75) {$d_1^+$};
    
    \node at (4.75, 2.75) {$b_1^-$};
    \node at (5.25, 2.75) {$b_1^+$};
    
    \node at (1.75, -4.25) {$d_3^-$};
    \node at (2.25, -4.25) {$d_3^+$};
    
    \node at (12.75, -4.25) {$b_4^-$};
    \node at (13.25, -4.25) {$b_4^+$};
    
    \node at (0.75, 4.75) {$\alpha$};
    \node at (8.75, 3.5) {$\beta$};
    
    \end{tikzpicture}
    \caption{Realizing a translation surface with trivial periods and rotation number one in the stratum $\mathcal H_1(2,2,2,2;\,-2^4)$. It is not hard to check that $\alpha$ has index three and $\beta$ has index one.}
    \label{fig:h12222-24}
\end{figure}

\noindent It remains to deal with the special case $m=2$. The argument above fails because, for $m=2$, we do not have any edge for realizing a translation surface with rotation number one in the stratum $\mathcal H_1(2,2,2,2;\,-2^4)$. However, the same proof still works for realizing a translation surface with trivial periods and rotation number two. 
We only realize the trivial representation in the connected component of $\mathcal H_1(2,2,2,2;\,-2^4)$ with rotation number one. Let $A,B,C$ and $D$ be the vertices of a square in $\C$. Let $a,b,c$ and $d$ denote the edges $AB,\,BC,\,CD$ and $DA$ respectively. Consider four copies of $(\C,\,dz)$ and perform the following slits.
\begin{itemize}
    \item[1.] In the first copy we slit all the edges $a_1,\,b_1,\,c_1$ and $d_1$,
    \smallskip
    \item[2.] then slit in the second copy the edges $a_2$ and $c_2$,
    \smallskip
    \item[3.] slit in the third copy the edge $d_3$ and, finally,
    \smallskip
    \item[4.] slit in the fourth copy the edge $b_4$.
\end{itemize}

\noindent Identify the edge $a_1^+$ with $a_2^-$ and $a_1^-$ with $a_2^+$. Similarly identify $c_1^+$ with $c_2^-$ and $c_1^-$ with $c_2^+$.  Then identify the edge $d_1^+$ with $d_3^-$ and $d_1^-$ with $d_3^+$. Finally identify $b_1^+$ with $b_4^-$ and $b_1^-$ with $b_4^+$. The resulting space is a genus one differential with trivial periods and it lies in the stratum $\mathcal H_1(2,2,2,2;\,-2^4)$. It remains to show that it has rotation number one as desired. This follows because we can find a pair of curves $\alpha,\,\beta$ so that $\textnormal{Ind}(\alpha)=3$ and $\textnormal{Ind}(\beta)=1$ as depicted in Figure \ref{fig:h12222-24}. 
\end{proof}

\noindent It remains the case $k=3$, that is the case of strata of the form $\mathcal H_1(m_1, m_2, m_3; -2^n)$ with $m_i$ all even. We can do induction from the two components of the lower stratum $\mathcal H_1(m_1 - 2, m_2 - 2, m_3; -2^{n-2})$, slit a saddle connection, say $s$ joining two zeros $P_1$ and $P_2$, and then add two copies of $(\C,\,dz)$ with the same slit.  If a closed curve crosses $s$, then after the operation its index alters by $2$, hence the rotation number remains unaltered. For this construction we need to check a few conditions.  First, we need 
\begin{equation}
    m_3 < 2(n-2) - (n-2) = n-2.
\end{equation} 

\noindent Notice that at least one of the three zeros satisfies this bound as soon as $n \geq 7$. In fact, in the case $m_i\geq n-2$ we get that $2n \geq 3(n-2)$ that implies $n\leq6$. We check the cases of small $n$ directly. Next, we need to find a saddle connection joining two specified zeros.  Indeed we can assume that there is a saddle connection joining any two of the three zeros to start with, and after adding more slit planes, clearly the statement of existing saddle connections between any two zeros still holds.  

\smallskip

\noindent So it reduces to check the induction base cases for strata $\mathcal H_1(2^3;\,-2^3)$, $\mathcal H_1(2,4,4,-2^5)$ and $\mathcal H_1(4,4,4,-2^6)$. Notice that, combinatorially, the stratum $\mathcal H_1(2,2,4;\,-2^4)$ is allowed but the trivial representation cannot be realized inside it because the Hurwitz type inequality \eqref{eq:triconstr} fails. We have the following Lemmas.

\begin{lem}
The trivial representation can be realized in each connected component of $\mathcal H_1(2,4,4,-2^5)$.
\end{lem}

\begin{proof}
\noindent In this case provide a unique construction and, as we shall see, the realization of the rotation number only depends on the choice of certain slits. The idea is to get the access to such a connected component from the stratum of genus zero differentials $\mathcal H_0(1,3,4,-2^5)$ as follows.
Let $A,\,B,\,C\in \C$ be distinct and not aligned points. Notice that the convex hull of $\{A,B,C\}$ is a non-degenerate triangle. Let $a,b,c$ denote the edges $AB$, $BC$ and $CA$. Consider six copies of $(\C,\,dz)$ and slit them as follows:
\begin{itemize}
    \item[1.] slit the first copy along the edge $b_1$,
    \smallskip
    \item[2.] slit the second copy along the edges $b_2$ and $c_2$, then
    \smallskip
    \item[3.] slit the remaining copies along $c_3$, $c_4$ and $c_5$ respectively;
    \smallskip
    \item[4.] finally slit the first copy along $a_1$ and either
    \smallskip
    \begin{itemize}
        \item[i.] slit the third copy along $a_3$ if we aim to realize a translation surface with rotation number one, otherwise
        \smallskip
        \item[ii.] slit the fourth copy along $a_4$ if we aim to realize a translation surface with rotation number two.
    \end{itemize}
\end{itemize}

\begin{figure}[!ht]
    \centering
    \begin{tikzpicture}[scale=1, every node/.style={scale=0.8}]
    
    \definecolor{pallido}{RGB}{221,227,227}
   
    \foreach \x [evaluate=\x as \coord using  \x] in {0,8} 
    {
    \pattern [pattern=north west lines, pattern color=pallido]
    (\coord,6.5)--(\coord+7,6.5)--(\coord+7,0.5)--(\coord,0.5)--(\coord,6.5);
    
    \draw[thin, red] (2, 5)--(5, 5);
    \draw[thin, orange] (10, 5)--(11.5, 2.5);
    \draw[thin, orange] (\coord+5, 5)--(\coord+3.5, 2.5);
    
    \fill [black] (\coord+2,5) circle (1.5pt);
    \fill [black] (\coord+5,5) circle (1.5pt);
    \fill [black] (\coord+3.5, 2.5) circle (1.5pt);
    
    \pattern [pattern=north west lines, pattern color=pallido]
    (\coord,-6.5)--(\coord+7,-6.5)--(\coord+7,-0.5)--(\coord,-0.5)--(\coord,-6.5);
    
    \draw[thin, red] (\coord+2, -2)--(\coord+5, -2);
    \draw[thin, orange] (\coord+2, -2)--(\coord+3.5, -4.5);
    
    \fill [black] (\coord+2, -2) circle (1.5pt);
    \fill [black] (\coord+5, -2) circle (1.5pt);
    \fill [black] (\coord+3.5, -4.5) circle (1.5pt);
    
    \node at (\coord+6.5, +1) {$\mathbb C$};
    \node at (\coord+6.5, -6) {$\mathbb C$};
    
    }
    
    \foreach \x [evaluate=\x as \coord using  \x] in {4} 
    {
    \pattern [pattern=north west lines, pattern color=pallido]
    (\coord,-13.5)--(\coord+7,-13.5)--(\coord+7,-7.5)--(\coord,-7.5)--(\coord,-13.5);
    
    \draw[thin, orange] (\coord+2, -9)--(\coord+3.5, -11.5);
    
    \fill [black] (\coord+2, -9) circle (1.5pt);
    \fill [black] (\coord+3.5, -11.5) circle (1.5pt);
    
    \node at (\coord+6.5, -13) {$\mathbb C$};
    
    }
    
    \draw[thin, blue] (2,5.5)--(5,5.5);
    \draw[thin, blue] (2,5.5) arc [start angle=90, end angle=270, radius = 0.5];
    \draw[thin, blue, <-] (2,4.5)--(4.5,4.5);
    \draw[thin, blue] (5,5.5) arc [start angle=90, end angle=-90, radius = 0.5];
    \draw[thin, blue] (4.825,4.5)--(5,4.5);
    
    \draw[thin, blue, <-] (13.5,5) arc [start angle=0, end angle=235, radius = 0.5];
    \draw[thin, blue] (13.5,5) arc [start angle=0, end angle=-115, radius = 0.5];
    
    \draw[thin, violet] (12.125, 3.25)--(13.5, 3.25);
    \draw[thin, violet] (13.5, 3.25) arc [start angle=270, end angle=360, radius = 0.5];
    \draw[thin, violet, ->] (14, 3.75)--(14, 5.5);
    \draw[thin, violet] (14, 5.5) arc [start angle=0, end angle=90, radius = 0.5];
    \draw[thin, violet] (13.5, 6)--(11, 6);
    \draw[thin, violet] (11, 6) arc [start angle=90, end angle=180, radius = 0.5];
    \draw[thin, violet] (10.5, 5.5)--(10.5, 4.35);
    
    \draw[thin, violet] (2.5, -3)--(2.5, -4.5);
    \draw[thin, violet] (2.5, -4.5) arc [start angle=180, end angle=270, radius = 1];
    \draw[thin, violet, dashed] (3.5, -5.5) arc [start angle=270, end angle=480, radius = 1];
    \draw[thin, violet] (3.5, -5.5)--(5, -5.5);
    \draw[thin, violet] (5, -5.5) arc [start angle=270, end angle=360, radius = 1];
    \draw[thin, violet] (6, -4.5)--(6, -2);
    \draw[thin, violet] (6,-2) arc [start angle=0, end angle=90, radius = 1];
    \draw[thin, violet] (5, -1)--(4,-1);
    \draw[thin, violet] (4, -1) arc [start angle=90, end angle=177.5, radius = 1];
    
    \draw[thin, violet, ->] (3, 5) arc [start angle=180, end angle=237.5, radius = 2];
    
    \draw[thin, violet, dashed] (10.5, -4.5) arc [start angle=180, end angle=120, radius = 1];
    \draw[thin, violet, dashed] (10.5, -4.5) arc [start angle=180, end angle=270, radius = 1];
    \draw[thin, violet, dashed] (11.5, -5.5)--(13, -5.5);
    \draw[thin, violet, dashed] (13, -5.5) arc [start angle=270, end angle=360, radius = 1];
    \draw[thin, violet, dashed] (14, -4.5)--(14, -2);
    \draw[thin, violet, dashed] (14, -2) arc [start angle=0, end angle=90, radius = 1];
    \draw[thin, violet, dashed] (13,-1)--(12,-1);
    \draw[thin, violet, dashed] (12, -1) arc [start angle=90, end angle=180, radius = 1];

    \node at (3.5, 4.75) {$a_1^-$};
    \node at (3.5, 5.25) {$a_1^+$};
        
    \node at (4.5, 3.5) {$b_1^+$};
    \node at (4, 4) {$b_1^-$};
    
    \node at (12.5, 3.5) {$b_2^+$};
    \node at (12, 4) {$b_2^-$};

    \node at (10.5, 3.5) {$c_2^+$};
    \node at (11, 4) {$c_2^-$};
    
    \node at (3.5, -2.25) {$a_3^-$};
    \node at (3.5, -1.75) {$a_3^+$};
    
    \node at (2.25, -2.85) {$c_3^+$};
    \node at (2.5,  -2.35) {$c_3^-$};
    
    \node at (11.5, -2.25) {$a_4^-$};
    \node at (11.5, -1.75) {$a_4^+$};
    
    \node at (10.5, -3.5) {$c_4^+$};
    \node at (11, -3) {$c_4^-$};

    \node at (6.5, -10.5) {$c_5^+$};
    \node at (7, -10) {$c_5^-$};
    
    \node at (1.25, 5) {$\alpha$};
    \node at (6.25, -3.5) {$\beta$};
    \node at (14.25, -3.5) {$\beta$};
    \end{tikzpicture}
    \caption{Realization of a genus one differential in $\mathcal H_1(2,4,4;\,-2^5)$ with trivial periods and prescribed rotation number. In both cases, the curve $\alpha$ can be taken as the blue curve depicted. The curve $\beta$ depends on which edge we decide to slit between $a_3$ and $a_4$. In the former case $\beta$ has index $2$ and the rotation number of the final structure will be two. In the latter case $\beta$ has index $3$ and the rotation number of the final structure will be one.}
    \label{fig:h12442-5}
\end{figure}
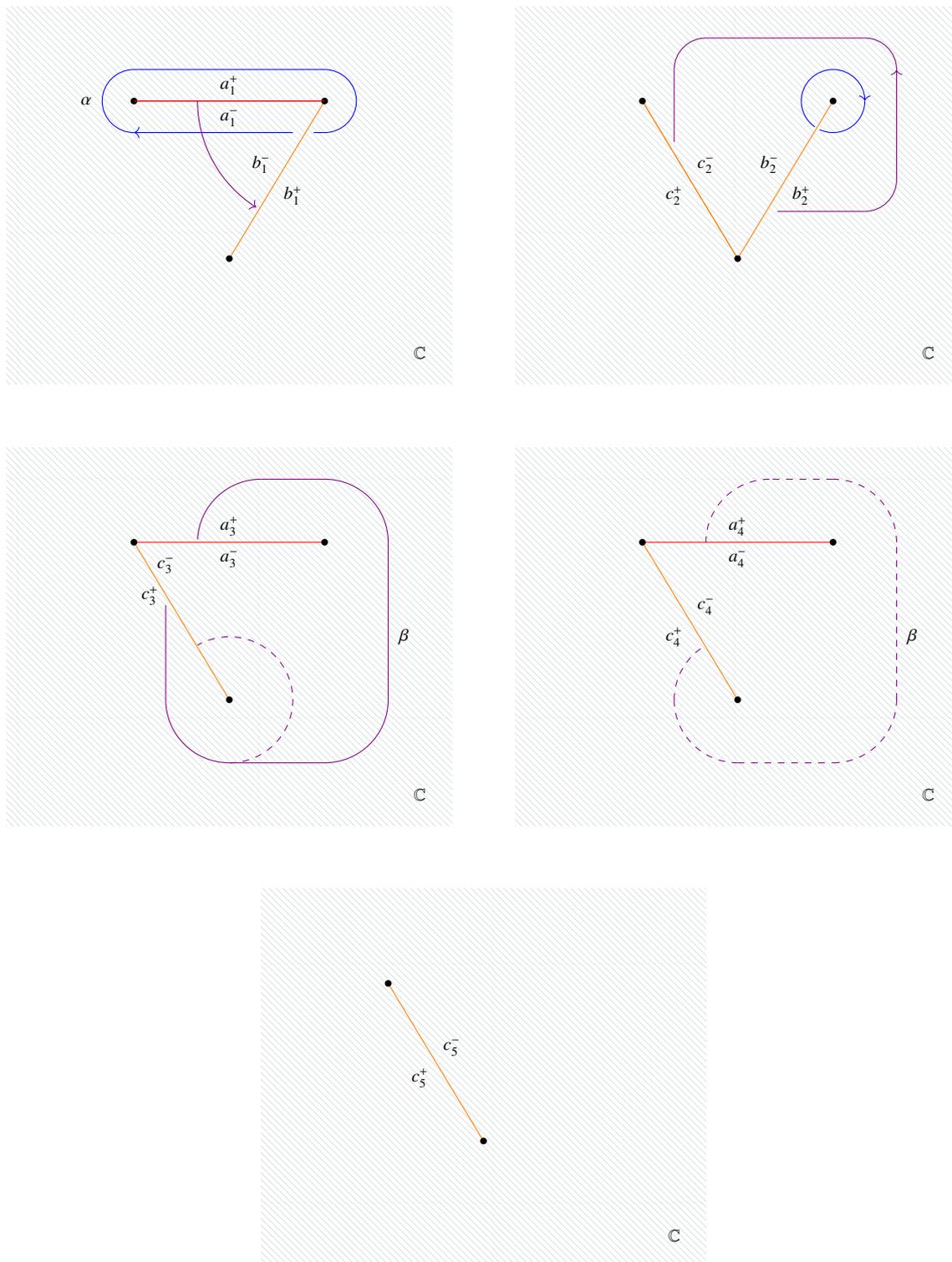

\noindent Label the resulting sides with $\pm$ according to our convention. Identify the edges as follows: $c_1^+$ with $c_2^-$ and $c_1^-$ with $c_2^+$. Next, identify $c_i^-$ with $c_{i+1}^+$ for $i=2,3,4$ and $c_5^-$ with $c_2^+$. The resulting surface is a genus zero differential in $\mathcal H_0(1,3,4;\,-2^5)$. The last step is to identify $a_1^+$ with $a_4^-$ and $a_1^-$ with $a_4^+$ in the case we aim to realize a genus one differential in $\mathcal H_1(2,4,4;\,-2^5)$ with trivial periods and rotation number one. Otherwise identify $a_1^+$ with $a_3^-$ and $a_1^-$ with $a_3^+$ in the case we aim to realize a genus one differential in $\mathcal H_1(2,4,4;\,-2^5)$ with trivial periods and rotation number two. See Figure \ref{fig:h12442-5}.
\end{proof}

\begin{lem}
The trivial representation can be realized in each connected component of $\mathcal H_1(4,4,4,-2^6)$.
\end{lem}

\begin{proof}
Even in this case provide a unique construction and, as we shall see, the realization of the rotation number only depends on the choice of certain slits. Let $A,\,B,\,C\in \C$ be distinct and not aligned points. Notice that the convex hull of $\{A,B,C\}$ is a non-degenerate triangle. 

\smallskip

\noindent Let $a,b,c$ denote the edges $AB$, $BC$ and $CA$. Consider six copies of $(\C,\,dz)$ and slit them as follows:
\begin{itemize}
    \item[1.] slit the first copy along the edge $a_1$,
    \smallskip
    \item[2.] slit the second copy along the edges $a_2$, $b_2$ and $c_2$,
    \smallskip
    \item[3.] slit the third and fourth copy along the edges $a_3$ and $a_4$ and the edges $b_3$ and $b_4$
    \smallskip
    \item[4.] slit the fifth and sixth copy along the edges $c_5$ and $c_6$, finally
    \smallskip
    \begin{itemize}
        \item[i.] if we aim to realize a translation surface with rotation number one then slit the fifth copy along the edge $a_5$, otherwise
        \item[ii.] if we aim to realize a translation surface with rotation number two slit the sixth copy along the edge $a_6$.
    \end{itemize}
\end{itemize}

\begin{figure}[!ht]
    \centering
    \begin{tikzpicture}[scale=1, every node/.style={scale=0.8}]
    
    \definecolor{pallido}{RGB}{221,227,227}

    \foreach \x [evaluate=\x as \coord using  \x] in {0,8} 
    {
    \pattern [pattern=north west lines, pattern color=pallido]
    (\coord,6.5)--(\coord+7,6.5)--(\coord+7,0.5)--(\coord,0.5)--(\coord,6.5);
    
    \draw[thin, orange] (\coord+2, 5)--(\coord+5,5);
    \draw[thin, orange] (10, 5)--(11.5, 2.5);
    \draw[thin, orange] (13, 5)--(11.5, 2.5);
    
    \fill [black] (\coord+2,5) circle (1.5pt);
    \fill [black] (\coord+5,5) circle (1.5pt);
    \fill [black] (11.5, 2.5) circle (1.5pt);
    
    \pattern [pattern=north west lines, pattern color=pallido]
    (\coord,-6.5)--(\coord+7,-6.5)--(\coord+7,-0.5)--(\coord,-0.5)--(\coord,-6.5);
    
    \draw[thin, red] (\coord+2, -2)--(\coord+5, -2);
    \draw[thin, orange] (\coord+5, -2)--(\coord+3.5, -4.5);
    
    \fill [black] (\coord+2, -2) circle (1.5pt);
    \fill [black] (\coord+5, -2) circle (1.5pt);
    \fill [black] (\coord+3.5, -4.5) circle (1.5pt);

    \pattern [pattern=north west lines, pattern color=pallido]
    (\coord,-13.5)--(\coord+7,-13.5)--(\coord+7,-7.5)--(\coord,-7.5)--(\coord,-13.5);
    
    \draw[thin, red] (\coord+2, -9)--(\coord+5, -9);
    \draw[thin, orange] (\coord+2, -9)--(\coord+3.5, -11.5);
    
    \fill [black] (\coord+2, -9) circle (1.5pt);
    \fill [black] (\coord+5, -9) circle (1.5pt);
    \fill [black] (\coord+3.5, -11.5) circle (1.5pt);
    
    \node at (\coord+6.5,   1) {$\mathbb C$};
    \node at (\coord+6.5,  -6) {$\mathbb C$};
    \node at (\coord+6.5, -13) {$\mathbb C$};
    
    }
    
    \draw[thin, blue, <-] (4.5,5) arc [start angle=180, end angle=-180 , radius = 0.5];
    \draw[thin, blue, <-] (12.5,5) arc [start angle=180, end angle=-120 , radius = 0.5];
    \draw[thin, blue, ->] (12.5,5) arc [start angle=180, end angle=240 , radius = 0.5];
    \draw[thin, blue, ->] (4.5,-2) arc [start angle=180, end angle=240 , radius = 0.5];
    \draw[thin, blue] (4.5,-2) arc [start angle=180, end angle=-120 , radius = 0.5];
    \draw[thin, blue] (13,-1.5) arc [start angle=90, end angle=-120 , radius = 0.5];
    \draw[thin, blue] (13, -1.5)--(10,-1.5);
    \draw[thin, blue] (10,-1.5) arc [start angle=90, end angle=270 , radius = 0.5];
    \draw[thin, blue, <-] (12.5, -2.5)--(10,-2.5);
    
    \draw[thin, violet] (4,-4.5) arc [start angle=0, end angle=50 , radius = 0.5];
    \draw[thin, violet] (4,-4.5) arc [start angle=0, end angle=-300, radius = 0.5];
    \draw[thin, violet] (11.5,-5) arc [start angle=-90, end angle=50, radius = 0.5];
    \draw[thin, violet] (11.5,-5)--(11, -5);
    \draw[thin, violet] (11,-5) arc [start angle=270, end angle=180, radius = 0.5];
    \draw[thin, violet, <-] (10.5, -4.5)--(10.5, -2);
    
    \draw[thin, violet, dashed] (2.5, -9) arc [start angle=180, end angle=90, radius = 0.5];
    \draw[thin, violet, dashed] (3, -8.5)--(5, -8.5);
    \draw[thin, violet, dashed] (5, -8.5) arc [start angle=90, end angle=0, radius = 0.5];
    \draw[thin, violet, dashed] (5.5, -9)--(5.5, -11.5);
    \draw[thin, violet, dashed] (5.5, -11.5) arc [start angle=0, end angle=-90, radius = 0.5];
    \draw[thin, violet, dashed] (3.5, -12)--(5, -12);
    
    \draw[thin, violet] (3.5, -12) arc [start angle=270, end angle=130, radius = 0.5];
    \draw[thin, violet] (3.5, -12) arc [start angle=-90, end angle=110, radius = 0.5];
    
    \draw[thin, violet] (11.5, -12) arc [start angle=270, end angle=130, radius = 0.5];
    \draw[thin, violet] (11.5, -12)--(13, -12);
    \draw[thin, violet] (13, -12) arc [start angle=-90, end angle=0, radius = 0.5];
    \draw[thin, violet] (13.5, -11.5)--(13.5, -9);
    \draw[thin, violet] (13.5, -9) arc [start angle=0, end angle=90, radius = 0.5];
    \draw[thin, violet] (13, -8.5)--(11, -8.5);
    \draw[thin, violet] (11, -8.5) arc [start angle=90, end angle=180, radius = 0.5];
    
    \draw[thin, violet] (11.75, 3) arc [start angle=60, end angle=120, radius = 0.5];
    
    \node at (3.5, 4.75) {$a_1^-$};
    \node at (3.5, 5.25) {$a_1^+$};
    
    \node at (11.5, 4.75) {$a_2^-$};
    \node at (11.5, 5.25) {$a_2^+$};
    
    \node at (12.5, 3.5) {$b_2^+$};
    \node at (12, 4) {$b_2^-$};
    
    \node at (10.5, 3.5) {$c_2^+$};
    \node at (11, 4) {$c_2^-$};
    
    \node at (3.5, -2.25) {$a_3^-$};
    \node at (3.5, -1.75) {$a_3^+$};
    
    \node at (4.5, -3.5) {$b_3^+$};
    \node at (4, -3) {$b_3^-$};
    
    \node at (11.5, -2.25) {$a_4^-$};
    \node at (11.5, -1.75) {$a_4^+$};
    
    \node at (12.5, -3.5) {$b_4^+$};
    \node at (12, -3) {$b_4^-$};
    
    \node at (3.5, -9.25) {$a_5^-$};
    \node at (3.5, -8.75) {$a_5^+$};
    
    \node at (2.5, -10.5) {$c_5^+$};
    \node at (3, -10) {$c_5^-$};
    
    \node at (11.5, -9.25) {$a_6^-$};
    \node at (11.5, -8.75) {$a_6^+$};
    
    \node at (10.5, -10.5) {$c_6^+$};
    \node at (11, -10) {$c_6^-$};
    
    \node at (9.25, -2) {$\alpha$};
    \node at (13.75, -10.5) {$\beta$};
    \node at (5.75, -10.5) {$\beta$};

    \end{tikzpicture}
    \caption{Realization of a genus one differential in $\mathcal H_1(4,4,4;\,-2^6)$ with trivial periods and prescribed rotation number. In both cases, the curve $\alpha$ can be taken as the blue curve depicted. The curve $\beta$ depends on which edge we decide to slit between $a_5$ and $a_6$. In the former case $\beta$ has index $3$ and the rotation number of the final structure will be one. In the latter case $\beta$ has index $4$ and the rotation number of the final structure will be two.}
    \label{fig:h14442-6}
\end{figure}
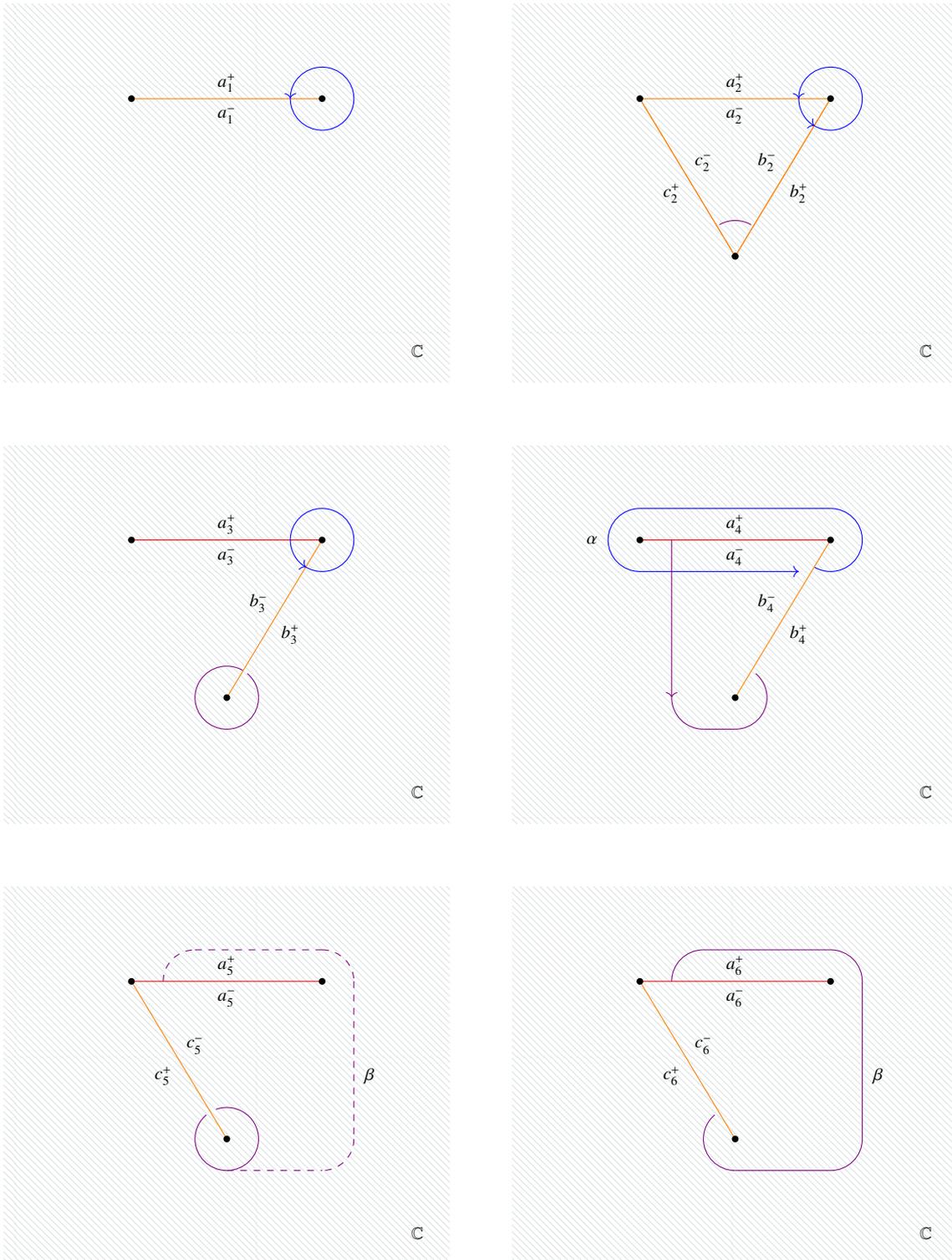

\noindent Label the resulting sides with $\pm$ according to our convention. Identify the edges as follows: $a_1^+$ with $a_2^-$ and $a_1^-$ with $a_2^+$. Next identify the edges $b_2^+$ with $b_3^-$, then $b_3^+$ with $b_4^-$ and, finally, $b_4^+$ with $b_2^-$. Similarly, identify the edges $c_2^+$ with $c_5^-$, then $c_5^+$ with $c_6^-$ and, finally, $c_6^+$ with $c_2^-$. The resulting surface is a genus zero differential in the stratum $\mathcal H_0(3,3,4;-2^6)$. The last step is to identify $a_3^+$ with $a_5^-$ and $a_3^-$ with $a_5^+$ in the case we aim to realize a genus one differential in $\mathcal H_1(4,4,4;\,-2^6)$ with trivial periods and rotation number one. Otherwise identify $a_3^+$ with $a_6^-$ and $a_3^-$ with $a_6^+$ in the case we aim to realize a genus one differential in $\mathcal H_1(4,4,4;\,-2^6)$ with trivial periods and rotation number two. See Figure \ref{fig:h14442-6}.
\end{proof}

\begin{lem}\label{lem:222-2-2-2}
The stratum $\mathcal H_1(2^3;\,-2^3)$ is exceptional; in fact the trivial representation can only be realized in the non-primitive connected component.
\end{lem}

\begin{proof}
Suppose that $(S,\omega)$ is a translation surface in $\mathcal H_1(2^3;\,-2^3)$ with trivial periods. Let $Z_i$ and $P_i$ be the zeros and poles of $\omega$ for $i=1,2,3$. The trivial period condition is equivalent to the existence of a triple cover of $S$ to $\mathbb{C}\mathbf{P}^1$ totally ramified at the $Z_i$ and having $P_1+P_2+P_3$ as the fiber over infinity. This implies the linear equivalence relation $3Z_1\sim 3Z_2\sim 3Z_3 \sim P_1+P_2+ P_3$. On the other hand, by assumption $2Z_1 + 2Z_2 + 2Z_3\sim 2P_1 + 2P_2 + 2P_3$. It follows that 
$Z_1+Z_2+Z_3 \sim 3Z_1+3Z_2+3Z_3-(2Z_1 + 2Z_2 +2Z_3) \sim P_1+P_2+P_3$, which implies by definition that $(S,\omega)$ belongs to the non-primitive component of $\mathcal H_1(2^3;\,-2^3)$.  
\end{proof}

\noindent This complete the proof of Proposition \ref{prop:genusonetri} and indeed the proof of Theorem \ref{thm:mainthm2} for genus one-differentials.

\medskip

\subsection{Hyperelliptic exact differentials}\label{ssec:hyptritwozeros} In subsection \S\ref{ssec:trirephyp} we have seen that the trivial representation cannot appear as the period character of a hyperelliptic translation surface with a single zero. However, it is still possible to realize trivial representation as the period character of some exact hyperelliptic differentials in some cases. As alluded above, the order of zeros of a meromorphic exact differential $\omega$, on a Riemann surface $X$, are subject to the constraint provided by the Hurwitz type inequality \eqref{eq:triconstr}:
\begin{equation*}
    m_j\le \sum_{i=1}^n p_i-n-1.
\end{equation*}

\noindent Such a formula, for strata $\mathcal{H}_g(m,m;-2p)$ and $\mathcal{H}_g(m,m;-p,-p)$ simplifies to 
\begin{equation}\label{eq:hypcondone}
    m\le 2p-2 \qquad m\le 2p-3
\end{equation} respectively. We also recall that, according to our Remark \ref{gbcond}, the order of zeros and poles are subject to the so-called Gauss-Bonnet condition
\begin{equation}\label{eq:hypcondtwo}
    2m=2g+2p-2, \,\,\textit{ i.e. }\,\, m=g+p-1;
\end{equation} 
otherwise the stratum would be empty. Equations \eqref{eq:hypcondone} and \eqref{eq:hypcondtwo} above combined together readily imply the following:

\begin{lem}
The trivial representation appears as the period of a translation surface in $\mathcal{H}_g(m,m;-2p)$ if and only if $g\le p-1$. Similarly, the trivial representation appears as the period of a translation surface with poles in $\mathcal{H}_g(m,m;-p,-p)$ if and only if $g\le p-2$.
\end{lem}

\noindent We next wonder whether the trivial representation appears as the period of some hyperelliptic translation surface. 

\begin{prop}\label{prop:hgdiffhyptriv}
Suppose the trivial representation can be realized in a stratum admitting a hyperelliptic component. Then it can be realized as the period character of some hyperelliptic translation surfaces with poles in the same stratum.
\end{prop}

\smallskip

\noindent The remaining part of this subsection is devoted to prove Proposition \ref{prop:hgdiffhyptriv}. We shall distinguish two cases as follows.

\subsubsection{Single higher order pole}\label{sssec:trihypsinglepole} Let $(X,\omega)$ be a genus $0$-meromorphic differential in $\mathcal{H}_0(p-1,p-1;-2p)$, where $p\ge2$. This structure can be easily realized as follows. On $(\C,\,dz)$ we consider two distinct points, say $P,\,Q\in\C$ and let $r$ be the unique straight line passing through them. Denote by $s_0$ the segment joining $P$ and $Q$ and let $r_1$ and $r_2$ be the sub-rays of $r$ leaving from $P$ and $Q$ respectively. We next bubble $p-1$ copies of $(\C,\,dz)$ along $r_1$ and, in the same fashion, we bubble $p-1$ copies of $(\C,\,dz)$ along $r_2$. The resulting structure $(X,\omega)$ is a translation surface with poles in $\mathcal{H}_0(p-1,p-1;-2p)$ having trivial periods. Notice that the straight line $r$ splits the complex plane $\C$ into two half-planes. We orient $r$ so that $s_0$ is oriented from $P$ to $Q$. According to this orientation, we denote by $H_1$ the half plane bounded by $r^-$ and by $H_2$ the half plane bounded by $r^+$. Around $P\in(X,\omega)$, we can single out $p+1$ sectors, say $H_1,\, S_1,\dots, S_{p-1},\,H_2$, in cyclic order. For $k=1,\dots,p-1$, each sector $S_k$ comes from a copy of $(\C,\,dz)$ glued along $r_1$ by construction. In particular, $S_k$ contains a segment, say $s_k$, based at $P$ that bounds a wedge of angle $2k\pi$ with $s_0$. Let $Q_k$ be extremal point of $s_k$ other than $P$.

\smallskip

\noindent Similarly, around $Q\in(X,\omega)$, we can single out other $p+1$ sectors, say $H_2,\, S_{p},\dots, S_{2p-2},\,H_1$, in cyclic order. In this case each sector comes from a copy of $(\C,\,dz)$ glued along the half-ray $r_2$. For $k=1,\dots,p-1$, each sector $S_{p+k-1}$, for $k=1,\dots,p-1$, comes from a copy of $(\C,\,dz)$ glued along $r_2$ by construction. In this case, $S_{p+k-1}$ contains a segment, say $s_{p+k-1}$, based at $Q$ that bounds a wedge of angle $2k\pi$ with $s_0$. Finally, let $P_k$ be extremal point of $s_{p+k-1}$ other than $Q$.

\smallskip

\noindent Any orientation on $s_0$ naturally yields a preferred orientation on each segment $s_k$ and $s_{p+k-1}$, for any $k$, because they all have the same image under the developing map. Recall that $s_0$ is oriented from $P$ to $Q$. As a consequence, $s_k$ is oriented from $P$ to $Q_k$ and $s_{p+k-1}$ is oriented from $P_k$ to $Q$ for all $k=1,\dots,p-1$. Let $1\le g\le p-1$. For $k=1,\dots,g$, slit $(X,\omega)$ along $s_k$ and $s_{p+k-1}$ and denote the resulting edges by $s_k^{\pm},\,s_{p+k-1}^{\pm}$, where the sign is taken according to our convention. Finally, identify $s_k^+$ with $s_{p+k-1}^-$ and $s_k^-$ with $s_{p+k-1}^+$. The resulting space is surface of genus $g$ equipped with a translation structure $(Y,\,\xi)\in\mathcal{H}_g(m,m;-2p)$, where $m=g+p-1$. 

\smallskip

\noindent By construction, $\xi$ is an exact differential on $Y$; that is $(Y,\xi)$ has trivial absolute periods. Since all sectors are cyclically ordered as $H_1,S_1,\dots,S_{p-1},H_2,S_{p},\dots,S_{2p-2}$, the structure $(Y,\xi)$ is also naturally equipped with an involution of order two with $2g+2$ fixed points given by: the mid-point of $s_0$, the mid-points of $s_k^{\pm}$ and of $s_{p+k-1}^{\pm}$ for $k=1,\dots,g$ and the pole, see Figure \ref{fig:hyptrigenusone}. Therefore the structure $(Y,\xi)$ is hyperelliptic with trivial periods as desired.

\begin{figure}[!ht] 
\centering
\begin{tikzpicture}[scale=1.05, every node/.style={scale=0.85}]
\definecolor{pallido}{RGB}{221,227,227}
\definecolor{pallido2}{RGB}{221,227,240}
    \pattern [pattern=north east lines, pattern color=pallido]
    (-6,3)--(-6,0)--(6,0)--(6,3)--(-6,3);
    
    \pattern [pattern=north east lines, pattern color=pallido]
    (-6,-3)--(-6,0)--(6,0)--(6,-3)--(-6,-3);

    \draw[thick, orange] (-2,0)--(2,0);
    \draw[ultra thin, dashed] (-2,0)--(-6,  0);
    \draw[ultra thin, dashed] ( 2,0)--( 6,  0);
    \draw[ultra thin, dashed] (-2,0)--(-2,  3);
    \draw[ultra thin, dashed] (-2,0)--(-2, -3);
    \draw[ultra thin, dashed] ( 2,0)--( 2,  3);
    \draw[ultra thin, dashed] ( 2,0)--( 2, -3);
    
    \draw[thin, black] (-2,0) to [bend left] (-4.5,2);
    \draw[thin, black] (-2,0) to [bend right] (-4.5,2);
    
    \draw[thin, black] (-2,0) to [bend left] (-4.5,-2);
    \draw[thin, black] (-2,0) to [bend right] (-4.5,-2);
    
    \draw[thin, black] (2,0) to [bend left] (4.5,2);
    \draw[thin, black] (2,0) to [bend right] (4.5,2);
    
    \draw[thin, black] (2,0) to [bend left] (4.5,-2);
    \draw[thin, black] (2,0) to [bend right] (4.5,-2);
    
    \fill [white] (-2,0) to [bend left] (-4.5,2) to [bend left] (-2,0);
    \fill [white] (-2,0) to [bend left] (-4.5,-2) to [bend left] (-2,0);
    
    \fill [white] (2,0) to [bend left] (4.5,2) to [bend left] (2,0);
    \fill [white] (2,0) to [bend left] (4.5,-2) to [bend left] (2,0);
    
    \node at (0,0) {$\times$};
    \draw[thin, black, ->] (0, 0.5) arc [start angle = 90, end angle = -260,radius = 0.5];
    
    \node [red]  at (-2.9, 1.3) {$\times$};
    \node [blue] at (-3.5, 0.6) {$\times$};
    \node [red]  at (2.9, -1.3) {$\times$};
    \node [blue] at (3.5, -0.6) {$\times$};
    \node [green] at (-3.5, -0.6) {$\times$};
    \node [green] at ( 3.5, 0.6) {$\times$};
    \node [violet] at (2.9, 1.3) {$\times$};
    \node [violet] at (-2.9, -1.3) {$\times$};
    
    \draw [thin, black] (4.5, 2) circle (2.5pt);
    \draw [thin, black] (4.5,-2) circle (2.5pt);
    \fill [white] ( 4.5, 2) circle (2.5pt);
    \fill [white] ( 4.5,-2) circle (2.5pt);
    \fill [black] (-4.5, 2) circle (2.5pt);
    \fill [black] (-4.5,-2) circle (2.5pt);
    
    \fill [black] (2,0) circle (2.5pt);
    \draw [thin, black] (-2,0) circle (2.5pt);
    \fill [white] (-2,0) circle (2.5pt);
    
\end{tikzpicture}
\caption{Realization of an exact hyperelliptic genus two-differential in $\mathcal{H}_2(4,4;-6)$. In this case the hyperelliptic involution has $6$ fixed points. Three out of four of these points are drawn in the picture with the symbol $\times$. Two symbols with the same color are identified on the final surface and hence they need to be counted as a single fix point. The remaining fix point is the puncture corresponding to the pole (in this case $2p=6$). }
\label{fig:hyptrigenusone}
\end{figure}
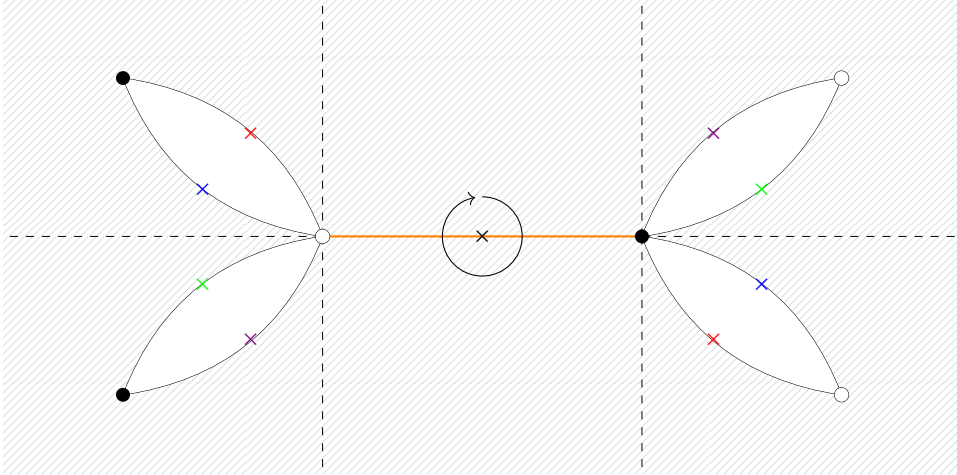

\subsubsection{Two higher order poles}\label{sssec:trihyptwopole}

This case is similar to the previous one and, in fact, the main difference is that we begin with two genus $0$-meromorphic differential $(X_i,\,\omega_i)\in\mathcal{H}_0(p-2;-p)$. These structures can be realized as follows: We begin with two copies of $(\C,\,dz)$ and we consider on each one two distinct points, say $P_i,\,Q_i\in\C$ such that
\begin{equation*}
    d\big(P_1,\,Q_1\big)=d\big(P_2,\,Q_2\big)
\end{equation*}

\noindent where $d(\cdot\,,\,\cdot)$ denotes the usual Euclidean metric. Let $l_i$ denote the unique straight line passing through $P_i$ and $Q_i$ and denote by $e_i$ the unique segment joining them. For $i=1,2$ we orient the line $l_i$ so that $e_i$ is oriented from $P_i$ to $Q_i$. On the first copy  of $(\C,\,dz)$, we define $r_1\subset l_1$ as the ray leaving from $P_1$ not passing through $Q_1$. Similarly, on the second copy of $(\C,\,dz)$ we define $r_2\subset l_2$ as the ray leaving from $Q_2$ and not passing through $P_2$. Next, we bubble $p-2$ copies of $(\C,\,dz)$ along $r_1$ and, similarly, we bubble $p-2$ copies of $(\C,\,dz)$ along $r_2$. Define the resulting structures as $(X_1,\,\omega_1)$ and $(X_2,\,\omega_2)$ respectively. Both structures have a pole of order $p$ and $\xi_1$ has a zero of order $p-2$ at $P_1$ whereas $\xi_2$ has a zero of order $p-2$ at $Q_2$. 

\begin{figure}[!ht]
    \centering
    \begin{tikzpicture}[scale=1.2, every node/.style={scale=0.8}]
    \definecolor{pallido}{RGB}{221,227,227}

    \foreach \x [evaluate=\x as \coord using  7*\x] in {0} 
    {
    \pattern [pattern=north west lines, pattern color=pallido]
    (\coord, -2)--(\coord+6, -2)--(\coord+6, 2)--(\coord, 2)--(\coord,-2);
    \draw[thin, orange] (\coord+3,0) to [out=20, in=160] (\coord+5,0);
    \draw[thin, orange] (\coord+3,0) to [out=-20, in=-160] (\coord+5,0);
    \draw[thin, black, dashed] (\coord+3,0)--(\coord+3, 2);
    \draw[thin, black, dashed] (\coord+3,0)--(\coord+3, -2);
    \draw[thin, black, dashed] (\coord+3,0)--(\coord, 0);
    
    \draw[ultra thin, black] (\coord+3,0) to [bend left] (\coord+1.5,1.5) to [bend left] (\coord+3,0);
    \draw[ultra thin, black] (\coord+1.5,1.5) to [bend left] (\coord+3,0);
    \draw[ultra thin, black] (\coord+3,0) to [bend left] (\coord+1.5,-1.5) to [bend left] (\coord+3,0);
    \draw[ultra thin, black] (\coord+1.5,-1.5) to [bend left] (\coord+3,0);
    
    \fill [white] (\coord+3,0) to [bend left] (\coord+1.5,1.5) to [bend left] (\coord+3,0);
    \fill [white] (\coord+3,0) to [bend left] (\coord+1.5,-1.5) to [bend left] (\coord+3,0);
    \fill [white] (\coord+5,0) to [out=160, in=20] (\coord+3,0) to [out=340, in=200] (\coord+5,0);

    \fill [black] (\coord+3,0) circle (2.25pt);
    \fill [black] (\coord+5,0) circle (2.5pt);
    \fill [white] (\coord+5,0) circle (2.25pt);
    \fill [black] (\coord+1.5,1.5) circle (2.5pt);
    \fill [white] (\coord+1.5,1.5) circle (2.25pt);
    \fill [black] (\coord+1.5,-1.5) circle (2.5pt);
    \fill [white] (\coord+1.5,-1.5) circle (2.25pt);
    \node[green] at (\coord+2.05, 0.5) {$\times$};
    \node[purple] at (\coord+2.5, 0.95) {$\times$};
    \node[red] at (\coord+2.05, -0.5) {$\times$};
    \node[blue] at (\coord+2.5, -0.95) {$\times$};
    \node[black] at (\coord+4, 0.21) {$\times$};
    \node[violet] at (\coord+4, -0.21) {$\times$};
    
    \node at (\coord+4.5, 0.35) {$e_1^-$};
    \node at (\coord+4.5, -0.35) {$e_1^+$};
    \node at (\coord+0.25, 1.75) {$\mathbb C$};
    \node at (\coord+0.25, -1.75) {$\mathbb C$};
    \node at (\coord+5, 1.5) {$(X_1,\,\omega_1)$};
    \node at (\coord+3.25, -1.5) {$S_1$};
    \node at (\coord+0.25, 0.5) {$S_2$};
    \node at (\coord+0.25, -0.5) {$S_3$};
    \node at (\coord+3.25, 0.375) {$P_1$};
    \node at (\coord+5.375, 0) {$Q_1$};
    }
    
    \foreach \x [evaluate=\x as \coord using  -0.625+ 7*\x] in {1} 
    {
    \pattern [pattern=north west lines, pattern color=pallido]
    (\coord, -2)--(\coord+6, -2)--(\coord+6, 2)--(\coord, 2)--(\coord,-2);
    \draw[thin, orange] (\coord+1,0) to [out=20, in=160] (\coord+3,0);
    \draw[thin, orange] (\coord+1,0) to [out=-20, in=-160] (\coord+3,0);
    \draw[thin, black, dashed] (\coord+3,0)--(\coord+3, 2);
    \draw[thin, black, dashed] (\coord+3,0)--(\coord+3, -2);
    \draw[thin, black, dashed] (\coord+3,0)--(\coord+6, 0);
    
    \draw[ultra thin, black] (\coord+3,0) to [bend left] (\coord+4.5,1.5) to [bend left] (\coord+3,0);
    \draw[ultra thin, black] (\coord+4.5,1.5) to [bend left] (\coord+3,0);
    \draw[ultra thin, black] (\coord+3,0) to [bend left] (\coord+4.5,-1.5) to [bend left] (\coord+3,0);
    \draw[ultra thin, black] (\coord+4.5,-1.5) to [bend left] (\coord+3,0);
    
    \fill [white] (\coord+3,0) to [bend left] (\coord+4.5,1.5) to [bend left] (\coord+3,0);
    \fill [white] (\coord+3,0) to [bend left] (\coord+4.5,-1.5) to [bend left] (\coord+3,0);
    \fill [white] (\coord+1,0) to [out=20, in=160] (\coord+3,0) to [out=200, in=340] (\coord+1,0);

    \fill [black] (\coord+3,0) circle (2.5pt);
    \fill [white] (\coord+3,0) circle (2.25pt);
    \fill [black] (\coord+1,0) circle (2.5pt);
    \fill [black] (\coord+4.5,1.5) circle (2.5pt);
    \fill [black] (\coord+4.5,-1.5) circle (2.5pt);
    \node[red] at (\coord+3.95, 0.5) {$\times$};
    \node[blue] at (\coord+3.5, 0.95) {$\times$};
    \node[green] at (\coord+3.95, -0.5) {$\times$};
    \node[purple] at (\coord+3.5, -0.95) {$\times$};
    \node[violet] at (\coord+2, 0.21) {$\times$};
    \node[black] at (\coord+2, -0.21) {$\times$};
    \node at (\coord+1.5, 0.35) {$e_2^-$};
    \node at (\coord+1.5, -0.35) {$e_2^+$};
    \node at (\coord+5.75, 1.75) {$\mathbb C$};
    \node at (\coord+5.75, -1.75) {$\mathbb C$};
    \node at (\coord+1, 1.5) {$(X_2,\,\omega_2)$};
    \node at (\coord+1, -1.5) {$S_4$};
    \node at (\coord+5.75, 0.5) {$S_6$};
    \node at (\coord+5.75, -0.5) {$S_5$};
    \node at (\coord+2.75, 0.375) {$Q_2$};
    \node at (\coord+0.625, 0) {$P_2$};
    }

    \end{tikzpicture}
    \caption{Realization of a hyperelliptic exact differential in the stratum $\mathcal{H}_2(5,5;-4,-4)$. The hyperelliptic involution has $6$ fixed points; all of them are drawn in the picture with the symbol $\times$. Two symbols with the same color are identified on the final surface and hence they need to be counted as a single fix point. }
    \label{fig:hyptrigenusonetwopoles}
\end{figure}
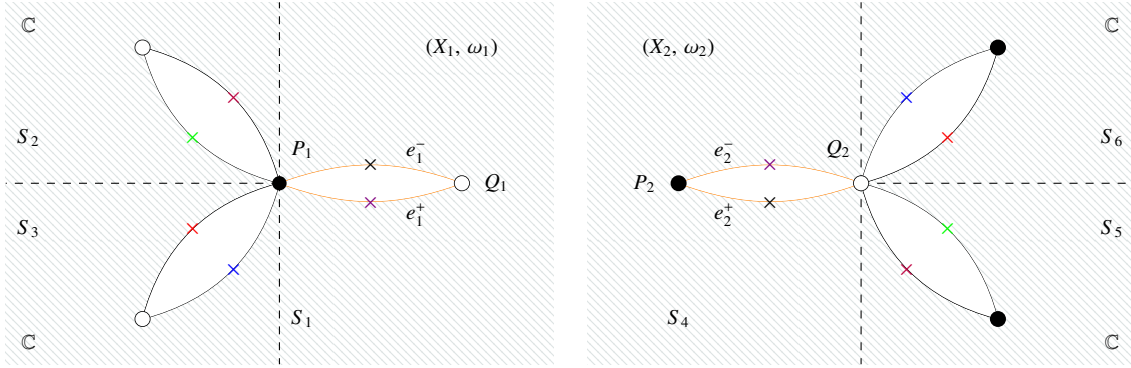

\noindent 

\smallskip

\noindent Let us focus on $(X_1,\,\omega_1)$. Around $P_1$ we can single out $p-1$ sectors, say $S_1,\dots,S_{p-1}$ in cyclic positive order, each of which is a wedge of angle $2\pi$. One of these sector, say $S_1$, is given by the first starting copy of $(\C,\,dz)$. All the other sectors $S_2,\dots,S_{p-1}$ come from the $p-2$ copies of $(\C,\,dz)$ bubbled along $r_1$. For every $k=2,\dots,p-1$, inside the sector $S_k$ there is a segment, say $s_k$, based at $P_1$ that forms a wedge of angle $2k\pi$ with $e_1$. Let ${Q_k}'$ be the extremal point of $s_k$ other than $P_1$.  Furthermore, the fixed orientation on the edge $e_1$ determines a preferred orientation on $s_k$ from $P_1$ to ${Q_k}'$. 

\smallskip

\noindent Similar considerations hold for $(X_2,\,\omega_2)$. Around $Q_2$ we can single out $p-1$ sectors, say $S_{p},\dots,S_{2p-2}$ in cyclic positive order, each of which is a wedge of angle $2\pi$. One of these sector, say $S_p$, is given by the second starting copy of $(\C,\,dz)$. The other sectors $S_{p+1},\dots,S_{2p-2}$ come from the $p-2$ copies of $(\C,\,dz)$ bubbled along $r_2$. For every $k=1,\dots,p-2$, each sector $S_{p+k}$ contains a segment, say $s_{p+k}$, based at $Q_2$ that forms a wedge of angle $2k\pi$ with $e_2$. Let ${P_k}'$ be the extremal point of $s_{p+k}$ other than $Q_2$. Even in this case, the already fixed orientation on $e_2$ determines a preferred orientation on $s_{p+k}$ from $Q_2$ to ${P_k}'$.

\smallskip

\noindent Slit $(X_1,\omega_1)$ and $(X_2,\omega_2)$ along $e_1$ and $e_2$ respectively, denote $e_1^{\pm}$ and $e_2^{\pm}$ the resulting sides and identify the edge $e_1^+$ with $e_2^-$ and the edge $e_1^-$ with $e_2^+$. The resulting space is a translation surface, say $(X,\omega)$, in the stratum $\mathcal{H}_0(p-1,p-1;-p,-p)$. In fact, recall that $e_i$ is oriented from $P_i$ to $Q_i$ for $i=1,2$; hence $P_1$ is identified with $P_2$ and, similarly, $Q_1$ is identified with $Q_2$ on $(X,\omega)$. If we denote $P,\,Q$ the resulting points, respectively, it is easy to check that $\omega$ has two zeros of order $p-1$ at $P$ and $Q$.

\smallskip

\noindent Let $1\le g\le p-2$. Slit $(X,\omega)$ along the segments $s_2,\dots,s_{g+1}$ (all leaving from $P$) and then slit along the segments $s_{p+1},\dots,s_{p+g}$ (leaving from $Q$). Denote the resulting edges as ${s_{k+1}}^{\pm}$ and ${s_{p+k}}^{\pm}$ as usual, according to our convention, for $k=1,\dots,g$. Next, identify $s_{k+1}^+$ with $s_{p+k}^-$ and identify $s_{k+1}^-$ with $s_{p+k}^+$. 

\smallskip 

\noindent The resulting space is a genus $g$ translation surface, say $(Z,\eta)\in\mathcal{H}_g(m,m;-p,-p)$, where $m=g+p-1$. By construction $(Z,\eta)$ has trivial periods and it admits a hyperelliptic involution with $2g+2$ fixed points given by: the mid-points of $e_1,e_2$ and the mid-points of $s_{k+1}^{\pm}$ and of $s_{p+k+1}^{\pm}$ for all $k=1,\dots,g$; see Figure \ref{fig:hyptrigenusonetwopoles}. This concludes the case with two higher order poles and indeed the proof of Proposition \ref{prop:hgdiffhyptriv}. 

\subsubsection{Realizing the trivial representation} As a consequence of subsections \S\ref{sssec:trihypsinglepole} and \S\ref{sssec:trihyptwopole} we have the following

\begin{cor}\label{cor:hgdiffhyptriv} Suppose the trivial representation can be realized in a stratum admitting two connected components, one of which is hyperelliptic and the other is not. Then it can be realized in each connected component of the same stratum. 
\end{cor}

\begin{proof}[Sketch of the proof] Suppose a stratum $\mathcal{H}_g(m,m;-\nu)$, where $\nu=\{2m\}$ or $\nu=\{m,m\}$ admits two connected components on of which is hyperelliptic; \textit{e.g.} $\mathcal H_2(5,5;-4,-4)$. We have seen above how to realize the trivial representation as the period of some hyperelliptic translation surface. We adopt the same notation above. In order to get an exact differential which is no longer hyperelliptic, it is sufficient to identify $s_{k+1}^+$ with $s_{2p-2-k}^-$ and $s_{k+1}^-$ with $s_{2p-2-k}^+$. The resulting space is also a genus $g$ translation surface in $\mathcal{H}_g(m,m;-p,-p)$ but it does not admit any hyperelliptic involution.
\end{proof}

\smallskip

\subsection{Meromorphic exact differentials with prescribed parity}\label{ssec:meroexdiffspin} We aim to realize the trivial representation as the period character of an exact meromorphic differential with prescribed spin parity. The gist of the idea is, once again, to use the induction on the genus. Similarly to the other cases, the process we shall develop here consists in adding one handle with trivial periods at each step, see subsection \S\ref{sssec:meroexdiffgencase}. Once again, in order for inductive process to start, we need to treat genus one surfaces in the first place and then explain how the induction works, see subsections \S\ref{ssec:meroexdiffgenone}. The main result of the present section is the following

\begin{prop}\label{prop:hgdiffspintriv}
Suppose the trivial representation can be realized in a stratum admitting two connected components distinguished by the spin parity. Then it can be realized as the period character of some translation surfaces with poles in both components of the same stratum with the only exceptions being the strata $\mathcal{H}_g(2^{g+2};\,-2^3)$ and $\mathcal{H}_g(2^{g+1};\,-4)$ for $g\ge2$. 
\end{prop}

\noindent The exceptional cases are handled by the following

\begin{prop}\label{prop:hgdiffspintriveset}
In the strata $\mathcal{H}_g(2^{g+2};\,-2^3)$ and $\mathcal{H}_g(2^{g+1};\,-4)$ for $g\ge2$, the trivial representation can only be realized in the connected component with parity equal to $g \,(\textnormal{mod }2)$.
\end{prop}

\noindent We begin with a few considerations.

\subsubsection{Reducing to lower genus surfaces}\label{sssec:reducing}

\noindent Suppose the trivial representation appear as the period character of some differential in a stratum $\mathcal{H}_g(2m_1,\dots,2m_k, -2p_1,\dots,-2p_n)$. For the moment we do not consider any parity.
According to \cite[Theorem B]{CFG}, this is possible if and only if 
\begin{equation}\label{eq:ordcond}
2m_j +1 \le \sum_{i=1}^n 2p_i-n
\end{equation}
and 
\begin{equation}\label{eq:degcond}
\sum_{j=1}^k 2m_j=\sum_{i=1}^n 2p_i+2g-2.
\end{equation}

\noindent Let us consider the string $\kappa=(2m_1,\dots,2m_k)$. We now introduce an algorithm to reduce $\kappa$ to a new, possibly shorter, string $\lambda=(2d_1,\dots,2d_l)$ such that
\begin{align}\label{eq:ordcondred}
    \sum_{j=1}^l 2d_j & \,=\,\sum_{i=1}^n 2p_i;
\end{align}
that is $\mathcal H_1(2d_1,\dots,2d_l;\,-2p_1,\dots,-2p_n)$ is a non-empty stratum of genus one-differentials.

\smallskip 

\noindent Such a reduction process is defined as follow. Let $\{2e_1,\dots,2e_t\}$ be any string of positive integers indexed in non-increasing order, we distinguish two cases according to the list below
\begin{itemize}
    \item we reduce as $\{2e_1,\dots,2e_{t-1}, 2e_t\}\longmapsto\{2e_1,\dots,2e_{t-1}, 2e_t-2\}$ in the case the last integer is than $2$, that is $e_t\geq 2$, otherwise
    \smallskip
    \item if the last integer is equal to $2$, then reduce as $\{2e_1,\dots, e_{t-1}, 2\}\longmapsto\{2e_1,\dots,2e_{t-1}\}$.
\end{itemize}

\smallskip

\noindent Starting with a string $\kappa$ as above, after $g-1$ reductions we end up with a new string $\lambda=(2d_1,\dots,2d_l)$ such that \eqref{eq:ordcondred} holds. Notice in the first place that, if $\kappa$ is indexed in non-increasing order then even $\lambda$ is indexed in the same way. We shall always assume that $2d_1\ge 2d_2\ge \cdots \ge 2d_l$ at each step, in particular in the constructions below the assumption $2m_i\ge 2m_j$ for $i\le j$ plays an important role. 

\smallskip

\noindent Observe that, if the trivial representation $\chi\colon\shomolzn\longrightarrow\C$ can be realized in a stratum $\mathcal{H}_g(\kappa;-\nu)$ then the trivial representation $\shomolzon\longrightarrow\C$ can be realized in the stratum $\mathcal{H}_1(\lambda;\,-\nu)$. Therefore, the gist of the proof is to undo the reduction by adding one handle with trivial periods time by time. The way this handle will be glued depends on how the reduction has been performed. We obtain in this way a finite sequence of translation surfaces with the last one being the desired structure.

\subsubsection{Inductive process: higher genus surfaces}\label{sssec:meroexdiffgencase} Let the induction start. The idea is to get the access to a genus $g+1$ stratum starting from a genus $g$ one, similarly to what was done in Section \S\ref{sec:hgcpar}, that is
\begin{equation}
    \mathcal{H}_{g}(2m_1,\dots,2m_k-2;\,-2p_1,\dots,-2p_n)\longmapsto \mathcal{H}_{g+1}(2m_1,\dots,2m_k;\,-2p_1,\dots,-2p_n)
\end{equation}
where $2m_k-2\ge0$. For genus one surfaces, we have already seen in subsection \S\ref{ssec:meroexdiffgenone} that as soon as the trivial representation $\chi\colon\shomolzon\longrightarrow \C$ can be realized in a given stratum then it can be realized as the period character of some translation surface with poles in each connected component of the same stratum with only three exceptions. By ignoring the exceptional cases for a moment, we can use these structures as the base case for an inductive foundation. In fact, every non-connected stratum $\mathcal H_1(2\kappa;\,-2\nu)$ admits at least one component with odd rotation number and at least one component with even rotation number. We shall use the former case for realizing genus $g$ meromorphic differentials with odd spin parity whereas we shall use the latter case for realizing genus $g$ meromorphic differentials with even spin parity. 

\smallskip

\noindent However, by applying the reduction just described above, it is possible to land to an exceptional stratum that we recall to be $\mathcal H_1(2,2;-4)$, $\mathcal H_1(2^3;\,-2^3)$ or $\mathcal H_1(3,3;\,-3,-3)$. Since we consider strata of even type, the polar part is necessarily a string of the form $(2p_1,\dots,2p_n)$. Since the reduction above preserves the polar part and only alters the string $\kappa=(2m_1,\dots,2m_k)$, it follows that we never land in the stratum $\mathcal H_1(3,3;\,-3,-3)$ of genus one differentials.

\smallskip

\noindent We next wonder: \textit{from which strata do we land on an exceptional stratum?} Once again, the key observation is that the reduction above preserves the polar part. In particular, we land in the stratum $\mathcal H_1(2,2;\,-4)$ if and only if the reduction starts from a stratum of the form $\mathcal H_g(2m_1,\dots,2m_k;-4)$ with $g\ge2$. The Gauss-Bonnet condition, see Remark \ref{gbcond}, implies that $m_1+\cdots+m_k=g+1$ and the Hurwitz type inequality \eqref{eq:triconstr} implies that $2m_i\leq 2$. Therefore, we land in the exceptional stratum $\mathcal H_1(2,2;\,-4)$ if and only if the reduction starts from a stratum of the form $\mathcal H_g(2^{g+1};\,-4)$ for $g\geq2$. Finally, we observe that the same argument shows that we land in the exceptional stratum $\mathcal H_1(2^3;\,-2^3)$ if and only if the reduction starts from strata of the form $\mathcal H_g(2^{g+2};\,-2^3)$ for $g\geq2$. For $g\geq2$, we define every stratum $\mathcal H_g(2^{g+1};\,-4)$ and $\mathcal H_g(2^{g+2};\,-2^3)$ as an \textit{exceptional stratum of even type} and we shall consider them in \S\ref{sssec:eset}. The remaining part of this subsection is devoted to prove Proposition \ref{prop:hgdiffspintriv} which handle non-exceptional strata of even type.

\begin{proof}[Proof of Proposition \ref{prop:hgdiffspintriv}] Let $\mathcal H_g(2m_1,\dots,2m_k;\,-2p_1,\dots,-2p_n)$ be a non-exceptional stratum of even type and suppose the Hurwitz type inequality \eqref{eq:triconstr} holds, that means the trivial representation can be realized in this stratum as the period character of some translation surface with trivial periods. Let $\kappa$ be the string encoding the order of zeros and reduce it as described in \S\ref{sssec:reducing}. After $g-1$ steps we get a new string, say $\lambda=(2d_1,\dots,2d_l)$ with $l\leq k$. In the reduced stratum $\mathcal H_1(\lambda;\,-\nu)$, the trivial representation can be realized as the period character of some genus one-differential with rotation number $1$ or $2$, say $(X_1,\,\omega_1)$. We then undo the reduction in order to realize an exact genus $g$ meromorphic differential in the initial stratum. Finally, it remains to compute the parity of the spin structure. We bubble a handle with trivial periods as described in subsection \S\ref{sssec:trialternhand} and, as we have already seen, such a surgery alters the spin parity. Therefore, for a given stratum $\mathcal H_g(2m_1,\dots,2m_k;\,-2p_1,\dots,-2p_n)$ we get the access to the even component from the odd component of the lower stratum of differentials $\mathcal H_{g-1}(2m_1,\dots,2m_k-2;\,-2p_1,\dots,-2p_n)$. Similarly, we get the access to the odd component of the stratum above from the even one of the lower stratum of differentials $\mathcal H_{g-1}(2m_1,\dots,2m_k-2;\,-2p_1,\dots,-2p_n)$. 

\smallskip

\noindent We now explain how to undo the reduction. For $1\le h\le g-1$, at each step we bubble a handle with trivial periods by using the alternative construction provided at subsection \S\ref{sssec:trialternhand} and denote by $(X_{h+1},\omega_{h+1})$ the resulting structure. By using the same notation, suppose first that at the $h^{th}$ step the reduction is of the form $\{2e_1,\dots,2e_{t-1},2\}\mapsto \{2e_1,\dots,2e_{t-1}\}$ -- this is the easiest case to handle. Let $P_{t-1}$ be the zero of order $2e_{t-1}$ and let $c_1,\,c_2,\,c_3$ be three paths all leaving from $P_{t-1}$ with length $\varepsilon$ and such that the angle between the paths $c_1$ and $c_2$ and the angle between the paths $c_2$ and $c_3$ are both $2\pi$. The angle between the paths $c_3$ and $c_1$ is $(4e_{t-1}-2)\pi$. Bubble a handle with trivial periods as described in \S\ref{sssec:trialternhand}. The resulting structure $(X_{h+1},\,\omega_{h+1})$ is a translation surface with poles and trivial periods. From now on, suppose that at the $h^{th}$ step the reduction is of the form $\{2e_1,\dots,2e_t\}\mapsto \{2e_1,\dots,2e_t-2\}$ with $2e_t-2\geq2$. We need the following 

\begin{lem}\label{lem:gemtechlem}
Let $(X,\omega)\in\mathcal H_g(m_1,\dots,m_k;\,-p_1,\dots,-p_n)$ be a translation surface with poles and trivial absolute periods. Let $s_1$ be a saddle connection joining two zeros, say $P$ and $Q$. Let $s_2$ be a geodesic segment leaving from $P$ such that $s_1$ and $s_2$ are twins. Then $s_2$ is an embedded geodesic segment such that $s_1\,\cap\,s_2\subseteq \{P,\,Q\}$.
\end{lem}

\noindent Suppose the Lemma holds. Let $(X_h,\omega_h) \in \mathcal H_g(2m_1,\dots,2m_k-2;\,-2p_1,\dots,-2p_n)$ and let $P_1,\dots, P_k$ be the zeros of $\omega$ of orders $2m_1,\dots,2m_{k-1},2m_k-2$ respectively. Let $P_i$ be a zero that realises
\begin{equation}
    d=\min_{\text{zeros of }\omega} d(\,P_i\,,\,P_k\,).
\end{equation}
Let $s_1$ be a saddle connection joining them and orient it from $P_i$ to $P_k$. Let $s_o$ and $s_2$ be two paths leaving from $P_i$ such that $s_o,\,s_1,\,s_2$ are pairwise twins; $s_o$ forms a wedge of angle $2\pi$ with $s_1$ on its left and $s_2$ forms a wedge of angle $2\pi$ with $s_1$ on its right. For simplicity, rename momentarily $P_k$ as $Q_1$ and then denote by $Q_o$ and $Q_2$ the extremal points of $s_o$ and $s_2$ other than $P_i$. There are three mutually disjoint possibilities that we now discuss
\begin{itemize}
    \item[1.] The extremal points $Q_o$, $Q_1$ and $Q_2$ are pairwise distinct,
    \smallskip
    \item[2.] two out of three extremal points coincides,
    \smallskip
    \item[3.] $Q_o$, $Q_1$ and $Q_2$ all coincide.
\end{itemize}

\begin{rmk}
Notice that, in principle, the extremal points $Q_o$ and $Q_2$ might be zeros of $\omega$. In the case where the extremal points $Q_o$ and $Q_2$ do not coincide with $Q_1$ but they are not regular points. For simplicity, suppose $Q_o=P_j$ for some $j\neq i,k$. We proceed exactly as above by ignoring the fact that $Q_o$ is not regular. The resulting structure lies in the stratum 
\begin{equation}
    \mathcal H_{h+1}(2m_1,\dots,\widehat{2m_j},\dots, 2m_j+2m_k;\,-2p_1,\dots,-2p_n).
\end{equation}
We eventually break the the zero of order $2m_j+2m_k$ into two zeros of orders $2m_j$ and $2m_k$ to get the desired result. A similar consideration holds if both $Q_o$ and $Q_2$ are not regular points. Since the constructions we are going to described do not depend on the nature of the extremal points, we suppose them to be regular for simplicity, unless otherwise specified, as the general case follows from the present remark. As an alternative, one may deform a little the given structure by "moving" the branch points so that the vertices $Q_o$ and $Q_2$ are both regular. Topologically, this kind of deformation has been described in \cite[Section \S6]{TA} and it is the geometric counterpart of the well-known Schiffer variations, see \cite{NS}; a classical tool in the study of Riemann surfaces and their moduli spaces. 
\end{rmk}

\smallskip

\noindent The case $1$ is the simplest one to deal with and it does not have exceptional sub-cases. We bubble a handle with trivial periods as described in \S\ref{sssec:trialternhand}. After slitting we get an hexagon with vertices $V_1,\dots,V_6$; see Figure \ref{fig:addtrihandleoddrot} above for a reference. We may define $V_1$ as the vertex arising from the branch point $P_i$ and then we label the remaining vertices in cyclic positive order. Two of these vertices, namely $V_1$ and $V_4$, have magnitude $(4m_{i}-2)\pi$ and $(4m_k-2)\pi$ respectively. All the other vertices, instead, have magnitude $2\pi$. Once we glue the "opposite" we get the desired handle with trivial periods and hence the resulting space is a translation surface $(X_{h+1},\,\omega_{h+1})\in \mathcal H_{h+1}(2m_1,\dots,2m_k;\,-2p_1,\dots,-2p_n)$ with trivial absolute periods. 

\smallskip

\begin{figure}[!ht]
    \centering
    \begin{tikzpicture}[scale=0.75, every node/.style={scale=1}]
    \definecolor{pallido}{RGB}{221,227,227}
    \pattern [pattern=north east lines, pattern color=pallido, rounded corners=5pt]  (30:7cm) -- (90:7cm) -- (150:7cm) -- (210:7cm) -- (270:7cm)--(330:7cm);
    
    \fill [white] (30:4cm) -- (90:4cm) -- (150:4cm) -- (210:4cm) -- (270:4cm)--(330:4cm);

    \newdimen\R
    \R=4cm
    \draw[thick] (330:\R) foreach \x in {30,90,...,330} { -- (\x:\R) };
    \foreach \x in {30,150,270} { 
        \draw[line width=0.5pt,black,fill=white] (\x:\R) circle (3pt);
    }
    \foreach \x in {90,210,330} { 
        \draw[line width=0.5pt,black,fill=black] (\x:\R) circle (3pt);
    }

    \node at (90:4.75cm) {$0$};
    \node at (330:4.75cm) {$0$};
    \node at (150:4.5cm) {$2$};
    \node at (30:5cm) {$2m_k-4$};
    \node at (270:4.5cm) {$0$};
    \node at (210:5cm) {$2m_{i}-2$};

    \end{tikzpicture}
    \caption{Magnitude of the vertices $V_1,\dots,V_6$ in the second case after breaking the zero $P_k$ of order $2m_k-2$. By drawing a blue and violet curve as in Figure \ref{fig:addtrihandleoddrot} it is possible to verify that both curves have even index and hence the spin structure changes after bubbling.}
    \label{fig:bubtrihandmagn}
\end{figure}
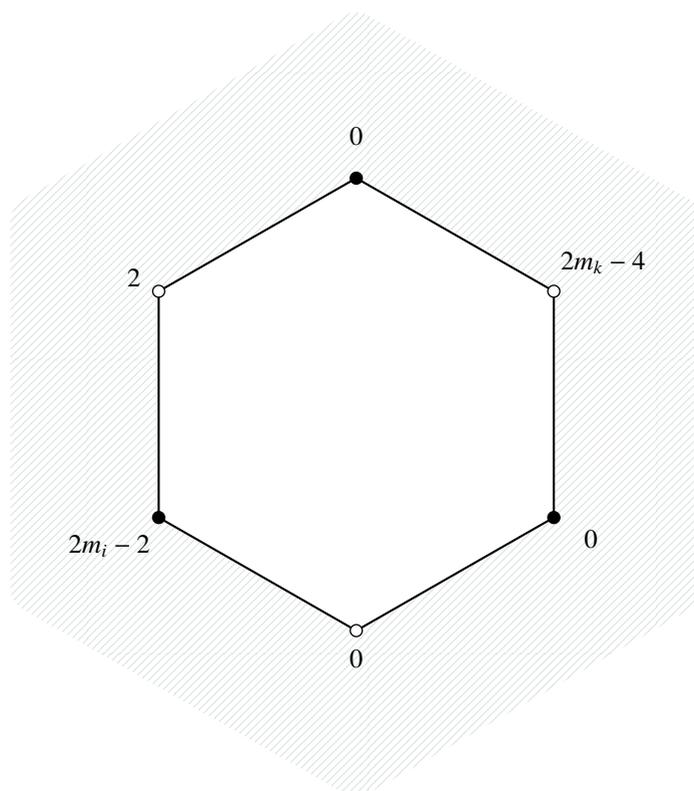

\noindent We now discuss the second case of the list above. Suppose in the first place that $Q_1$ coincides with $Q_o$ without loss of generality. Since two twins close up, we cannot immediately bubble a handle with trivial periods as in the previous case. We first need the following intermediate step. Suppose $2m_k-2\ge4$, that is $m_k\ge3$. The residual case $m_k=2$ needs a special treatment we shall consider later. Break $Q_1$ into two zeros $Q_{11}$ and $Q_{12}$ of orders $2$ and $2m_k-4\ge2$ respectively. Once the zero is broken, there are three twins with extremal points given by $Q_o$, $Q_{11}$ and $Q_{12}$ which are now all distinct. In this way we have reduced the construction to the case $1$ above with the only difference that, after slitting, the vertex $V_4$ of the resulting hexagon has magnitude $(4m_k-6)\pi$ and there is now another vertex, say $V_2$, with magnitude $6\pi$, see Figure \ref{fig:bubtrihandmagn}. Once we glue the "opposite" we get the desired handle with trivial periods and hence a translation surface with poles $(X_{h+1},\,\omega_{h+1})\in \mathcal H_{h+1}(2m_1,\dots,2m_k;\,-2p_1,\dots,-2p_n)$ with trivial absolute periods. As a second possibility, we next suppose the edges $s_o$ and $s_2$ close up, that means $Q_o$ and $Q_2$ now coincide. In this case, $Q_o$ turns out to be a branch point of order $2m_j$, with $j<k$. Since $2m_j\ge 2m_k> 2m_k-2\ge 2$, it follows that $2m_j\ge4$ and hence we can break it into two zeros of even order each; \textit{e.g.} $2$ and $2m_j-2$. We now proceed as above in order to get the desired structure with trivial absolute periods.

\smallskip

\noindent The third case is similar to the second one. Once again we cannot immediately bubble a handle with trivial periods. Suppose $2m_k-2\ge6$, that is $m_k\ge4$. The remaining cases $m_k=2,3$ need a special treatment and we shall consider them below. Break $Q_1$ into three zeros of orders $2$, $2m_k-6$ and $2$ respectively. Notice that $2m_k-6\ge2$. After slitting, the vertices $V_1$ and $V_4$ have magnitude $(4m_i-2)\pi$ and $(4m_k-10)\pi$ and there are now two vertices, say $V_2$ and $V_5$ with magnitude $6\pi$. Once again, we glue the opposite sides and the resulting structure is a translation surface $(X_{h+1},\,\omega_{h+1})$ with trivial periods in $\mathcal H_{h+1}(2m_1,\dots,2m_k;\,-2p_1,\dots,-2p_n)$.

\smallskip

\noindent In order to complete the proof we treat the residual cases mentioned above. Let us summarise them for the reader's convenience

\begin{itemize}
    \item[1.] $Q_1$ coincides with $Q_o$ or $Q_2$ and $m_k=2$,
    \smallskip
    \item[2.] $Q_o$, $Q_1$ and $Q_2$ all coincides and $m_k=2$,
    \smallskip
    \item[3.] $Q_o$, $Q_1$ and $Q_2$ all coincides and $m_k=3$.
\end{itemize}

\noindent All constructions above were performed by defining $s_1$ as the saddle connection that joins $P_k$ with the closest zero and then defining $s_o$ and $s_2$ as the twin $2\pi$ far away from $s_1$ on the left and right respectively. In order to deal with these residual cases we just need to find suitable twins on order to perform the desired bubbling.

\smallskip

\noindent Let $s_1$ be a saddle connection that joins $P_k$ with the nearest $P_i$ and let $s_2,\dots,s_{2m_i+1}$ be paths leaving from $P_i$ such that $s_i,\,s_j$ are twins for all $1\le i,\,j\le 2m_i+1$. In the case $m_k=2$, then there are at most three edges having $P_i$ and $Q_1$ as extremal points. If $s_1$ is the sole saddle connection joining $P_i$ and $Q_1$ then there is nothing to prove. Moreover, since $P_i$ is a branch point of even order $2m_i$, and $2m_i\ge 2m_k=4$, there two adjacent edges, say $s_l$ and $s_{l+1}$ such that
\begin{itemize}
    \item[\textit{i.}] their extremal points other than $P_i$ are both different from $Q_1$, and
    \smallskip
    \item[\textit{ii.}] they respectively form an angle $2\pi$ and $4\pi$ with respect to a saddle connection joining $P_i$ and $Q_1$. In fact, this is always possible unless $2m_i=4$. However, this case has already been covered above because we can find three edges with distinct extremal points.
\end{itemize}

\noindent If $s_l$ and $s_{l+1}$ close up, then the common extremal point other than $P_i$ is a branch point, say $P_j$, of order $2m_j\ge 2m_k>2$. Split $P_j$ into two branch points of orders $2$ and $2m_j-2$. Then we can bubble a handle with trivial periods as already described above. The resulting space is a genus $h+1$ translation surface with trivial absolute periods and prescribed parity as desired. Finally, the case $m_k=3$ is similar. The only difference here is that, in this case, there are at most five edges having $P_i$ and $Q_1$ as the extremal points. Once again, since $2m_i\ge 2m_k\ge6$, there are two adjacent edges $s_l$ and $s_{l+1}$ such conditions \textit{i.} and \textit{ii.} above holds. Then we may proceed as above to bubble a handle with trivial periods and then obtain a translation surface with poles and trivial absolute periods with the desire parity. This last case completes the proof of Proposition \ref{prop:hgdiffspintriv}.
\end{proof}

\noindent We are left with the following

\begin{proof}[Proof of Lemma \ref{lem:gemtechlem}]
In the first place we recall that any translation surface with poles and trivial absolute periods has at least two zeros otherwise the Hurwitz type inequality \eqref{eq:triconstr} does not hold. Let $P$ and $Q$ be a pair of zeros that realize 
\begin{equation}
    d=\min_{\text{zeros of }\omega} d(\,\cdot\,,\,\cdot\,)
\end{equation}
where $d(\,\cdot\,,\,\cdot\,)$ denotes the usual Euclidean distance. Let $s_1$ be a saddle connection joining them. The segment $s_1$ is embedded as it cannot have self-intersections. In fact, if there was a self-intersection at $R\in s_1$ then $s_1$ would contain a closed geodesic loop with trivial period. Denote the segments $P \rightarrow R\rightarrow R\rightarrow Q$ by $\gamma_1,\gamma_2,\gamma_3$ respectively. The intersection at $R$ must be transverse: if not, the geodesic segments would coincide. Moreover, the closed loop $\gamma_2$ has trivial period by assumption and it necessarily passes through at least two branch points. Then $R$ is a zero for $\omega$ but this contradicts the minimality of $d$ unless $R\in\{P,\,Q\}$. On the other hand, $R$ cannot coincide with $P$ or $Q$ otherwise $s_1$ can be shortened. Therefore $s_1$ is embedded. Let $s_2$ be another segment leaving from $P$ such that $s_1$ and $s_2$ are twins, see Definition \ref{def:twins}. For the same reason as above, $s_2$ cannot intersects itself and hence it is embedded. It remains to show that $s_1\,\cap\,s_2\subseteq \{P,\,Q\}$. In the case that $s_1\,\cap\,s_2= \{P,\,Q\}$ we are done and $s_2$ is another saddle connection joining $P$ and $Q$. Otherwise, in the case $s_1\,\cap\,s_2\subseteq \{P,\,R\}$ with $R\neq Q$, then $P\rightarrow R\rightarrow P$ is a simple closed geodesic loop with trivial period. Once again, $R$ must be a branch point but this contradicts the minimality of $d$.
\end{proof}

\smallskip

\subsubsection{Exceptional strata of even type}\label{sssec:eset} It remains to consider the exceptional cases of Proposition \ref{prop:hgdiffspintriv}, that is to prove Proposition \ref{prop:hgdiffspintriveset}. We have the following

\begin{lem}
For every $g\ge2$, the trivial representation can only be realized in the connected component of $\mathcal{H}_g(2^{g+2};\,-2^3)$ with parity equal to $g \,(\textnormal{mod }2)$.
\end{lem}

\begin{proof}
By Lemma~\ref{lem:222-2-2-2} the claim holds for the case of $g = 1$. We will prove the general case by induction on $g$. Denote by $\mathcal{TR}_g$ the locus of exact differentials in $\mathcal{H}_g(2^{g+2};\,-2^3)$. Note that $\mathcal{TR}_g$ equivalently parameterises triple covers of $\mathbb{C} \mathbf{P}^1$ totally ramified at $Z_1,\ldots, Z_{g+2}$ and having $P_1+P_2+P_3$ as the fiber over infinity.  

\smallskip

\noindent First we show that connected components of $\mathcal{TR}_g$ correspond to integers $0\leq k \leq (g+2)/2$ such that $2k \equiv g+2 \pmod{3}$. To see this, we use the monodromy description of triple covers parameterized in a connected component of $\mathcal{TR}_g$. Up to permuting the branch points and relabeling the three sheets, we can assume that the monodromy cycles at the first $k$ branch points are given by the permutation $c = (1,2,3)$ and the monodromy cycles at the last $g+2-k$ branch points are given by the permutation $c^{-1} = (1,3,2)$, where $c^k (c^{-1})^{g+2-k} = {\rm id}$. This is equivalent to requiring that $2k - g - 2$ is divisible by $3$. 

\smallskip

\noindent Denote by $\mathcal{TR}_g(k)$ the connected component by using $k$ monodromy cycles of $c$ and $g+2-k$ monodromy cycles of $c^{-1}$. Without loss of generality we can assume that $k \geq 2$. Next we will exhibit a degenerate cover in the boundary of each $\mathcal{TR}_g(k)$ (in the sense of admissible covers \cite[\S3G]{HarrisMorrison}) by gluing an element in $\mathcal{TR}_{g-1}(k-2)$ and an element in $\mathcal {TR}_1(3)$ at a separating node. To see this, let two branch points both with monodromy cycle $c$ approach each other in the target $\mathbb{C} \mathbf{P}^1$.  The resulting (admissible) cover is given by gluing the following two sub-covers. One is over $\mathbb{C} \mathbf{P}^1$ with the two chosen branch points of monodromy cycle $c$, and with an extra branch point at the node out of the vanishing cycle enclosing the two merged points whose monodromy cycle is $(c \cdot c)^{-1} = c$.  The other is over $\mathbb{C} \mathbf{P}^1$ with the remaining branch points of $k-2$ monodromy cycles $c$ and $g+2-k$ monodromy cycles of $c^{-1}$, with an extra branch point at the node whose monodromy cycle is $c^{-1}$. 

\smallskip

\noindent Finally we show that the spin parity of differentials in $\mathcal{TR}_g(k)$ is $g \pmod{2}$. By the induction hypothesis, the spin parities of $\mathcal{TR}_{g-1}(k-2)$ and $\mathcal{TR}_1(3)$ are $g-1$ and $1 \pmod{2}$, respectively.  Moreover, the degenerate differential in the boundary of $\mathcal{TR}_g(k)$ described above has the same parity as those in the interior since it is of compact type, i.e., the gluing node separates the two sub-surfaces. Therefore, by the fact that the spin parity on a curve of compact type is given by the sum of parities of its components \cite[(3.2)]{Cornalba}, the parity of the interior of $\mathcal{TR}_g(k)$ is equal to $(g-1) + 1 = g \pmod{2}$. 
\end{proof}

\smallskip

\noindent We finally show the following

\begin{lem}
For every $g\ge2$, the trivial representation can only be  realized in the connected component of $\mathcal{H}_g(2^{g+1};\,-4)$ with parity equal to $g \,(\textnormal{mod }2)$.
\end{lem}

\begin{proof}
This case essentially follows from the one in the preceding lemma. In both cases, the related triple covers arise from the same description, that is, totally ramified at $g+2$ points.  Let $Z_1, \ldots, Z_{g+1}$ and $Q$ be the $g+2$ ramification points. Let $P_1 + P_2 + P_3$ be an unramified fiber divisor. Then in the former case we consider the spin line bundle associated to the divisor $Z_1 + \cdots + Z_{g+1} + Q - P_1 - P_2 - P_3$, while in the latter case it is $Z_1 + \cdots + Z_{g+1} - 2Q$.  Since $3Q \sim P_1 + P_2 + P_3$, the two spin line bundle classes are linearly equivalent, hence they have the same parity. Alternatively, one can just apply the same monodromy and degeneration argument as in the preceding lemma.  
\end{proof}

\noindent In $g=4$, for instance, there are two distinct monodromy patterns corresponding to $k=0$ and $k=3$. That means $\mathcal{TR}_4$ has two connected components in $\mathcal H_4(2^6;\,-2^3)$ as well as in $\mathcal H_4(2^5;\,-4)$. The pictures below depict a realization of translation surfaces with trivial periods in both components of $\mathcal{TR}_4\subset \mathcal H_4(2^6;\,-2^3)$. A direct computation of the Arf-invariant shows that in both cases the parity, see Section \S\ref{sssec:spinpar}, is even as expected. See Figures \ref{fig:genus4evenspin1} and \ref{fig:genus4evenspin2} for these constructions.

\begin{figure}[!ht]
    \centering
    \begin{tikzpicture}[scale=1, every node/.style={scale=0.8}]
    
    \definecolor{pallido}{RGB}{221,227,227}
   
    \foreach \x [evaluate=\x as \coord using  \x] in {0,8} 
    {
    \pattern [pattern=north west lines, pattern color=pallido]
    (\coord,7)--(\coord+7,7)--(\coord+7,0)--(\coord,0)--(\coord,7);
    
    \draw[thin, orange] (\coord+2.25,5.5)--(\coord+4.75,5.5);
    \draw[thin, orange] (\coord+1.5, 4)--(\coord+1.5,2);
    \draw[thin, orange] (\coord+5.5, 4)--(\coord+5.5,2);
    
    \draw[thin, violet,->] (\coord+2.25,6.25)--(\coord+4.75,6.25);
    \draw[thin, violet] (\coord+2.25,4.75)--(\coord+4.75,4.75);
    \draw[thin, violet] (\coord+2.25,6.25) arc [start angle=90, end angle=270, radius = 0.75];
    \draw[thin, violet] (\coord+4.75,6.25) arc [start angle=90, end angle=-90, radius = 0.75];
    
    \draw[thin, violet, <-] (4.75,4)--(4.75,2);
    \draw[thin, violet] (6.25,4)--(6.25,2);
    \draw[thin, violet] (4.75, 4) arc [start angle=180, end angle=0, radius = 0.75];
    \draw[thin, violet] (4.75, 2) arc [start angle=180, end angle=360, radius = 0.75];
    
    \draw[thin, blue, <-] (1.55,3)--(5.45,3);
    
    \fill [black] (\coord+2.25,5.5) circle (1.5pt);
    \fill [black] (\coord+4.75,5.5) circle (1.5pt);
    
    \fill [black] (\coord+1.5,2) circle (1.5pt);
    \fill [black] (\coord+1.5,4) circle (1.5pt);
    \fill [black] (\coord+5.5,2) circle (1.5pt);
    \fill [black] (\coord+5.5,4) circle (1.5pt);
    }

    \draw[thin, blue, <-] (9.6,3)--(13.4,3);
    
    \draw[thin, blue] (1.45, 3.5)--(1.25, 3.5);
    \draw[thin, blue] (1.25, 3.5) arc [start angle=270, end angle=180, radius = 0.25];
    \draw[thin, blue, ->] (1,3.75)--(1, 6.25);
    \draw[thin, blue] (1, 6.25) arc [start angle=180, end angle=90, radius = 0.25];
    \draw[thin, blue] (1.25, 6.5)--(2.75,6.5);
    \draw[thin, blue] (2.75, 6.5) arc [start angle=90, end angle=0, radius = 0.25];
    \draw[thin, blue] (3, 6.25)--(3,5.5);
    
    \draw[thin, blue,->] (11, 5.5)--(11,4);
    \draw[thin, blue] (11,4) arc [start angle=0, end angle=-90, radius = 0.5];
    \draw[thin, blue] (10.5, 3.5)--(9.5,3.5);
    
    \draw[thin, blue] (9.5, 3)--(9,3);
    \draw[thin, blue] (9,3) arc [start angle=270, end angle=180, radius = 0.25];
    \draw[thin, blue] (8.75, 3.25)--(8.75, 6.25);
    \draw[thin, blue] (8.75,6.25) arc [start angle=180, end angle=90, radius = 0.25];
    \draw[thin, blue, ->] (9, 6.5)--(13.75,6.5);
    \draw[thin, blue] (13.75,6.5) arc [start angle=90, end angle=0, radius = 0.25];
    \draw[thin, blue] (14, 3.25)--(14, 6.25);
    \draw[thin, blue] (14, 3.25) arc [start angle=0, end angle=-90, radius = 0.25];
    \draw[thin, blue] (13.55, 3)--(13.75, 3);
    
    \foreach \x [evaluate=\x as \coord using  \x] in {4} 
    {
    \pattern [pattern=north west lines, pattern color=pallido]
    (\coord,-8)--(\coord+7,-8)--(\coord+7,-1)--(\coord,-1)--(\coord,-8);
    
    \draw[thin, orange] (\coord+2.25,-2.5)--(\coord+4.75,-2.5);
    \draw[thin, orange] (\coord+1.5, -4)--(\coord+1.5,-6);
    \draw[thin, orange] (\coord+5.5, -4)--(\coord+5.5,-6);
    
    \fill [black] (\coord+2.25, -2.5) circle (1.5pt);
    \fill [black] (\coord+4.75, -2.5) circle (1.5pt);
    \fill [black] (\coord+1.5,-6) circle (1.5pt);
    \fill [black] (\coord+1.5,-4) circle (1.5pt);
    \fill [black] (\coord+5.5,-6) circle (1.5pt);
    \fill [black] (\coord+5.5,-4) circle (1.5pt);
    
    }
    
    \draw[thin, violet, <-] (8.75,-4)--(8.75,-6);
    \draw[thin, violet] (10.25,-4)--(10.25,-6);
    \draw[thin, violet] (8.75, -4) arc [start angle=180, end angle=0, radius = 0.75];
    \draw[thin, violet] (8.75, -6) arc [start angle=180, end angle=360, radius = 0.75];
    
    \draw[thin, blue] (5.45, -5.5)--(5.25, -5.5);
    \draw[thin, blue] (5.25, -5.5) arc [start angle=90, end angle=180, radius = 0.25];
    \draw[thin, blue] (5, -5.75)--(5,-6.75);
    \draw[thin, blue] (5, -6.75) arc [start angle=180, end angle=270, radius = 0.25];
    \draw[thin, blue, ->] (5.25, -7)--(10.5,-7);
    \draw[thin, blue] (10.5, -7)arc [start angle=270, end angle=360, radius = 0.25];
    \draw[thin, blue] (10.75, -6.75)--(10.75, -5.75);
    \draw[thin, blue] (10.75, -5.75) arc [start angle=0, end angle=90, radius = 0.25];
    \draw[thin, blue] (10.5, -5.5)--(9.55, -5.5);
    
    \draw[thin, blue, ->] (7, -2.55)--(7,-4.25);
    \draw[thin, blue] (7,-4.25) arc [start angle=0, end angle=-90, radius = 0.25];
    \draw[thin, blue] (6.75, -4.5) -- (5.75, -4.5);
    \draw[thin, blue] (5.25, -4.5) arc [start angle=270, end angle=180, radius = 0.25];
    \draw[thin, blue, ->] (5, -4.25) -- (5, -2.25);
    \draw[thin, blue] (5, -2.25) arc [start angle=180, end angle=90, radius = 0.25];
    \draw[thin, blue] (5.25, -2)--(6.75,-2);
    \draw[thin, blue] (6.75, -2) arc [start angle=90, end angle=0, radius = 0.25];
    
    \node at (1.25, 2.5) {$a_1^+$};
    \node at (1.75, 2.5) {$a_2^-$};
    \node at (9.25, 2.5) {$a_2^+$};
    \node at (9.75, 2.5) {$a_3^-$};
    \node at (5.25,  -5) {$a_3^+$};
    \node at (5.75,  -5) {$a_1^-$};
    
    \node at ( 5.25, 2.5) {$b_1^+$};
    \node at ( 5.75, 2.5) {$b_2^-$};
    \node at (13.25, 2.5) {$b_2^+$};
    \node at (13.75, 2.5) {$b_3^-$};
    \node at ( 9.25,  -5) {$b_3^+$};
    \node at ( 9.75,  -5) {$b_1^-$};
    
    \node at ( 3.5, 5.75) {$c_1^+$};
    \node at ( 3.5, 5.25) {$c_2^-$};
    \node at (11.5, 5.75) {$c_2^+$};
    \node at (11.5, 5.25) {$c_3^-$};
    \node at ( 7.5,-2.25) {$c_3^+$};
    \node at ( 7.5,-2.75) {$c_1^-$};
    
    \node at (0.75, 5.5) {$\alpha_1$};
    \node at (6.75, -3.75) {$\alpha_1$};
    \node at (8.5, 5.5) {$\alpha_2$};
    \node at (3.5, 3.25) {$\alpha_2$};
    \node at (10.75, 4.25) {$\alpha_3$};
    \node at (4.75, -3) {$\alpha_3$};
    \node at (11.5, 3.25) {$\alpha_4$};
    \node at ( 7.5,-6.75) {$\alpha_4$};
    
    \node at (3.5, 6.5) {$\beta_1$};
    \node at (6.5,3) {$\beta_2$};
    \node at (9.25, 5.5) {$\beta_3$};
    \node at (10.5,-5) {$\beta_4$};
    
    \end{tikzpicture}
    \caption{Realization of a translation surface with poles and even parity in $\mathcal H_4(2^6;\,-2^3)$. Here the monodromy pattern is equal to $k=0$. Unlike the notation used so far, here the edges $a_i^+$, $b_i^+$ and $c_i^+$ are identified with $a_i^-$, $b_i^-$ and $c_i^-$ respectively. Moreover, the indices of the curves $\alpha_i$ and $\beta_i$ are according to the Table \ref{tab:indeces1} below.
    }
    \label{fig:genus4evenspin1}
\end{figure}
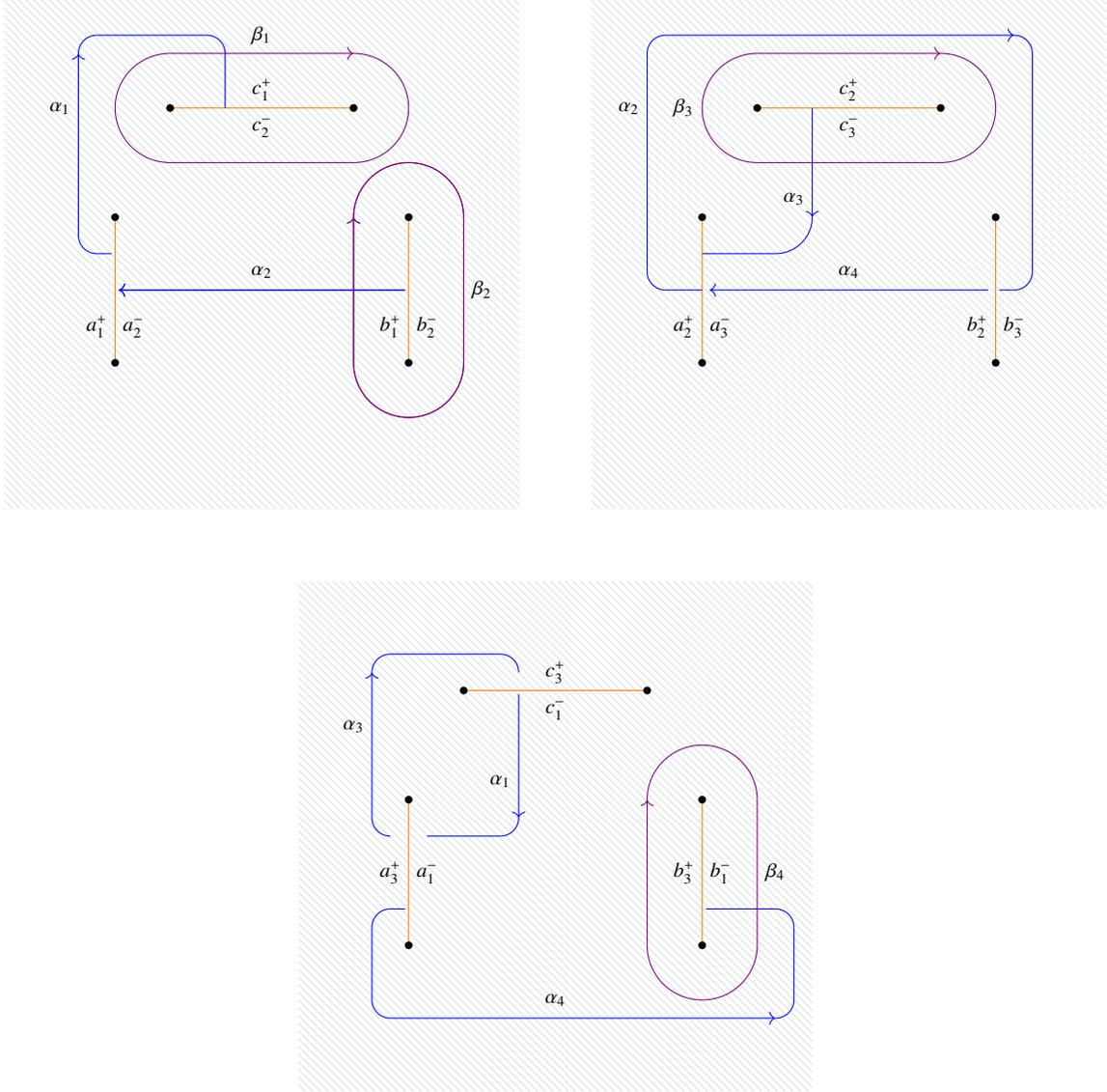
\begin{table}[!ht]
    \centering
    \begin{tabular}{lcc}
    \midrule
      $i$ & $\textnormal{Ind}(\alpha_i)$ & $\textnormal{Ind}(\beta_i)$\\ \midrule
      1 & 1 & 1\\ \midrule
      2 & 1 & 1\\ \midrule
      3 & 1 & 1\\ \midrule
      4 & 1 & 1\\ \midrule
    \end{tabular}
    \caption{Indices of curves in Figure \ref{fig:genus4evenspin1}.}
    \label{tab:indeces1}
\end{table}

\begin{figure}[!ht]
    \centering
    \begin{tikzpicture}[scale=0.8, every node/.style={scale=0.7}]
    
    \definecolor{pallido}{RGB}{221,227,227}
   
    \foreach \x [evaluate=\x as \coord using  \x] in {0,10} 
    {
    \pattern [pattern=north west lines, pattern color=pallido]
    (\coord,9)--(\coord+9,9)--(\coord+9,0)--(\coord,0)--(\coord,9);
    
    \draw[thin, orange] (3.5, 2.5)--(3.5,6.5);
    \draw[thin, orange] (\coord+5.5, 2.5)--(\coord+5.5,6.5);
    
    \draw[thin, orange] (\coord+3.5, 6.5)--(\coord+1.5, 4.5);
    \draw[thin, orange] (11.5, 4.5)--(13.5, 2.5);
    
    \draw[thin, orange] (15.5, 2.5)--(17.5, 4.5);
    \draw[thin, orange] (5.5, 6.5)--(7.5, 4.5);
    
    \fill [black] (\coord+5.5, 6.5) circle (1.5pt);
    \fill [black] (\coord+5.5, 2.5) circle (1.5pt);
    \fill [black] (\coord+3.5, 6.5) circle (1.5pt);
    \fill [black] (\coord+3.5, 2.5) circle (1.5pt);
    \fill [black] (\coord+1.5, 4.5) circle (1.5pt);
    \fill [black] (\coord+7.5, 4.5) circle (1.5pt);
    
    }
    
    \foreach \x [evaluate=\x as \coord using  \x] in {5} 
    {
    \pattern [pattern=north west lines, pattern color=pallido]
    (\coord-1,-1)--(\coord+9,-1)--(\coord+9,-10)--(\coord-1,-10)--(\coord-1,-1);
    
    \draw[thin, orange] (8.5, -7.5)--(8.5, -3.5);
    \draw[thin, orange] (8.5, -7.5)--(6.5, -5.5);
    
    \draw[thin,orange] (10.5, -7.5)--(12.5, -5.5);
    \draw[thin,orange] (12.5, -5.5)--(10.5, -3.5);
    
    \fill [black] (\coord+5.5, -3.5) circle (1.5pt);
    \fill [black] (\coord+5.5, -7.5) circle (1.5pt);
    \fill [black] (\coord+3.5, -3.5) circle (1.5pt);
    \fill [black] (\coord+3.5, -7.5) circle (1.5pt);
    \fill [black] (\coord+1.5, -5.5) circle (1.5pt);
    \fill [black] (\coord+7.5, -5.5) circle (1.5pt);
    
    }
    
    
    \draw[thin, blue] (2.5, 5.5)--(2.5,8);
    \draw[thin, blue] (2.5,8) arc [start angle=180, end angle=90, radius = 0.5];
    \draw[thin, blue, <-] (3,8.5)--(8,8.5);
    \draw[thin, blue] (8,8.5) arc [start angle=90, end angle=0, radius = 0.5];
    \draw[thin, blue] (8.5,8)--(8.5,4);
    \draw[thin, blue] (8.5,4) arc [start angle=0, end angle=-90, radius = 0.5];
    \draw[thin, blue] (8,3.5)--(6.5,3.5);
    \draw[thin, blue] (6.5,3.5) arc [start angle=270, end angle=180, radius = 0.5];
    \draw[thin, blue] (6,4)--(6,5.5);
    \draw[thin, blue] (5.5,6) arc [start angle=90, end angle=0, radius = 0.5];
    
    \draw[thin, blue] (15.5, 6)--(15,6);
    \draw[thin, blue, <-] (15,6) arc [start angle=90, end angle=180, radius = 0.5];
    \draw[thin, blue] (14.5,5.5) arc [start angle=360, end angle=270, radius = 0.5];
    \draw[thin, blue] (13,5)--(14,5);
    \draw[thin, blue] (12.5,5.5) arc [start angle=180, end angle=270, radius = 0.5];
    
    \node at (5.5, 8.25) {$\alpha_1$};
    \node at (15, 6.25) {$\alpha_1$};
    
    \draw[thin, blue] (15, 2.5) arc [start angle=180, end angle=405, radius = 0.5];
    \draw[thin, blue] (15, 2.5)--(15, 4);
    \draw[thin, blue] (15, 4) arc [start angle=0, end angle=90, radius = 0.5];
    \draw[thin, blue, ->] (13,4.5)--(14.5,4.5);
    \draw[thin, blue] (13, 4.5) arc [start angle=90, end angle=180, radius = 0.5];
    \draw[thin, blue] (12.5,4)--(12.5,3.5);
    
    \draw[thin, blue] (7.5, -8)--(7.5,-6.5);
    \draw[thin, blue] (8,-8.5) arc [start angle=270, end angle=180, radius = 0.5];
    \draw[thin, blue, <-] (8,-8.5)--(9.5,-8.5);
    \draw[thin, blue] (9.5,-8.5) arc [start angle=270, end angle=360, radius = 0.5];
    \draw[thin, blue] (10, -7.5)--(10,-8);
    \draw[thin, blue] (10,-7.5) arc [start angle=180, end angle=45, radius = 0.5];
    
    \node at (13.5, 4.25) {$\alpha_2$};
    \node at (8.75, -8.25) {$\alpha_2$};
    
    \draw[thin, blue, ->] (3.5, 3) arc [start angle=90, end angle=360, radius = 0.5];
    \draw[thin, blue] (4,2.5)--(4,7);
    \draw[thin, blue] (4, 7) arc [start angle=180, end angle=90, radius = 0.5];
    \draw[thin, blue] (4.5,7.5)--(6,7.5);
    \draw[thin, blue] (6,7.5) arc [start angle=90, end angle=0, radius = 0.5];
    \draw[thin, blue] (6.5,7)--(6.5,5.5);
    \draw[thin, blue] (11.5,-5)--(11.5, -4.5);
    \draw[thin, blue] (11.5,-5) arc [start angle=0, end angle=-90, radius = 0.25];
    \draw[thin, blue, ->] (11.25,-5.25)--(9.25, -5.25);
    \draw[thin, blue] (9.25,-5.25) arc [start angle=90, end angle=180, radius = 0.25];
    \draw[thin, blue] (9, -6.75)--(9, -5.5);
    \draw[thin, blue] (8.75,-7) arc [start angle=270, end angle=360, radius = 0.25];
    \draw[thin, blue] (8.5, -7)--(8.75,-7);
    
    \node at (4.25, 4.5) {$\alpha_3$};
    \node at (10.25, -5) {$\alpha_3$};
    
    \draw[thin, blue] (12, 4.5) arc [start angle=0, end angle=45, radius = 0.5];
    \draw[thin, blue, <-] (12, 4.5) arc [start angle=0, end angle=-45, radius = 0.5];
    \draw[thin, blue] (6, -5.5) arc [start angle=180, end angle=315, radius = 0.5];
    \draw[thin, blue, <-] (6, -5.5)--(6, -3.5);
    \draw[thin, blue] (6, -3.5) arc [start angle=180, end angle=90, radius = 0.5];
    \draw[thin, blue] (6.5,-3)--(8.5,-3);
    \draw[thin, blue] (8.5,-3) arc [start angle=90, end angle=-90, radius = 0.5];
    \draw[thin, blue] (3.5, 6) arc [start angle=90, end angle=180, radius = 0.5];
    \draw[thin, blue] (3,5.5)--(3,4.5);
    \draw[thin, blue] (3, 4.5) arc [start angle=0, end angle=-90, radius = 0.5];
    \draw[thin, blue] (2.5, 4)--(1.5,4);
    \draw[thin, blue, <-] (1.5, 4) arc [start angle=270, end angle=45, radius = 0.5];
    
    \node at (0.75, 4.5) {$\alpha_4$};
    \node at ( 7.5, -2.75) {$\alpha_4$};
    \node at (12.25, 4.75) {$\alpha_4$};
    
    
    \draw[thin, violet] (3.5, 3.5)--(2.5, 3.5);
    \draw[thin, violet, ->] (2.5, 3.5) arc [start angle=90, end angle=270, radius = 1.5];
    \draw[thin, violet] (2.5, 0.5)--(6.25,0.5);
    \draw[thin, violet] (6.25, 0.5) arc [start angle=270, end angle=360, radius = 0.5];
    \draw[thin, violet] (6.75, 1)--(6.75, 5.25);
    
    \draw[thin, violet] (11.75, -4.5)--(11.75, -3);
    \draw[thin, violet] (11.75,-3) arc [start angle=0, end angle=180, radius = 1];
    \draw[thin, violet, ->] (9.75,-3)--(9.75,-4);
    \draw[thin, violet] (9.75,-4) arc [start angle=0, end angle=-90, radius = 0.5];
    \draw[thin, violet] (9.25,-4.5) arc [start angle=90, end angle=180, radius = 0.5];
    \draw[thin, violet] (8.75,-5)--(8.75,-6.25);
    \draw[thin, violet] (8.75,-6.25) arc [start angle=0, end angle=-90, radius = 0.25];
    
    \node at (0.75, 1.75) {$\beta_1$};
    \node at (10.75, -1.75) {$\beta_1$};
    
    \draw[thin, violet] (15.35, 3.5)--(15, 3.5);
    \draw[thin, violet] (15, 3.5) arc [start angle=90, end angle=180, radius = 0.5];
    \draw[thin, violet] (14.5,3)--(14.5,2);
    \draw[thin, violet] (14.5, 2) arc [start angle=180, end angle=270, radius = 0.5];
    \draw[thin, violet, <-] (15, 1.5)--(17.5, 1.5);
    \draw[thin, violet] (17.5, 1.5) arc [start angle=270, end angle=360, radius = 0.5];
    \draw[thin, violet] (18,2)--(18,5);
    \draw[thin, violet] (18,5) arc [start angle=0, end angle=90, radius = 0.5];
    \draw[thin, violet] (17.5, 5.5)--(17, 5.5);
    \draw[thin, violet] (17,5.5) arc [start angle=90, end angle=180, radius = 0.5];
    \draw[thin, violet] (16.5,5)--(16.5,4);
    \draw[thin, violet] (16.5, 4) arc [start angle=0, end angle=-90, radius = 0.5];
    \draw[thin, violet] (15.75, 3.5)--(16, 3.5);
    \draw[thin, violet, <-] (5.5, 3.5) arc [start angle=90, end angle=445, radius = 1];
    
    \node at (5.5, 1.25) {$\beta_2$};
    \node at (16.5, 1.25) {$\beta_2$};
    
    \draw[thin, violet] (15.5,5.5)--(16,5.5);
    \draw[thin, violet] (16,5.5) arc [start angle=270, end angle=360, radius = 0.5];
    \draw[thin, violet] (16.5,6)--(16.5,7);
    \draw[thin, violet] (16.5, 7) arc [start angle=0, end angle=90, radius = 0.5];
    \draw[thin, violet, <-] (16, 7.5)--(14.5,7.5);
    \draw[thin, violet] (14.5, 7.5) arc [start angle=90, end angle=180, radius = 0.5];
    \draw[thin, violet] (14, 7)--(14,6);
    \draw[thin, violet] (14, 6) arc [start angle=360, end angle=180, radius = 0.5];
    
    \draw[thin, violet] (3,6)--(3,7.5);
    \draw[thin, violet] (3, 7.5) arc [start angle=180, end angle=90, radius = 0.5];
    \draw[thin, violet] (3.5, 8)--(4,8);
    \draw[thin, violet] (4, 8) arc [start angle=90, end angle=0, radius = 0.5];
    \draw[thin, violet] (4.5, 7.5)--(4.5, 7);
    \draw[thin, violet, <-] (4.5, 7)--(4.5, 6);
    \draw[thin, violet] (4.5, 6) arc [start angle=180, end angle=270, radius = 0.5];
    \draw[thin, violet] (5,5.5)--(5.5,5.5);
    
    \node at (4.75, 6) {$\beta_3$};
    \node at (15.25, 7.75) {$\beta_3$};
    
    \draw[thin, violet, <-] (8, 4.5) arc [start angle=360, end angle=135, radius = 0.5];
    \draw[thin, violet] (8, 4.5) arc [start angle=0, end angle=130, radius = 0.5];
    
    \draw[thin, violet] (7.25, -6)--(7.25,-5);
    \draw[thin, violet] (7.25, -5) arc [start angle=0, end angle=90, radius = 0.5];
    \draw[thin, violet] (6.75, -4.5)--(5.5, -4.5);
    \draw[thin, violet] (5.5, -4.5) arc [start angle=90, end angle=180, radius = 0.5];
    \draw[thin, violet] (5,-5)--(5,-8.5);
    \draw[thin, violet] (5, -8.5) arc [start angle=180, end angle=270, radius = 0.5];
    \draw[thin, violet, ->] (5.5, -9)--(12.5,-9);
    \draw[thin, violet] (12.5, -9) arc [start angle=270, end angle=360, radius = 0.5];
    \draw[thin, violet] (13,-8.5)--(13,-5.5);
    \draw[thin, violet] (13, -5.5) arc [start angle=0, end angle=135, radius = 0.5];
    
    \draw[thin, violet] (12, -5.5) arc [start angle=180, end angle=135, radius = 0.5];
    \draw[thin, violet] (12, -5.5) arc [start angle=0, end angle=-90, radius = 0.5];
    \draw[thin, violet, ->] (11.5, -6)--(10,-6);
    \draw[thin, violet] (10,-6) arc [start angle=90, end angle=180, radius = 0.5];
    \draw[thin, violet] (9.5,-6.5)--(9.5, -7.5);
    \draw[thin, violet] (9.5, -7.5) arc [start angle=0, end angle=-90, radius = 0.5];
    \draw[thin, violet] (9,-8)--(8.5,-8);
    \draw[thin, violet] (8.5, -8) arc [start angle=270, end angle=180, radius = 0.5];
    \draw[thin, violet] (8,-7.5)--(8,-7);
    
    \draw[thin, violet] (14, 2.5) arc [start angle=0, end angle=135, radius = 0.5];
    \draw[thin, violet] (14, 2.5) arc [start angle=0, end angle=-90, radius = 0.5];
    \draw[thin, violet, ->] (13.5, 2)--(12.75,2);
    \draw[thin, violet] (12.75, 2) arc [start angle=270, end angle=180, radius = 0.5];
    \draw[thin, violet] (12.25, 2.5)--(12.25, 3.5);
    
    \node at (7.5, 5.25) {$\beta_4$};
    \node at (13.25, 1.75) {$\beta_4$};
    \node at (8.75, -9.25) {$\beta_4$};
    \node at (10.75, -6.25) {$\beta_4$};

    
    \node at (2.5, 5) {$b^+$};
    \node at (2, 5.25) {$a^-$};
    \node at (13, 5.5) {$a^+$};
    \node at (12.5, 5.75) {$b^-$};
    
    \node at (3.75, 4) {$d^+$};
    \node at (3.25, 4) {$c^-$};
    \node at (8.75, -4.375) {$c^+$};
    \node at (8.25, -4.375) {$d^-$};
    
    \node at (12.75, 3) {$e^-$};
    \node at (13, 3.5) {$f^+$};
    \node at (7.75, -7) {$f^-$};
    \node at (8, -6.5) {$e^+$};
    
    \node at (5.75, 4) {$h^+$};
    \node at (5.25, 4) {$g^-$};
    \node at (15.75, 5) {$g^+$};
    \node at (15.25, 5) {$h^-$};
    
    \node at (6.625, 5.125) {$l^+$};
    \node at (6.875, 5.5) {$i^-$};
    \node at (11.25, -3.75) {$l^-$};
    \node at (11, -4.25) {$i^+$};
    
    \node at (11, -6.75) {$m^-$};
    \node at (11.25, -7.25) {$n^+$};
    \node at (17.5, 4) {$m^+$};
    \node at (17.25, 4.5) {$n^-$};
    \end{tikzpicture}
    \caption{Realization of a translation surface with poles and even parity in $\mathcal H_4(2^6;\,-2^3)$. Here the monodromy pattern is equal to $k=3$. Unlike the notation used so far, here the edges $a^+,\dots,n^+$ are identified with $a^-,\dots,n^-$ respectively. Moreover, the indices of the curves $\alpha_i$ and $\beta_i$ are according to the Table \ref{tab:indeces2} below.}
    \label{fig:genus4evenspin2}
\end{figure}
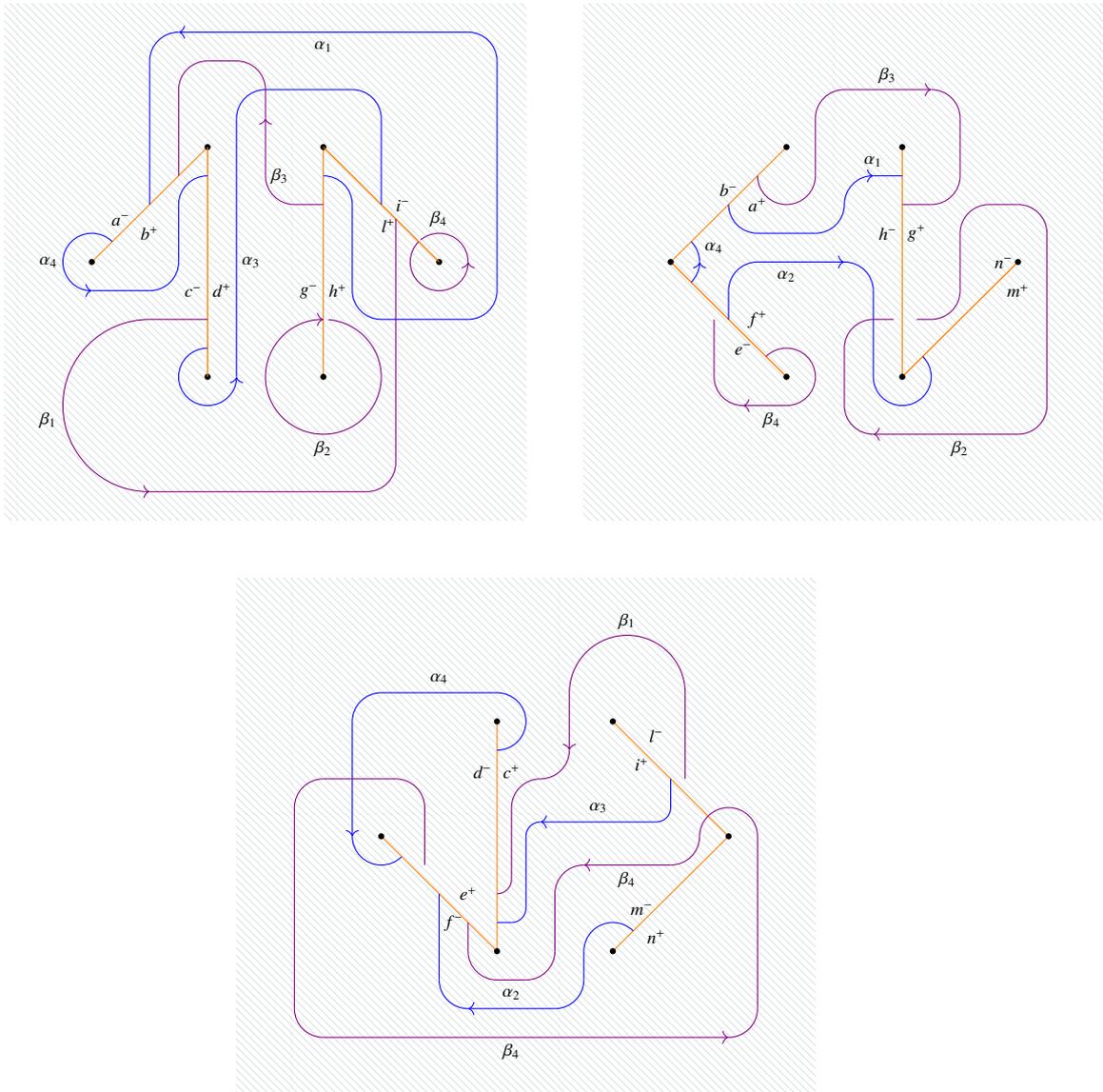
\begin{table}[!ht]
    \centering
    \begin{tabular}{lcc}
    \midrule
      $i$ & $\textnormal{Ind}(\alpha_i)$ & $\textnormal{Ind}(\beta_i)$\\ \midrule
      1 & 1 & 1\\ \midrule
      2 & 0 & 2\\ \midrule
      3 & 0 & 0\\ \midrule
      4 & 2 & 3\\ \midrule
    \end{tabular}
    \caption{Indices of curves in Figure \ref{fig:genus4evenspin2}.}
    \label{tab:indeces2}
\end{table}

\noindent In both cases, by using formula \eqref{eq:spinparity}, it is possible to see that both structures have even parity and belong to different connected components of $\mathcal{TR}_4$.

\appendix

\section{Proof strategy flowcharts}\label{appfdgo}

\noindent The proof of Theorem \ref{mainthm} is long and it involves a case-by-case discussion according to the diagrams below.

\subsection{Genus-one surfaces} Let $\chi\colon\shomolzon\longrightarrow \C$ be a representation. Suppose we want to realize $\chi$ in a connected component of the stratum $\mathcal{H}_1(m;-p_1,\dots,-p_n)$, namely as the holonomy of some translation surface with a single zero of maximal order. In Section \S\ref{genusonemero} we provide a way for realizing $\chi$ as the holonomy of some translation surface with poles in a prescribed stratum with prescribed rotation number. Table \ref{fig:flowdiagram} provides an outline of the strategy adopted. The multiple zero case is handled by Corollary \ref{multzeros}.

\begin{table}[!ht]
    \centering
    \begin{tikzpicture}[scale=1, node distance = 2cm]
    \node (reptype) [rb] {$\chi$ of trivial-ends type?};
    
    \node (allzres) [rb, right of=reptype, xshift=2cm ] {Section \S\ref{resnotzero}\\ All residues different from zero?};
    \node (2punct) [rb, right of=allzres, xshift=2cm ] {Only two punctures?};
    \node (a3punct) [rb, below of=2punct, xshift=0cm ] {Subsection \S\ref{apntr} \\ Is $\chi_n$ rational?};
    \node (2pntr) [sb, right of=2punct, xshift=2cm ] {Subsection \S\ref{2pntr}};
    \node (mpntr1) [sb, right of=a3punct, xshift=2cm, yshift=0cm ] {Paragraph \S\ref{chinonrat}};
    \node (mpntr2) [sb, below of=mpntr1, yshift=0.75cm ] {Paragraph \S\ref{chirat}};
    \node (opzr) [sb, below of=allzres, yshift=0cm ] {Subsection \S\ref{optr}};
    
    \node (order)  [rb, below of=reptype, yshift=-3cm] {All poles of order two?};
    \node (order2) [rb, right of=order, xshift=2cm] {Section \S\ref{genusoneordertwo} \\ Rotation number $1$?};
    \node (rot1)   [rb, right of=order2, xshift=2cm] {Subsection \S\ref{rot1} \\ Positive volume?};
    \node (rot1p)  [sb, right of=rot1, xshift=2cm, yshift=0cm] {Paragraph \S\ref{pk1}};
    \node (rot1n)  [sb, below of=rot1p, xshift=0cm, yshift=0.75cm] {Paragraph \S\ref{nk1}};
    \node (rot2)   [rb, below of=order2, yshift=-1cm] {Subsection \S\ref{rot2}\\ Rotation number $2$ \\ Positive volume?};
    \node (rot2p)  [rb, right of=rot2, xshift=2cm] {Even number of punctures?};
    \node (rot2n)  [rb, below of=rot2p, yshift=-1cm] {Even number of punctures?};
    \node (rot2pe) [sb, right of=rot2p, xshift=2cm, yshift=0cm] {Paragraph \S\ref{pvep}};
    \node (rot2po) [sb, below of=rot2pe, xshift=0cm, yshift=0.75cm] {Paragraph \S\ref{pvop}};
    \node (rot2ne) [sb, right of=rot2n, xshift=2cm, yshift=0cm] {Paragraph \S\ref{npvonp}};
    \node (rot2no) [sb, below of=rot2ne, xshift=0cm, yshift=0.75cm] {Paragraph \S\ref{nvep}};
    \node (orderho) [rb, below of=order, yshift=-6cm] {Poles of the same order?};
    \node (orderhogen) [sb, right of=orderho, xshift=2cm] {Section \S\ref{morezeros} \\ Corollary \ref{genpol}};
    \node (orderhop) [rb, below of=orderho] {Section \S\ref{genusoneorderhigherthantwo}\\ Positive volume?};
    \node (orderhop3) [rb, right of=orderhop, xshift=2cm] {More than two punctures?};
    \node (orderhop3y)  [sb, right of=orderhop3, xshift=2cm, yshift=0cm] {Paragraph \S\ref{pv2p}};
    \node (orderhop3n)  [sb, below of=orderhop3y, xshift=0cm, yshift=0.75cm] {Paragraph \S\ref{pvmp}};
    \node (orderhop3bis) [rb, below of=orderhop3, yshift=-1cm] {More than two punctures?};
    \node (orderhop3bissp) [rb, right of=orderhop3bis, xshift=2cm] {Single puncture?};
    \node (orderhop3bisspc)  [sb, right of=orderhop3bissp, xshift=2cm, yshift=0cm] {Paragraph \S\ref{npvop}};
    \node (orderhop3bisspno)  [sb, below of=orderhop3bisspc, yshift=0.75cm] {Paragraph \S\ref{npv2p}};
    \node (orderhop3bisy)  [sb, below of=orderhop3bisspno, yshift=0.75cm] {Paragraph \S\ref{npvmp}};

    \draw[->] (reptype) -- node[anchor=east] {yes} (order);
    \draw[->] (reptype) -- node[anchor=south] {no} (allzres);
    \draw[->] (allzres) -- node[anchor=south] {yes} (2punct);
    \draw[->] (2punct) -- node[anchor=south] {yes} (2pntr);
    \draw[->] (2punct) -- node[anchor=east] {no} (a3punct);
    \draw[->] (allzres) -- node[anchor=east] {no} (opzr);
    \draw[->] (a3punct) -- node[anchor=south] {no} (mpntr1);
    \draw[->] (a3punct) |- node[anchor=south] {\qquad \qquad \qquad yes} (mpntr2);
    \draw[->] (order) -- node[anchor=south] {yes} (order2);
    \draw[->] (order2) -- node[anchor=south] {yes} (rot1);
    \draw[->] (rot1) -- node[anchor=south] {yes} (rot1p);
    \draw[->] (rot1) |- node[anchor=south] {\qquad \qquad \qquad no} (rot1n);
    \draw[->] (order2) -- node[anchor=east] {no} (rot2);
    \draw[->] (rot2) -- node[anchor=south] {yes} (rot2p);
    \draw[->] (rot2) |- node[anchor=south] {\qquad \qquad \qquad no} (rot2n);
    
    \draw[->] (rot2p) -- node[anchor=south] {yes} (rot2pe);
    \draw[->] (rot2p) |- node[anchor=south] {\qquad \qquad \qquad no} (rot2po);
    \draw[->] (rot2n) -- node[anchor=south] {no} (rot2ne);
    \draw[->] (rot2n) |- node[anchor=south] {\qquad \qquad \qquad yes} (rot2no);
    \draw[->] (order) -- node[anchor=east]{no} (orderho);
    \draw[->] (orderho) -- node[anchor=south]{no} (orderhogen);
    \draw[->] (orderho) -- node[anchor=east]{yes} (orderhop);
    \draw[->] (orderhop) -- node[anchor=south]{yes} (orderhop3);
    \draw[->] (orderhop3) -- node[anchor=south] {no} (orderhop3y);
    \draw[->] (orderhop3) |- node[anchor=south] {\qquad \qquad \qquad yes} (orderhop3n);
    \draw[->] (orderhop) |- node[anchor=south] {\qquad \qquad \qquad no} (orderhop3bis);
    \draw[->] (orderhop3bis) -- node[anchor=south] {no} (orderhop3bissp);
    \draw[->] (orderhop3bissp) -- node[anchor=south] {yes} (orderhop3bisspc);
    \draw[->] (orderhop3bissp) |- node[anchor=south] {\qquad \qquad \qquad no} (orderhop3bisspno);
    \draw[->] (orderhop3bis) |- node[anchor=south] {\qquad \qquad \qquad yes} (orderhop3bisy);

    \end{tikzpicture}
    \caption{Flowchart for the proof of Theorem \ref{mainthm} for genus-one differentials with a single zero in Section \S\ref{genusonemero}.}  
    \label{fig:flowdiagram}
\end{table}

\subsection{Hyperelliptic translation surfaces} Let $\chi\colon\shomolzn\longrightarrow \C$ be a representation. We want to realize $\chi$ as the holonomy of some translation surfaces with poles and a hyperelliptic involution, see Definition \ref{hypdef}. Table \ref{fig:flowdiagramhyp} provides an outline of the strategy adopted.

\begin{table}[!ht]
    \centering
    \begin{tikzpicture}[scale=0.95, node distance = 2cm]
    \node (reptype) [rb] {$\chi$ non-trivial \\ trivial-ends type?};
    
    
    \node (onepunct) [rb, right of=reptype, xshift=3cm ] {Section \S\ref{ssec:hypzeres}\\ One punctured surface?};
    \node (1punct) [sb, right of=onepunct, xshift=3cm ] {Subsection \S\ref{sssec:hyponehop}};
    \node (2punct) [sb, right of=onepunct, xshift=3cm, yshift=-1.25cm ] {Subsection \S\ref{sssec:hyptwohop}};

    
    \node (simplepole) [rb, below of=reptype, yshift=-0.5cm] {There are two punctures\\ Simple poles?};
    \node (nosimplepoles) [sb, right of=simplepole, xshift=3cm, yshift=0cm ] {Section \S\ref{ssec:hypnotzeres}};
    \node (simplepoles) [rb, below of=simplepole, yshift=0cm] {Section \S\ref{ssec:hypnotzeressimp}\\ Is $\chi$ discrete?};
    \node (nodis) [rb, right of=simplepoles, xshift=3cm] {Subsection \S\ref{sssec:hypnotzeressp}\\ Is $\chi$ real-collinear?};
    \node (nodisnocoll) [sb, right of=nodis, xshift=3cm] {Paragraph \S\ref{par:hypnotcol}};
    \node (nodiscoll) [sb, right of=nodis, xshift=3cm, yshift=-1.25cm ] {Paragraph \S\ref{par:hypcol}};
    \node (dis) [rb, below of=simplepoles, yshift=-0.5cm ] {$\text{Im}(\chi)\subset\Z\oplus i\,\Z$\\ Rank one?};
    \node (distwo) [sb, right of=dis, xshift=3cm ] {Subsection \S\ref{sssec:disranktwo}};
    \node (disone) [rb, below of=dis, yshift=0cm ] {Subsection \S\ref{sssec:disrankone}\\ Is the residue greater than $2g$?};
    \node (disonegreater) [sb, right of=disone, xshift=3cm ] {Paragraph \S\ref{par:easy}};
    \node (disone2) [rb, below of=disone, yshift=0cm ] {Is the residue in\\ $g < r \le 2g$?};
    \node (disonelower) [sb, right of=disone2, xshift=3cm ] {Paragraph \S\ref{par:hard}};
    \node (notrealizable) [sb, below of=disone2] {Not realizable\\ \cite[Proposition 12.10]{CFG}};

    \draw[->] (reptype) -- node[anchor=east] {no} (simplepole);
    \draw[->] (reptype) -- node[anchor=south] {yes} (onepunct);
    \draw[->] (onepunct) -- node[anchor=south] {yes} (1punct);
    \draw[->] (onepunct) |- node[anchor=south] {\qquad \qquad \qquad \qquad no} (2punct);
    \draw[->] (simplepole) -- node[anchor=south] {no} (nosimplepoles);
    \draw[->] (simplepole) -- node[anchor=east] {yes} (simplepoles);
    \draw[->] (simplepoles) -- node[anchor=south] {no} (nodis);
    \draw[->] (nodis) -- node[anchor=south] {no} (nodisnocoll);
    \draw[->] (nodis) |- node[anchor=south] {\qquad \qquad \qquad \qquad yes}(nodiscoll);

    \draw[->] (simplepoles) -- node[anchor=east] {yes} (dis);
    \draw[->] (dis) -- node[anchor=south] {no} (distwo);
    \draw[->] (dis) -- node[anchor=east] {yes} (disone);
    \draw[->] (disone) -- node[anchor=south] {yes} (disonegreater);
    \draw[->] (disone) -- node[anchor=east] {no} (disone2);
    \draw[->] (disone2) -- node[anchor=south] {yes} (disonelower);
    \draw[->] (disone2) -- node[anchor=east] {no} (notrealizable);

    \end{tikzpicture}
    \caption{Flowchart for the proof of Theorem \ref{mainthm} for genus-$g$ differentials in Section \S\ref{sec:hgchyp}.}  
    \label{fig:flowdiagramhyp}
\end{table}

\bibliographystyle{amsalpha}
\bibliography{pccs}

\end{document}